\documentclass[11pt]{amsart}
\usepackage[utf8]{inputenc}

\usepackage{overpic, amssymb, times, graphicx, tikz-cd}
\usepackage[all]{xy}
\usepackage[backref=page]{hyperref}
\usepackage{verbatim}
\hypersetup{
    colorlinks=true,
    linkcolor=blue,
    filecolor=blue,      
    urlcolor=blue,
    citecolor=blue,
}
\DeclareMathOperator{\area}{area}
\DeclareMathOperator{\Span}{span}

\DeclareMathOperator{\im}{im}
\DeclareMathOperator{\dom}{dom}
\DeclareMathOperator{\Diff}{Diff}

\DeclareMathOperator{\ind}{ind}

\DeclareMathOperator{\Symp}{Symp}
\DeclareMathOperator{\Int}{int}

\DeclareMathOperator{\Id}{Id}
\DeclareMathOperator{\Maslov}{M}
\DeclareMathOperator{\CZ}{CZ}
\DeclareMathOperator{\bcs}{bcs}
\DeclareMathOperator{\Flow}{Flow}

\DeclareMathOperator{\SL}{SL}

\DeclareMathOperator{\PD}{PD}
\DeclareMathOperator{\tr}{tr}
\DeclareMathOperator{\rot}{rot}
\DeclareMathOperator{\lk}{lk}
\DeclareMathOperator{\tb}{tb}
\DeclareMathOperator{\wl}{wl}
\DeclareMathOperator{\cycword}{cw}
\DeclareMathOperator{\word}{w}
\DeclareMathOperator{\cross}{cr}
\DeclareMathOperator{\sgn}{sgn}
\DeclareMathOperator{\Diag}{Diag}
\DeclareMathOperator{\const}{const}
\DeclareMathOperator{\Ret}{Ret}

\DeclareMathOperator{\SFT}{SFT}

\newcommand{\Proj}{\mathbb{P}}
\newcommand{\C}{\mathbb{C}}
\newcommand{\R}{\mathbb{R}}
\newcommand{\Z}{\mathbb{Z}}

\newcommand{\Q}{\mathbb{Q}}
\newcommand{\disk}{\mathbb{D}}
\newcommand{\grad}{\nabla}
\newcommand{\bigO}{\mathcal{O}}
\newcommand{\region}{\mathcal{R}}
\newcommand{\Igrading}{\mathcal{I}_{\Lambda}}
\newcommand{\action}{\mathcal{A}}
\newcommand{\energy}{\mathcal{E}}
\newcommand{\rect}{\mathcal{D}}
\newcommand{\delbar}{\overline{\partial}}
\newcommand{\Cinfty}{\mathcal{C}^{\infty}}
\newcommand{\Rthree}{(\R^{3},\xi_{std})}
\newcommand{\Lie}{\mathcal{L}}
\newcommand{\Sthree}{(S^{3},\xi_{std})}

\newcommand{\Circle}{S^{1}}

\newcommand{\LambdaZero}{\Lambda^{0}}
\newcommand{\LambdaPlus}{\Lambda^{+}}
\newcommand{\LambdaMinus}{\Lambda^{-}}
\newcommand{\LambdaPM}{\Lambda^{\pm}}
\newcommand{\SurgL}{\R^{3}_{\LambdaPM}}
\newcommand{\SurgLClosed}{S^{3}_{\LambdaPM}}
\newcommand{\SurgXi}{\xi_{\LambdaPM}}
\newcommand{\SurgLxi}{(\SurgL, \SurgXi)}
\newcommand{\SurgLxiClosed}{(\SurgLClosed, \SurgXi)}
\newcommand{\SurgLPrime}{\R^{3}_{\Lambda}}
\newcommand{\SurgXiPrime}{\xi_{\Lambda}}
\newcommand{\SurgLxiPrime}{(\SurgLPrime, \SurgXiPrime)}
\newcommand{\pxy}{\pi_{xy}}

\newcommand{\SLtwoR}{\SL(2, \R)}

\newcommand{\half}{\frac{1}{2}}
\newcommand{\partialCH}{\partial_{CH}}
\newcommand{\partialLCH}{\partial_{LCH}}
\newcommand{\partialRSFT}{\partial_{RSFT}}
\newcommand{\partialSFT}{\partial_{SFT}}

\newcommand{\be}{\begin{enumerate}}
\newcommand{\ee}{\end{enumerate}}

\newcommand{\Mxi}{(M,\xi)}
\newcommand{\MxiOT}{(M_{OT}, \xi_{OT})}

\newtheorem{thm}{Theorem}[section]

\newtheorem{ex}[thm]{Example}
\newtheorem{assump}[thm]{Assumptions}

\newtheorem{prop}[thm]{Proposition}
\newtheorem{properties}[thm]{Properties}

\newtheorem{defn}[thm]{Definition}

\newtheorem{lemma}[thm]{Lemma}
\newtheorem{cor}[thm]{Corollary}

\newtheorem{q}[thm]{Question}
\newtheorem{rmk}[thm]{Remark}

\newtheorem{warn}[thm]{Warning}

\title{Combinatorial Reeb dynamics on punctured contact $3$-manifolds}
\author{Russell Avdek}
\thanks{The author is partly supported by the grant KAW 2016.0198 from the Knut and Alice Wallenberg Foundation}
\email{russell.avdek@math.uu.se}
\date{\today}

\begin{document}

\begin{abstract}
Let $\LambdaPM = \LambdaPlus \cup \LambdaMinus \subset \Rthree$ be a contact surgery diagram determining a closed, connected contact $3$-manifold $\SurgLxiClosed$ and an open contact manifold $\SurgLxi$. Following \cite{BEE:LegendrianSurgery, Ekholm:SurgeryCurves}, we demonstrate how $\LambdaPM$ determines a family $\alpha_{\epsilon}$ of contact forms for $\SurgLxi$ whose closed Reeb orbits are in one-to-one correspondence with cyclic words of composable Reeb chords on $\LambdaPM$. We compute the homology classes and integral Conley-Zehnder indices of these orbits diagrammatically and develop algebraic tools for studying holomorphic curves in surgery cobordisms between the $\SurgLxi$.

These new techniques are used to describe the first known examples of closed, tight contact manifolds with vanishing contact homology: They are contact $\frac{1}{k}$ surgeries along the right-handed, $\tb=1$ trefoil for $k > 0$, which are known to have non-zero Heegaard-Floer contact classes by \cite{LS:TightI}.
\end{abstract}
\maketitle

\setcounter{tocdepth}{1}
\tableofcontents
\newpage

\section{Introduction}\label{Sec:Introduction}

The main objects of interest in this paper are \emph{contact $3$-manifolds} and their \emph{Legendrian submanifolds}. A \emph{contact form} on an oriented $3$-manifold $M$ is a $1$-form $\alpha \in \Omega^{1}(M)$ for which $\alpha \wedge d\alpha > 0$ with respect to the orientation of $M$. A contact $3$-manifold is a pair $\Mxi$ consisting of an oriented $3$-manifold $M$ together with an oriented $2$-dimensional distribution $\xi \subset TM$ which is the kernel of a contact form $\alpha$ satisfying $d\alpha|_{\xi} > 0$ with respect to the orientation on $\xi$. We say that \emph{$\alpha$ is a contact form for $\Mxi$}. A Legendrian submanifold of $\Mxi$ is a link which is tangent to $\xi$. We'll typically denote Legendrian submanifolds by $\Lambda$ or $\LambdaZero$.

Given a contact $1$-form $\alpha$ for some $\Mxi$ its \emph{Reeb vector field}, $R$, is determined by the equations
\begin{equation*}
    \alpha(R) = 1,\quad d\alpha(R, \ast) = 0.
\end{equation*}
For the purposes of studying invariants of $\Mxi$ and its Legendrian submanifolds defined by counting holomorphic curves \cite{EGH:SFTIntro, EtnyreNg:LCHSurvey, Hutchings:ECHNotes, BiasedSH} we are interested in finding contact forms on a given $\Mxi$ for which $R$ is easy to analyze. Specifically we want to have visibility into the closed orbits of $R$ as well the \emph{chords} of Legendrians $\LambdaZero \subset \Mxi$; that is, the orbits of $R$ parameterized by compact intervals which both begin and end on $\LambdaZero$.

Let $\Rthree$ denote the standard contact structure on Euclidean $3$-space where
\begin{equation*}
    \xi_{std} = \ker(\alpha_{std}),\ \alpha_{std} = dz - y dx
\end{equation*}
and let $\Sthree$ denote the standard contact structure on the unit $3$-sphere $S^{3}$ where
\begin{equation*}
\xi_{std} = \ker\left(\sum_{1}^{2} x_{i}dy_{i} - y_{i}dx_{i}\right).
\end{equation*}
A \emph{contact surgery diagram} is a Legendrian link
\begin{equation*}
\LambdaPM = \LambdaPlus \cup \LambdaMinus \subset \Rthree.
\end{equation*}
Performing contact $\pm 1$ surgery on the components of the $\Lambda^{\pm}$ as defined in \cite{DG:Surgery} produces a contact $3$-manifold which we will denote by $\SurgLxi$. By considering $\Rthree$ as being contained in $\Sthree$ we can view the surgery diagram $\LambdaPM$ as determining a closed contact $3$-manifold $\SurgLxiClosed$, with $\SurgLxi$ obtained by removing a point from $\SurgLxiClosed$. As proved by Ding and Geiges in \cite{DG:Surgery} -- see also \cite{Avdek:ContactSurgery} -- every closed, connected contact $3$-manifold $\Mxi$ can be described as $\SurgLxiClosed$ for some choice of $\LambdaPM$.

For the remainder of this introduction we assume basic familiar with contact surgery, Weinstein handle attachment, and symplectic field theory ($\SFT$). Further background and references will be provided in Section \ref{Sec:Prerecs}.

\subsection{Combinatorial Reeb dynamics on punctured contact $3$-manifolds}

The primary purpose of this article is to describe a family of particularly well-behaved contact forms $\alpha_{\epsilon}$ for $\SurgLxi$ which are determined by the surgery diagram $\LambdaPM$. Our intention is to extend the analysis of Reeb dynamics appearing in work of Bourgeois, Ekholm, and Eliashberg \cite{BEE:LegendrianSurgery, Ekholm:SurgeryCurves} to allow for contact $+1$ surgeries. In particular, the following theorem states that their ``chords-to-orbits correspondence'' is applicable to any closed contact $3$-manifold.\footnote{Contact $-1$ surgery -- also known as Legendrian surgery -- describes how the convex boundaries of Liouville domains are modified by critical-index Weinstein handle attachment.}

\begin{thm}\label{Thm:Main}
Let $\LambdaPM$ be a contact surgery diagram presented in the front projection, where each component is equipped with an orientation. Possibly after a Legendrian isotopy of $\LambdaPM$ which preserves the front projection up to isotopy there is
\be
\item a constant $\epsilon_{0}$,
\item a neighborhood $N_{\epsilon_{0}}$ of $\LambdaPM$ in $\R^{3}$, and 
\item a family of contact forms $\alpha_{\epsilon}$ with Reeb vector fields $R_{\epsilon}$ parameterized by $\epsilon < \epsilon_{0}$ on $\SurgLxi$
\ee
such that the following conditions hold:
\be
\item For any $\epsilon < \epsilon_{0}$, performing contact surgery along a neighborhood $N_{\epsilon} \subset N_{\epsilon_{0}}$ produces $\SurgLxi$ equipped with the contact form $\alpha_{\epsilon}$.
\item $\alpha_{\epsilon} = \alpha_{std}$ on the complement of $N_{\epsilon}$.
\item For any $\epsilon < \epsilon_{0}$ there is a one-to-one correspondence between cyclic words of composable $\partial_{z}$ chords of $\LambdaPM$ and closed orbits of $R_{\epsilon}$ (Theorem \ref{Thm:ChordOrbitCorrespondence}).
\item For a given cyclic word of chords $w$, there exists $\epsilon_{w} < \epsilon_{0}$ such that the orbits of $R_{\epsilon}$ corresponding to $w$ are hyperbolic for $\epsilon < \epsilon_{w}$ (Theorem \ref{Thm:Mod2CZ}).
\item There is pair of sections $(X, Y)$ of $\SurgLxi$ determined by $\LambdaPM$ and its orientation, providing a symplectic trivialization of the restriction of $(\SurgXi, d\alpha_{\epsilon})$ to all closed orbits of $R_{\epsilon}$. The zero-locus $X^{-1}(0) = Y^{-1}(0)$ is a link contained in $(\R^{3} \setminus N_{\epsilon}) \subset \SurgL$ whose connected components are given by transverse push-offs of the components of $\LambdaPM$ with non-zero rotation number (Theorem \ref{Thm:FramingSummary}).
\item The integral Conley-Zehnder indices $\CZ_{X, Y}$ (Theorem \ref{Thm:IntegralCZ}) and homology classes (Theorem \ref{Thm:H1}) of the closed orbits of $R_{\epsilon}$ can be computed combinatorially from the surgery diagram.
\ee
\end{thm}

By ``computed combinatorially'', we mean computed via extensions of methods typically used to set up chain complexes for the Legendrian contact homology ($LCH$) \cite{EtnyreNg:LCHSurvey} or Legendrian rational symplectic field theory ($LRSFT$) \cite{Ng:RSFT} of $\LambdaPM$. Analogous results are stated for chords of Legendrian links $\LambdaZero \subset \SurgLxi$ throughout the paper, providing a ``chord-to-chord'' correspondence with diagrammatically computable Maslov indices. The content of Theorem \ref{Thm:Main} is sufficient to compute some algebraic invariants of tight contact structures on the lens space $L(2, 1)$ and $\Circle \times S^{2}$ as shown in Section \ref{Sec:CHUnknot}.

The dynamics analysis of Theorem \ref{Thm:Main} can be supplemented with a direct limit argument as in \cite[Section 4]{EkholmNg} to obtain a description of the Reeb dynamics on the closed contact manifolds $\SurgLxiClosed$ associated to a contact surgery diagram, which introduces a pair of embedded elliptic orbits.\footnote{See, for example, \cite[Section 4.1]{Bourgeois:ContactIntro} and \cite[Example 1.8]{Hutchings:ECHNotes}.} We will not pursue analysis of closed contact manifolds in this paper as the open manifolds $\SurgLxi$ have particularly friendly geometries which we'll leverage in applications.

\subsection{Constrained topology of holomorphic curves and applications}

The secondary purpose of this article is to develop tools for studying holomorphic curves in symplectizations of the $\SurgLxi$ and in surgery cobordisms between them. Our intention in to make ``hat versions'' of holomorphic curve invariants of $\SurgLxi$ -- as defined by Colin, Ghiggini, Honda, and Hutchings in \cite[Section 7.1]{CGHH:Sutures} -- more computationally accessible. Theorem \ref{Thm:Main} already provides us with rather complete descriptions of the chain complexes underlying such invariants.\footnote{There is some subtlety for $\widehat{ECH}$: In order to compute relative $ECH$ indices, the links underlying collections of simple Reeb orbits should be known, whereas we will describe the homotopy classes of closed Reeb orbits. Such link embeddings can be computed as solutions to matrix arithmetic problems described in Section \ref{Sec:Embeddings}.} In particular, we'll be interested in the hat version of contact homology $(CH)$:
\begin{equation*}
\widehat{CH}\SurgLxiClosed = CH\SurgLxi.
\end{equation*}
Hat versions of other holomorphic curve invariants of $\SurgLxiClosed$ such as embedded contact homology ($\widehat{ECH}$) and the $SFT$ algebra ($\widehat{SFT}$) are defined analogously.\footnote{We use $SFT$ to denote the $SFT$ algebra, while $\SFT$ -- without italics -- refers to Eliashberg, Givental, and Hofer's framework for defining holomorphic curve invariants of contact and symplectic manifolds of \cite{EGH:SFTIntro}.}

We demonstrate the utility of our tools in two applications: First we provide a (slightly) new proof of the vanishing of $CH$ of overtwisted contact manifolds (Eliashberg and Yau \cite{Yau:VanishingCH}) using surgery-theoretic methods (Section \ref{Sec:OTSurgery}). Second, we prove the following theorem (Section \ref{Sec:Trefoil}):

\begin{figure}[h]\label{Fig:TrefoilFront}
	\begin{overpic}[scale=.7]{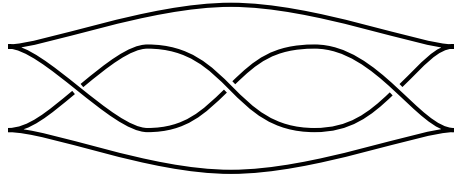}
	\end{overpic}
\caption{The Legendrian trefoil of Theorem \ref{Thm:Trefoil} shown in the front projection.}
\end{figure}

\begin{thm}\label{Thm:Trefoil}
If $\LambdaMinus = \emptyset$ and $\LambdaPlus$ has a component which is a right-handed trefoil, then 
\begin{equation*}
CH\SurgLxiClosed = \widehat{CH}\SurgLxiClosed = 0
\end{equation*}
In particular, contact $\frac{1}{k}$ surgery on the right-handed, $\tb = 1$ trefoil for $k > 0$ produces a closed, tight contact manifold $\SurgLxiClosed$ with vanishing contact homology. See figure \ref{Fig:TrefoilFront}.
\end{thm}

The development of our tools (Section \ref{Sec:FoliationsAndQuivers}) starts with a variation of the construction of transverse knot filtrations of holomorphic curve invariants from \cite[Section 7.2]{CGHH:Sutures}: Lines in $\R^{3}$ directed by $\partial_{z}$ over points $(x, y) \in \R^{2} \setminus \pi_{x,y}(N_{\epsilon})$ determine infinite energy holomorphic planes $\C_{x, y}$ in $\R\times \SurgL$. The $\C_{x, y}$ form a holomorphic foliation whose existence constrains the topology of curves á la proofs of uniqueness-of-symplectic-manifold theorems \cite{Eliash:Filling, Gromov:JCurves, GZ:FourBall, Hind:Filling, McDuff:RationalRuled, McDuff:Filling, Wendl:Foliations}. Counting intersections $\C_{x, y} \cdot U$ of these planes with finite energy curves $U$ asymptotic to collections $\gamma^{\pm}$ of closed $R_{\epsilon}$ orbits yields locally constant, $\Z_{\geq 0}$-valued functions on $\SFT$ moduli spaces -- topological invariants determined by the relative homology classes 
\begin{equation*}
[\pi_{\SurgL}\circ U] \in H_{2}(\SurgL, \gamma^{\pm})
\end{equation*}
of holomorphic curves. Surgery cobordisms may be similarly considered when equipped with special almost complex structures described in Section \ref{Sec:NStandard}. By tracking these intersections, we can
\be
\item show that certain disks appearing in Ng's combinatorially-defined Legendrian $RSFT$ \cite{Ng:RSFT} determine rigid holomorphic planes in $\R \times \SurgL$ (Section \ref{Sec:PlaneBubbling}). This follows a Lagrangian-boundary version of Hofer's bubbling argument \cite{Hofer:OTWeinstein} in which case the $\C_{x, y}\cdot U$ completely dictate the ways in which certain families of holomorphic disks can degenerate into multi-level $\SFT$ buildings.
\item equip the $\widehat{CH}$ chain complexes with a new grading, denoted $\Igrading$, which depends on the surgery diagram (Section \ref{Sec:IntersectionGrading}). Variants of this grading may similarly be applied to any holomorphic curve invariant of $\SurgLxi$.
\ee

In the proof of Theorem \ref{Thm:Trefoil}, we show that $+1$ surgery on the $\tb=1$ trefoil provides a $\CZ_{X, Y} = 2$ closed orbit $\gamma$ of $R_{\epsilon}$ with $\partial_{CH}\gamma = \pm 1 \in \Q$. Computations of Conley-Zehnder indices, homology classes, and $\Igrading$ shows that any $\ind = 1$ rational holomorphic curves positively asymptotic to $\gamma$ must be a plane, and that such planes may be counted using our bubbling argument.

Theorem \ref{Thm:Trefoil} provides the first examples of closed, tight contact manifolds with $CH = 0$.\footnote{Due to $CH$ functoriality under Liouville cobordism, Honda's tight contact manifold which becomes overtwisted after contact $-1$ surgery \cite{Honda:OTSurgery} already provides an example of a contact manifold with convex boundary whose \emph{sutured contact homology} \cite{CGHH:Sutures} is zero.} The tightness of $\frac{1}{k}$ surgeries on the $\tb = 1$ trefoil is provided by computations of Heegaard-Floer ($HF$) contact classes \cite{HKM:ContactClass, OS:ContactClass} by Lisca and Stipsicz  in \cite[Section 3]{LS:TightI}. As the $HF$ contact class contains the same information as the $ECH$ contact class \cite{CGH:HFequalsECH, KLT:HFSW}, and both $ECH$ and $\SFT$ count holomorphic curves of arbitrary topological type -- in particular, arbitrary genus -- it would be interesting to know if there is some $\SFT$ invariant of this contact manifold which is non-vanishing. Broadening the scope of this inquiry, we ask the following:

\begin{q}
For $3$-dimensional contact manifolds, does $CH\Mxi \neq 0$ imply that the $HF=ECH$ contact class of $\Mxi$ is non-zero? Do there exist tight contact manifolds of dimension greater than three with vanishing contact homology?
\end{q}

We note that using the algebraic formalism of \cite{EGH:SFTIntro}, the vanishing of contact homology is equivalent to the vanishing of $SFT$ according to \cite{AlgebraicallyOvertwisted}.

\subsection{Outline of this paper}

In Section \ref{Sec:Prerecs} we outline notation and background information which will be used throughout the rest of the paper. Section \ref{Sec:ChordNotation} is also primarily concerned with notation, associating algebraic data to chords of Legendrian links in $\Rthree$ which will be used to package invariants of chords and closed orbits in the surgered contact manifolds $\SurgLxi$.

Sections \ref{Sec:ModelGeometry} through \ref{Sec:Homology} carry out the computational details of Theorem \ref{Thm:Main} and analogous results for chords of Legendrian links $\LambdaZero \subset \SurgLxi$. In Section \ref{Sec:SurgeryCobordisms} we describe handle-attachment cobordisms between the $\SurgLxi$ associated to surgeries along their Legendrian knots. The construction of these cobordisms -- slight modifications of \cite{Ekholm:SurgeryCurves, Weinstein:Handles} -- provides us with model geometry facilitating analysis of holomorphic curves.

Section \ref{Sec:FoliationsAndQuivers} describes holomorphic curves in symplectizations of and surgery cobordisms between the $\SurgLxi$. The algebraic tools described in that section are prerequisite for the applications appearing in Section \ref{Sec:Applications}, culminating in the proof of Theorem \ref{Thm:Trefoil}.

Content pertaining to Legendrian links $\LambdaZero \subset \SurgLxi$ may be skipped by readers only interested in the applications of Section \ref{Sec:Applications}. This material is included to provide a complete picture of relative $\SFT$ chain complexes in anticipation of their use in future applications.

\subsection*{Acknowledgments}

We send our gratitude to Erkao Bao, Guillaume Dreyer, Tobias Ekholm, Ko Honda, Yang Huang, and Vera Vértesi for interesting conversations. Special thanks goes to Fabio Gironella and András Stipsicz, as well as Georgios Dimitroglou Rizell for their interest in this project and invitations to give talks in their seminars. Finally, we thank our referees for their thoughtful and detailed commentary.

\section{Prerequisites}\label{Sec:Prerecs}

\subsection{General notation}

Throughout this paper $\delta_{\ast, \ast}$ -- with a double subscript -- will denote the Kronecker delta and $\lfloor \ast \rfloor$ will be the floor function $\R \rightarrow \Z$. A \emph{collection} will be a set in which elements are allowed to have non-trivial multiplicity. We use set notation for collections. For example $\{1, 1, 2\}$ is a collection with $\{ 1, 1, 2\} \setminus \{1 \} = \{ 1, 2\}$ and $\{ 1, 1, 2\} \cup \{ 2, 3 \} = \{ 1, 1, 2, 2, 3\}$. We'll use often use collections and ordered collection to organize chords and orbits as they may appear in $CH$, $ECH$, $LCH$, etc.

Unless otherwise specified, we use $I$ to denote a connected $1$ manifold and for a positive number $\epsilon$, we write $I_{\epsilon}=[-\epsilon, \epsilon]$. For $a > 0$, the circle $\R/a\Z$ will be denoted $\Circle_{a}$ and without the subscript, $\Circle = \Circle_{1}$. The unit disk of dimension $n$ and radius $C$ centered about $x \in \R^{n}$ will be denoted $\disk_{C}^{n}(x)$. We'll typically use the simplified notation $\disk^{n}=\disk_{1}^{n}(0)$ and $\disk$ for $\disk^{2}$. The complex projective space will be written $\Proj^{n}$.

For a closed manifold $M$, $\widehat{M}$ will denote the open manifold obtained from $M$ by removing a point or closed disk. When $\Mxi$ is a closed contact manifold, $\widehat{\Mxi}$ will denote $\Mxi$ with a point or standard Darboux disk removed. We say that $\widehat{\Mxi}$ is a \emph{punctured contact manifold}.\footnote{In \cite{CGHH:Sutures}, the notation $M(1)$ is used for what we call $\widehat{M}$.}

For a space $M$, we denote homology and cohomology groups as $H_{\ast}(M)$ and $H^{\ast}(M)$, respectively. Integral coefficients will be assumed unless otherwise explicitly stated. When $M$ is a closed manifold, $\PD$ will be used to denote the Poincaré duality isomorphism in either direction $H_{i} \leftrightarrow H^{\dim(M) - i}$. Abusing notation, we also use $\PD$ to denote the associated isomorphisms for punctured manifolds $\widehat{M}$ in degrees $i \neq 0$. By a \emph{$\Q$ homology sphere}, we mean a closed or punctured $3$-manifold with finite $H_{1}$ (implying that $H_{2} = 0$ by the universal coefficients theorem, cf. \cite[Corollary 3.3]{Hatcher:AlgebraicTopology}).

For a vector bundle $E$ over a manifold $M$, the space of $\Cinfty$ sections will denoted $\Gamma(E)$. The space of nowhere zero sections -- which may be empty -- will be denoted $\Gamma_{\neq 0}(E)$. Provided that $E$ has finite rank $n$ and trivializations $(V_{i})$ and $(W_{i})$ of $E$ over some set $U\subset M$, transformations of the form $\sum_{i, j} a_{i, j}W_{i}\otimes V^{*}_{j}$ can be written as matrices with respect to which we say that $(V_{i})$ is the \emph{incoming basis} and $(W_{i})$ is the \emph{outgoing basis}. In such situations, provided $a_{1},\dots, a_{n}\in \Cinfty(U)$, $\Diag(a_{1}, \dots, a_{n})$ will be the diagonal matrix with $a_{1}$ in the top-left corner and $J_{0}$ will denote standard complex multiplication where applicable. The Euler class of a finite dimensional bundle will be written $e(E)$ and Chern classes will be written $c_{k}(E)$ when the bundle is equipped with a (homotopy class of) complex structure. We will be predominantly interested in the case $E = \xi$ for a $3$-dimensional contact manifold $\Mxi$ in which case the Euler and first Chern classes coincide: $e(\xi) = c_{1}(\xi)$.

\subsection{Vector fields and almost complex structures}

In this section we review vector fields and almost complex structures typically encountered in symplectic and contact geometry, primarily for the purpose of establishing conventions which often vary in the literature. We'll use Option 1 of \cite{Wendl:Signs}. See that article or \cite[Remark 3.3]{MS:SymplecticIntro} for further discussion.\footnote{Regarding work we'll be frequently referencing: Our signs for symplectic forms on cotangent bundles will be opposite that of \cite{Ekholm:SurgeryCurves} and our signs for Hamiltonian vector fields are opposite that of \cite{BH:ContactDefinition, BH:Cylindrical, CGHH:Sutures}.}

Let $(W, \beta)$ be a $2n$-dimensional exact symplectic manifold. That is, $W$ is an oriented $2n$-manifold on which $d\beta$ is symplectic. We call such $\beta$ a \emph{Liouville form} or \emph{symplectic potential}. If $H\in\Cinfty(W)$ is a smooth function with values in $\R$ or $\Circle$, the associated \emph{Hamiltonian vector field}, denoted $X_{H}$, is the unique solution to the equation
\begin{equation*}
    d\beta(\ast, X_{H}) = dH.
\end{equation*}
Clearly $H$ is constant along the flow-lines of $X_{H}$ and $X_{H}$ depends only on $d\beta$ (rather than $\beta$). If $J$ is an almost complex structure for which $g_{J}$, defined by 
\begin{equation*}
    g_{J}(u, v) = d\beta(u, J v),\ u, v \in T_{p}\Sigma
\end{equation*}
is a $J$-invariant Riemannian metric, then
\begin{equation*}
    X_{H} = J\grad H
\end{equation*}
where $\grad H$ is the gradient of $H$ with respect to $g_{J}$ solving $g_{J}(\grad H, \ast) = dH$. We say that such $J$ is a \emph{compatible} almost-complex structure.

The \emph{Liouville vector field}, denoted $X_{\beta}$, on $W$ is the unique solution to the equation 
\begin{equation*}
    d\beta(X_{\beta}, \ast) = \beta.
\end{equation*}
If $W$ is compact and $X_{\beta}$ points outward along the boundary of $W$, we say that the pair $(W, \beta)$ is a \emph{Liouville domain}. Given a function $H\in \Cinfty(W)$, the $1$-form $\beta_{H} = \beta + dH$ is also a primitive for $d\beta$ such that
\begin{equation*}
    X_{\beta_{H}} = X_{\beta} + X_{H}.
\end{equation*}
By our choice of convention, Hamiltonian and Liouville vector fields interact with $d\beta$ as follows:
\begin{equation*}
    \beta(X_{H}) = d\beta(X_{\beta}, X_{H}) = dH(X_{\beta}).
\end{equation*}

Given a contact manifold $\Mxi$ equipped with a contact form $\alpha$, action of the chords and closed orbits of its Reeb vector field may be computed as
\begin{equation*}
    \action(\gamma) = \int_{\gamma}\alpha.
\end{equation*}

\subsection{Contact and symplectic manifolds}

Here we review some contact and symplectic manifolds which will appear throughout this article.

\subsubsection{Cotangent bundles}

Our convention for Liouville forms on the cotangent bundle $T^{*}L$ of a smooth manifold $L$ will be to use the form $(T^{*}L, \lambda_{can})$ with $\lambda_{can} = p_{i} dq_{i}$ in a local coordinate system $(q_{i})$ on $L$. Provided such coordinates on $L$, we use $(p_{i}, q_{i})$ as local coordinates on $T^{*}L$ so that $d\lambda_{can}$ is symplectic with respect to the induced orientation.

\subsubsection{Contactizations}\label{Sec:Contactizations}

Provided an exact symplectic manifold $(W, \beta)$, we have a contact form $dz + \beta$ on $I \times W$. We will refer to both the contact manifold $(I \times W, \ker(dz + \beta))$ and the pair $(\R \times W, dz+\beta)$ as the \emph{contactization of $(W, \beta)$}.

It's easy to see that deformations of an exact symplectic manifold give rise to contactomorphic contactizations. For if $H \in \Cinfty(W, \R)$ then the contactization of $(W, \beta + dH)$ is equivalent to the contactization of $(W, \beta)$ by the transformation
\begin{equation*}
(t, w) \mapsto (t + H, w).
\end{equation*}

We'll further analyze geometry of contactizations in Sections \ref{Sec:ContactizationForms} and \ref{Sec:SteinProduct}. The quintessential example of a contactization is the \emph{$1$-jet space} of a closed manifold, which is the contactization of its cotangent bundle.

\subsubsection{Symplectizations}

Provided $\Mxi$ and $\alpha$ as above, $(\R\times M, e^{t}\alpha)$ is an exact symplectic manifold called the $\emph{symplectization}$ of the pair $(M, \alpha)$. By considering diffeomorphisms of the form $(t, x) \mapsto (t + f(x), x)$ on $\R\times M$ for $f \in \Cinfty(M, (0, \infty))$ it is clear that the symplectization is independent of the choice of $\alpha$ for $\xi$, up to symplectomorphism.

For any constant $C$, we will likewise refer to $([C, \infty)\times M, e^{t}\alpha)$ as the \emph{positive half-infinite symplectization} and $((-\infty, C]\times M, e^{t},\alpha)$ as the \emph{negative half-infinite symplectization} of the pair $(M, \alpha)$. For constants $C < C'$, we will call $([C, C']\times M, e^{t}\alpha)$ a \emph{finite symplectization} of the pair $(M, \alpha)$.

Here we can compute
\begin{equation*}
    X_{\beta} = \partial_{t},\ X_{t} = e^{-t}R.
\end{equation*}
Hence there is a one-to-one correspondence between periodic orbits of $R$ and $1$-periodic orbits of $X_{t}$ by associating to each $\gamma$ in $M$ the loop $(\log(\action(\gamma)), \gamma)$ in the symplectization.

\subsubsection{Liouville cobordisms between closed and punctured contact manifolds}

Here we review some standard vocabulary regarding symplectic cobordisms, modified to deal with punctured contact manifolds. What are sometimes called ``strong symplectic cobordisms'' we will simply refer to as \emph{symplectic cobordisms} for notational simplicity. What are sometimes called ``exact symplectic cobordisms'' we will refer to as \emph{Liouville cobordisms}. Our reasoning is that there exist symplectic cobordisms which are exact symplectic manifolds, but which are not ``exact symplectic cobordisms'' -- cf. \cite[Section 2.4]{Wendl:NonExact}. See that paper or \cite[Chapter 12]{OS:SurgeryBook} for a review of various notions of fillings and cobordisms with emphasis on low dimensions. We will only be concerned with Liouville cobordisms in this paper.

Let $\Mxi$ be a closed contact manifold of dimension $2n + 1$ and $p \in M$ a point. We say that a contact form $\alpha$ for $\xi$ defined on $M \setminus \{ p \}$ is \emph{standard at infinity} if there exists a ball $B_{p}$ about $p\in M$, a positive constant $C$, and a diffeomorphism
\begin{equation*}
\Phi: \big( B_{p} \setminus \{ p\} \big) \rightarrow \big( \R^{2n + 1} \setminus \disk^{2n+1}_{C}(0) \big)
\end{equation*}
such that $\Phi^{*}(dz - y_{i}dx_{i}) = \alpha$ and $|\Phi(\gamma(t))| \rightarrow \infty$ for paths $\gamma(t)$ in $B_{p} \setminus \{ p\}$ tending towards $p$.

A \emph{Liouville cobordism between contact manifolds $(M^{+}, \xi^{+})$ and $(M^{-}, \xi^{-})$} is a compact exact symplectic manifold $(W, \lambda)$ for which 
\be
\item $\partial W = M^{+} - M^{-}$,
\item the Liouville vector field $X_{\lambda}$ points into $W$ along $M^{-}$ and out of $W$ along $M^{+}$, and
\item $\lambda|_{TM^{\pm}}$ is a contact form for $\xi^{\pm}$.
\ee
We call $M^{+}$ the \emph{convex boundary of $(W, \lambda)$} and $M^{-}$ the \emph{concave boundary of $(W, \lambda)$}. We may think of a Liouville domain as cobordism whose concave boundary is the empty set.

A \emph{Liouville cobordism between punctured contact manifolds $\widehat{(M^{+}, \xi^{+})}$ and $\widehat{(M^{-}, \xi^{-})}$} is defined analogously as in the case where the $(M^{\pm}, \xi^{\pm})$ are closed. However we require that there exists a region
\begin{equation*}
I_{C} \times (\R^{2n + 1} \setminus \disk^{2n+1}_{C}(0)) \subset W, \quad \{ \pm C \} \times (\R^{2n + 1} \setminus \disk^{2n+1}_{C}(0)) \subset M^{\pm}
\end{equation*}
along which $\lambda = e^{t}(dz - y_{i}dx_{i})$ and such that the $t = \pm C$ slices provide standard at infinity neighborhoods of the punctures of the $M^{\pm}$.

We won't bother to specify that a Liouville cobordism is between closed or punctured contact manifolds, as it should be clear from the context. In either case, we may define the \emph{completion} of a Liouville cobordism to be the non-compact exact symplectic manifold obtained from a Liouville cobordism by appending a positive half-infinite symplectization to a collar of its convex boundary and a negative half-infinite symplectization to a collar of its concave boundary. We denote the completion of such a cobordism $(W, \lambda)$ as $(\overline{W}, \overline{\lambda})$.

\subsection{Remarks on $\SLtwoR$}

We briefly review some properties of $\SLtwoR$ which will be useful for analyzing Reeb dynamics on contact $3$-manifolds. By definition, $\SLtwoR$ coincides with $\Symp(2, \R)$ -- the space of matrices preserving the standard symplectic form $dx\wedge dy$.

An element $A \in \SLtwoR$ has characteristic polynomial
\begin{equation}\label{Eq:DetTr}
    \det(A - \lambda\Id) = \lambda^{2} - \tr(A)\lambda + 1
\end{equation}
using which, eigenvalues of $A$ can be found using the quadratic formula. The \emph{non-degenerate} elements are those for which $1$ is not an eigenvalue. A non-degenerate element $A$ falls into one of two classes:
\be
\item $A$ is called \emph{elliptic} if its eigenvalues lie on the unit circle or equivalently, $|\tr(A)| < 2$.
\item $A$ is called \emph{hyperbolic} if its eigenvalues are elements of $\R$ or equivalently $|\tr(A)| > 2$.
\ee

Hyperbolic elements are further classified as \emph{positive (resp. negative) hyperbolic} if the eigenvalues are positive (resp. negative) real numbers. The classification of $A \in \SLtwoR$ as elliptic, positive hyperbolic, or negative hyperbolic depends only on the conjugacy class of $A$.

\subsection{Conley-Zehnder indices of Reeb orbits in contact $3$-manifolds}\label{Sec:CZOverview}

Throughout the remaining subsections covering Reeb dynamics and contact homology, we follow the expositions \cite{Bourgeois:ContactIntro} of Bourgeois (which covers all dimensions) and \cite[Section 3.2]{Hutchings:ECHNotes} of Hutchings (which specifically focuses on the $3$-manifolds). Let $\gamma$ be a closed Reeb orbit of a contact manifold $\Mxi$ equipped with a contact form $\alpha$ for $\xi$ whose Reeb vector field will be denoted $R$. We assume $\gamma$ is embedded, comes with a parameterization $\gamma(t)$, and write $\gamma^{k}$ for its $k$-fold iterate with $k > 0$. 

As the Reeb flow preserves $\xi$, the Poincaré return map for time $t=\action(\gamma)$ sends $\xi|_{\gamma(0)}$ to itself and -- provided a symplectic basis of $(\xi|_{\gamma(0)}, d\alpha)$ -- determines a matrix $\Ret_{\gamma}\in \SLtwoR$. The orbit $\gamma$ will be called non-degenerate, elliptic, positive (negative) hyperbolic if the matrix $\Ret_{\gamma}$ has the associated property. We say that the contact form $\alpha$ is \emph{non-degenerate} if all of its Reeb orbits are non-degenerate.\footnote{In practice, one is typically interested in studying sequences of contact forms $\alpha_{n}$ with ``nice'' limiting behavior such that there exists a sequence $C_{n} \rightarrow \infty$ so that the orbits of $\alpha_{n}$ of action $\leq C_{n}$ are non-degenerate. See, for example \cite{BH:ContactDefinition, BH:Cylindrical, Bourgeois:Thesis, Ekholm:SurgeryCurves}. We take a similar approach in this article.}

\begin{rmk}
Having a non-degenerate contact form for which all closed orbits are hyperbolic -- as is the case with the contact forms $\alpha_{\epsilon}$ of Theorem \ref{Thm:Main} -- is generally desirable as branched covers of trivial cylinders over elliptic orbits can have negative index \cite[Section 1]{HT:GluingI}. Likewise, in $ECH$ chain complexes only simple covers of hyperbolic orbits are considered, whereas multiple covers of elliptic orbits cannot be avoided \cite{Hutchings:ECHNotes}. See also \cite{BH:Cylindrical, Rooney:ECH} where analysis of holomorphic maps is simplified by considering only hyperbolic orbits.
\end{rmk}

Suppose that $\gamma$ is a non-degenerate orbit equipped with a framing $s \in \Gamma_{\neq 0}(\xi|_{\gamma})$. By extending $s$ to a symplectic trivialization of the normal bundle $(\xi|_{\gamma}, d\alpha)$ to $\gamma$, we can write the restriction of the linearized flow to $\xi|_{\gamma}$ as a path $\phi = \phi(t)$ in $\SLtwoR$. Then we define the \emph{Conley-Zehnder index of the orbit $\gamma$ with framing $s$}, denoted $\CZ_{s}(\gamma)$, to be the Conley-Zehnder index $\CZ(\phi)$ of the path $\phi$.

If $\gamma$ is hyperbolic, $\phi$ rotates the eigenspaces of $\Ret_{\gamma}$ by an angle $\pi n$  for some $n\in \Z$ in which case
\begin{equation*}
    \CZ_{s}(\gamma^{k}) = k n.
\end{equation*}
Negative hyperbolic orbits have $n$ odd and positive hyperbolic orbits have $n$ even. If $\gamma$ is elliptic, $\phi$ rotates the eigenspaces of $\Ret_{\gamma}$ by some angle $\theta \in \R \setminus 2\pi\Z$ in which case the Conley-Zehnder index is computed
\begin{equation*}
    \CZ_{s}(\gamma^{k}) = 2\left\lfloor \frac{k\theta}{2\pi} \right\rfloor + 1.
\end{equation*}

Note that $\CZ_{s}$ depends only on the isotopy class of the framing $s$. If we write $s + n$ for a framing whose isotopy class is given by twisting $s$ by $n$ meridians, then
\begin{equation}\label{Eq:MeridianTwist}
    \CZ_{s + n}(\gamma^{k}) = \CZ_{s}(\gamma^{k}) - 2nk.
\end{equation}

An orbit $\gamma^{k}$ is \emph{bad} if the parity of its Conley-Zehnder index disagrees with that of the underlying embedded orbit $\gamma$. Orbits which are not bad are \emph{good}. Hence (when $\dim(M) = 3$) the only bad orbits are even covers of negative hyperbolic orbits. See Remarks 1.9.2 and 1.9.6 of \cite{EGH:SFTIntro}.

Note that as $\CZ_{s}(\gamma) \bmod_{2}$ is independent of $s$, so is the property that an orbit is good or bad. We write $\CZ_{2}(\gamma) \in \Z/2\Z$ for the index modulo-$2$ which satisfies
\begin{equation}\label{Eq:CZTwo}
    \sgn \circ \det (\Ret_{\gamma} - \Id) = (-1)^{\CZ_{2} + 1}.
\end{equation}

The following method of computing the Conley-Zehnder index of a path $\phi(t), t \in [0, 1]$ of symplectic matrices is due to Robbin-Salamon \cite{RS:Index}. For a path $\phi: [0, 1] \rightarrow \SLtwoR$ a point $t \in [0, 1]$ is \emph{crossing} if $1$ is an eigenvalue of $\phi(t)$. Writing
\begin{equation}\label{Eq:SympPathDerivative}
    \frac{\partial \phi}{\partial t}(t) = J_{0} S(t) \phi(t)
\end{equation}
for symmetric matrices $S(t)$, we say that a crossing $t$ is \emph{regular} if the quadratic form $\Gamma(t)$ defined as the restriction of $S(t)$ to $\ker(\phi(t) - \Id)$ is non-degenerate. For a path $\phi$ with only regular crossings, we can compute $\CZ(\phi)$ as
\begin{equation}\label{Eq:RSCZ}
\CZ(\phi) = \half \sgn(\Gamma(0)) + \sum_{t > 0\ \text{crossing}}\sgn(\Gamma(t)).
\end{equation}
Also of utility for computation is the \emph{loop property} of $\CZ$ which states that given $k \in \Z$ and a non-degenerate path $\phi$, the path $\tilde{\phi}(t) = e^{i 2\pi k t}\phi(t)$ has
\begin{equation}\label{Eq:CZLoop}
 \CZ(\tilde{\phi}) = 2k + \CZ(\phi).
\end{equation}

\subsection{Holomorphic curves in symplectizations and the index formula}

Now suppose that $\alpha$ is a non-degenerate contact form for some contact $3$-manifold $\Mxi$ and that $J$ is an almost-complex structure which is \emph{adapted to the symplectization} $(\R\times M, e^{t}\alpha)$. That is:
\be
\item $J$ is compatible with $d(e^{t}\alpha)$,
\item it is $t$-invariant and preserves $\xi$, and
\item $J\partial_{t} = R$.
\ee

Let $\gamma^{+} = \{ \gamma^{+}_{1},\dots, \gamma^{+}_{m^{+}} \}$ and $\gamma^{-} = \{ \gamma^{-}_{1},\dots, \gamma^{-}_{m^{-}} \}$ be collections of Reeb orbits with $\gamma^{+}$ non-empty and let $(\Sigma, j)$ be a Riemann surface with marked points $\{ p^{+}_{1},\dots, p^{+}_{m^{+}}, p^{-}_{1}, \dots, p^{-}_{m^{-}} \}$. We write $\Sigma'$ for $\Sigma$ with its marked points removed. We say that $(t, U): \Sigma' \rightarrow \R\times M$ is \emph{holomorphic} if
\begin{equation*}
    \delbar(t, U) = \half \bigr( T(t, U) + J\circ T(t, U)\circ j \bigr)
\end{equation*}
vanishes. If we wish to specify $J$ and $j$, we'll say that the map is \emph{$(J, j)$ holomorphic}. This is equivalent to the conditions
\begin{equation}\label{Eq:DelbarBreakdown}
dt = U^{*}\alpha \circ j,\quad J\pi_{\alpha}\circ TU = \pi_{\alpha}\circ TU \circ j
\end{equation}
where $\pi_{\alpha}: TM \rightarrow \xi$ is the projection $V \mapsto V - \alpha(V)R$. We provide a few simple examples.

\begin{ex}[Trivial strips, planes, and cylinders]\label{Ex:TrivialObjects}
Provided a map $\gamma: I \rightarrow M$ parameterizing a Reeb trajectory for a connected $1$-manifold $I$, $\R \times \im(\gamma) \subset \R \times M$ is an immersion with $J$-complex tangent planes. Some examples of particular interest:
\be
\item If $I$ is compact with non-empty boundary parameterizing a chord of $R$ with endpoints on a Legendrian submanifold $\Lambda$, we'll call $\R \times \im(\gamma)$ a \emph{trivial strip}.
\item If $I = \R$ and the map $\gamma$ is an embedding, we'll say that $\R \times \im(\gamma)$ is a \emph{trivial plane}.
\item If $I = \Circle_{a}$ parameterizing a Reeb orbit of action $a$, then we'll say that $\R \times \im(\gamma)$ is a \emph{trivial cylinder}.
\ee
\end{ex}

Given a holomorphic map $(t, U): (\Sigma, j) \rightarrow (\R\times M, J)$, we say that the puncture $p_{i}^{+}$ is \emph{positively-asymptotic} to the orbit $\gamma_{i}^{+}$ if there exists a neighborhood $[C, \infty) \times \Circle$ of $p_{i}^{+}$ in $\Sigma$ with coordinates $r, \theta$ for which $j$ is the standard cylindrical complex structure so that $t(r, \theta) \rightarrow \infty$ and $U(r, \theta)$ tends to a parameterization of $\gamma^{+}_{i}$ as $r\rightarrow \infty$. Likewise, we say that the puncture $p_{i}^{-}$ is \emph{negatively-asymptotic} to the orbit $\gamma_{i}^{-}$ if $t(r, \theta) \rightarrow -\infty$ and $U(r, \theta)$ tends to a parameterization of $-\gamma^{+}_{i}$ as $r\rightarrow \infty$. Allowing $j$ and the location of the marked points to vary and then modding out by reparameterization in the domain, we write $\mathcal{M}_{(t, U)}$ for the \emph{moduli space of holomorphic maps} asymptotic to the $\gamma^{\pm}$ containing the map $(t, U)$.

The \emph{index} of a holomorphic map as above is defined by the formula
\begin{equation}\label{Eq:DelbarIndex}
\begin{gathered}
\ind((t, U)) = \CZ_{s}(\gamma^{+}) - \CZ_{s}(\gamma^{-}) - \chi(\Sigma') + 2c_{s}(U) \in \Z\\
\CZ_{s}(\gamma^{\pm}) = \sum_{i=1}^{m^{\pm}}\CZ_{s}(\gamma^{\pm}_{i}).
\end{gathered}
\end{equation}
The \emph{relative first Chern class} $c_{s}(U)$ is the signed count of zeros of $U^{*}\xi$ over $\Sigma'$ using a section which coincides with $s$ near the punctures. Note that $\ind$ is independent of $s$. In ideal geometric settings, $\mathcal{M}_{(t, U)}$ is a manifold near the point $(t, U)$ of dimension $\ind((t, U))$.

\begin{rmk}
Here we are disregarding asymptotic markers for orbits which are required for a rigorous functional-analytic setup for moduli spaces and curve counts. We refer to \cite{BH:ContactDefinition, Pardon:Contact} for details.
\end{rmk}

The \emph{energy} of a holomorphic curve is defined
\begin{equation*}
    \energy(t, U) = \int_{\Sigma'}d\alpha = \sum_{1}^{m^{+}}\action(\gamma_{i}^{+}) - \sum_{1}^{m^{-}}\action(\gamma_{i}^{-}).
\end{equation*}
The energy is non-negative and is zero if and only if $(t, U)$ is a branched cover of a trivial cylinder. Energies of curves will be presumed finite unless otherwise explicitly stated.

\subsection{Contact homology and its variants}\label{Sec:SFTOverview}

We now give a brief overview of contact homology and symplectic field theory. As in previous subsections, we focus specifically on the case of contact $3$-manifolds.

For each closed Reeb orbit $\gamma$ with framing $s$, we define its degree as $|\gamma|_{s} = \CZ_{s}(\gamma) - 1 \in \Z$. This degree modulo $2$ will be denoted $|\gamma|$. We write $CC(\alpha)$ for the supercommutative algebra with unit $1$ generated by the good Reeb orbits of $\alpha$ over $\Q$. Here supercommutativity means $\gamma_{1}\gamma_{2} = (-1)^{|\gamma_{1}||\gamma_{2}|}\gamma_{2}\gamma_{1}$. We note that $CC(\alpha)$ has two canonical gradings:
\be
\item The \emph{degree grading} given by $|\gamma_{1}\cdots\gamma_{n}| := \sum_{1}^{n}|\gamma_{i}| \in \Z/2\Z$.
\item The \emph{$H_{1}$ grading} given by $[\gamma_{1}\cdots\gamma_{n}] := \sum_{1}^{n}[\gamma_{i}] \in H_{1}(M)$.
\ee
For $i \in \Z/2\Z$ and $h \in H_{1}(M)$, we will use the notation $CC_{i, h}$ to denote the relevant graded $\Q$-subspaces. The contact homology differential 
\begin{equation*}
    \partialCH: CC_{i, h} \rightarrow CC_{i-1, h}
\end{equation*}
is defined by counting $\ind = 1$ (possibly perturbed) solutions to $\delbar(t, U) = 0$ with one positive puncture, any number of negative punctures, and genus $0$. For such curves $(t, U)$ positively asymptotic to some $\gamma^{+}$ and negatively asymptotic to $\gamma^{-}_{j}$ simultaneously framed with some choice of $s$, Equation \eqref{Eq:DelbarIndex} becomes
\begin{equation}\label{Eq:DelbarCHIndex}
    \ind((t, U)) = |\gamma^{+}|_{s} - \sum_{j}|\gamma^{-}_{j}|_{s} + 2c_{s}(U).
\end{equation}
Each such solution contributes a term to $\partial \gamma^{+}$ of the form $m(\gamma^{+}; \gamma^{-}_{i})\gamma_{1}^{-}\cdots \gamma_{n}^{-}$ with $m(\gamma^{+}; \gamma^{-}_{i})\in\Q$. If there are no negative punctures we get a term of the form $m(\gamma^{+})1$ and we set $\partial_{CH} 1 = 0$. The differential is then extended to products of orbits using the graded Leibniz rule 
\begin{equation*}
    \partialCH (\gamma_{1}\gamma_{2}) = (\partial \gamma_{1})\gamma_{2} + (-1)^{|\gamma_{1}|}\gamma_{1}(\partial \gamma_{2})
\end{equation*}
and to sums of products linearly. 

\begin{defn}
The resulting differential graded algebra $\ker(\partialCH) / \im(\partialCH)$ is defined to be the \emph{contact homology of $\Mxi$}, denoted $CH\Mxi$. As in the case of $CC(\alpha)$, $CH\Mxi$ also has degree and $H_{1}$ gradings. We write $CH_{i, h}\Mxi$ for subspace of $CH\Mxi$ with degree $i$ and $H_{1}$ grading $h$.
\end{defn}

This theory, first proposed in \cite{EGH:SFTIntro} by Eliashberg, Givental, and Hofer has been proven to be rigorously defined and independent of all choice involved by Bao-Honda in \cite{BH:ContactDefinition} and Pardon in \cite{Pardon:Contact}. We defer to these citations for the specifics of how the coefficients $m(\gamma^{+}; \gamma^{-}_{i}) \in \Q$ are computed and details around any required perturbations of $\delbar$. For the purposes of this paper, it suffices to know that for generic $J$ adapted to the symplectization of a contact manifold
\be
\item curves which are somewhere injective may be assumed regular,
\item regularity for these curves may be achieved by perturbations of $J$ in arbitrarily small neighborhoods of the closed orbits of $R$, and
\item that assuming such regularity, the moduli space of holomorphic planes positively asymptotic to a closed, embedded orbit will be a manifold (rather than an orbifold), so that such $\ind = 1$ planes can be counted over $\Z$.
\ee

Additional algebraic structures -- which require more sophisticated underlying chain complexes -- may be constructed as follows:
\be
\item By counting $\ind=1$, genus-$0$ holomorphic curves with arbitrary numbers of positive and negative punctures via a differential $\partialRSFT$, the \emph{rational $SFT$ algebra} ($RSFT$) may be defined.
\item By counting $\ind=1$ holomorphic curves with arbitrary genus and numbers of positive and negative punctures via a differential $\partialSFT$, the \emph{$SFT$ algebra} ($SFT$) may be defined.
\ee
See \cite{EGH:SFTIntro} for a more complete picture or the exposition \cite[Lecture 12]{Wendl:SFTNotes} for further details regarding these invariants.\footnote{At the time of writing, rigorous definitions of $RSFT$ and $SFT$ are under construction using a variety of frameworks. We refer to \cite{BH:ContactDefinition, MZ:RSFT, Pardon:Contact} for accounts of the current state of the development of SFT.} For other $RSFT$-like algebraic structures associated to counts of rational curves with multiple positive punctures see \cite{MZ:RSFT} which constructs such invariants and provides an overview of recent additions to the literature. 

The lecture notes \cite{Bourgeois:ContactIntro} and Section 1.8 of \cite{Pardon:Contact} also contains a rather exhaustive list of additional structures such as grading refinements and twisted coefficient systems for contact homology. We won't address such additional structures in this article, except in the following simple situations.

\begin{prop}[Canonical $\Z$ gradings]\label{Prop:CanonicalZGrading}
The $0\in H_{1}(M)$ part of $CH\Mxi$ is a subalgebra of $CH$. Suppose that $\Gamma_{\neq 0}(\xi)$ is non-empty (equivalently $c_{1}(\xi) = 0$).
\be
\item The $\Z$-valued degree gradings $|\cdot|_{s}$ on $CC(\alpha)$ determine $\Z$-valued gradings on $CH_{\ast, 0}\Mxi$ and are independent of the choice of $s\in \Gamma_{\neq 0}(\xi)$. 
\item Moreover if $H^{1}(M)=H_{2}(M)=0$, then the $\Z$-valued degree gradings $|\cdot|_{s}$ on $CC(\alpha)$ determine $\Z$-valued gradings $CH\Mxi$ which are independent of the choice of $s\in \Gamma_{\neq 0}(\xi)$.
\ee
\end{prop}

We get canonical $\Z$ gradings on $CH$ when we have a non-degenerate Reeb vector field with only homologically trivial Reeb orbits or when $M$ is a $3$-dimensional $\Q$ homology sphere.

\begin{proof}
The fact that $CH_{\ast, 0}\Mxi$ is a subalgebra of $CH\Mxi$ is clear from the fact that $\partial_{CH}$ preserves $H_{1}$ and that $CC_{\ast, 0}$ is closed under products.

Provided $s \in \Gamma_{\neq 0}(\xi)$, extend $s$ to a trivialization $\xi \rightarrow \C$. For our extension, we may use $Js$ for an almost complex structure $J$ on $\xi$. In this way, we see that any other non-vanishing section $s'$ defines a map $M\rightarrow \C^{*} \simeq \Circle$ and recall that homotopy classes of maps to $\Circle$ are in bijective correspondence with elements of $H^{1}$ \cite[Theorem 4.57]{Hatcher:AlgebraicTopology}. Write $[s' - s] \in H^{1}$ for the cohomological element provided by this correspondence. If $\gamma$ is a closed orbit of some $\alpha$ for $\Mxi$, then $[s'-s]\cdot [\gamma]\in \Z$ equals the difference in meridians between the framings of $\xi|_{\gamma}$ determined by $s$ and $s'$. Then $\CZ_{s}(\gamma) - \CZ_{s'}(\gamma)$ will be determined by this framing difference according to Equation \eqref{Eq:MeridianTwist}.

If $[\gamma]=0\in H_{1}$, then the above tells us $\CZ_{s}(\gamma) = \CZ_{s'}(\gamma)$, so that the gradings $|\gamma|_{s}$ on $CC_{\ast, 0}$ are independent of choice of non-vanishing $s$. If $H^{1}(M) = 0$ then $s'$ is necessarily homotopic to $s$, so that all of the gradings $|\cdot|_{s}$ are equivalent on $CC_{\ast, \ast}$. As $s$ is non-vanishing, the $c_{s}$ term in Equation \eqref{Eq:DelbarCHIndex} is always $0$, meaning that $\partial$ always lowers degree $|\cdot|_{s}$ by exactly $1$ and so the $\Z$-valued degree gradings on $CC$ determines a $\Z$ grading on homology.

To complete the proof, we must show that the $\Z$ grading is independent of choices used to compute $CH$. Proofs of invariance of $CH$ (cf. as they appear in \cite{BH:ContactDefinition, Pardon:Contact}) are obtained by considering the symplectization of $(M, \alpha)$ -- for some $\alpha$ -- equipped with almost complex structures which are adapted to $\alpha$ at the negative end $(-\infty, -C]\times M$ of the symplectization and adapted to $H\alpha$ at the positive end $[C, \infty) \times M$ for some $C > 0$ and $H \in \Cinfty(M, (0, \infty))$. In such a scenario, $T (\R \times M)$ can be split as the direct sum $\Span_{\R}(\partial_{t}, J\partial_{t})\oplus \xi$, and we can extend $s$ over $\R \times M$ in the obvious way to frame Reeb orbits at both ends of $\R \times M$. The isomorphisms between the contact homologies of the ends of the cobordism is defined by counting $\ind = 0$ holomorphic curves in $\R \times M$, which by the index formula of Equation \eqref{Eq:DelbarCHIndex} must preserve the $\Z$ grading.
\end{proof}

The variant of contact homology which will be of the most interest to us is the \emph{hat version}, denoted $\widehat{CH}\Mxi$ and defined in \cite{CGHH:Sutures}. To define this theory for $\Mxi$, we can equip $\widehat{M}$ with a standard-at-infinity $\alpha$ for $\widehat{\xi} = \xi|_{\widehat{M}}$, choose an appropriately convex $J$ on $\widehat{\xi}$, and compute $CH$ as above. We describe such $J$ for the $\SurgLxi$ in Section \ref{Sec:NStandard}.\footnote{In \cite{CGHH:Sutures}, less restrictive conditions are placed on $\alpha$ and $J$ to define $\widehat{CH}$ within the framework of the more general sutured contact homology. We choose more restrictive conditions so as to simplify our exposition and avoid general discussion of sutured contact manifolds and their completions as well as to simplify $J$-convexity arguments.}

The following theorem summarizes some properties of $\widehat{CH}$ laid out in the introduction of \cite{CGHH:Sutures} (coupled with some well-known results):

\begin{thm}\label{Thm:CHHatOverview}
The invariant $\widehat{CH}\Mxi$ satisfies the following properties:
\be 
\item For the standard contact $3$-sphere $\Sthree$, $\widehat{CH}\Sthree = \Q1$.
\item If $\Mxi$ is overtwisted, then $\widehat{CH}\Mxi = 0$.
\item For a contact-connected sum $(M_{1},\xi_{1})\# (M_{2}, \xi_{2})$
\begin{equation*}
\widehat{CH}((M_{1},\xi_{1})\# (M_{2}, \xi_{2})) \simeq \widehat{CH}(M_{1},\xi_{1}) \otimes \widehat{CH}(M_{2}, \xi_{2}).   
\end{equation*}
\item The inclusion $\widehat{\Mxi} \rightarrow \Mxi$ induces an algebra homomorphism
\begin{equation*}
\widehat{CH}\Mxi \rightarrow CH\Mxi.
\end{equation*}
Consequently $CH\Mxi \neq 0$ implies $\widehat{CH}\Mxi\neq 0$.
\item A Liouville cobordism $(W, \lambda)$ with convex boundary $\widehat{(M^{+}, \xi^{+})}$ and concave boundary $\widehat{(M^{-}, \xi^{-})}$ determines an algebra homomorphism
\begin{equation*}
\Phi_{(W, \lambda)}: \widehat{CH}(M^{+}, \xi^{+}) \rightarrow \widehat{CH}(M^{-}, \xi^{-}).
\end{equation*}
\item Consequently if $\Mxi$ admits a Liouville filling then both $CH\Mxi$ and $\widehat{CH}\Mxi$ are non-zero.
\ee
\end{thm}

The fifth item, which will refer to as \emph{Liouville functoriality}, does not explicitly appear in the literature for $\widehat{CH}$, though it follows from a simple combination of existing arguments and constructions. Liouville functoriality is established for closed contact manifolds in \cite{BH:ContactDefinition, Pardon:Contact}. To extend the results to punctured contact manifolds, one needs to establish $\SFT$ compactness \cite{SFTCompactness} of (possibly perturbed) moduli spaces of holomorphic curves positively asymptotic to closed orbits of a standard-at-infinity contact form on $\widehat{(M^{+}, \xi^{+})}$ and negatively asymptotic to closed orbits of a standard-at-infinity form on $\widehat{(M^{-}, \xi^{-})}$ in the completion of $(W, \lambda)$. To obtain compactness, we may restrict to almost complex structures $J$ which are $t$-invariant over the neighborhood of the puncture of the $M^{\pm}$ to ensure that sequences of curves cannot escape the completed cobordism through the horizontal boundary of the symplectization of the puncture. Our definition of Liouville cobordism between punctured contact manifolds and the $J$ of Section \ref{Sec:NStandard} ensure that these desired hypotheses are in place. Perturbations of $\delbar$ required to achieve transversality for the counting of curves and gluing of multi-level $\SFT$ buildings can be implemented in arbitrarily small neighborhoods of closed Reeb orbits \cite[Section 5]{BH:ContactDefinition}, so that such perturbations do not interfere with convexity. In this way, the compactness results of \cite[Section 5]{CGHH:Sutures} carry over without issue.

For the last item in Theorem \ref{Thm:CHHatOverview}, Liouville functoriality tells us that a Liouville filling of a closed contact manifold induces an algebra homomorphism from $CH\Mxi$ to $\Q$ (also known as an augmentation). Therefore $CH\Mxi \neq 0$ implying $\widehat{CH}\Mxi \neq 0$ by the fourth item.

\subsubsection{Relative contact homology}

We now briefly review $\SFT$ invariants of a Legendrian link $\Lambda \subset \Mxi$. For the case $\Mxi = \Rthree$, we recommend the exposition \cite{EtnyreNg:LCHSurvey} with the general theory laid out in Section 2.8 of \cite{EGH:SFTIntro}.

Provided a Legendrian link $\Lambda \subset \Mxi$ and a contact form $\alpha$ for $\xi$, consider the space of chords of $R$ which begin and end on $\Lambda$. A chord $r = r(t)$ is \emph{non-degenerate} if it satisfies the transversality condition
\begin{equation*}
    \Flow_{R}^{\action(r)}(T_{r(0)}\Lambda) \pitchfork T_{r(\action(r))}\Lambda \subset \xi_{r(\action(r))}.
\end{equation*}
We then say that the pair $(\alpha, \Lambda)$ is \emph{non-degenerate} if all chords for the pair and all closed orbits of $R$ are non-degenerate. Provided non-degeneracy, we consider a $\Z/2\Z$-graded super-commutative algebra $CC(\alpha, \Lambda)$ generated by the chords of $\Lambda$ and the good closed orbits of $R$.\footnote{We are skipping definition of the gradings of chords in the general case. See \cite{EtnyreNg:LCHSurvey,Ng:RSFT} for gradings in the case of Legendrians in $\Rthree$.} As in the non-relative case $CC(\alpha, \Lambda)$ comes with an additional homological grading, given by the relative homology classes of chords and orbits in $H_{\ast}(M, \Lambda)$.

We may then define a differential 
\begin{equation*}
\partialLCH: CC_{ i, h}(\alpha, \Lambda) \rightarrow CC_{ i-1, h}(\alpha, \Lambda) 
\end{equation*}
for $i \in \Z/2\Z$ and $h \in H_{1}(M, \Lambda)$ as follows: For a chord $r$, $\partial_{LCH}$ counts $\ind=1$ holomorphic disks in the symplectization of $\Mxi$ with
\be
\item A single boundary puncture positively asymptotic to $r$,
\item any number $m$ of boundary punctures negatively asympotic to chords $r^{-}_{i}$ of $\Lambda$,
\item $\partial \disk$ with its punctures removed mapped to the Lagrangian cylinder over $\Lambda$, and
\item $n$ interior punctures negatively asymptotic to closed orbits $\gamma^{-}_{j}$
\ee
Each such disk contributes a term of the form $m(r^{+};r^{-},\gamma^{-})r^{-}_{1}\cdots r^{-}_{m}\gamma^{-}_{1}\cdots\gamma^{-}_{n}$ to $\partial_{LCH}r^{+}$. For a closed orbit $\gamma^{+}$, the differential $\partial_{LCH}\gamma^{+}$ coincides with the contact homology differential of $\gamma^{+}$. The differential is then extended to products and sums of products using the Leibniz rule and linearity as in the case of non-relative contact homology.

\begin{defn}
The resulting differential graded algebra $\ker(\partialLCH)/\im(\partialLCH)$ is defined to be the \emph{Legendrian contact homology} of the triple $(M, \xi, \Lambda)$, denoted $LCH(M, \xi, \Lambda)$. As in the case of $CC_{\Lambda}$, $LCH$ has degree and relative $H_{1}$ gradings.
\end{defn}

The computation $\partial_{LCH}^{2} = 0$ and proof of invariance for links in $\Rthree$ -- in which case there are no closed Reeb orbits -- is carried out in \cite{EES:LegendriansInR2nPlus1} with a proof of the general case sketched in \cite{EGH:SFTIntro}. In the case $\Lambda \subset \Rthree$, a combinatorial version of $LCH$ -- originally due to Chekanov \cite{Chekanov:LCH} -- may be computed by counting immersions of disks into the $xy$-plane with boundary mapped to the Lagrangian projection of $\Lambda$, in which case $\partial_{LCH}^{2} = 0$ may be proved diagrammatically. Additional algebraic structures may derived from the triple $(M, \xi, \Lambda)$ by considering disks with multiple positive punctures as in \cite{Ekholm:Z2RSFT, Ng:RSFT}. Again, we point to \cite{EtnyreNg:LCHSurvey} for further references regarding proofs that the combinatorially and analytically defined invariants coincide for $\Rthree$ as well as extensions and generalizations of $LCH$ in both algebraic and geometric directions.

\subsection{Legendrian knots and links in $\Rthree$}\label{Sec:LegendrianOverview}

Legendrian knots and links will be denoted by $\Lambda$ with sub- and super-scripts. Throughout this article, we assume that each component of $\Lambda$ is equipped with a predetermined orientation. For a Legendrian link $\Lambda$ in a contact manifold $\Mxi$ with contact form $\alpha$ and Reeb vector field $R$
\begin{equation*}
\Flow^{\delta}_{R}(\Lambda)
\end{equation*}
for $\delta> 0$ arbitrarily small will be called the \emph{push-off of $\Lambda$}. The Legendrian isotopy class of the pair $(\Lambda, \Flow^{\delta}_{R}(\Lambda))$ is independent of $R$ and $\delta$. We write $\lambda_{\xi}$ for the Legendrian isotopy class of the push-off.

For a Legendrian link $\Lambda$ in $\Rthree$, the front- and Lagrangian projections will be denoted by $\pi_{xz}$ and $\pxy$ respectively. We will use front projections as our default starting point for analysis of $\Lambda$ from which we will transition to the Lagrangian projection -- see Section \ref{Sec:StandardNeighborhoods}. 

Assuming that $\Lambda$ has a single connected component, its \emph{classical invariants} are
\be
\item the Thurston-Bennequin number $\tb(\Lambda)$,
\item  the rotation number, $\rot(\Lambda)$, which depends on an orientation of $\Lambda$, and
\item the smooth topological knot underlying $\Lambda$.
\ee

In the Lagrangian projection, we may compute $\tb(\Lambda)$ as the writhe and $\rot(\Lambda)$ as winding number. Geometrically, the Thurston-Bennequin number is defined  as the linking number 
\begin{equation*}
\tb(\Lambda) = \lk(\Lambda, \lambda_{\xi})
\end{equation*}
whereas $\rot$ is defined as the degree of the Gauss map of $T\Lambda$ in $\xi_{std}$ with respect to a nowhere vanishing trivialization. If we replace $\R^{3}$ with any contact $\Q$ homology sphere, then $\tb$ is defined for null-homologous Legendrian knots and $\rot$ is defined for all Legendrian knots using the framings of Proposition \ref{Prop:CanonicalZGrading}. See also Definition \ref{Def:RotGeneral}.

Classical invariants of a Legendrian knot $\Lambda \subset \Rthree$ are constrained by the \emph{slice-Bennequin bound} of \cite{Rudolph}:
\begin{equation}\label{Eq:SliceBennequin}
\half(\tb(\Lambda) + | \rot(\Lambda)| + 1) \leq g_{s}(\Lambda) \leq g(\Lambda).
\end{equation}
Here $g_{s}(\Lambda)$ is the smooth slice genus of the topological knot underlying $\Lambda$ and $g(\Lambda)$ is the Seifert genus. See \cite[Section 3]{Etnyre:KnotNotes} for an overview of related results.

Let $\Lambda$ be a Legendrian knot in a contact manifold $\Mxi$. Take a cube $I_{\epsilon}^{3} \subset M, \epsilon > 0$ with coordinates $x, y, z$ such that
\begin{equation*}
\xi = \ker(\alpha_{std}), \quad \Lambda \cap I_{\epsilon}^{3} = \{ y = z = 0 \},
\end{equation*}
and $\partial_{x}$ orients $\Lambda$. Then $\Lambda$ is locally described by the left-hand side of Figure \ref{Fig:ZigZag}. The \emph{positive and negative stabilizations} of $\Lambda$, denoted $S_{+}(\Lambda)$ and $S_{-}(\Lambda)$, are defined as the Legendrian knots determined by modifying $\Lambda$ in the front projection of $I_{\epsilon}^{3}$ as described in the right-hand side of Figure \ref{Fig:ZigZag}. We say that a Legendrian knot $\Lambda$ is \emph{stabilized} if it a positive or negative stabilization of some $\Lambda' \subset \Mxi$.

\begin{figure}[h]
\begin{overpic}[scale=.4]{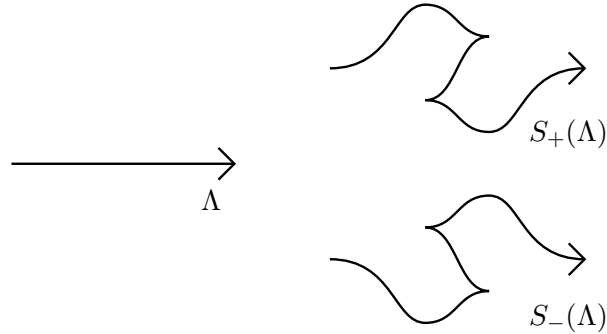}
\put(33, 20){$\Lambda$}
\put(90, 32){$S_{+}(\Lambda)$}
\put(90, 0){$S_{-}(\Lambda)$}
\end{overpic}
\caption{Positive and negative stabilizations of $\Lambda$ described in the front projection as in \cite[Figure 19]{Etnyre:KnotNotes}.}
\label{Fig:ZigZag}
\end{figure}

\subsection{Contact surgery}

Contact surgery -- first defined in \cite{DG:Surgery} -- provides a way of performing Dehn surgery on a Legendrian link $\LambdaPM$ so that the surgered manifold carries a contact structure uniquely determined by $\LambdaPM$ and $\Mxi$. We recommend Ozbagci and Stipsicz's \cite{OS:SurgeryBook} as a general reference.

We take the coefficients of the components of the sublinks $\LambdaPlus$ (resp. $\LambdaMinus$) to be $+1$ (resp. $-1$). Intuitively speaking, contact $-1$ ($+1$) surgery removes a neighborhood of a Legendrian knot of the form $I_{\epsilon} \times I_{\epsilon} \times \Circle$ -- the first coordinate being directed by $\partial_{z}$ -- and then glues it back in using a positive (negative) Dehn twist along $\{\epsilon\}\times I_{\epsilon} \times \Circle$. The construction may be formalized using the gluing theory of convex surfaces. A rigorous account of the construction will be carried out in Section \ref{Sec:ModelGeometry}. For $k\in \Z\setminus \{0\}$ one may analogously perform \emph{contact $\frac{1}{k}$ surgery} on a Legendrian knot $\Lambda$ by applying $-k$ Dehn twists as above. We will take as definition that contact $\frac{1}{k}$ surgery for $k \neq 0$ is given by performing contact $\sgn(k)$ surgery on $|k|$ parallel push-offs of $\Lambda$.

We write
\begin{equation*}
\Lambda = \LambdaPlus \cup \LambdaZero \cup \LambdaMinus \subset \Rthree
\end{equation*}
to specify a Legendrian link $\LambdaZero$ sitting inside of the contact manifold $\SurgLxi$. Since the neighborhoods of the components of $\Lambda$ defining surgery many be chosen to be disjoint from $\LambdaZero$, we may consider it to be a Legendrian link in $\SurgLxi$ post-surgery. The superscript $0$ on $\LambdaZero$ may be thought of as indicating a trivial $\frac{1}{0}=\infty$ surgery in the usual notation of Kirby calculus.

In Section \ref{Sec:SurgeryCobordisms} we will review how contact surgeries may be viewed as the result of handle attachments. We refer the reader to \cite{OS:SurgeryBook} for a review in the low-dimensional case and to \cite{SteinToWeinstein} for the general case.

\begin{thm}\label{Thm:SurgeryOverview}
We summarize some known results about contact surgery relevant to this paper:
\be
\item The contact manifold obtained by contact $+1$ surgery on the Legendrian unknot with $\tb=-1$ yields the standard fillable contact structure $\xi_{std}$ on $\Circle \times S^{2}$.
\item Applying contact $-1$ surgery on a Legendrian knot in $\Mxi$ produces the same contact manifold as is obtained by attaching a Weinstein handle to the convex boundary of the symplectization of $\Mxi$.
\item Then performing $\pm1$ surgery on a Legendrian knot $\Lambda \subset \Mxi$ followed by $\mp1$ surgery on a push-off $\lambda_{\xi}$ leaves $\Mxi$ unchanged.
\item A contact $3$-manifold is overtwisted if and only it can be described as the result of a contact $+1$ surgery along a stabilized Legendrian knot $\Lambda$ in some $\Mxi$.\footnote{One proof is obtained by proving the ``if'' statement using \cite{Ozbagci:Stabilization} and proving ``only if'' by following the proof of Theorem \ref{Thm:OTCH}. Alternatively, one can apply a handle-slide to \cite[Theorem 5.5(2)]{Avdek:ContactSurgery}.}
\item If $\Lambda \subset \Rthree$ satisfies $\tb(\Lambda) = 2g_{s}(\Lambda) -1$, then contact $\frac{1}{k}$ surgery on $\Lambda$ produces a tight contact manifold \cite{LS:TightI} for any $k \in \Z$.
\item For a Legendrian knot $\Lambda \subset \Rthree$ and an integer $k > 0$, contact $\frac{1}{k}$ surgery on $\Lambda$ produces a symplectically fillable contact $3$-manifold if and only if both $k = 1$ and $\Lambda$ bounds a Lagrangian disk in the standard symplectic $4$-disk \cite{PlusOneFilling}.
\ee
\end{thm}

\section{Notation and algebraic data associated to chords}\label{Sec:ChordNotation}

In this section we describe notation and algebraic data associated to chords of Legendrian links which will be used throughout the remainder of the paper. We take $\Lambda \subset \Rthree$ to be a non-empty link with sub-links $\Lambda^{+}, \Lambda^{-}$, and $\Lambda^{0}$ -- any of which may be empty. We write $\Lambda^{\pm}=\LambdaPlus \cup \LambdaMinus$.

\begin{assump}
It is assumed throughout that $\Lambda$ is \emph{chord generic}, meaning that all chords of $\Lambda$ are non-degenerate and that distinct chords are disjoint as subsets of $\R^{3}$.
\end{assump}

\subsection{Surgery coefficients and chords of $\Lambda$}

It will be convenient to write $\Lambda = \cup \Lambda_{i}$ with the subscript $i$ indexing the connected components of $\Lambda$. Using this notation, we use $c_{i} \in \{-1, 0, +1\}$ to indicate that $\Lambda_{i} \subset \Lambda^{c_{i}}$.

Denote by $r_{j}$ the Reeb chords of $\Lambda$ with the contact form $\alpha_{std} = dz - y dx$, which are in one-to-one correspondence with the double points of the Lagrangian projection $\pxy$. We write $\sgn_{j} \in \{\pm 1\}$ for the sign of the crossing of $\Lambda$ in the Lagrangian projection associated with the chord $r_{j}$ in accordance with the orientation of $\Lambda$.

We define $l^{-}_{j}$ to be the index $i$ of the $\Lambda_{i}$ on which $r_{j}$ begins and $l^{+}_{j}$ to be the index of the component of $\Lambda$ on which $r_{j}$ ends. The \emph{tip} of a chord $r_{j}$ is the point $q_{j}^{+} \in \Lambda_{l^{+}_{j}}$ where the chord ends. The \emph{tail} of $r_{j}$ is the point $q^{-}_{j} \in \Lambda_{l^{-}_{j}}$ at which the chord $r_{j}$ begins. We write the surgery coefficient of the components of $\Lambda$ corresponding to $l^{\pm}_{j}$ as $c_{j}^{\pm}$. That is,
\begin{equation*}
c_{j}^{\pm} = c_{l_{j}^{\pm}}.
\end{equation*}

\subsection{Words of chords}

An ordered pair of chords $(r_{j_{1}}, r_{j_{2}})$ is \emph{composable} if $l^{+}_{j_{1}} = l^{-}_{j_{2}}$.  A \emph{word of Reeb chords for $\Lambda$} is a formal product of chords $w=r_{j_{1}}\cdots r_{j_{n}}$ which which each pair $(r_{j_{k}}, r_{j_{k+1}})$ is composable for $k=1,\cdots,n-1$. 

We say that a word of Reeb chords $r_{j_{1}}\cdots r_{j_{n}}$ is a \emph{word of chords with boundary on $\LambdaZero$} if $r_{j_{1}}$ begins on $\LambdaZero$ and $r_{j_{n}}$ ends on $\LambdaZero$ and all other endpoints of chords touch components of $\LambdaPlus \cup \LambdaMinus$.

A \emph{cyclic word of Reeb chords for $\Lambda$}, denoted $r_{j_{1}}\cdots r_{j_{n}}$, is a word of Reeb chords for which $(r_{j_{n}},r_{j_{1}})$ is composable. Cyclic permutations of cyclic words are considered to be equivalent:
\begin{equation*}
    r_{j_{1}}r_{j_{2}}\cdots r_{j_{n}} = r_{j_{2}}\cdots r_{j_{n}}r_{j_{1}}.
\end{equation*}
When speaking of cyclic words of Reeb chords on $\Lambda$, we will implicitly assume that it is a cyclic word of Reeb chords on $\LambdaPlus \cup \LambdaMinus$.

The \emph{word length} of a word $w$ of Reeb chords is the number of individual chords it contains and will be denoted $\wl(w)$. The actions of each $r_{j}$ will be denoted $\action_{j}$ and the \emph{action} of a word is defined
\begin{equation*}
    \action(r_{j_{1}}\cdots r_{j_{n}}) = \sum_{k=1}^{n}\action_{j_{k}}.
\end{equation*}

\subsection{Capping paths}

Provided a composable pair of chords $(r_{j_{1}}, r_{j_{2}})$, their \emph{capping path} is the unique embedded, oriented segment of $\Lambda_{l^{+}_{j_{1}}} = \Lambda_{l^{-}_{j_{2}}}$, traveling in the direction of the orientation of $\Lambda$ from the tip of $r_{j_{1}}$ to the tail of $r_{j_{2}}$. The capping path will be denoted $\eta_{j_{1}, j_{2}}$. 

The analogously defined path, which travels opposite the orientation of $\Lambda_{l^{+}_{j_{1}}}$ will be denoted $\overline{\eta}_{j_{1}, j_{2}}$ and called the \emph{opposite capping path}. We will use $\zeta_{j_{1}, j_{2}}$ to denote one of either $\eta_{j_{1}, j_{2}}$ or $\overline{\eta}_{j_{1},j_{2}}$. By setting $\overline{\overline{\eta}}_{j_{1},j_{2}} = \eta_{j_{1},j_{2}}$, we can define $\overline{\zeta}_{j_{1},j_{2}}$ in the obvious way.

\subsubsection{Rotation angles and numbers}

Denote by $G$ the Gauss map sending the unit tangent bundle of $\R^{2}$ to $\Circle_{2\pi}$ with 
\begin{equation*}
G(\cos(t)\partial_{x} + \sin(t)\partial_{y}) = t.
\end{equation*}
This determines a map $G_{\Lambda}: \Lambda \rightarrow \Circle_{2\pi}$ assigning to each point in $\Lambda$ the unit tangent vector at that point determined by the orientation on $\Lambda$.

For any path $\zeta: [0, 1] \rightarrow \Lambda$, we can associate an angle $\theta(\zeta)\in \R$ as follows: Composing $\zeta$ with $G_{\Lambda}$ determines a map 
\begin{equation*}
\phi = G_{\Lambda}\circ \zeta: [0, 1] \rightarrow \Circle_{2\pi}.
\end{equation*}
Denoting by $\widetilde{\phi}$ the lift of this map to $\R$, the \emph{rotation angle} of $\zeta$, denoted $\theta(\zeta)$ is defined
\begin{equation*}
    \theta(\zeta) = \widetilde{\phi}(1) - \widetilde{\phi}(0).
\end{equation*}
If $q: \Circle \rightarrow \R^{3}$ is a parameterization of a component $\Lambda_{i}$ of $\Lambda$ then the rotation angle of the associated path $[0, 1]\rightarrow \Lambda_{i}$ is $2\pi \rot(\Lambda_{i})$.

The \emph{rotation angle} of a composable pair $(r_{j_{1}}, r_{j_{2}})$, denoted $\theta_{j_{1}, j_{2}} \in \R$, will later help us to compute Conley-Zehnder indices of closed Reeb orbits. It is defined as $\theta_{j_{1}, j_{2}} = \theta(\eta_{j_{1}, j_{2}})$. We write $\overline{\theta}_{j_{1}, j_{2}}$ for the rotation angle computed with the opposite capping path $\overline{\eta}_{j_{1}, j_{2}}$ whence
\begin{equation}\label{Eq:AngleSum}
    \theta_{j_{1}, j_{2}} - \overline{\theta}_{j_{1}, j_{2}} = 2\pi\rot(\Lambda_{l^{+}_{j_{1}}}).
\end{equation}
The \emph{rotation number of a composable pair of chords $(r_{j_{1}}, r_{j_{2}})$}, denoted $\rot_{j_{1}, j_{2}}$, is defined as
\begin{equation*}
\rot_{j_{1}, j_{2}} = \left\lfloor \frac{\theta_{j_{1}, j_{2}}}{\pi} \right\rfloor \in \Z.
\end{equation*}

\subsubsection{Crossing monomials}

Now we define the \emph{crossing monomials} which will later facilitate our computations of the homology classes of Reeb orbits of the $R_{\epsilon}$. Consider a collection of variables $\mu_{i}$ indexed by the connected components $\Lambda_{i}$ of $\Lambda$.

The \emph{crossing monomial of a chord $r_{j}$}, denoted $\cross_{j}$, is defined by the equation
\begin{equation}\label{Eq:ChordCrossingMonomial}
\cross_{j} = \half \left(\left(c_{j}^{-} + \sgn_{j}\right)\mu_{l^{-}_{j}} + \left(c^{+}_{j} + \sgn_{j}\right)\mu_{l^{+}_{j}} \right) \in \bigoplus \Z \mu_{i}.
\end{equation}

\begin{figure}[h]
\begin{overpic}[scale=.8]{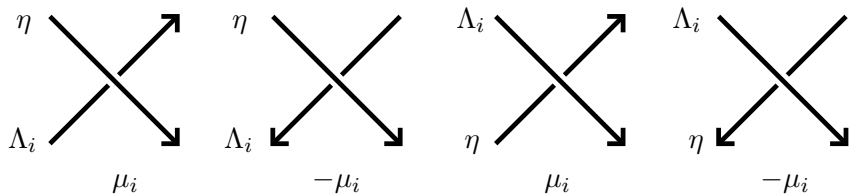}
\put(-4, 15){$\eta$}
\put(-5, 0){$\Lambda_{i}$}
\put(8, -5){$\mu_{i}$}

\put(23, 15){$\eta$}
\put(22, 0){$\Lambda_{i}$}
\put(33, -5){$-\mu_{i}$}

\put(52, 0){$\eta$}
\put(51, 15){$\Lambda_{i}$}
\put(62, -5){$\mu_{i}$}

\put(80, 0){$\eta$}
\put(78, 15){$\Lambda_{i}$}
\put(89, -5){$-\mu_{i}$}
\end{overpic}
\vspace*{.7cm}
\caption{Each subfigure gives a local picture of a crossing of a capping path $\eta$ with a component of $\Lambda$ in the Lagrangian projection. Labelings of the strands appears at the left of each subfigure, with the local contribution to the crossing number appearing below. Each subfigure may be rotated by $\pi$.}
\label{Fig:CrossingPairs}
\end{figure}

The \emph{crossing monomial of a composable pair of chords $(r_{j_{1}}, r_{j_{2}})$}, denoted $\cross_{j_{1}, j_{2}}$ is defined by the following formula:
\begin{equation*}
\cross_{j_{1}, j_{2}} = \sum_{q^{-}_{j^{-}} \in \Int(\eta_{j_{1}, j_{2}})} \sgn_{j^{-}} \mu_{l^{+}_{j^{-}}} + \sum_{q^{+}_{j^{+}} \in \Int(\eta_{j_{1}, j_{2}})} \sgn_{j^{+}} \mu_{l^{-}_{j^{+}}} \in \bigoplus \Z \mu_{i}.
\end{equation*}
The contributions are as described in Figure \ref{Fig:CrossingPairs}.

\begin{rmk}[Crossing monomials for connected $\Lambda$]\label{Rmk:CrossingConnectedLambda}
When $\Lambda$ consists of a single connected component we get a single surgery coefficient $c$ and a single $\mu$. In this case $\cross_{j}  = (c + \sgn_{j})\mu \in 2\mu\Z$ and $\cross_{j_{1}, j_{2}} = m \mu$ where $m$ is the number of times the interior capping path $\eta_{j_{1}, j_{2}}$ touches tip the tips and tails of chords, counted with signs given by the $\sgn_{j}$.
\end{rmk}

\subsection{Broken closed strings}\label{Sec:BCSandMaslov}

We temporarily work with an arbitrary contact $3$-manifold $\Mxi$ containing a Legendrian submanifold $\Lambda$. Equip $\Mxi$ with a contact form $\alpha$ and write $\kappa_{j}$ for the chords of $\Lambda$, which will be assumed non-degenerate. Words of chords with boundary on $\Lambda$ and cyclic words of chords on $\Lambda$ are defined as above in the obvious fashion.

Let $\kappa_{k}$, $k=1,\dots, n$ be a sequence of chords on $\Lambda$ and let $a_{k} \in \{\pm 1\}$. Let $\zeta_{k}$ be a collection of oriented arcs $\zeta_{k}:[0, 1]\rightarrow \Lambda$ starting at the endpoint (starting point) of $\kappa_{k}$ if $a_{k}$ is positive (negative) and ending at the starting point (ending point) of $\kappa_{k+1}$ if $a_{k+1}$ is positive (negative). Assume that the $a_{k}$ and $\zeta_{j_{k}}$ are such that
\begin{equation}\label{Eq:BrokenClosedString}
b = (a_{1}\kappa_{1})\ast \zeta_{1}\ast \cdots \ast (a_{k}\kappa_{n})\ast \zeta_{n}
\end{equation}
forms a closed, oriented loop, where $\ast$ denotes concatenation and $\pm \kappa_{k}$ is $\kappa_{k}$ parameterized with positive (negative) orientation.

\begin{defn}
We call a map $b$ as in Equation \eqref{Eq:BrokenClosedString} a \emph{broken closed string on $\Lambda$}. We call the $a_{k}$ \emph{asymptotic indicators}. We consider broken closed strings which differ by cyclic rotation of indices involved to be equivalent and say that a broken closed string is \emph{parameterized} if a fixed ordering of the indices is in use. We also consider broken closed strings which differ by homotopy of the $\zeta_{k}$ (relative to their endpoints) to be equivalent.
\end{defn}

\begin{ex}\label{Ex:BoundaryBCS}
Let $(t, U)$ be a holomorphic map from a disk with boundary punctures $\{ p_{j}\}$ removed $\disk\setminus \{p_{j}\}$ to the symplectization of $\Mxi$, with boundary punctures asymptotic to chords of $\Lambda$. Suppose that
\be
\item the $p_{j}$ are indexed according to their counterclockwise ordering along $\partial \disk$,
\item the $p_{j}$ are $a_{j}$-asymptotic to chords $\kappa_{j}$ ($a_{j} = 1$ for positively asymptotic and $a_{j}=-1$ for negatively asymptotic), and
\item $U(\partial \disk \setminus \{ p_{j}\}) \subset \Lambda$ with $\zeta_{j}$ denoting the restriction of $U$ to the component of $\partial \disk \setminus \{ p_{j}\}) \subset \Lambda$ whose oriented boundary is $p_{j+1} - p_{j}$.
\ee
With the data $a_{j}$, $\kappa_{j}$, and $\zeta_{j}$ specified by $(t, U)$ as above, Equation \eqref{Eq:BrokenClosedString} is a broken closed string on $\Lambda$. Of particular interest are broken closed strings determined by disks appearing in the $LRSFT$ differential for Legendrian links in $\Rthree$ \cite{Ng:RSFT}.

This may be generalized in the obvious way to holomorphic maps $(t, U)$ whose domain is a compact Riemann surface $(\Sigma, j)$ decorated with interior punctures (asymptotic to closed Reeb orbits) and boundary punctures (asymptotic to chords). Then any boundary component of $\Sigma$ determines a broken closed string on $\Lambda$.
\end{ex}

\begin{defn}\label{Def:HolomorphicBoundaryComponent}
A broken closed string determined by a holomorphic map as in Example \ref{Ex:BoundaryBCS} will be called a \emph{holomorphic boundary component}.
\end{defn}

We note that the $\zeta_{k}$ in the definition of a holomorphic boundary component may be constant: For example, if $(t, U)$ is a trivial strip with domain $\R \times I_{\epsilon}$ for some chord $\kappa$, consider
\begin{equation*}
b = \kappa \ast \zeta_{1} \ast (-\kappa) \ast \zeta_{2}
\end{equation*}
with $\zeta_{1}$ being a constant path with value the tip of $\kappa$ and $\zeta_{2}$ a constant path with value the tail of $\kappa$.

\begin{ex}\label{Ex:OrbitBCS}
Suppose that $\LambdaPM$ is a contact surgery diagram and let $w = r_{j_{1}}\cdots r_{j_{n}}$ be a cyclic word of composable Reeb chords on $\Lambda$. There are $2^{n}$ parameterized broken closed strings associated to this cyclic word, given by all of the ways that we may choose orientations for the capping path starting at the tip of each $r_{j_{k}}$ and ending at the tail of each $r_{j_{k+1}}$:
\begin{equation*}
\begin{aligned}
\{& r_{j_{1}}\ast \eta_{j_{1},j_{2}} \ast \cdots \ast r_{j_{n}}\ast \eta_{j_{n},j_{1}},\\
&r_{j_{1}}\ast \eta_{j_{1},j_{2}} \ast \cdots \ast r_{j_{n}}\ast \overline{\eta}_{j_{n},j_{1}},\\
&\dots\\
&r_{j_{1}}\ast \overline{\eta}_{j_{1},j_{2}} \ast \cdots \ast r_{j_{n}}\ast \eta_{j_{n},j_{1}},\\
&r_{j_{1}}\ast \overline{\eta}_{j_{1},j_{2}} \ast \cdots \ast r_{j_{n}}\ast \overline{\eta}_{j_{n},j_{1}}\}.
\end{aligned}
\end{equation*}
\end{ex}

\begin{defn}\label{Def:OrbitString}
We call each of the broken closed strings described in Example \ref{Ex:OrbitBCS} an \emph{orbit string} associated to $w$.
\end{defn}

When dealing with orbit strings, the $r_{j}$ are determined by the indices of the capping paths involved, and so will be omitted from our notation.

Note that a broken closed string on a Legendrian submanifold of dimension $n$ in a contact manifold of dimension $2n + 1$ for $n > 1$ is uniquely determined by its chords up to homotopy through broken closed strings. We will later see in Section \ref{Sec:H1Orbits} that a parameterized capping string provides instructions for homotoping a Reeb orbit of $\SurgLxi$ into the complement of a neighborhood of $\LambdaPM$ in $\R^{3}$.

\subsection{Maslov indices of broken closed strings}

Here we define Maslov indices on broken closed strings on Legendrians in contact $3$-manifolds, which are relevant to index computation of holomorphic curves. Essentially, we are packaging terminology appearing in the above subsection so as to be cleanly plugged into index computations appearing in \cite{EES:LegendriansInR2nPlus1, Ekholm:Z2RSFT}.  See Section \ref{Sec:IndexFormulae}.

We assume that $\dim(M)=3$ and that we are working with $\kappa_{k}, a_{k}, \zeta_{k}$ for $k=1,\dots,n$ as described in the previous subsection determining a broken closed string $b$ whose domain we take to be $\dom(b) = \Circle$. We remark on the case $\dim(M) > 3$ later in this subsection. Our discussion follows \cite[Section 3]{Ekholm:Z2RSFT}. We write $q^{-}_{k} \in \Lambda$ for the starting point of each $\kappa_{k}$ and $q^{+}_{k}$ for its endpoint.

We assume that $\xi$ is equipped with an adapted almost complex structure $J$ and suppose that we have a trivialization $s: \xi|_{\im(b)}\rightarrow \C$ of $\xi$ over the image of a broken closed string $b$ in $M$ which identifies the symplectic structure $d\alpha$ and complex structure $J$ on the target with the standard symplectic and complex structures on $\C$. The trivialization $s$ provides us with an identification
\begin{equation*}
    b^{*}\xi \simeq \C\times \Circle
\end{equation*}
Denote by $\mathcal{L}(\xi)\rightarrow M$ the bundle whose fiber $\mathcal{L}(\xi|_{x})\simeq \Circle_{\pi}$ over a point $x\in M$ is the space of unoriented Lagrangian subspaces -- that is unoriented real lines -- in $(\xi_{x}, d\alpha)$.\footnote{We use the circle of radius $\pi$, $\Circle_{\pi}$, rather than $\Circle_{2\pi}$ due to our ignoring the orientations of the lines involved.} Then $s$ likewise determines an identification
\begin{equation*}
    b^{*}\mathcal{L}(\xi) \simeq \Circle_{\pi}\times \Circle.
\end{equation*}
Over the subset of $\Circle$ parameterizing the $\zeta_{k}$, we have a section of this bundle determined by the unoriented Gauss map:
\begin{equation*}
    t \mapsto T_{b(t)}\Lambda \subset \xi_{b(t)}
\end{equation*}
Using $s$, this section determines a map $\phi^{G}$ over this subset to $\Circle_{\pi}$. We now describe how to extend this section over the subset of $\Circle$ parameterizing the $a_{k}\kappa_{k}$. 

For each chord $\kappa_{k}$, the time $t\in [0, \action(\kappa_{k})]$ flow of $R$ determines a path in $\SLtwoR$ by writing $\Flow_{R}^{t}(\xi_{q^{-}_{k}})$ in the standard basis of $\R^{2}$ determined by $s$. This likewise determines a section of $\mathcal{L}(\xi)$ over the chord by $\Flow_{R}^{t}(T_{q^{-}_{k}}\Lambda)$. As we've assumed that $\kappa_{k}$ is non-degenerate,
\begin{equation*}
\Flow_{R}^{\action(\kappa_{k})}(T_{q^{-}_{k}}\Lambda) \neq T_{q^{+}_{k}}\Lambda
\end{equation*}
as Lagrangian subspaces of $\xi_{q^{+}_{k}}$. In order to assign a Maslov number to $b$, we must make a correction to obtain a closed loop of Lagrangian subspaces:
\be
\item If $a_{k}=1$, then the orientation of $b$ and the chord coincide. To form a closed loop we join $\Flow_{R}^{\action(\kappa_{k})}(T_{q^{-}_{k}}\Lambda)$ to $T_{q^{+}_{k}}\Lambda$ by making the smallest possible clockwise rotation to $\Flow_{R}^{\action(\kappa_{k})}(T_{q^{-}_{k}}\Lambda)$.
\item If $a_{k}=-1$, then the orientation of $b$ and the chord disagree. To form a closed loop of Lagrangian subspaces along $b$, we start at the endpoint of the chord, follow the negative flow of $R$, and then join $\Flow_{R}^{-\action(\kappa_{k})}(T_{q^{+}_{k}}\Lambda)$ to $T_{q^{-}_{k}}\Lambda$ by making the smallest clockwise rotation possible.
\ee

Denote by $\phi_{b, s}:\Circle \rightarrow \Circle_{\pi}$ the map obtained.

\begin{defn}\label{Def:MaslovIndex}
We call the degree of the map $\phi_{b, s}$ described above the \emph{Maslov index of the broken closed string $b$ with respect to the framing $s$}, denoted $\Maslov_{s}(b) \in \Z$. It is easy to see that $\Maslov_{s}(b)$ does not depend on the cyclic ordering of its indices involved so that it is well-defined.
\end{defn}

The following easily follows from the construction of $\Maslov_{s}$.

\begin{prop}
Let $b$ be a broken closed string on $\Lambda \subset \Mxi$ with a trivialization $s$ of $\xi|_{\im(b)}$. Smooth homotopies of such trivializations $s$ leave $\Maslov_{s}(b)$ unchanged. The $\bmod_{2}$ reduction of $\Maslov_{s}(b)$ is independent of $s$, so that we may define $\Maslov_{2}(b)\in \Z/2\Z$ as an invariant of $b$. 

Now suppose that $\Gamma_{\neq 0}(\xi)$ is non-empty as in Proposition \ref{Prop:CanonicalZGrading} which clearly applies to any $\Lambda \subset \Rthree$:
\be
\item If $b$ is homologically trivial in $M$ then $\Maslov_{s}(b)$ is independent of $s \in \Gamma_{\neq 0}(\xi)$. 
\item If $H_{2}(M) = H^{1}(M) = 0$ then $\Maslov_{s}(b)$ is independent of $s \in \Gamma_{\neq 0}(\xi)$, regardless of the homotopy class of $b$ in $M$.
\ee
\end{prop}

\subsection{Generalizations and comparison with existing conventions}

\subsubsection{Generalized crossing signs and Maslov indices}

Crossing signs generalize to $n$-dimensional Legendrian submanifolds inside contact manifolds of dimension $2n+1$ as follows. As above, consider a generic chord $\kappa$ on an oriented Legendrian submanifold $\Lambda \subset \Mxi$ parameterized by an interval $[0, a]$ given by the flow of some $R$. Then we may define $\sgn(\kappa)$ by
\begin{equation*}
\big( \wedge^{n}T_{\kappa(a)}\Lambda \big) \wedge \big(\wedge^{n}\Flow^{a}_{R}(T_{\kappa(0)}\Lambda) \big) = \sgn(\kappa) \big( \wedge^{2n}\xi_{\kappa(a)} \big)
\end{equation*}
as an orientation on $\xi_{\kappa(a)}$. Note that $\sgn(\kappa)$ is independent of the orientation of $\Lambda$ if and only if $\Lambda$ is connected. However the product of $\sgn$ over the chords appearing in a broken closed string is always independent of choice of orientation.

We also briefly address generalizations of the Maslov index to higher dimensions. Provided a contact manifold $\Mxi$ of dimension $2n+1$, we write $\mathcal{L}(2n) = U(n)/O(n)$ for the space of (unoriented) Lagrangian planes in the standard symplectic vector space and define the bundle
\begin{equation*}
\mathcal{L}(2n) \hookrightarrow \mathcal{L}(\xi) \twoheadrightarrow M    
\end{equation*}
as above without modification. Provided a trivialization
\begin{equation*}
    s: b^{*}(\xi) \rightarrow \C^{n}\times \Circle
\end{equation*}
we can view the sections of $b^{\ast}\mathcal{L}(\xi)$ as maps from $\Circle$ to $U(n)/O(n)$, in which case $\Maslov_{s}(b)$ may be defined and computed as the usual Maslov index of loops in the Lagrangian Grassmanian. See, for example \cite[Theorem 2.35]{MS:SymplecticIntro}. The required ``clockwise rotation'' correction in arbitrary dimensions is described by the paths $\mathrm{f}_{j}(s)$ appearing in Section 5.9 of \cite{EES:LegendriansInR2nPlus1}.

\subsubsection{Conventions for capping paths}
We briefly address how our conventions for capping paths and rotation angles differ from those used to construct gradings in Legendrian contact homology. See, for example, the exposition \cite[Section 3.1]{EtnyreNg:LCHSurvey}. Assume that $\Lambda \subset \Rthree$ consists of a single component and has a designated basepoint $\ast$ not coinciding with the tip or tail of any chord.

For a chord $r_{j}$, exactly one of $\eta_{j, j}$ or $\overline{\eta}_{j, j}$ will pass through $\ast$. Denoting by $\phi_{j}$ the rotation angle of the path not passing through $\ast$, the $LCH$ grading is defined -- by a slight manipulation of conventional notation -- as
\begin{equation*}
    |r_{j}| = \left\lfloor \frac{\phi_{j}}{\pi} \right\rfloor - 1
\end{equation*}
This is very similar to our computation of rotation numbers except that
\be
\item knots along which we are performing surgery do not have basepoints,
\item our capping paths do not necessarily begin and end at endpoints of the same chord, and
\item our capping paths follow the orientation of $\Lambda$ by default.
\ee
We will see that our conventions for computation arises naturally when computing Conley-Zender indices of Reeb orbits of the $R_{\epsilon}$ using the framing construction of Section \ref{Sec:Framing}. This convention is also convenient as it will simplify the statements of homology classes of closed Reeb orbits in Section \ref{Sec:Homology}.

Our framing construction can be modified so as naturally lead to computations of rotation angles using basepoints as in $LCH$. See Remark \ref{Rmk:MuToLambdaFraming}. By Equation \eqref{Eq:AngleSum}, if $\rot(\Lambda) = 0$ then our computation of rotation angles coincide when the endpoints of a capping path lie over the same chord:
\begin{equation*}
    \theta_{j, j} = \overline{\theta}_{j, j} = \phi_{j}.
\end{equation*}

\subsubsection{Conventions for broken closed strings}
In \cite[Definition 3.1]{Ng:RSFT}, broken closed strings have discontinuities at Reeb chords, whereas our broken closed strings are continuous maps. We have chosen to define broken closed strings to include the data of the chords in question, so as reduce ambiguity when discussing chords on Legendrians contained in surgered contact manifolds $\SurgLxi$.

\section{Model geometry for Legendrian links and contact surgery}\label{Sec:ModelGeometry}

In this section we construct neighborhoods of Legendrian links and then perform contact surgery on $\LambdaPM$ using these neighborhoods to describe the contact manifolds $\SurgLxi$ and the contact forms $\alpha_{\epsilon}$.

Our strategy is to develop highly specialized models for the objects involved in contact surgery, determining Reeb vector fields on surgered contact manifolds which are linear in a way which will be made precise in Section \ref{Sec:OverlappingRectangles}. The main benefits of this approach are that the proofs of the following will be considerably simplified:
\be 
\item the chord-to-orbit (Theorem \ref{Thm:ChordOrbitCorrespondence}) and chord-to-chord (Theorem \ref{Thm:ChordsToChords}) correspondences.
\item Conley-Zehnder index (\ref{Thm:IntegralCZ}) and Maslov index (Theorem \ref{Thm:MaslovComputation}) computations.
\ee
We will also be able to determine the embeddings of simple closed orbits in surgered manifolds as fixed points of explicitly defined affine endomorphisms of $\R^{2}$ (Section \ref{Sec:Embeddings}). While we don't pursue computation in this paper, we anticipate this being of utility in future applications.

The primary disadvantage to our contact forms being so specialized is that surgery cobordisms between the $\SurgLxi$ will be less explicitly defined and will require greater effort in their construction (Section \ref{Sec:SurgeryCobordisms}). Furthermore, we will be imposing restrictions on the Lagrangian projections of Legendrian links in the style of \cite{Ng:ComputableInvariants}, so that our analysis -- which is applicable to all Legendrian \emph{isotopy classes} $\LambdaPM$ -- will not be applicable to all chord generic Legendrian links in $\Rthree$.

\begin{rmk}
Our approach to contact surgery is quite similar to that of Foulon and Hasselblatt in \cite{FH:Anosov}. In that paper surgery is defined using a model Dehn twist as in our Section \ref{Sec:ModelDehnTwists}. 

In \cite{BEE:LegendrianSurgery, Ekholm:SurgeryCurves} Bourgeois, Ekholm, and Eliashberg describe surgeries as the result of critical-index Weinstein handle attachments and then study the resulting Reeb dynamics. This contrasts with our approach in that we will first describe our contact forms $\alpha_{\epsilon}$ and then build specialized Weinstein handles have the $\alpha_{\epsilon}$ as the restriction of their Liouville form  to their contact boundaries.

The approaches to contact surgery in this article, \cite{FH:Anosov}, and \cite{BEE:LegendrianSurgery, Ekholm:SurgeryCurves} all have at least one feature in common: Shrinking the size of the surgery locus is used to control Reeb dynamics.
\end{rmk}

\subsection{Almost complex structures, metrics, and the Gauss map}

We will want our Legendrians and their neighborhoods to interact nicely with an almost complex structures $J_{0}$ and a metric $g_{\R^{3}}$ which we now describe.

Define vector fields $X, Y \in \Gamma(\xi_{std})$ by lifting the derivatives of the usual coordinates:
\begin{equation*}
    X = \partial_{x} + y\partial_{z},\quad Y=\partial_{y}.
\end{equation*}
We define a complex structure $J_{0}$ on $\xi_{std}$ as the lift of the usual complex structure on $\R^{2} = \C$:
\begin{equation}\label{Eq:Jstd}
    J_{0}X = Y,\quad J_{0}Y = -X.
\end{equation}
This determines an almost complex structure adapted to the symplectization $(\R\times\R^{3}, e^{t}\alpha_{std})$ which we'll also call $J_{0}$, defined
\begin{equation*}
J_{0}\partial_{t} = \partial_{z},\quad J_{0}\partial_{z} = -\partial_{t}.
\end{equation*}
This almost complex structure determines a $J_{0}$-invariant metric $g_{\R^{3}}$ on $\R^{3}$, defined
\begin{equation*}
g_{\R^{3}}(u, v) = \alpha(u)\alpha(v) + d\alpha(\pi_{\alpha} u, J_{0}\pi_{\alpha} v), \quad \pi_{\alpha}(u) = u - \alpha(u)\partial_{z} \in \xi_{std}.
\end{equation*}
The metric yields a simple formula for the lengths of vectors in $\xi_{std}$:
\begin{equation}\label{Eq:LegVectorLengths}
Z = aX + bY \in \xi_{std} \quad \implies \quad |Z| = \sqrt{a^{2} + b^{2}}.
\end{equation}

\begin{comment}
Provided a component $\Lambda_{i}$ of $\Lambda$, we write $G_{i}: \Circle \rightarrow \R/2\pi\R$ for the Gauss map associated to the parameterization with respect to the basis $(X, Y)$ of $\xi_{std}$. Assuming that $\Lambda_{i}$ is parameterized $\Lambda_{i}(q):\Circle \rightarrow \R^{3}$ with constant speed $v_{i}$, then we have
\begin{equation*}
    \frac{\partial L_{i}}{\partial t} = v_{i}\exp(J_{0}G_{i}(q)) \in \xi_{std}
\end{equation*}
in the basis $(X, Y)$.
\end{comment}

\subsection{Good position and Lagrangian resolution}\label{Sec:LagrangianResolution}

\begin{defn}\label{Def:GoodPosition}
We say that a Legendrian link $\Lambda \subset \R^{3}$ is in \emph{good position} if it is chord generic and for each double point $(x_{0},y_{0})\in\R^{2}$ of its Lagrangian projection $\pxy(\Lambda)$ there exists a neighborhood within which
\be
\item the over-crossing arc admits a parameterization satisfying $(x, y)(q) = (x_{0} + q, y_{0} - q)$ and
\item the under-crossing arc admits a parameterization satisfying $(x, y)(q) = (x_{0} + q, y_{0} + q)$
\ee
\end{defn}

Good position guarantees that the Gauss map of a parameterization of $\Lambda$ evaluates to $\frac{3\pi}{4}$ or $\frac{7\pi}{4}$ near an over-crossing and to $\frac{\pi}{4}$ or $\frac{5\pi}{4}$ near an under-crossing.\footnote{In \cite{BEE:LegendrianSurgery}, it is presumed that the tangent map of Reeb flow along a chord $r$ sends $T_{r(0)}\Lambda \subset \xi$ to the subspace $JT_{r(a)}\Lambda$, which is achieved by an appropriate choice of almost complex structure on the contact hyperplane of the manifold containing $\Lambda$. In our case, this is achieved by assuming that $\Lambda$ is in good position. We will see in the proof of Theorem \ref{Thm:IntegralCZ} that our analysis is contingent upon this assumption. Similarly precise perturbations of Legendrian submanifolds near endpoints of chords appear in \cite{EES:LegendriansInR2nPlus1} for the purpose of guaranteeing transversality of moduli spaces used to compute differentials for the contact homology of Legendrians in $(\R^{2n+1}, \xi_{std})$.} Likewise, the condition ensures that capping paths of composable pairs of chords satisfy
\begin{equation*}
\theta_{j_{1}, j_{2}} \bmod_{2\pi} \in \left\{ \frac{\pi}{2}, \frac{3\pi}{2} \right\}.    
\end{equation*}

\begin{figure}[h]\begin{overpic}[scale=.6]{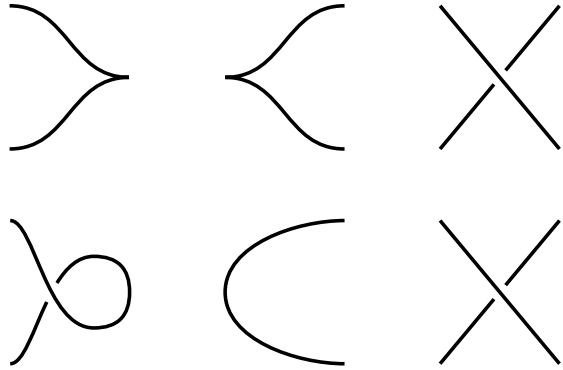}
\end{overpic}
\caption{The first row of subfigures shows segments of a Legendrian link appearing in the front projection. Directly below each subfigure is how it appears in the Lagrangian resolution.}
\label{Fig:LagrangianResolution}
\end{figure}

\begin{prop}\label{Prop:GoodAndUnitSpeed}
Provided a front projection of a Legendrian link $\Lambda$, we may perform a Legendrian isotopy so that the following properties are satisfied:
\be
\item $\Lambda$ is in good position.
\item The Lagrangian diagram is obtained by resolving singularities of the front as depicted in Figure \ref{Fig:LagrangianResolution}.
\item The arc length of each connected component of $\Lambda$ with respect to $g_{\R^{3}}$ is $1$.
\ee
\end{prop}

\begin{proof}
The proof proceeds in three steps. The first step establishes the first two desired properties of $\Lambda$. This step is essentially the proof of Proposition 2.2 from \cite{Ng:ComputableInvariants} and so we will omit the details. The only modification required to ensure a link is in good position after Legendrian isotopy is to control $\frac{\partial z}{\partial x}$ of a parameterization of $\Lambda$ near the right-pointing cusps and what are called ``exceptional segments'' in that proof. In particular, $\frac{\partial z}{\partial x}$ can be made quadratic with highest-order coefficient $\half$ ($-\half$) on neighborhoods of the positive (negative) endpoints of chords.

In our second step, we modify $\Lambda$ so that the arc length of each component is arbitrarily small while maintaining our desired conditions on the Lagrangian projection. For $\rho > 0$, consider the linear transformation $\phi_{\rho}$ of $\R^{3}$, defined $\phi_{\rho}( x, y, z) = (\rho x, \rho y, \rho^{2}z)$. Then $\phi_{\rho}^{\ast}\alpha_{std} = \rho^{2}\alpha_{std}$ so that each $\phi_{\rho}$ is a contact transformation. The map $\phi_{\rho}$ also has the following useful properties:
\be
\item It preserves the angles of vectors in $\xi_{std}$.
\item If $\Lambda_{i}$ is a Legendrian curve with arc-length $\ell$, then $\phi_{\rho}(\Lambda_{i})$ has arc-length $\rho \ell$.
\ee
Take the family of Legendrians $\phi_{e^{-T}}(\Lambda)$, $T \in [0, T_{0}]$ with $T_{0}$ large enough so that each connected component of $\phi_{e^{-T_{0}}}(\Lambda)$ will have arc length $\leq 1$. This interpolation between $\Lambda$ and $\phi_{e^{-T_{0}}}(\Lambda)$ determines a $1$-parameter family of Legendrian submanifolds and so may be realized by a Legendrian isotopy.

In the case that $\Lambda$ connected, we choose $T_{0}$ so that the arc-length is exactly equal to $1$ after the isotopy and conclude the proof. When $\Lambda$ is disconnected a final, third step is required. In this step we increase the arc-lengths of the connected components of $\Lambda$ so that they are all $1$ while preserving good position and the smooth isotopy type of the Lagrangian projection.

\begin{figure}[h]\begin{overpic}[scale=.8]{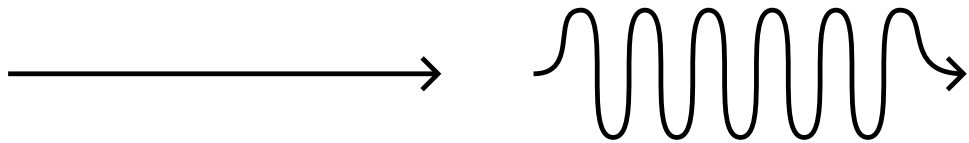}
\end{overpic}
\caption{Locally modifying a Legendrian in the Lagrangian projection by a rapidly oscillating function $\tilde{y}$ to increase its arc length.}
\label{Fig:Wiggle}
\end{figure}

We demonstrate how to increase arc-lengths so as to achieve the desired result. Consider a segment of $\Lambda_{i}$ along which the $x$-derivative is non-zero, parameterized via the $x$ variable $x \mapsto (x, y(x), z(x))$ with $x \in [-\delta, \delta]$ for an arbitrarily small positive constant $\delta$. We assume that the Lagrangian projection of the segment does not touch any double points. Let $\tilde{y} \in \Cinfty([-\delta, \delta])$ be a function with compact support contained in $(-\delta, \delta)$ and for which $\int_{-\delta}^{\delta} \tilde{y}dx = 0$. Consider perturbations $\Lambda_{i, T}$ of $\Lambda_{i}$ parameterized by $T \in [0, 1]$ which modify $\Lambda_{i}$ along our segment to take the form
\begin{equation*}
x \mapsto \left(x, y + T\tilde{y}, z + T\int_{-\delta}^{x}\tilde{y}dx\right).
\end{equation*}
The vanishing of the integral of $\tilde{y}$ ensures that the $z$-values at the endpoints of the segment are unaffected by the perturbation. By making $\sup \tilde{y}$ small and $\int |\frac{\partial \tilde{y}}{\partial x}|dx$ very large, we can ensure that for $T \in [0, 1]$ our perturbations introduce no new double points in the Lagrangian projection and that $\Lambda_{i, T}$ has arc-length as large as we like, say $2$ when $T=1$. See Figure \ref{Fig:Wiggle}. Hence for some $T_{0} \in [0, 1]$ the arc length of $\Lambda_{i, T_{0}}$ will be exactly $1$.

Apply such perturbations to each connected component of $\Lambda$ so that no new double points are created and neighborhoods of double points are unaffected. Each perturbation is realizable by a Legendrian isotopy. Thus we have obtained a Legendrian isotopy of $\Lambda$ having all of the desired properties.
\end{proof}

\begin{figure}[h]\begin{overpic}[scale=.8]{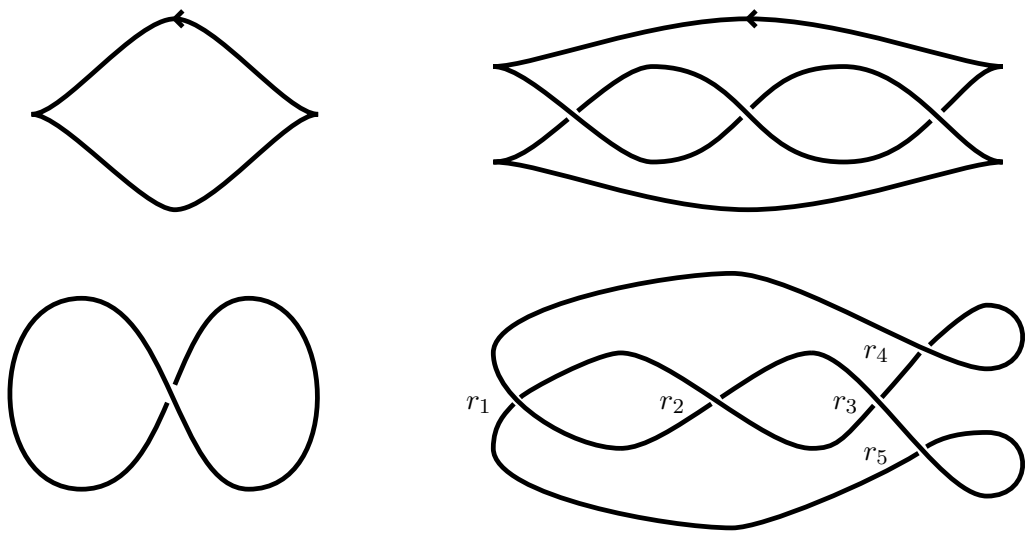}
\put(45, 12){$r_{1}$}
\put(64, 12){$r_{2}$}
\put(81, 12){$r_{3}$}
\put(84, 17){$r_{4}$}
\put(84, 7){$r_{5}$}
\end{overpic}
\caption{The left column shows Legendrian $\tb=-1$ unknot in the front and Lagrangian projections. A right-handed trefoil knot with $\tb = 1$ and $\rot=0$ is shown in the front and Legendrian projections on the right. The Reeb chords of the Lagrangian projection of the trefoil are labeled $r_{i}$.}
\label{Fig:LagrangianResolutionEx}
\end{figure}

Provided $\Lambda$ as a front projection diagram, we call the Lagrangian projection of (an isotopic copy of) $\Lambda$ obtained as in the above proposition the \emph{Lagrangian resolution} of the front diagram. Figure \ref{Fig:LagrangianResolutionEx} displays Lagrangian resolutions of an unknot and a trefoil. Following \cite{Ng:ComputableInvariants}, we say that a front projection of a Legendrian link $\Lambda$ is \emph{nice} if there exists some $x_{0} \in \R$ for which all right-pointing cusps have $x$-value $x_{0}$. Its not difficult to see that any $\Lambda$ can be isotoped to have a nice front projection.

\subsection{Conventions for link diagrams}

We will not concern ourselves with specific requirements of good position or arc-length when drawing Legendrian links in the Lagrangian projection and consider such a diagram to be valid if it recovers the Lagrangian projection of a Legendrian link after an isotopy of the $xy$-plane. In particular, we will not take care to ensure that angles at crossings are precise or that the components of $\R^{2}\setminus \pxy(\Lambda)$ satisfy the area requirements of \cite[Section 2]{Etnyre:KnotNotes}.

Throughout, Legendrian knots with surgery coefficient $+1$ will be colored blue and knots with surgery coefficient $-1$ will be colored red. If the coefficient of a knot is not already determined or the knot corresponds to a component of $\LambdaZero$, it will be colored black.

\subsection{Standard neighborhoods}\label{Sec:StandardNeighborhoods}

Before stating the properties we will want our neighborhoods of $\Lambda$ to have, we will create model neighborhoods near under- and over-crossings of chords. The neighborhood construction is completed in Proposition \ref{Prop:NeighborhoodConstruction}.

\subsubsection{Model neighborhoods near endpoints of chords}

Here we describe a construction of a neighborhood of $\Lambda$ along the arcs described in Definition \ref{Def:GoodPosition}. We can reparameterize the arcs to have unit speed so that they take the form
\begin{equation*}
q \mapsto (x_{0} + \frac{q}{\sqrt{2}}, y_{0} - \frac{q}{\sqrt{2}}, z_{0} + \frac{y_{0}q}{\sqrt{2}} - \frac{q^{2}}{4})
\end{equation*}
near an over-crossing and
\begin{equation*}
q \mapsto (x_{0} + \frac{q}{\sqrt{2}}, y_{0} + \frac{q}{\sqrt{2}}, z_{0} + \frac{y_{0}q}{\sqrt{2}} + \frac{q^{2}}{4})
\end{equation*}
along an under-crossing. For $\epsilon > 0$ sufficiently small, we extend these embeddings to embeddings of $I_{\epsilon}\times I_{\epsilon} \times I_{2\epsilon}$ into $\R^{3}$ using coordinates $(z, p, q)$. Near an over-crossing, this embedding takes the form
\begin{equation}\label{Eq:OvercrossingNeighborhoodModel}
\Phi^{+}_{x_{0},y_{0},z_{0}}(z, p, q)= (x_{0} - \frac{p}{\sqrt{2}} + \frac{q}{\sqrt{2}}, y_{0} - \frac{p}{\sqrt{2}} - \frac{q}{\sqrt{2}}, z_{0} + z + y_{0}\frac{q-p}{\sqrt{2}} + \frac{p^{2}}{4} + \frac{pq}{2} - \frac{q^{2}}{4}).
\end{equation}
Near an under-crossing arc this takes the form
\begin{equation}\label{Eq:UndercrossingNeighborhoodModel}
\Phi^{-}_{x_{0},y_{0},z_{0}}(z, p, q)= (x_{0} + \frac{p}{\sqrt{2}} + \frac{q}{\sqrt{2}}, y_{0} - \frac{p}{\sqrt{2}} + \frac{q}{\sqrt{2}}, z_{0} + z + y_{0}\frac{p+q}{\sqrt{2}} - \frac{p^{2}}{4} + \frac{pq}{2} + \frac{q^{2}}{4}).
\end{equation}

\begin{properties}\label{Properties:CrossingNeighborhoods}
The following properties are satisfied by the $\Phi^{\pm}_{x_{0},y_{0},z_{0}}$:
\be
\item $\Phi^{\pm}_{x_{0},y_{0},z_{0}}(0, 0, q)$ provides a parameterization of $\Lambda$ with unit speed.
\item $(\Phi^{\pm}_{x_{0},y_{0},z_{0}})^{*}\alpha_{std} = dz + p dq$.
\item With respect to the basis $P=\partial_{p}, Q=\partial_{q}-p\partial_{z}$, $J_{0} = \left( \begin{smallmatrix}
0 & -1 \\ 1 & 0
\end{smallmatrix}\right)$.
\item $\pxy\circ \Phi^{\pm}_{x_{0}, y_{0}, z_{0}}$ is an affine map.
\item The images of $\pxy\circ \Phi^{\pm}_{x_{0},y_{0},z_{0}}$ overlap in squares of the form $I_{\epsilon}\times I_{\epsilon}$ near a crossing (see Figure \ref{Fig:OverlappingRectangles}).
\ee
\end{properties}

\subsubsection{Neighborhood construction}

We now assume that $\Lambda$ satisfies the conclusions of Proposition \ref{Prop:GoodAndUnitSpeed}.

\begin{prop}\label{Prop:NeighborhoodConstruction}
For $\epsilon_{0}$ sufficiently small, there exists a neighborhood $N_{\epsilon_{0}, i}$ of each $\Lambda_{i}$ parameterized by an embedding
\begin{equation*}
    \Phi_{i}: I_{\epsilon_{0}}\times I_{\epsilon_{0}} \times \Circle \rightarrow \R^{3}
\end{equation*}
with coordinates $(z, p, q)$ such that the following conditions are satisfied:
\be
\item $\Phi_{i}^{\ast}\alpha_{std}= dz + p dq$.
\item The $N_{\epsilon_{0}, i}$ are disjoint.
\item $\Phi_{i}(0, 0, q)$ provides a unit-speed parameterization of $\Lambda_{i}$.
\item $J_{0}$ is $z$-invariant in $N_{\epsilon_{0}, i}$ and with respect to the basis $P = \partial_{p}, Q = \partial_{q} - p\partial_{z}$ it satisfies
\begin{equation*}
\Phi_{i}^{\ast}J_{0} = \left( \begin{smallmatrix}
0 & -1 \\ 1 & 0
\end{smallmatrix}\right) + \bigO(p).
\end{equation*}
\item Near the endpoints $(x_{j},y_{j},z_{j}^{\pm})$ with $z_{j}^{+} > z_{j}^{-}$ of each chord $r_{j}$ of touching $\Lambda$, we can can find a matrix of the form $M=\Diag(1, 1,1)$ or $\Diag(1, -1, -1)$ so that
\begin{equation*}
\Phi_{i}(z, p, q) = \Phi_{x_{j},y_{j},z_{j}^{\pm}}^{\pm}\circ M(z, p, q-q_{j}^{\pm})
\end{equation*}
where the $\Phi_{x_{j},y_{j},z_{j}^{\pm}}$ are as in Properties \ref{Properties:CrossingNeighborhoods}.
\ee
\end{prop}

\begin{proof}
Presuming that $\Lambda_{i}$ is parameterized with a variable $q$ with respect to which it has unit speed, we pick an arbitrarily small positive constant $\epsilon_{1}$ and define a map $I_{\epsilon_{1}} \times \Circle \rightarrow \R^{3}$ as
\begin{equation*}
\phi_{1}: (p, q) \mapsto \exp_{\Lambda_{i}(q)}\Big(-p J_{0}\frac{\partial \Lambda_{i}}{\partial q}(q) + h_{1}(p, q) \Big)
\end{equation*}
where $h_{1} \in \Cinfty(I_{\epsilon_{1}}\times \Circle, \R^{3})$ vanishes up to second order in $p$ and is chosen so that it produces the map
\begin{equation*}
(p, q) \mapsto \Phi_{x_{j},y_{j},z_{j}^{\pm}}^{\pm}\circ M(0, p, q - q^{\pm}_{j})
\end{equation*}
near the endpoints of the chords of $\Lambda$ as in the statement of the proposition. Here $\exp$ is the exponential map with respect to the metric $g_{\R^{3}}$ and the matrix $M$ is as in the statement of the proposition. 

Since the tangent map of the exponential map is the identity along the zero-section,
\begin{equation*}
T\exp_{\Lambda_{i}(q)} = \Id: T_{\Lambda_{i}(q)}\R^{3} \rightarrow T_{\Lambda_{i}(q)}\R^{3}
\end{equation*}
and $h_{1}$ is $\bigO(p^{2})$, we have
\begin{equation*}
\begin{aligned}
\frac{\partial \phi_{1}}{\partial p}|_{p=0} &= -J_{0}\frac{\partial \Lambda_{i}}{\partial q} + \frac{\partial h_{1}}{\partial p}|_{p = 0} = -J_{0}\frac{\partial \Lambda_{i}}{\partial q} = -J_{0}e^{J_{0}G_{i}}X = -e^{J_{0}G_{i}(q)}Y\\
\frac{\partial \phi_{1}}{\partial q}|_{p=0} &= \frac{\partial}{\partial q}\exp_{\Lambda_{i}(q)}(0) = \frac{\partial \Lambda_{i}}{\partial q} = e^{J_{0}G_{i}}X.
\end{aligned}
\end{equation*}
Therefore the tangent map for $\phi_{1}$ can be expressed along $\{ p = 0 \}$ as a matrix
\begin{equation}\label{Eq:ExpTangentMap}
T\phi_{1}|_{p=0} = -J_{0}e^{J_{0}G_{i}(q)}
\end{equation}
with incoming basis $(P, Q)$ and outgoing basis $(X, Y)$. This map will be an embedding when restricted to some $I_{\epsilon_{1}}\times \Circle$ for $\epsilon_{1}$ sufficiently small. 

From Equation \eqref{Eq:ExpTangentMap} we compute $\phi_{1}^{*}d\alpha_{std}= dp\wedge dq$ along $\{p = 0\}$. More generally, we can write $\phi_{1}^{*}d\alpha_{std} = Fdp\wedge dq$ for some smooth function $F$ satisfying $F|_{\{p=0\}} = 1$. Hence $F$ will be strictly positive on some tubular neighborhood of $\{p = 0\} \subset I_{\epsilon_{1}}\times \Circle$ so that $\phi_{1}^{*}d\alpha_{std}$ will be symplectic on some $I_{\epsilon_{2}}\times \Circle$ for $\epsilon_{2}$ sufficiently small. 

Applying a fiber-wise Taylor expansion to $\phi_{1}^{*}\alpha_{std}$ along the annulus $I_{\epsilon_{2}}\times \Circle$, we write
\begin{equation*}
\phi_{1}^{*}\alpha_{std} = (f_{0} + p f_{1} + p^{2}f_{2} + f_{hot})dp + (g_{0} + p g_{1} + p^{2}g_{2} + g_{hot})dq
\end{equation*}
where
\be
\item $f_{hot}$ and $g_{hot}$ are functions of $p$ and $q$ which are $\bigO(p^{3})$ and
\item $f_{0},\dots, g_{2}$ are functions of $q$.
\ee
As $\Lambda_{i}$ is Legendrian and $J_{0}$ preserves the contact structure, we must have $f_{0} = g_{0} = 0$. Then computing
\begin{equation*}
\phi_{1}^{*}d\alpha_{std} = d\phi_{1}^{*}\alpha_{std} = \left( g_{1} + p\left(2g_{2} - \frac{\partial f_{1}}{\partial q}\right) -p^{2}\frac{\partial f_{2}}{\partial q} + \frac{\partial g_{hot}}{\partial p} - \frac{\partial f_{hot}}{\partial q} \right)dp\wedge dp,
\end{equation*}
we must have $g_{1} = 1$ so that 
\begin{equation*}
    \phi_{1}^{*}\alpha_{std} = (p f_{1} + p^{2}f_{2} + f_{hot})dp + (p + p^{2}g_{2} + g_{hot})dq.
\end{equation*}

We can eliminate the $f_{1}$ term in this equation with a perturbation in the $z$-direction. With $h_{2} = \half p^{2}f_{1}$ we have $dh_{2} = p f_{1}dp + \half p^{2}\frac{\partial f_{1}}{\partial q}dq$. Hence
\begin{equation*}
    \phi_{2}(p, q) = \phi_{1}(p, q) - (0, 0, h_{2}(p, q)),
\end{equation*}
admits an expansion of the form
\begin{equation*}
    \phi_{2}^{*}\alpha_{std} = (p^{2}f_{2} + f_{hot})dp + (p + p^{2}g_{2} + g_{hot})dq.
\end{equation*}
To ensure that this map is an embedding, we restrict its domain to $I_{\epsilon_{3}}\times\Circle$ for some sufficiently small $\epsilon_{3} \leq \epsilon_{2}$. We note that the $f_{2}$ and $g_{2}$ here may differ from those in the Taylor expansion of $\phi_{1}^{\ast}\alpha_{std}$.

Now we'll apply a Moser argument as in \cite[Section 3.2]{MS:SymplecticIntro} to modify $\phi_{2}$ by precomposing it with an isotopy to produce a map $\phi_{3}$ so that $\phi_{3}^{*}d\alpha_{std} = dp\wedge dq$. Due to the facts that the annulus is not closed and that we'll require the result to be an codimension $1$ embedding, we cannot simply quote \cite[Section 3.2]{MS:SymplecticIntro}. 

Writing $\phi_{2}^{*}\alpha_{std} = p dq + \sigma$ and solving for a vector field $X_{\sigma}$ satisfying $dp\wedge dq(\ast, X_{\sigma}) = \sigma$ we see that $\sigma$ and $X_{\sigma}$ have coefficient functions vanishing up to second order:
\begin{equation*}
    \sigma = \bigO(p^{2})dp + \bigO(p^{2})dq,\quad X_{\sigma} = \bigO(p^{2})\partial_{p} + \bigO(p^{2})\partial_{q}.
\end{equation*}
Writing $\Flow_{X_{\sigma}}^{t}$ for the time $t$ flow of $X_{\sigma}$, choose $\epsilon_{4} \leq \epsilon_{3}$ so that $\Flow_{X_{\sigma}}^{t}(I_{\epsilon_{4}}\times\Circle) \subset I_{\epsilon_{3}}\times \Circle$ for all $t \in [0, 1]$ and define
\begin{equation*}
    \phi_{3}(p, q) = \phi_{2}\circ \Flow_{X_{\sigma}}^{1}(p, q): I_{\epsilon_{4}}\times \Circle \rightarrow \R^{3}.
\end{equation*}
The Moser argument shows that $\phi_{3}^{*}d\alpha_{std} = dp\wedge dq$ as desired. Moreover our conditions on $X_{\sigma}$ imply that $\Flow_{X_{\sigma}}^{1}$ must agree with the identity mapping up to third order along $\{ p = 0 \}$. Hence we can continue to write $\phi_{3}^{*}\alpha_{std} = pdq + \sigma$ for some $\sigma$ which vanishes up to second-order in $p$. Using $\phi_{3}^{*}d\alpha_{std} = dp\wedge dq$, we know that $\sigma$ is closed, and since it mush vanish along $\{ p=0 \}$, we conclude that it is exact. Hence $\phi_{3}^{*}\alpha_{std} = pdq + dh_{4}$ for some $h_{4} \in \Cinfty(I_{\epsilon_{3}}\times\Circle, \R)$. Possible restricting to some $I_{\epsilon_{4}}\times\Circle$, we define
\begin{equation*}
    \phi_{4}(p, q) = \phi_{3}(p, q) - (0, 0, h_{4})
\end{equation*}
so that $\phi_{4}$ is an embedding whence $\phi_{4}^{*}\alpha_{std} = p dq$. Now define
\begin{equation*}
\Phi_{i}(z, p, q) = \phi_{4}(p, q) + (0, 0, z).
\end{equation*}
Restricting to some $I_{\epsilon_{0}}\times I_{\epsilon_{0}}\times\Circle$ for $\epsilon_{0}$ sufficiently small, we can ensure that $\sqcup \Phi_{i}$ is an embedding. By construction of the $\Phi_{i}$, we have
\begin{equation*}
\Phi^{*}_{i}\alpha_{std} = dz + p dq.
\end{equation*}

Regarding the formula for $J_{0}$ in the basis $(P, Q)$, note that this is satisfied for the map $\phi_{1}$ and that subsequent perturbations -- $\phi_{2}, \phi_{3}, \phi_{4}$ -- preserve $(P, Q)$ up to second order in $p$. The $z$-invariance of $J_{0}$ is clear from the definition of the $\Phi_{i}$ and $z$-invariance of the almost complex structure on $\xi_{std}$. 

For the last condition stated in the proposition, we note that $\phi_{1}$ produces the desired result by definition of the function $h_{1}$. As all other required conditions are satisfied by $\phi_{1}$ where the last condition is required to be satisfied as per Properties \ref{Properties:CrossingNeighborhoods}. The perturbations of $\phi_{1}$ carried out in the remainder of the proof are trivial where this condition is required to be satisfied. Indeed, near the endpoints of chords $h_{2}$ (used to define $\phi_{2}$), $\sigma$ (used to define $\phi_{3}$), and $h_{4}$ (used to define $\phi_{4}$) all vanish.
\end{proof}

\begin{assump}\label{Assump:StandardNeighborhood}
We assume throughout the remainder of this article that the Legendrian link $\Lambda$ is in good position, has unit arc-length with respect to $g_{\R^{3}}$, and write
\begin{equation*}
N_{\epsilon} = \cup_{i} N_{\epsilon, i}
\end{equation*}
for a neighborhood of $\Lambda$ as described in the above proposition with $\epsilon \leq \epsilon_{0}$. We call the set $\{z = 0 \} \subset N_{\epsilon, i}$ the \emph{ribbon} of $\Lambda_{i}$. From the above proof, we may assume that the image of the projection of the ribbon of $\Lambda_{i}$ to the $xy$-plane coincides with the image of the projection of $N_{\epsilon, i}$.
\end{assump}

\subsection{Transverse push-offs}

The boundary of the ribbon of a component $\Lambda_{i}$ of $\Lambda$ consists of two knots
\begin{equation}
    T^{+}_{i, \epsilon} = \{ z=0, p=\epsilon \}\quad T^{-}_{i, \epsilon} = \{ z=0, p=-\epsilon \}.
\end{equation}

\begin{defn}\label{Def:Pushoff}
With $\epsilon$ fixed, the knots $T^{+}_{i, \epsilon}$ and $T^{-}_{i, \epsilon}$ will be called the \emph{positive and negative transverse push-offs} of $\Lambda_{i}$. We orient both of these knots so that $\partial_{q} > 0$ in the coordinate system on $N_{\epsilon, i}$.
\end{defn}

The positive (negative) transverse push-off is positively (negatively) transverse to $\xi_{\Lambda}$. Because these knots live on the boundary of $N_{\epsilon, i}$, we may consider them as living within either $\SurgLxi$ or $\Rthree$.

\subsection{Model Dehn twists}\label{Sec:ModelDehnTwists}

In this and the following subsection we describe contact forms on $\SurgL$ which will facilitate analysis on Reeb orbits after contact $\pm 1$-surgery. We begin by providing an explicit model for a Dehn twist and then describe the gluing map used to define contact $\pm 1$-surgery explicitly.

Provided a smooth function $f: \R \rightarrow \Circle$, we define $\tau_{f} \in \Diff^{+}(\R \times \Circle)$ by
\begin{equation*}
    \tau_{f}(p, q) = (p , q + f(p)).
\end{equation*}
and note that $\tau_{-f} = \tau^{-1}_{f}$. We'll call the map $\tau_{f}$ a \emph{positive (resp. negative) Dehn twist by $f$} if:
\be
\item the derivative of $f$ has compact support in $\R$.
\item $\int_{\R}\frac{\partial f}{\partial p}dp = -1$ (resp. $+1$).
\ee
A positive (resp. negative) Dehn twist by $f$ is a positive (resp. negative) Dehn twist in the usual sense of the expression. We compute
\begin{equation}\label{Eq:TwistBetaPullback}
    \tau_{f}^{\ast}p dq = p dq + p\frac{\partial f}{\partial p}dp,\quad \tau_{f}^{*}(dp\wedge dq) = dp\wedge dq
\end{equation}
so that $\tau_{f}$ is always a symplectomorphism with respect to $dp\wedge dq$ but does not preserve $p dq$ unless $f$ is constant. For any $f$ and $\epsilon > 0$ we write
\begin{equation*}
    f_{\epsilon}(p) = f\left(\frac{p}{\epsilon}\right).
\end{equation*}

\begin{assump}\label{Assump:TwistProperties}
Throughout the remainder of this paper, $f$ will denote a function for which $\tau_{f}$ is a negative Dehn twist whose derivative $\frac{\partial f}{\partial p}$ is
\be
\item non-negative,
\item an even function of $p$,
\item supported on $I_{1}=[-1, 1]$, and
\item bounded in absolute value point-wise by $1$.
\ee
\end{assump}

\begin{figure}[h]\begin{overpic}[scale=1.0]{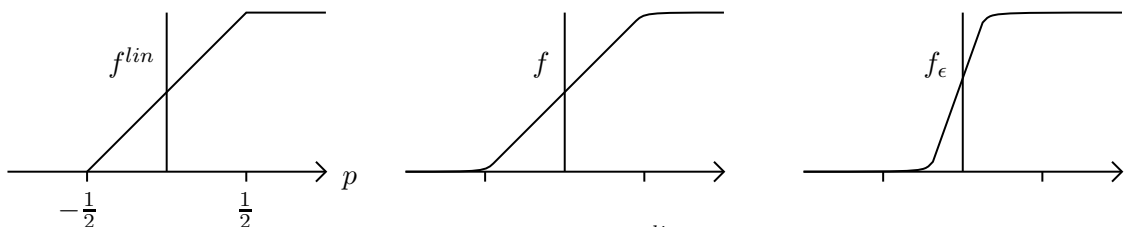}
\put(30, 0){$p$}
\put(4.5, -3){$-\half$}
\put(20.5, -3){$\half$}
\put(9, 10){$f^{lin}$}
\put(47, 10){$f$}
\put(82, 10){$f_{\epsilon}$}
\end{overpic}
\caption{The functions $f^{lin}$, $f$, and $f_{\epsilon}$.}
\label{Fig:Flin}
\end{figure}

We think of $f$ as being a smooth approximation to a piece-wise linear function $f^{lin}$ defined
\begin{equation}\label{Eq:fLinearApprox}
f^{lin}(p) =
  \begin{cases}
    0, & p\in \left(-\infty, -\half\right] \\
    p + \half, & p \in I_{\half} \\
    1 & p \in \left[\half, \infty\right).
  \end{cases}
\end{equation}
See Figure \ref{Fig:Flin}.

The following proposition gathers some properties of the deviation of twists by $f_{\epsilon}$ from preserving $p dq$ as described in Equation \eqref{Eq:TwistBetaPullback}.

\begin{prop}\label{Prop:TwistProperties}
Suppose that $f$ satisfies Assumptions \ref{Assump:TwistProperties} and for $\epsilon \in (0, 1)$ define
\begin{equation*}
    H_{\epsilon}(p) = \int_{-\infty}^{p}P\frac{\partial f_{\epsilon}}{\partial p}(P)dP.
\end{equation*}
Then $H$ is well-defined, is zero on the complement of $I_{\epsilon}$, symmetric, and satisfies $-\epsilon \leq H_{\epsilon} \leq 0$ point-wise.
\end{prop}

\begin{proof}
The first two statements are clear from the compact support and symmetry of the derivative of $f$. Then using the fact that $\frac{\partial f_{\epsilon}}{\partial p}$ is supported on $I_{\epsilon}$ we have
\begin{equation*}
|H_{\epsilon}(p)| = \frac{1}{\epsilon}\left|\int_{-\epsilon}^{p}P\frac{\partial f}{\partial p}\left(\frac{P}{\epsilon}\right)dP\right| \leq \frac{1}{\epsilon}\sup_{p}\left|\frac{\partial f}{\partial p}\right|\int_{-\epsilon}^{\epsilon}|P|dP = \epsilon\sup_{p}\left|\frac{\partial f}{\partial p}\right| \leq \epsilon.
\end{equation*}

For $p \leq 0$, $H_{\epsilon}(p)$ is an integral of a non-positive function and so must be non-positive. Then for $p \geq 0$
\begin{equation*}
H_{\epsilon}(p) = H_{\epsilon}(0) + \int_{0}^{p}P\frac{\partial f_{\epsilon}}{\partial p}(P)dP = H_{\epsilon}(0) - \int_{-p}^{0}P\frac{\partial f_{\epsilon}}{\partial p}(P)dP = H_{\epsilon}(-p)
\end{equation*}
by the symmetry of the derivative of $f$.
\end{proof}

\subsection{Gluing maps}\label{Sec:GluingMaps}

Now we define the gluing maps to define contact surgery on $\Lambda$ and contact forms $\alpha_{\epsilon}$ on the surgered manifold $\SurgL$. 

Let $\epsilon_{0}$ be a sufficiently small as described in Proposition \ref{Prop:NeighborhoodConstruction} and choose $\epsilon \in (0, \epsilon_{0})$. We decompose a neighborhood of each $\partial N_{\epsilon, i}$ into top, side, and bottom pieces as shown in Figure \ref{Fig:TopBottomSide}:
\be
\item $T_{\delta, \epsilon} = \{ z \geq \epsilon - \delta \}$,
\item $S_{\delta, \epsilon} = \{ |p| \geq \epsilon - \delta \}$,
\item $B_{\delta, \epsilon} = \{ z \leq -\epsilon + \delta \}$.
\ee
To perform contact surgery along $\Lambda_{i}$ with surgery coefficient $c_{i}$ we define a map $\phi_{c_{i}, f, \epsilon, \delta}$ in coordinates $(z, p, q)$ as follows: 
\begin{equation}\label{Eq:GluingMap}
\phi_{c, f, \epsilon, \delta}(z, p, q) \sim
  \begin{cases}
    (z - c_{i} H_{\epsilon}(p), p, q + c f_{\epsilon}(p)), & \text{along}\ T_{\delta, \epsilon} \\
    (z, p, q), & \text{along}\ S_{\delta, \epsilon} \cup B_{\delta, \epsilon}
  \end{cases}
\end{equation}
where $H_{\epsilon}$ is as described in Proposition \ref{Prop:TwistProperties}. Due to the properties of $f_{\epsilon}$ and $H_{\epsilon}$ described in the previous section, we have that $\phi_{c, f, \epsilon, \delta}$ agrees on the overlaps of the top, bottom, and sides of $N_{\epsilon, i}$ for $\delta$ sufficiently small. Therefore the map determines a smooth gluing.

\begin{figure}[h]\begin{overpic}[scale=.8]{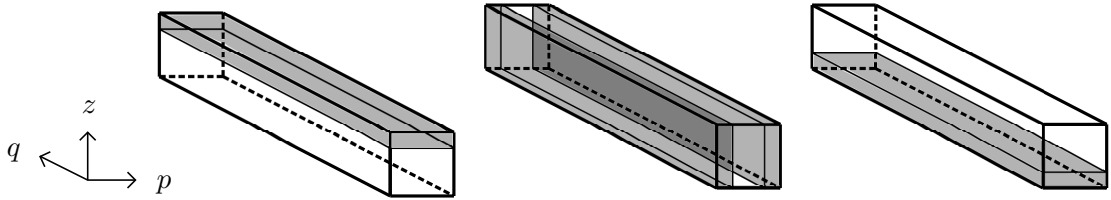}
\put(11, 1){$p$}
\put(-3, 4){$q$}
\put(4, 8){$z$}
\end{overpic}
\caption{From left to right, the top, side, and bottom pieces of our neighborhood are shaded.}
\label{Fig:TopBottomSide}
\end{figure}

The tangent map of the gluing map is given
\begin{equation}\label{Eq:TPhi}
T\phi_{c_{i}, f, \epsilon, \delta} = \partial_{z}\otimes\left(dz - c_{i} p \frac{\partial f_{\epsilon}}{\partial p}dp\right) + \partial_{p}\otimes dp + \partial_{q}\otimes\left(dq + c_{i} \frac{\partial f_{\epsilon}}{\partial p}dp\right)
\end{equation}
along $T_{\delta, \epsilon}$ and $T\phi_{c_{i}, f, \epsilon, \delta}=\Id$ along $S_{\delta, \epsilon}\cup B_{\delta, \epsilon}$ so that
\begin{equation*}
\phi_{c_{i}, f, \epsilon, \delta}^{*}(dz + pdq) = dz + pdq.
\end{equation*}
The gluing map therefore determines a contact form $\alpha_{c_{i}, f, \epsilon, \delta}$ on the manifold $\R^{3}_{\Lambda_{i}}$ obtained by performing the surgery and hence a contact structure $\SurgXi = \ker(\alpha_{c_{i}, f, \epsilon, \delta})$ on this manifold. Shrinking $\delta$ amounts to a restriction of the domain of the map and so does not affect the associated contact manifold.

\begin{defn}\label{Def:AlphaEpsilon}
For $\epsilon \in (0, \epsilon_{0})$ we write $\alpha_{\epsilon}$ for the contact form on $\SurgL$ determined by performing surgery using the gluings $\phi_{c_{i}, f, \epsilon, \delta}$ as described in Equation \eqref{Eq:GluingMap} to each connected component $N_{\epsilon, i}$ of $N_{\epsilon}$. The Reeb vector field of $\alpha_{\epsilon}$ will be denoted $R_{\epsilon}$.
\end{defn}

\section{Chords-to-orbits and chords-to-chords correspondences}\label{Sec:Orbits}

In this section we study the dynamics of the Reeb vector fields $R_{\epsilon}$ for the contact forms $\alpha_{\epsilon}$ for $\SurgLxi$ as described in Definition \ref{Def:AlphaEpsilon}. Our results are summarized by the following:

\begin{thm}\label{Thm:ChordOrbitCorrespondence}
There exist one-to-one correspondences between:
\be
\item Closed orbits of $R_{\epsilon}$ in $\SurgLxi$ and cyclic words of chords on $\LambdaPM \subset \Rthree$.
\item Chords of $R_{\epsilon}$ with boundary on $\LambdaZero \subset \SurgLxi$ and words of chords with boundary on $\LambdaZero \subset \Rthree$.
\ee
\end{thm}

A description of the correspondences will be given below.

\begin{defn}
Via the above theorem, we use the notation
\begin{equation*}
(r_{j_{1}}\cdots r_{j_{n}})
\end{equation*}
to denote either a closed orbit of $R_{\epsilon}$ or a chord of $\LambdaZero \subset \SurgLxi$ whose underlying word is $r_{j_{1}}\cdots r_{j_{n}}$.
\end{defn}

After establishing Theorem \ref{Thm:ChordOrbitCorrespondence}, we estimate the actions of chords and closed orbits in $\SurgLxi$ in Section \ref{Sec:ActionEstimates}. Then in Section \ref{Sec:Embeddings}, we describe equations whose solutions determine the embeddings of closed Reeb orbits and allow exact calculation of their actions. While we do not provide a closed form solutions to these equations, their analysis provides the following

\begin{thm}\label{Thm:Mod2CZ}
For each $n > 0$, there exists $\epsilon_{n}$ such that for all $\epsilon \leq \epsilon_{n}$ all orbits $\gamma$ of word length $\leq n$ are hyperbolic with
\begin{equation*}
\CZ_{2}(\gamma) = \sum_{k=1}^{n} \left(\rot_{j_{k},j_{k+1}} +\  \delta_{1, c_{j_{k}}^{+}}\right) \in \Z/2\Z.
\end{equation*}
Moreover, if either $\LambdaPlus = \emptyset$ or $\LambdaMinus = \emptyset$, then all closed orbits of $R_{\epsilon}$ are hyperbolic for all $\epsilon < \min\{\half, \epsilon_{0}\}$.
\end{thm}

Throughout this section $\gamma$ will denote a closed orbit of $R_{\epsilon}$ and $\kappa$ will denote a chord of $\LambdaZero \subset \SurgLxi$.

\subsection{Overlapping rectangles}\label{Sec:OverlappingRectangles}

In order to state our chord-to-orbit and chord-to-chord correspondences we need to introduce the objects which will define them -- embedded squares in $\SurgLxi$ which record the positions of Reeb orbits as they propagate through the manifold. Along the way, we slightly refine the specifications of the function $f$ in our surgery construction so as to reduce our analysis of dynamics of $R_{\epsilon}$ to analysis of affine linear transformations.

With $\epsilon$ -- the constant which governs the size of $N_{\epsilon}$ -- sufficiently small, the projection of $N_{\epsilon}$ to the $xy$-plane will have overlaps only at rectangles centered about double points of the Lagrangian projection of $\Lambda$. There is a unique rectangle $\rect_{j} \subset \R^{2}$ for each chord $r_{j}$. As per Assumptions \ref{Assump:StandardNeighborhood} and Properties \ref{Properties:CrossingNeighborhoods}, each $\rect_{j}$ is the image of a map of the form
\begin{equation*}
   (p, q) \mapsto (x_{0} + p - q, y_{0} + p + q) 
\end{equation*}
for $(p, q) \in I_{\epsilon_{1}}\times I_{\epsilon_{2}}$ for some $\epsilon_{i}\in (0, \infty)$ with $(x_{0}, y_{0}) \in \R^{2}$ being the coordinates of the double point of $\Lambda$ in the $xy$-plane corresponding to $r_{j}$.

We write $\rect^{ex}_{j}$ for the lift of this disk to the top of the $N_{\epsilon, l^{-}_{j}}$ and $\rect^{en}_{j}$ for the lift to the bottom of $N_{\epsilon, l^{+}_{j}}$. The superscripts are indicative of the fact that closed orbits of $R_{\epsilon}$ enter $N_{\epsilon}$ through the $\rect^{en}_{j}$ and exit $N_{\epsilon}$ through the $\rect^{ex}_{j}$. See Lemma \ref{Lemma:ExitEntry} below.

Again using Assumptions \ref{Assump:StandardNeighborhood} and Properties \ref{Properties:CrossingNeighborhoods}, we have that each $\rect^{en}_{j}, \rect^{ex}_{j}$ can be described as 
\begin{equation}\label{Eq:DDefinition}
    \{ z = \pm\epsilon,\ q \in [q^{\pm}_{j} - \delta, q^{\pm}_{j} + \delta]\}
\end{equation}
for some $\delta \in (0, \infty)$ with respect to the coordinates $(z,p,q)$ provided by Proposition \ref{Prop:NeighborhoodConstruction} on the ``outside'' of the surgery handle.

\begin{figure}[h]\begin{overpic}[scale=.8]{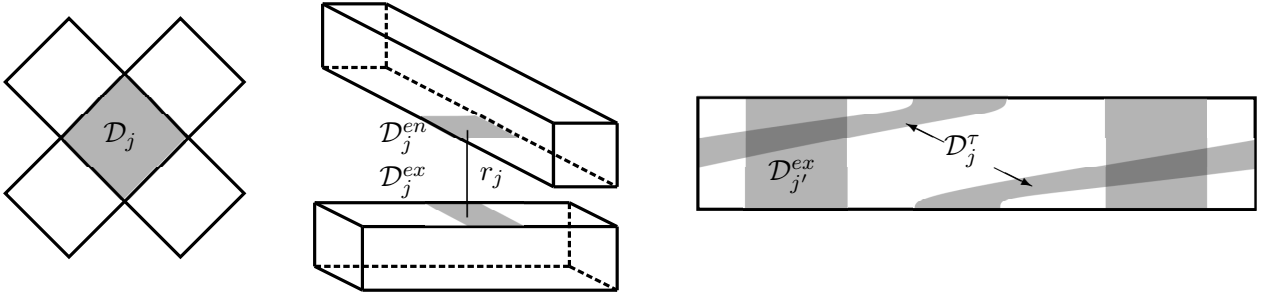}
    \put(8, 12){$\rect_{j}$}
    \put(30, 12){$\rect^{en}_{j}$}
    \put(30, 8.5){$\rect^{ex}_{j}$}
    \put(38, 9){$r_{j}$}
    \put(61, 9){$\rect^{ex}_{j'}$}
    \put(75, 11){$\rect^{\tau}_{j}$}
    \put(75, 12){\vector(-2, 1){3}}
    \put(79, 10){\vector(2, -1){3}}
\end{overpic}
\caption{On the left, we see the $xy$-projection of the ribbon of $\Lambda$ overlapping at a rectangle $\rect_{j}$. In the middle -- with a slightly offset point of view -- we see the $\rect^{\ast}_{j}$ touching the endpoints of a chord $r_{j}$. Here the boxes represent portions of $N_{\epsilon}$. On the right, we see $\tau_{f_{\epsilon}}$ applied to one rectangle intersecting other rectangles. In this portion of the diagram $\partial_{p}$ points upward and $\partial_{q}$ points to the left.}
\label{Fig:OverlappingRectangles}
\end{figure}

If we flow $\rect^{en}_{j}$ through the surgery handle in which it is contained, we will see it pass through the top $\{ z=\epsilon \}$ in a set $\rect^{\tau}_{j}$ which when projected onto the $(p, q)$ coordinates is of the form
\begin{equation*}
    \rect^{\tau}_{j} = \tau^{c_{j}^{+}}_{f_{\epsilon}}(\{ q \in [q_{0} - \delta, q_{0} + \delta]\})
\end{equation*}
for some $q_{0} \in \Circle$ and $\delta > 0$. This set will intersect the each $\rect^{ex}_{j'}$ for $j'\neq j$ in a connected set diffeomorphic to a square. These intersections are depicted as the dark gray regions in the right-hand side of Figure \ref{Fig:OverlappingRectangles}.

\begin{assump}\label{Assump:DiskIntersection}
For a fixed $\epsilon$, we refine our choice of $f$ in Assumptions \ref{Assump:TwistProperties} so that it is affine with derivative equal to $1$ on a some $I_{1-\delta} \subset I_{1}$ with $\delta$ chosen sufficiently small so that each $\rect^{\tau}_{j} \cap \rect^{ex}_{j'}$ with $j\neq j'$ is determined by a pair of linear inequalities
\begin{equation*}
    \rect^{\tau}_{j} \cap \rect^{ex}_{j'} = \{ q \in [q_{0} - \delta_{1}, q_{0} + \delta_{1}],\  a + b q \in [\delta_{2}, \delta_{3}] \}
\end{equation*}
for constants $a, b, \delta_{1},\delta_{2}, \delta_{3}$.
\end{assump}

\begin{properties}\label{Properties:DiskIntersection}
Under Assumptions \ref{Assump:DiskIntersection}, we have that at any point $(p, q) \in \rect^{en}_{j}$ for which $\tau_{f_{\epsilon}}(p, q) \in \rect^{ex}_{j'}$ then
\begin{equation*}
\frac{\partial f_{\epsilon}}{\partial p}(p) = \frac{1}{\epsilon}\quad H_{\epsilon}(p) = H_{\epsilon}(0) + \frac{p^{2}}{2\epsilon}
\end{equation*}
where $i = l^{+}_{j}$. At such points we can write $\tau^{c_{i}}_{f_{\epsilon}}$ as
\begin{equation*}
(p, q) \mapsto \left(p, q + \half + \frac{c_{i}p}{\epsilon}\right).
\end{equation*}
\end{properties}

\subsection{Cyclic words from Reeb orbits}

Here we prove the easy part of the of the (cyclic words) $\leftrightarrow$ (closed orbits) correspondence, showing that each $\gamma$ uniquely determines a cyclic word of chords on $\LambdaPlus \cup \LambdaMinus$.

\begin{lemma}\label{Lemma:ExitEntry}
Any closed orbit $\gamma$ of $R_{\epsilon}$ must pass through $N_{\epsilon}$. Every time $\gamma$ enters $N_{\epsilon}$, it must pass through some $\rect^{en}_{j}$ and every time it exists $N_{\epsilon}$ it must pass through some $\rect^{ex}_{j}$.
\end{lemma}

\begin{proof}
The Reeb vector field $R_{\epsilon}$ agrees with $\partial_{z}$ on the complement of $N_{\epsilon}$ and flows $\rect^{ex}_{j}$ into $\rect^{en}_{j}$. The orbit $\gamma$ must pass through $N_{\epsilon}$ as otherwise $z(\gamma)$ would take on arbitrarily large values implying that $\gamma$ is not closed. If when passing through some component $N_{\epsilon, i}$ of the surgery handles $\gamma$ exits the top of $N_{\epsilon, i}$ in the complement of the $\rect^{ex}_{j}$ then again $z(\gamma)$ would tend to $\infty$ as we follow the trajectory of the orbit. Likewise, if $\gamma$ enters some $N_{\epsilon, i}$ in the complement of the $\rect^{en}_{j}$ then following the orbit backwards in time, we see that $z(\gamma)$ is unbounded from below.
\end{proof}

Then $\gamma$ must intersect some non-empty finite collection of the $\rect^{ex}_{j}$. Let $j_{1}\cdots j_{n}$ be the indices of the $\rect^{ex}_{j}$ through which $\gamma$ passes, ordered in accordance with a parameterization of $\gamma$. 

\begin{defn}
We define the \emph{cyclic word map} as
\begin{equation*}
    \cycword(\gamma) = r_{j_{1}}\cdots r_{j_{n}}
\end{equation*}
and write $\wl(\gamma)$ for the word length of $\cycword(\gamma)$.
\end{defn}

\subsection{Reeb orbits from cyclic words}

In this section we describe how a cyclic word of composable Reeb chords uniquely determines an closed orbit of $R_{\epsilon}$. Let $r_{j_{1}}\cdots r_{j_{n}}$ be a cyclic word and consider the squares $\rect^{\ast}_{j_{k}}, \ast = en, ex, \tau$ as described in the previous subsection.

\begin{figure}[h]\begin{overpic}[scale=.7]{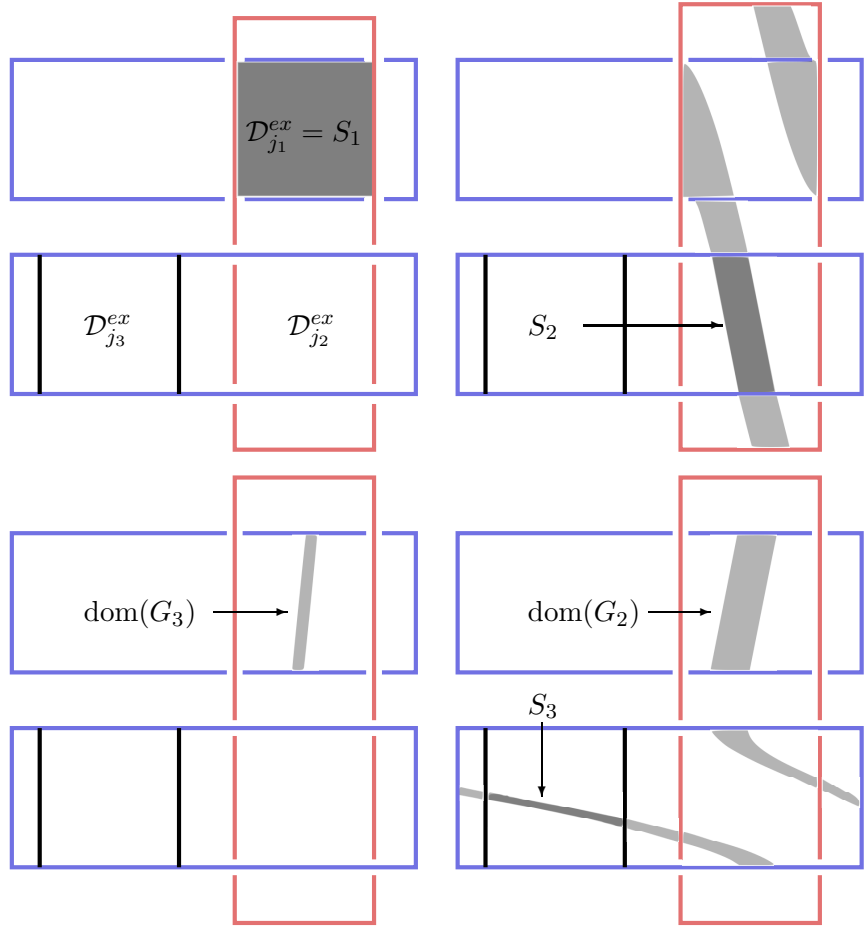}
        \put(25.5, 85){$\rect^{ex}_{j_{1}} = S_{1}$}
        \put(30, 64){$\rect^{ex}_{j_{2}}$}
        \put(8, 64){$\rect^{ex}_{j_{3}}$}
        \put(56, 64){$S_{2}$}
        \put(62, 65){\vector(1,0){15}}
        \put(56, 33){$\dom(G_{2})$}
        \put(69, 34){\vector(1, 0){6.75}}
        \put(8, 33){$\dom(G_{3})$}
        \put(22,34){\vector(1,0){8}}
        \put(56, 23){$S_{3}$}
        \put(57.5,22){\vector(0,-1){8}}
    \end{overpic}
    \vspace{5mm}
\caption{Following the sub-figures clockwise we see the sets $S_{k}$ and $\dom(G_{k})$ drawn schematically. The $S_{k}$ are shaded dark gray. Each rectangle represents the ribbon of some component of $\Lambda$ cut at some value of $q$, layered over each other as indicated by the crossings so that the value of $z$ increases as we traverse each sub-figure clockwise. Within each rectangle, the sides of shorter (longer) length are directed by $p$ ($q$, respectively). In the top-right we see $F^{h}_{1}\circ F^{co}_{1}(S_{1})$ as a subset of the top of $N_{\epsilon, l^{+}_{j_{1}}} = N_{\epsilon, l^{-}_{j_{2}}}$. Taking the intersection of this set with $\rect^{ex}_{j_{2}}$ determines $S_{2}$.}
\label{Fig:IterativeDynamics}
\end{figure}

\begin{thm}\label{Thm:ChordsToOrbits}
For $\epsilon \leq \epsilon_{0}$ and each word $w = r_{j_{1}}\cdots r_{j_{n}}$, there is a unique closed Reeb orbit $\gamma_{w}$ of $R_{\epsilon}$ for which $\cycword(\gamma_{w}) = w$.
\end{thm}

Our logic follows directly from arguments in \cite[Section 6.1]{BEE:LegendrianSurgery} -- carried out in detail in \cite{Ekholm:SurgeryCurves} -- which are simplified by our reduction of dynamics to that of affine transformations in Section \ref{Sec:OverlappingRectangles}.

\begin{proof}
The proof follows from an analysis of $\Flow_{R_{\epsilon}}^{t}$ applied to the disk $\rect^{ex}_{j_{1}}$. Recall that this disk is contained in the ``top'' of a surgery handle $N_{\epsilon, l^{-}_{j_{1}}}$. 

Write $S_{1} = \rect^{ex}_{j_{1}}$ and let $G_{1}=\Id_{S_{1}}$. Consider the following iterative process for which Figure \ref{Fig:IterativeDynamics} serves as a visual aid:
\be
\item \textbf{Flow through the handle complement}: There is a function $t(p, q)$ solving for the minimal $t > 0$ such that $\Flow_{R_{\epsilon}}^{t(p, q)}$ applied to $(p, q)\in S_{1}$ is an element of the square $\rect^{en}_{j_{1}}$ directly above $S_{1}$. Write $F^{co}_{1}(p, q)=\Flow_{R_{\epsilon}}^{t(p, q)}(p, q)$ whose image is the square $S_{1}'$, which is contained in the bottom of $N_{\epsilon, l^{+}_{j_{1}}}$. By the results of Section \ref{Sec:OverlappingRectangles}, $S'_{1}=\rect^{en}_{j_{1}}$. Briefly, $F^{co}_{1}$ is the flow of our square $\rect^{en}_{j_{1}}$ through the handle complement.
\item \textbf{Flow through the handle}: Similarly define a function $F^{h}_{1}$ which flows $S'_{1} \subset \{ z=-\epsilon \} \subset N_{\epsilon, l^{+}_{j_{1}}}$ up to the top, $\{z=\epsilon\}$, of the surgery handle. The square $F^{h}(S'_{1})$ will appear in the coordinates $(z, p, q)$ on the ``outside'' of the surgery handle as the application of a (positive or negative) Dehn twist to $S'_{1}$. That is, in the notation of Section \ref{Sec:OverlappingRectangles}, $F^{h}_{1}(S'_{1})=\rect^{\tau}_{j_{2}}$ is the flow through the handle.
\item \textbf{Trim}: We write $S_{2}=F^{h}_{1}(S'_{1}) \cap \rect^{ex}_{j_{2}}$ for the intersection of $F^{h}_{1}(S'_{1})$ with the next square in the sequence $\rect^{ex}_{j_{k}}$ determined by $w$. Then $S_{2}$ is contained in the top of $N_{\epsilon, l^{-}_{j_{2}}}$. We get a diffeomorphism $G_{2}$ from $\dom(G_{2}) = (F^{h}_{1}\circ F^{co}_{1})^{-1}(S_{2})\subset S_{1}$ to $\im(G_{2}) = S_{2}$ by $G_{2} = F^{h}_{1}\circ F^{co}_{1}$.
\item \textbf{Repeat}: We now inductively repeat the process by applying it to $S_{k} \subset \rect^{ex}_{j_{k}}$. We analogously define $F^{co}_{k}, F^{h}_{k}$ with domain $S_{k}$ then apply $F^{h}_{k}\circ F^{co}_{k}$ to flow $S_{k}$ up through the next handle in the sequence $N_{\epsilon, l^{+}_{j_{k}}}$ whose image we trim to define $S_{k+1}$. This determines a diffeomorphism 
\begin{equation*}
G_{k+1} = F^{h}_{k}\circ F^{co}_{k}\circ \cdots F^{h}_{1}\circ F^{co}_{1}: (\dom(G_{k+1}) \subset S_{1}) \rightarrow (S_{k+1} \subset \rect^{ex}_{j_{k+1}}).
\end{equation*}
\ee

Making use of the results in Section \ref{Sec:OverlappingRectangles} we have the following collection of observations:
\be
\item Each $F^{co}_{k}$ -- considered with domain $\rect^{ex}_{j_{k}}$ in which $S_{k}$ is contained -- is an affine transformation with respect to the $(p, q)$ coordinates of the components of $N_{\epsilon}$. Each $F^{co}_{k}$ sends $\rect^{ex}_{j_{k}}$ diffeomorphically to $\rect^{en}_{j_{k}}$ and is a symplectomorphism with respect to $d\alpha_{\epsilon}$.
\item Each $F^{h}_{k}$ -- considered with domain $\rect^{en}_{j_{k}}$ -- is non-linear as can by seen by looking where $p$ is extremal. It is also a symplectomorphism with respect to $d\alpha_{std}$. The restriction of $F^{h}_{k}$ to $(F^{h}_{k})^{-1}(\rect^{ex}_{j})$ for each $j$ is an affine transformation by Properties \ref{Properties:CrossingNeighborhoods}.
\item We see by induction that $S_{k}$ is a connected, non-empty quadrilateral determined by a non-degenerate pair of linear inequalities, one of which is of the form $q \in [q_{0} - \delta, q_{0} + \delta]$.
\item Combining the above -- with the fact that a composition of affine transformations is an affine transformation -- $\dom(G_{k})$ is a quadrilateral determined by a pair of linear inequalities, one of which is the trivial $p \in [-\epsilon, \epsilon]$.
\item Each trimming step monotonically decreases the area with respect to $d\alpha_{\epsilon}$ and for each $k$ we have $\dom(G_{k+1}) \subsetneq \dom(G_{k})$:
\begin{equation*}
    0 < \int_{S_{k+1}}d\alpha_{\epsilon} < \int_{S_{k}}d\alpha_{\epsilon},\quad 0 < \int_{\dom(G_{k+1})}d\alpha_{\epsilon} < \int_{\dom(G_{k})}d\alpha_{\epsilon}.
\end{equation*}
\ee

\begin{figure}[h]\begin{overpic}{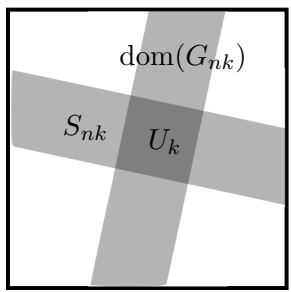}
    \put(50, 50){$U_{k}$}
    \put(20, 55){$S_{nk}$}
    \put(40, 80){$\dom(G_{nk})$}
\end{overpic}
\caption{overlaps of the sets $S_{nk}$ and $\dom(G_{nk})$ within in the set $S_{1}=\rect^{ex}_{j_{1}}$.}
\label{Fig:IterativeOverlaps}
\end{figure}

Now observe that $\dom(G_{n})$ stretches across $\rect^{ex}_{j_{1}}$ in the $p$ direction and that $S_{n}$ stretches across $\rect^{ex}_{j_{1}}$ in the $q$ direction. Since both sets are convex, $U_{1} = \dom(G_{n}) \cap S_{n}$ must be nonempty and convex. We likewise define $U_{k}$ as the intersection of $S_{nk}$ with $\dom(G_{nk})$ for all $k > 0$. See Figure \ref{Fig:IterativeOverlaps}.

The $U_{k}$ satisfy $U_{k+1} \subsetneq U_{k}$ and we claim that $\area(U_{k}) \rightarrow 0$ as $k \rightarrow \infty$. To see this, recall that $f_{\epsilon}$ is linear on $p \in [-\epsilon(1 - \delta), \epsilon(1 - \delta)]$ where $\delta \in (0, 1)$ is as described in Assumptions \ref{Assump:DiskIntersection}. By the conditions which characterize $\delta$, for each $k \geq 0$ the set of points in $S_{k}$ which reach $S_{k+1}$ via the map $F^{h}_{k}\circ F^{co}_{k}$ must be contained in the set $S^{lin}_{k} = \{ p \in [-\epsilon(1 - \delta), \epsilon(1 - \delta)] \} \cap S_{k}$. By the fact that $S^{k}$ is a rectangle stretching across the $p$-coordinate of the annulus, $\area(S^{lin}_{k}) = (1 - \delta)\area(S_{k})$. By the definition of $S^{lin}_{k}$ and the fact that $F^{h}_{k}\circ F^{co}_{k}$ is symplectic, $\area(S_{k+1}) \leq \area(S^{lin}_{k})$. Inductively, we conclude $\area(S_{k}) \leq (1 - \delta)^{k - 1}\area(S_{1})$. Since $U_{k}$ is contained in $S_{nk}$, our claim is established.

By our construction, any Reeb orbit with word $w^{k}$ must intersect $\rect^{ex}_{j_{1}}$ at a point in $\dom(G_{nk})$ which is sent to itself via $G_{nk}$. Hence such a point of intersection must lie in $U_{k}$. By considering multiple covers of the orbit $\gamma_{w}$ -- whose existence we seek to establish -- we see that if such a point of intersection lies in $U_{1}$ then it must lie in $U_{k}$ for all $k > 0$. We therefore define 
\begin{equation*}
U_{\infty}  = \cap_{1}^{\infty} U_{k} \subset \rect^{ex}_{j_{1}}
\end{equation*}
which by our previous analysis consists of a single point. 

To complete our proof, it suffices to show that $G_{nk}(U_{\infty}) = U_{\infty}$ for all $k > 0$. This amounts to unwinding the definitions established in the proof so far. If we write $k = k_{1} + k_{2}$ for any pair of natural numbers $k_{1}, k_{2}$, then we must have
\begin{equation*}
    G_{nk_{1}}(\dom(G_{nk})) \subset \dom(G_{nk_{2}})
\end{equation*}
as otherwise $G_{nk_{2}}\circ G_{nk_{1}}(\dom(G_{nk}))$ would not be contained in $\rect^{ex}_{j_{1}}$. On the other hand, $\dom(G_{nk}) \subsetneq \dom(G_{nk_{1}})$ implies that
\begin{equation*}
    G_{nk_{1}}(\dom(G_{nk})) \subsetneq S_{nk_{1}}.
\end{equation*}
Combining the above two equations we conclude that
\begin{equation*}
    U_{\infty} = \cap_{1}^{\infty}(S_{nk}\cap \dom(G_{nk})) = (\cap_{1}^{\infty}S_{nk}) \cap (\cap_{1}^{\infty}\dom(G_{nk}))
\end{equation*}
satisfies $G_{nk}(U_{\infty})=U_{\infty}$.
\end{proof}

\subsection{Reeb chords of $\LambdaZero$ after surgery}

In this section, we describe open-string versions of our results for closed Reeb orbits, establishing the chords-to-chord correspondence of Theorem \ref{Thm:ChordsToChords}.

\begin{defn}
Suppose that a chord $\kappa$ of $\LambdaZero \subset \SurgLxi$ passes through a sequence of the $\rect^{\ast}_{j}$ of the form
\begin{equation*}
    \rect^{en}_{j_{1}}, \rect^{ex}_{j_{2}}, \rect^{en}_{j_{2}},\dots, \rect^{ex}_{j_{n-1}}, \rect^{en}_{j_{n-1}}, \rect^{ex}_{j_{n}}.
\end{equation*}
Then we write $\word(\kappa) = r_{j_{1}}\cdots r_{j_{n}}$. We call the association $\kappa \mapsto \word(\kappa)$ the \emph{word map}.
\end{defn}

\begin{thm}\label{Thm:ChordsToChords}
For each $\epsilon < \epsilon_{0}$, the word map $\word$ determines a one-to-one correspondence between words of chords with boundary on $\LambdaZero \subset \Rthree$ and Reeb chords of $\LambdaZero \subset \SurgLxi$ determined by the contact form $\alpha_{\epsilon}$. For each such word $w$, the associated chord $\kappa_{w}$ is non-degenerate for all $\epsilon < \epsilon_{0}$.
\end{thm}

The chords with word length $1$ are those which exist for $\LambdaZero \subset \R^{3}$ prior to surgery, while the rest of the chords in Theorem \ref{Thm:ChordsToChords} are created after the performance of surgery along $\LambdaPM$.

\begin{proof}
The proof is analogous to the proof of Theorem \ref{Thm:ChordsToOrbits}, although considerably simpler.

Let $w = r_{j_{1}}\cdots r_{j_{n}}$ be a word of chords on $\LambdaZero$ with word length $n>1$. By Equation \ref{Eq:OvercrossingNeighborhoodModel}, flowing $\LambdaZero$ up to $N_{\epsilon}$ along a chord sends $\LambdaZero$ to a strand in $N_{\epsilon}$ of the form $q=q_{0}$ which we call $A^{'}_{1}$. Flow this arc up to the top of $N_{\epsilon}$ and take its intersection with $\rect^{ex}_{j_{2}}$ to obtain an arc we'll call $A_{1}$. Define arcs $A_{k}$ for $k>1$ as follows:
\be
\item \textbf{Flow through the handle complement}: Flow $A_{k-1} \subset \rect^{ex}_{j_{k}}$ up to $\rect^{en}_{j_{k}}$ using the map $F^{co}_{k}$ as in the proof of Theorem \ref{Thm:ChordsToOrbits}.
\item \textbf{Flow through the handle}: Now we apply the map $F^{h}_{k}$ as defined in Theorem \ref{Thm:ChordsToOrbits} to flow $F^{co}_{k}(A_{k-1})$ up to the top of $N_{\epsilon}$.
\item \textbf{Trim}: Define $A_{k} = F^{h}_{k}\circ F^{co}_{k}(A_{k-1}) \cap \rect^{ex}_{j_{k}}$.
\item \textbf{Repeat}: Repeat the above steps until we obtain an arc $A_{n}\subset \rect^{ex}_{j_{n}}$.
\ee

Again, following the logic of the proof of Theorem \ref{Thm:ChordsToOrbits} using the linearity conditions of Section \ref{Sec:OverlappingRectangles}, each $A_{k}\subset \rect^{ex}_{j_{k}}$ is a line segment which wraps across $\rect^{ex}_{j_{k}}$ in the $q$-direction. In other words, each admits a parameterized of the form:
\begin{equation*}
    A_{k} = \{ (aq + b, q):\ q\in [q_{0}-\delta, q_{0} + \delta] \}
\end{equation*}
for some constants $a\neq 0, b, q_{0}, \delta$. Since flowing $\LambdaZero$ downward to $\rect^{ex}_{j_{n}}$ along the chord $r_{j_{n}}$ is a set of the form $q = q_{0}$, the intersection of this set with $A_{k}$ consists of a single point. Since $A_{k}$ wraps across $\rect^{ex}_{j_{n}}$ in the $q$ direction, we must have that this intersection is transverse. By construction, the collection of such intersections are in one-to-one correspondence with the collection of chords of $\LambdaZero \subset \SurgLxi$.

For words of length $1$, the restriction of $R_{\epsilon}$ to the complement of the surgery handles is $\partial_{z}$ so that words of length $1$ correspond exactly to the chords of $\LambdaZero$ present prior to surgery.
\end{proof}

\subsection{Action estimates}\label{Sec:ActionEstimates}

To obtain refined estimates of the actions of the chords and orbits of $R_{\epsilon}$ we'll need the following lemmas. The first lemma tells us how much time it takes to flow from the top of $N_{\epsilon}$ to the bottom in a neighborhood of a chord $r_{j}$.

\begin{lemma}\label{Lemma:ActionHandleComplement}
Let $r_{j}$ be some chord of $\Lambda \subset \Rthree$ with action $\action(r_{j})$ and parameterize the disk $\rect^{ex}_{j} \subset \partial N_{\epsilon}$ with coordinates $(p, q)$ as in Equation \eqref{Eq:UndercrossingNeighborhoodModel}. Then for each $(p, q) \in \rect^{en}_{j}$, there exists a minimal-length chord from $\rect^{ex}_{j}$ to $\rect^{en}_{j}$ starting at $(P, Q)$ with action
\begin{equation*}
    t= \action(r_{j}) - 2\epsilon - pq.
\end{equation*}
\end{lemma}

\begin{proof}
This is a straight-forward calculation, so we omit the details. For a given $j$, write $(p^{ex}, q^{ex})$ and $(p^{en}, q^{en})$ for the coordinates on $\rect^{ex}_{j}$ and $\rect^{en}_{j}$ provided by Equations \eqref{Eq:UndercrossingNeighborhoodModel} and \eqref{Eq:OvercrossingNeighborhoodModel} respectively. Then $p^{en} = -q^{ex}$ and $q^{en} = p^{ex}$. Plug these into the equations provided to compute the differences in the $z$ coordinates and consider the fact that $R_{\epsilon}=\partial_{z}$ on $\R^{3} \setminus N_{\epsilon}$.
\end{proof}

Our second lemma tells how much time it takes for an orbit to flow through one of the surgery handles.

\begin{figure}[h]\begin{overpic}[scale=2.25]{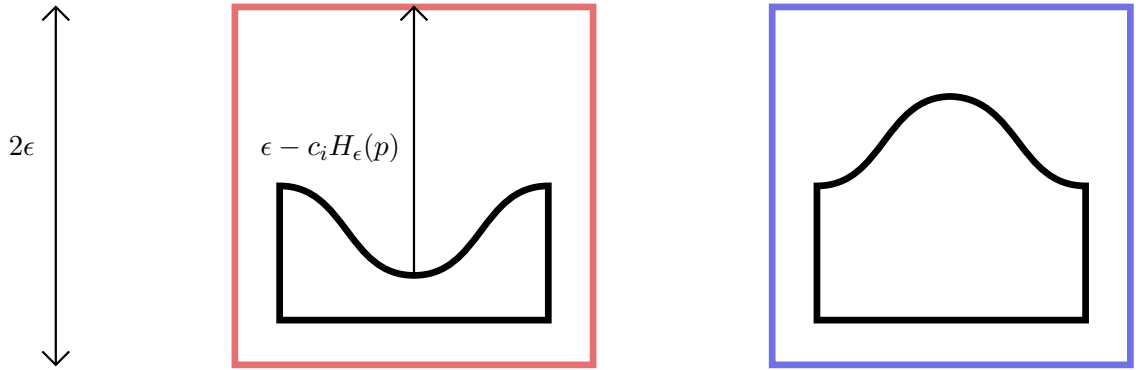}
\put(-3, 20){$2\epsilon$}
\put(20, 20){$\epsilon - c_{i}H_{\epsilon}(p)$}
\end{overpic}
\caption{The squares represent $\{ q = q_{0} \} \subset N_{\epsilon, i}$ slices of the $N_{\epsilon}$ at components $N_{\epsilon, i}$ with surgery coefficient $c_{i}=-1$ (left) and $c_{i}=1$ (right). The black arcs represent the boundaries of the gluing region, as it intersects each slice.}
\label{Fig:NeighborhoodActionDelta}
\end{figure}

\begin{lemma}\label{Lemma:ActionThroughHandle}
For some $j$, again consider coordinates $(p, q)$ on $\rect^{ex}_{j} \subset \partial N_{\epsilon}$ as provided by Equation \eqref{Eq:UndercrossingNeighborhoodModel}. Then the time it takes a point in $\rect^{en}_{j}$ to reach this point via the flow of $R_{\epsilon}$ is
\begin{equation*}
t = 2\epsilon + c_{i}H_{\epsilon}(p).
\end{equation*}
\end{lemma}

This becomes obvious if we look at the graph of the ``top'' part of the gluing map of Equation \eqref{Eq:GluingMap}. See Figure \ref{Fig:NeighborhoodActionDelta}. By comparing Proposition \ref{Prop:TwistProperties} with the definition of the gluing map in Equation \eqref{Eq:GluingMap}, actions increase slightly as we pass through a surgery handle with coefficient $-1$ and decrease slightly as we pass through a surgery handle with coefficient $+1$.

\begin{prop}\label{Prop:ActionEstimate}
For all closed Reeb orbits $\gamma$ of $R_{\epsilon}$, we have
\begin{equation*}
    |\action(\gamma) - \action(\cycword(\gamma))| < 3\epsilon \wl(\gamma).
\end{equation*}
For each chord $r$ of $R_{\epsilon}$ with boundary on $\LambdaZero \subset \SurgLxi$, we have
\begin{equation*}
    |\action(r) - \action(\word(\gamma))| < 3 \epsilon \wl(r).
\end{equation*}
\end{prop}

This is obvious from Lemmas \ref{Lemma:ActionHandleComplement} and \ref{Lemma:ActionThroughHandle} together with Proposition \ref{Prop:TwistProperties}.

\subsection{Calculating orbit embeddings}\label{Sec:Embeddings}

Let $\gamma = (r_{j_{1}}\cdots r_{j_{n}})$ be a closed orbit of some $R_{\epsilon}$. Let $(p_{k},q_{k})$ be coordinates on the squares $\rect^{ex}_{j_{k}}$ described by Equation \eqref{Eq:OvercrossingNeighborhoodModel}. Suppose that in these coordinates, $\gamma$ passes through the points $(P_{k},Q_{k})$. If $\gamma$ is simply covered and we compute the exact values of the $(P_{k},Q_{k})$, then we can see the knot formed by $\gamma$ inside of $\SurgL$ and be able to compute the action $\action(\gamma)$ exactly. In this section, we describe how these $(P_{k},Q_{k})$ can be calculated. The analysis here will be the starting point for the computation of Conley-Zehnder indices.

\begin{comment}
\begin{prop}\label{Prop:ActionSolution}
Suppose that $\gamma = (r_{j_{1}}\cdots r_{j_{n}})$ is a closed orbit of $R_{\epsilon}$ which passes through the $\rect^{ex}_{j_{k}}$ at points $(P_{k}, Q_{k})$ according to the coordinates $(p, q)$ on each $\rect^{en}_{j_{k}}$ provided by Equation \eqref{Eq:OvercrossingNeighborhoodModel}. Then the action of $\gamma$ is
\begin{equation*}
\action(\gamma) = \action(\cycword(\gamma)) + \sum_{k=1}^{n} \big(c_{i_{k}}H_{\epsilon_{i_{k}}}(0) + P_{k}(c_{i_{k}}\frac{P_{k}}{\epsilon_{i_{k}}} - Q_{k})\big)
\end{equation*}
with $i_{k} = l^{+}_{j_{k}}$
\end{prop}
\end{comment}

In the above notation, we can describe $(P_{1}, Q_{1})$ as a fixed point of an affine transformation 
\begin{equation*}
A + b: \R^{2} \rightarrow \R^{2},\quad A \in \SLtwoR, b\in \R^{2}
\end{equation*}
as follows.

Starting at (a subset of) $\rect^{ex}_{j_{k}}$, apply $\Flow_{R_{\epsilon}}$ to pass through the handle complement to $\rect^{en}_{j_{k}}$ and then through the surgery handle $N_{\epsilon, l^{+}_{j_{k}}}$ to $\rect^{ex}_{j_{k+1}}$. As we are only interested in the set of points in $\rect^{ex}_{j_{k}}$ along which the $\tau_{\pm f_{\epsilon}}$ are linear, we can write this as a map $A_{k} + b_{k}$ with $A_{k} \in \SLtwoR$. The $b_{k} \in \R^{2}$ term is required by Equation \eqref{Eq:OvercrossingNeighborhoodModel} centering the $q$ coordinate about the endpoint of the Reeb chord $r_{k+1}$.

Hence we may write
\begin{equation}\label{Eq:AtoAn}
\begin{aligned}
A+b &= (A_{n} + b_{n})\circ \cdots \circ (A_{1} + b_{1})\\
	&= (A_{n}\cdots A_{1}) + (A_{n}\cdots A_{2}) b_{1} + \cdots + A_{n}b_{n-1} + b_{n}
\end{aligned}
\end{equation}
with $(P_{1},Q_{1})$ being the fixed point of this map. By linearity of the equations involved and our prior knowledge (Theorem \ref{Thm:ChordsToOrbits}) that there exists a unique fixed point, we may as well consider the $A_{k} + b_{k}$ to be transformations of $\R^{2}$. We can then solve for $u_{1} = (P_{1},Q_{1})$ as
\begin{equation*}
\begin{aligned}
u_{1} &= (\Id - A)^{-1}b\\
&= (\Id - A_{n}\cdots A_{1})^{-1}\big( A_{n}\cdots A_{2} b_{1} + \cdots + A_{n}b_{n-1} + b_{n})
\end{aligned}
\end{equation*}
Provided $u_{1}$, we can then find the $u_{k} = (P_{k}, Q_{k})$ by applying the $(A_{k} + b_{k})$:
\begin{equation*}
u_{k+1} = A_{k}u_{k} + b_{k} = (A_{k} + b_{k})\cdots (A_{1} + b_{1})u_{1}.
\end{equation*}

\begin{prop}\label{Prop:AnExplicit}
In the above notation,
\begin{equation*}
\begin{gathered}
(-1)^{\rot_{j_{k},j_{k+1}}}A_{k} =  \begin{pmatrix}
0 & -1 \\ 1 & -\frac{c_{j_{k}}}{\epsilon}
\end{pmatrix} = J_{0}\begin{pmatrix}
1 & -\frac{c_{j_{k}}}{\epsilon} \\ 0 & 1
\end{pmatrix},\\
(-1)^{\rot_{j_{k},j_{k+1}}}b_{k} = \begin{pmatrix}
0 \\ \half - d_{j_{k},j_{k+1}}
\end{pmatrix} = J_{0}\begin{pmatrix}
\half - d_{j_{k},j_{k+1}} \\ 0
\end{pmatrix}
\end{gathered}
\end{equation*}
where $d_{j_{k},j_{k + 1}}$ is the minimal length of a capping path for the pair $(r_{j},r_{j_{k+1}})$ projected to the $xy$ plane using the standard Euclidean metric on $\R^{2}$.
\end{prop}

\begin{proof}
We can determine $A_{k} + b_{k}$ as a composition of the following elementary mappings:
\be
\item The change of coordinates from $\rect^{ex}_{j_{k}}$ to $\rect^{en}_{j_{k}}$ which we see when flowing points $(p, q)$ through the handle complement:
\begin{equation*}
(p, q) \mapsto (-q, p).
\end{equation*}
\item The flow from $\rect^{en}_{j_{k}}$ to the top of $N_{\epsilon, l_{j_{k}}^{+}}$ which, according to Properties \ref{Properties:DiskIntersection} is given
\begin{equation*}
(p, q) \mapsto \left(p, q + \half + \frac{c_{j_{k}}^{+}}{\epsilon}p\right)
\end{equation*}
\item A shift in the $q$ coordinate so that $(0, 0)$ is identified with the tail of $r_{j_{k+1}}$. Here $d^{in}_{k,k+1}$ is the magnitude of this shift when $\Lambda_{j_{k}}$ is parameterized with $\Phi_{i^{+}}$ as in Proposition \ref{Prop:NeighborhoodConstruction}:
\begin{equation*}
(p, q) \mapsto (p, q - d^{in}_{k,k+1})
\end{equation*}
\item A mapping of the coordinates on the top of $N_{\epsilon, l^{+}_{j_{k}}}$ to $\rect^{ex}_{k+1}$.
\begin{equation*}
(p, q) \mapsto (-1)^{\rot_{j_{k},j_{k+1}}}(p, q).
\end{equation*}
\ee
The result of composing the above maps produces
\begin{equation*}
(p, q) \mapsto (-1)^{\rot_{j_{k},j_{k+1}}}\left(-q, p + \half - d^{in}_{k, k+1} - \frac{c_{j_{k}}^{+}}{\epsilon}q\right).
\end{equation*}
\end{proof}

\subsection{Hyperbolicity and the $\Z/2\Z$ index}

For a given closed orbit $\gamma = (r_{j_{1}}\cdots r_{j_{n}})$, we can use the above formula to write its Poincaré return map as $\Ret_{\gamma} = A_{n}\cdots A_{1}$ where the $A_{k}$ are given by Equation \eqref{Eq:AtoAn}. By using the calculation of the $A_{k}$ in Proposition \ref{Prop:AnExplicit}, we have an explicit representation of $\Ret_{\gamma}$ as
\begin{equation}\label{Eq:MatrixReturn}
\begin{aligned}
(-1)^{\rot}\Ret_{\gamma} &= \prod_{K=1}^{n}J_{0}\begin{pmatrix}
1 & -c^{+}_{j_{n + 1 - K}}\epsilon^{-1} \\
0 & 1
\end{pmatrix}\\
&= J_{0}\begin{pmatrix}
1 & -c^{+}_{j_{n}}\epsilon^{-1} \\
0 & 1
\end{pmatrix}\cdots J_{0}\begin{pmatrix}
1 & -c^{+}_{j_{1}}\epsilon^{-1} \\
0 & 1
\end{pmatrix}\\
&= J_{0}^{n} + \sum_{K=1}^{n} \left( \sum_{k\in I_{K}} \left(\prod_{i=1}^{K} -c^{+}_{j_{k_{i}}}\right)M_{k}\right) \epsilon^{-K}\\
\rot &= \sum_{K=1}^{n} \rot_{j_{K},j_{K+1}} \\
M_{k} &=  J_{0}^{n-k_{K}}\Diag(0,1)J_{0}^{k_{K} - k_{K-1} -1} \cdots J_{0}^{k_{2} - k_{1} - 1}\Diag(0,1)J_{0}^{k_{1} - 1}\\
I_{K} &= \{ k = (k_{1},\dots,k_{K})\ :\ 1\leq k_{1} < \cdots < k_{K} \leq n \}.
\end{aligned}
\end{equation}
The equality in the third line involving the $M_{k}$ easily follows from an induction on $n$.

Observe that $I_{n}$ consists of a single element $(1,\dots, n)$ so that the $K=n$ term in the above formula is
\begin{equation}\label{Eq:Mn}
\epsilon^{-n}(\prod^{n}_{k=1} -c^{+}_{j_{k}})\Diag(0, 1) = \epsilon^{-n}(-1)^{\#(c^{+}_{j_{k}} = 1)}\Diag(0, 1)
\end{equation}
Thus for a fixed word, $\tr(\Ret_{\gamma})$ can be expressed as a polynomial in $\epsilon^{-1}$ whose highest-order term given by the above expression. 

\begin{proof}[Proof of Theorem \ref{Thm:Mod2CZ}]
For $\epsilon_{w}$ sufficiently small the $\epsilon^{-n}$ term in the polynomials for $\tr(\Ret_{\gamma})$ determines their sign for all $\epsilon < \epsilon_{w}$ and words of length $\leq n$ as there are only finitely many cyclic words less than a given length. Possibly making $\epsilon_{w}$ smaller, we can guarantee that the absolute values of the traces are bounded below by $2$. To compute $\CZ_{2}$ we apply Equations \eqref{Eq:DetTr} and \eqref{Eq:CZTwo} noting that $\det (\Ret - \Id) = 2 - \tr(\Ret)$.

If one of $\LambdaPlus$ or $\LambdaMinus$ is empty, then each orbit of word length $n$ has return map
\begin{equation*}
\Ret_{\gamma} = \pm M_{a}^{n},\quad M_{a} = \begin{pmatrix}
0 & -1 \\ 1 & a
\end{pmatrix},\quad a = \pm \epsilon^{-1}.
\end{equation*}
If $\epsilon < \half$, then $M_{a}$ is hyperbolic and so is conjugate to $\Diag(\lambda, \lambda^{-1})$ with 
\begin{equation*}
\lambda = \half (a + \sqrt{a^{2} - 4}),\quad \lambda^{-1} = \half (a - \sqrt{a^{2} - 4})
\end{equation*}
implying that $\Ret_{\gamma}$ is conjugate to $\pm\Diag(\lambda^{n}, \lambda^{-n})$. In this case, it's clear that $|\tr(\Ret_{\gamma})| > 2$ independent of $n$, implying that all closed orbits of $R_{\epsilon}$ are hyperbolic for $\epsilon < \half$.
\end{proof}

\section{The semi-global framing $(X, Y)$}\label{Sec:Framing}

Having computed the $\Z/2\Z$ Conley-Zehnder indices of the closed Reeb orbits of $R_{\epsilon}$ we now seek to compute $\Z$-valued indices with respect to a framing as well as Maslov indices of broken closed strings on $\LambdaZero \subset \SurgLxi$. 

In this section we describe sections of $\SurgXi$ which we will later use to compute these indices. This will allow us to draw a cycle representing $\PD(c_{1}(\SurgXi)) = \PD(e(\SurgXi))$ as a link in the Lagrangian projection: See Figure \ref{Fig:Rot1Unknot} for an example. The results of this section are summarized as follows:

\begin{thm}\label{Thm:FramingSummary}
For each $\epsilon < \epsilon_{0}$ there are sections $X, Y \in \Gamma(\xi_{\Lambda})$ such that the following conditions hold:
\be
\item $(X, Y) = (\partial_{x}+y\partial_{z}, \partial_{y})$ on $\SurgL \setminus N_{\epsilon} \simeq \R^{3}\setminus N_{\epsilon}$.
\item $(X, Y)$ is a symplectic basis of $(\SurgXi, d\alpha_{\epsilon})$ at each point contained in a closed Reeb orbit of $R_{\epsilon}$.
\item $X^{-1}(0)=Y^{-1}(0)$ is a union of connected components of $\cup_{i}T^{c_{i}}_{i}$ where the $T^{\pm}_{i}$ are the transverse push-offs of the $\Lambda_{i}$ as described in Definition \ref{Def:Pushoff}.
\ee
Using the $(X, Y)$, the first Chern class of $\SurgXi$ may be computed as
\begin{equation*}
\PD(c_{1}(\xi_{\LambdaPM})) = \sum_{\Lambda_{i} \subset \LambdaPM} -c_{i}\rot(\Lambda_{i})[T^{c_{i}}_{i}] = \sum_{\Lambda_{i} \subset \LambdaPM} \rot(\Lambda_{i})\mu_{i} \in H_{1}(\SurgL).
\end{equation*}
\end{thm}

The classes $\mu_{i}$ are given by a standard presentation of $H_{1}(\SurgL)$ determined by the surgery diagram, which we will describe in Section \ref{Sec:Homology}.

Theorem \ref{Thm:FramingSummary} may be compared with \cite[Proposition 2.3]{Gompf:Handlebodies} where a similar result is stated for Chern classes integrated over $2$-cycles in Stein surfaces and with \cite[Section 3]{EO:OBInvariants} where Chern classes are computed when performing surgery along Legendrians lying in pages of open book decompositions.

\begin{figure}[h]\begin{overpic}[scale=.7]{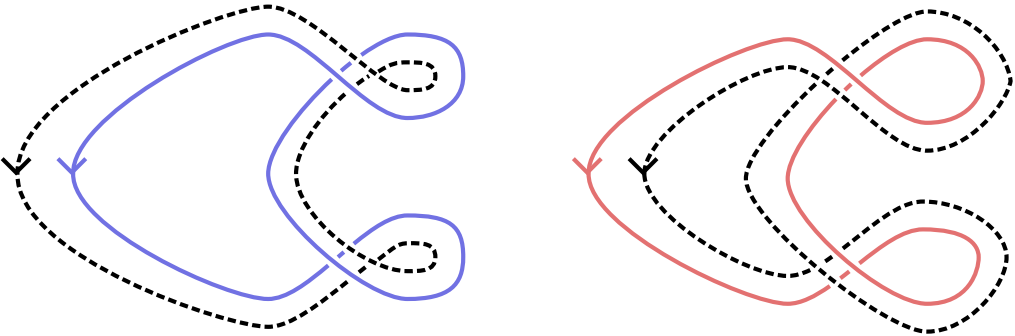}\end{overpic}
\caption{Here we consider contact $\pm 1$-surgery on the Legendrian unknot with $\rot(\Lambda) = 1$. In each case, Theorem \ref{Thm:FramingSummary} provides a framing of $\xi_{\Lambda}$ on the complement of a transverse push-off of $\Lambda$ which travels along the right hand side of $\Lambda$ when the surgery coefficient is $+1$ and along the left side of $\Lambda$ when the coefficient is $-1$. These push-offs are depicted as the dashed, black circles.}
\label{Fig:Rot1Unknot}
\end{figure}

For notational simplicity, we assume throughout this section that $\LambdaPM$ has a single connected component unless otherwise stated. Accordingly, we temporarily drop the indices $i$ appearing in the notation of Section \ref{Sec:ChordNotation}. The surgery coefficient of this knot will be denoted $c$.

Our framing is constructed in three steps: 
\be
\item We start with a framing of $\xi_{std}$ over the complement of $N_{\epsilon}$ and express it in terms of our local coordinate system $(z, p ,q)$ along the boundary of our surgery handles.
\item Next, we describe an explicit extension of this framing throughout most of the handle. We will need this explicit description to compute Conley-Zehnder and Maslov indices in Section \ref{Sec:CZMaslov}.
\item Finally, we describe the zero locus of this extension.
\ee

\subsection{Change of bases between trivializations}

Consider the following pairs of sections of $\xi_{std}$ and $\SurgXi$, which form symplectic bases:
\begin{equation*}
\begin{gathered}
X = \partial_{x} + y\partial_{x},\quad Y = \partial_{y}\\
P_{in} = \partial_{p},\quad Q_{in} = \partial_{q} - p\partial_{z}\\
P_{out} = \partial_{p},\quad Q_{out} = \partial_{q} - p\partial_{z}.
\end{gathered}
\end{equation*}
These come from
\be
\item the coordinate systems $(x, y, z)$ on $\R^{3}$, 
\item the coordinates $(z, p, q)$ on $N_{\epsilon}$ viewed ``from the outside'' of the surgery handle prior to surgery, and 
\item $(z, p, q)$ on $N_{\epsilon}$ viewed ``from the inside'' of the surgery handle after surgery,
\ee
respectively. After performing surgery, the pairs $(X, Y)$ and $(P_{in}, Q_{in})$ are well defined on the complement of a neighborhood of the form $N_{\epsilon'}$ for some $\epsilon' < \epsilon$. Our strategy will be to apply a series of change-of-basis transformations to extend the framing $(X, Y)$ of $\SurgXi$ throughout the surgery handle in so far as cohomological obstruction -- $c_{1}(\SurgXi)$ -- will allow.

First we describe change-of-bases from $(P_{in}, Q_{in})$ to $(P_{out}, Q_{out})$. Following Equation \eqref{Eq:TPhi} the restriction of the tangent map of the gluing map $\phi_{c, f, \epsilon, \delta}$  -- defined in Equation \eqref{Eq:GluingMap} -- to $\SurgXi$ can be written
\begin{equation}\label{Eq:TangentGluingMap}
T\phi_{c, f, \epsilon, \delta}(z, p, q)|_{\xi} =
    \begin{pmatrix}
        1 & 0 \\
        c\frac{\partial f_{\epsilon}}{\partial p}(p) & 1
        \end{pmatrix}
\end{equation}
along $T_{\delta, \epsilon}$ and as $\Diag(1, 1)$ along $B_{\delta, \epsilon} \cup S_{\delta, \epsilon}$. Here the incoming basis is $(P_{in}, Q_{in})$, the outgoing basis is $(P_{out}, Q_{out})$, and coordinates $(z, p, q)$ correspond to the coordinate system inside of the surgery handle.

Now we describe change-of-bases from $(P_{out}, Q_{out})$ to $(X, Y)$. To this end, let $G$ be the Gauss map for a parametrization of $\Lambda$ as described in Section \ref{Sec:StandardNeighborhoods}. Using the construction of $N_{\epsilon}$ in Proposition \ref{Prop:NeighborhoodConstruction}, we can write the change of basis at a point $(p, q, z)$
\begin{equation}\label{Eq:TangentInclusionMap}
    E(p, q)e^{J_{0}(G(q) - \frac{\pi}{2})}
\end{equation}
where $E = \Diag(1, 1) + \bigO(p)$. Here the incoming basis is $(P_{out}, Q_{out})$, the outcoming basis is $(X, Y)$, and coordinates $(z, p, q)$ correspond to the coordinate system on ``the outside'' -- the complement of the surgery handle in $N_{\epsilon}$.

By composing the changes of bases described above in Equations \eqref{Eq:TangentGluingMap} and \eqref{Eq:TangentInclusionMap} and then inverting we can write $(X, Y)$ in the basis $(P_{in}, Q_{in})$ on a neighborhood of $\partial N_{\epsilon}$ as follows: Along $B_{\delta, \epsilon}\cup S_{\delta, \epsilon}$ the change of basis is given by
\begin{equation}\label{Eq:XYBottom}
e^{J_{0}(\frac{\pi}{2}- G(q))}E^{-1}(p, q).
\end{equation}
Along $T_{\delta, \epsilon}$ the transition map is
\begin{equation}\label{Eq:XYTop}
    \begin{pmatrix}
        1 & 0 \\
        -c\frac{\partial f_{\epsilon}}{\partial p}(p) & 1
        \end{pmatrix}e^{J_{0}(\frac{\pi}{2}- G(q + c f_{\epsilon}(p)))}E^{-1}(p, q + c f_{\epsilon}(p)).
\end{equation}
Here the incoming basis is $(X, Y)$, the outcoming basis is $(P_{in}, Q_{in})$, and coordinates $(z, p, q)$ correspond to the coordinate system inside of the surgery handle. Then where they are defined, Equations \eqref{Eq:XYBottom} and \eqref{Eq:XYTop} provide $X$ and $Y$ as a linear combination of $P_{in}, Q_{in}$ by multiplying the above expressions on the left by $\left(\begin{smallmatrix}1 \\ 0\end{smallmatrix}\right)$ and $\left(\begin{smallmatrix}0 \\ 1\end{smallmatrix}\right)$, respectively.

\subsection{Framing extension up to obstruction}\label{Sec:FramingExtension}

We use the above equations to extend the framing $(X, Y)$ of $\SurgXi$ inside of the surgery handle. To this end, let $\delta> 0$ be an arbitrarily small constant and consider a smooth function $\nu:I_{\epsilon} \rightarrow [0, 1]$ with the following properties:
\be
\item $\nu(-\epsilon) = 0$ and $\nu(\epsilon)=1$,
\item all of its derivatives vanish outside of $I_{\epsilon - \delta}$.
\ee

\begin{figure}[h]\begin{overpic}[scale=.7]{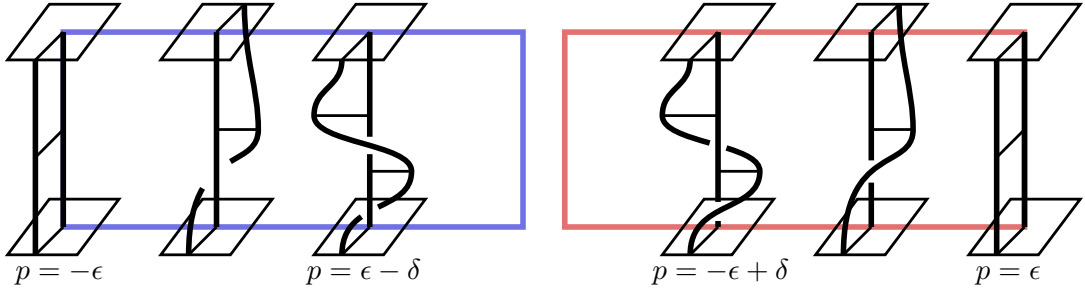}
\put(1, -2){$p=-\epsilon$}
\put(28, -2){$p=\epsilon-\delta$}
\put(60, -2){$p=-\epsilon + \delta$}
\put(90, -2){$p=\epsilon$}
\end{overpic}
\vspace{5mm}
\caption{On the left we have the extension of the framing $(X, Y)$ through the surgery handle over a square of the form $\{ q = q_{0}, p < \epsilon - \delta\}$ in $N_{\epsilon}$ when $c=1$ and $\rot(\Lambda) = 1$. On the right is the case $c=-1$, $\rot(\Lambda) = 1$. Here $\partial_{p}$ points to the right, $\partial_{q}$ points in to the page, and $\partial_{z}$ points upwards. Along the bottom of the square, the framing is constant. At each $p$, the framing twists with respect to the trivialization $(\partial_{p}, \partial_{q} - p\partial_{z})$ according to the twisting of Gauss map along the path in $\Lambda$ from $q_{0}$ to $q_{0} + f_{\epsilon}(p)$. For $c=1$, moving from left to right, we eventually get to $p_{0} = \epsilon - \delta$ such that $f_{\epsilon}(p) = 1$ for $p > p_{0}$.}
\label{Fig:XYExtension}
\end{figure}

When $c=1$, we use Equations \eqref{Eq:XYBottom} and \eqref{Eq:XYTop} to extend the definitions of $(X, Y)$ over the set $\{ p < \epsilon - \delta\} \subset N_{\epsilon}$ using the family of matrices
\begin{equation}\label{Eq:XYExtensionPlus1}
\begin{gathered}
\begin{pmatrix}
1 & 0 \\
-\frac{\partial f_{\epsilon}}{\partial p}(\zeta^{+}) & 1
\end{pmatrix}e^{J_{0}(\frac{\pi}{2}- G(\eta^{+}))}E^{-1}\left(\zeta^{+}, \eta^{+}\right), \\
\zeta^{+}(z, p) = p\nu(z) - \epsilon(1-\nu(z))\\
\eta^{+}(z, p, q) = q + f_{\epsilon}\left(\zeta^{+}(z, p)\right).
\end{gathered}
\end{equation}
Note that $\zeta^{+} = -\epsilon$ along $\{ z = -\epsilon \} \cup \{ p = -\epsilon \}$ and that $\zeta^{+} = p$ along $z=\epsilon$. By these properties and the properties of $f_{\epsilon}$ and its derivatives in Section \ref{Sec:GluingMaps}, we have that this family of matrices agrees with Equation \eqref{Eq:XYBottom} along $B_{\delta, \epsilon}$ and with Equation \eqref{Eq:XYTop} along $T_{\delta, \epsilon}$.

Likewise when $c=-1$, we extend the definitions of $(X, Y)$ over the set $\{ p > -\epsilon + \delta\} \subset N_{\epsilon}$ using the family of matrices which provide $(X, Y)$ in the basis $(P_{in}, Q_{in})$
\begin{equation}\label{Eq:XYExtensionMinus1}
\begin{gathered}
\begin{pmatrix}
1 & 0 \\
\frac{\partial f_{\epsilon}}{\partial p}(\zeta^{-}) & 1
\end{pmatrix}e^{J_{0}(\frac{\pi}{2}- iG(\eta^{-}))}E^{-1}(\zeta^{-}, \eta^{-})), \\
\zeta^{-}(z, p) = p\nu(z) + \epsilon(1 - \nu(z))\\
\eta^{-}(z, p, q) = q - f_{\epsilon}(\zeta^{-}(z, p)).
\end{gathered}
\end{equation}
Note that $\zeta^{-} = \epsilon$ along $\{ z = -\epsilon \} \cup \{ p = \epsilon \}$ and that $\zeta^{-} = p$ along $z=\epsilon$. As in the $c=1$ case, this family of matrices agrees with Equation \eqref{Eq:XYBottom} along $B_{\delta, \epsilon}\cup S_{\delta, \epsilon}$ and with Equation \eqref{Eq:XYTop} along $T_{\delta, \epsilon}$.

The extension of the fields $(X, Y)$ through the surgery handle $N_{\epsilon}$ is illustrated in Figure \ref{Fig:XYExtension}.

\subsection{Obstruction to global definition of $(X, Y)$}

The Chern class $c_{1}(\SurgXi)$ agrees with the Euler class of $\SurgXi$ and so can be represented as the zero locus of a generic section $s\in\Gamma(\SurgXi)$. In attempting to extend the definition of $(X, Y)$ over the squares $S^{+}_{q_{0}, \epsilon, \delta} := \{ q=q_{0}, \epsilon - \delta \leq p \leq \epsilon \}$ when $c=1$ and $S^{+}_{q_{0}, \epsilon, \delta} := \{ q=q_{0}, -\epsilon \leq p \leq -\epsilon + \delta \}$ when $c=-1$ we may complete the proof of Theorem \ref{Thm:FramingSummary}.

\begin{proof}[Proof of Theorem \ref{Thm:FramingSummary}]
We attempt to extend $X$ throughout the entirety of the handle, assuming that $\delta$ is small enough so that $f_{\epsilon}$ is constant on each component of $[-\epsilon, -\epsilon + \delta) \cup (\epsilon - \delta, \epsilon]$. For the case of $+1$-contact surgery we study Equation \eqref{Eq:XYExtensionPlus1}.  We orient $S^{+}_{q_{0}, \epsilon, \delta}$ so that $\partial_{q}$ points positively through it. Parameterize the oriented boundary of each $S^{+}_{q_{0}, \epsilon, \delta}$ with a piece-wise smooth curve $\gamma = \gamma(t)$ so that
\begin{equation*}
\frac{\partial \gamma}{\partial t} = \begin{cases}
    \partial_{z} & p=\epsilon - \delta\\
    \partial_{p} &  z=\epsilon\\
    -\partial_{z} & p=\epsilon\\
    -\partial_{p} & z=\epsilon.
\end{cases}
\end{equation*}

Applying the vector $\left(\begin{smallmatrix}1 \\ 0\end{smallmatrix}\right)$ to the left of Equation \eqref{Eq:XYTop} gives us the section $X$ as a linear combination of $P_{in}$ and $Q_{in}$ along $\gamma$. By throwing away the shearing and rescaling terms in Equation \eqref{Eq:XYExtensionPlus1}, this section is homotopic through non-vanishing sections of $\SurgXi$ to a section of the form
\begin{equation}\label{Eq:XLoopHomotopy}
t \mapsto   \begin{cases}
    e^{-J_{0}G(q_{0} + f_{\epsilon}(t\epsilon))}, & p=\epsilon - \delta, t\in [-1, 1]\\
	e^{-J_{0}G(q_{0} + f_{\epsilon}(\epsilon))}, & \{ z=-\epsilon \} \cup \{p=\epsilon\} \cup \{z=\epsilon\}.
  \end{cases}
\end{equation}
This is homotopic to
\begin{equation*}
t \mapsto e^{\const - 2\pi i t\rot(\Lambda) }, \quad t\in [0, 1].
\end{equation*}
Therefore a generic extension of $X$ over each $S^{+}_{q_{0}, \epsilon, \delta}$ will have $-\rot(\Lambda)$ zeros counted with multiplicity. Taking a generic extension of $X$ over $\{ p > \epsilon - \delta \}$ will then be an oriented link which transversely intersects each square with multiplicity $\rot(\Lambda)$. Pushing this zero locus through the side $ p = \epsilon $ of the surgery handle provides $\PD(c_{1}(\SurgXi)) = -\rot(\Lambda)\lambda_{\xi}$.

The case $c=-1$ is similar: We only check signs. Consider a parameterization of boundary of the square $S^{-}_{q_{0}, \epsilon, \delta}$ with a loop $\gamma$ satisfying
\begin{equation*}
\frac{\partial \gamma}{\partial t} = \begin{cases}
    -\partial_{z} & p=-\epsilon + \delta\\
    -\partial_{p} & z=-\epsilon\\
    \partial_{z} & p=-\epsilon\\
    \partial_{p} & z=\epsilon.
\end{cases}
\end{equation*}
Then following Equation \eqref{Eq:XYExtensionMinus1} the analog of Equation \eqref{Eq:XLoopHomotopy} for the $c=-1$ case is
\begin{equation*}
t \mapsto   \begin{cases}
e^{-J_{0}G(q_{0} - f_{\epsilon}(t\epsilon))}, & p=-\epsilon + \delta, t\in [-1, 1]\\
e^{-J_{0}G(q_{0} - f_{\epsilon}(\epsilon))}, & \{ z=\epsilon \} \cup \{p=-\epsilon\} \cup \{z=-\epsilon\}
  \end{cases}
\end{equation*}
so that the zero locus of the extension of the vector field $X$ throughout the handle is homologous to $\rot(\Lambda) \lambda_{\xi}$.
\end{proof}

\begin{rmk}\label{Rmk:MuToLambdaFraming}
We sketch how the framing $(X, Y)$ can be modified so that its zero locus is contained in a union of meridians of the $\Lambda_{i}$. Take a meridian $\mu_{i}$ of $\Lambda_{i}$ and handle-slide it through $N_{\epsilon}$ to obtain a longitude $-c_{i}\lambda_{i}$ which we may take to be $-c_{i}T^{c_{i}}$. 

This homotopy, say parameterized by $[0, 1]$ may be chosen so that the surface $S$ it sweeps out is an embedded cylindrical cobordism parameterized by an annulus $[0, 1]\times \Circle$. Then we can find a family $(X_{t}, Y_{t})$ of sections of $\xi_{\Lambda}$ whose zero-loci are contained in $\{t\} \subset \Circle$, so that $(X_{1}, Y_{1})$ will vanish along some union of the $\Lambda_{i}$ as desired.

If $\gamma$ is a Reeb orbit of $R_{\epsilon}$ then according to Equation \eqref{Eq:MeridianTwist} we can compute $\CZ_{X_{1}, Y_{1}}$ from $\CZ_{X, Y}$ by counting the number of intersections of $\gamma$ with $S$, which measures the meridonial framing difference.
\end{rmk}

\subsection{Rotation numbers and Chern classes in arbitrary contact $3$-manifolds}

We briefly state how the above can be generalized to understand how $c_{1}$ changes after contact surgery on an arbitrary contact manifold $\Mxi$. A section $s\in \Gamma(\xi)$ determines a homotopy class of oriented trivialization of $\xi$ on the complement of $s^{-1}(0)$ by considering $\xi_{x} = \Span_{\R}(s_{x}, J s_{x})$ for an almost complex structure $J$ on $\xi$ compatible with $d\alpha$ for a contact form $\alpha$ for $\xi$ and $x \in M \setminus s^{-1}(0)$. Suppose that $s$ is transverse to the zero section and non-vanishing along a neighborhood $N_{\epsilon}$ of $\Lambda$. Write $\eta_{s}$ for the oriented link 
\begin{equation*}
T_{s} = s^{-1}(0) \subset (M \setminus N_{\epsilon}) = (M_{\Lambda} \setminus N_{\epsilon}).
\end{equation*}

\begin{defn}\label{Def:RotGeneral}
The \emph{rotation number} $\rot_{s}(\Lambda)$ is the winding number of $\partial_{q}$ in $\xi_{\Lambda}$ determined by the trivialization of $\xi|_{\Lambda}$ provided by $s$.
\end{defn}

Note that changing the orientation of $\Lambda$ multiplies the rotation number by $-1$ and that $\rot_{s}$ agrees with the standard definition of the rotation number for oriented Legendrians in $\Rthree$ if we take $s \in \Gamma_{\neq 0}(\xi_{std})$. More generally, the rotation number of a Legendrian knot $\Lambda$ is canonically defined whenever $\xi$ admits a non-vanishing section and at least one of $H^{1}(M)=0$ or $[\LambdaZero] = 0 \in H_{1}(M)$ holds as in Proposition \ref{Prop:CanonicalZGrading}. 

We note that Definition \ref{Def:RotGeneral} may be applied to Legendrian knots in $\SurgLxi$ even when these hypotheses are not satisfied: If such a knot $\LambdaZero$ is contained in $\R^{3} \setminus N_{\epsilon} = \SurgL \setminus N_{\epsilon}$, then $\rot_{X, Y}(\LambdaZero)$ may be computed using the typical methods for Legendrian knots in $\Rthree$ as described in Section \ref{Sec:LegendrianOverview}. This follows immediately from the first condition listed in Theorem \ref{Thm:FramingSummary}.

\begin{prop}\label{Prop:C1General}
Following the notation in the preceding discussion and writing $\lambda_{\xi}$ for a longitude of $\Lambda$ determined by $\xi$, and $\mu$ for a meridian of $\Lambda$, the Chern class $c_{1}(\xi_{\Lambda})$ for the contact manifold $(M_{\Lambda}, \xi_{\Lambda})$ obtained by performing contact $\pm 1$-surgery on $\Lambda \subset M$ is determined by the formula
\begin{equation*}
    \PD(c_{1}(\xi_{\Lambda})) = [T_{s}] \mp \rot_{s}(\Lambda)\lambda_{\xi} = [T_{s}] \pm \rot_{s}(\Lambda)\mu \in H_{1}(M_{\Lambda}).
\end{equation*}
\end{prop}

This can be proved using the same strategy as Theorem \ref{Thm:FramingSummary}, replacing $X$ with $s$.

\section{Conley-Zehnder and Maslov index computations}\label{Sec:CZMaslov}

The goal of this section is to compute the integral Conley-Zehnder indices of closed orbits of the $R_{\epsilon}$ and the Maslov indices of broken closed strings on $\LambdaZero \subset \SurgLxi$ using the framing $(X, Y)$ defined in Section \ref{Sec:Framing}.

\begin{thm}\label{Thm:IntegralCZ}
For each $n > 0$, there exists $\epsilon_{0}$ such that for all $\epsilon\leq \epsilon_{0}$ all orbits $\gamma$ of word length $\leq n$ are hyperbolic with
\begin{equation*}
\CZ_{X, Y}(\gamma) = \sum_{k=1}^{n} (\rot_{j_{k},j_{k+1}} +\ \delta_{1, c_{j_{k}}^{+}}) \in \Z.
\end{equation*}
\end{thm}

\begin{thm}\label{Thm:MaslovComputation}
Let $b$ be a broken closed string on $\LambdaZero \subset \SurgLxi$ of the form
\begin{equation*}
b = \zeta_{1}\ast (a_{1}\kappa_{1}) \ast \cdots \zeta_{n}\ast (a_{n}\kappa_{n})
\end{equation*}
where each $\zeta_{k}$ is a path in $\LambdaZero$ and each $\kappa_{k}$ is a chord of $\LambdaZero$ with respect to $R_{\epsilon}$. By Theorem \ref{Thm:ChordsToChords}, we can write 
\begin{equation*}
\kappa_{k} =(r_{k_{1}}\cdots r_{k_{n_{k}}})
\end{equation*}
for some word of chords with boundary on $\LambdaZero \subset \Rthree$. In this notation,
\begin{equation*}
\Maslov_{X, Y}(b) = \sum_{1}^{n} \bigg(\frac{\theta(\zeta_{k})}{\pi} - 
    \half + a_{k}m_{X, Y}(\kappa_{k})\bigg), \quad
m_{X, Y}(\kappa_{k}) = \sum_{l=1}^{n_{k}-1}\left(\rot_{j_{k_{l}},j_{k_{l+1}}} + \delta_{1, c_{k_{l}}}^{+}\right).
\end{equation*}
\end{thm}

Proving the above theorems requires further analysis of Equations \eqref{Eq:XYExtensionPlus1} and \eqref{Eq:XYExtensionMinus1}. The analysis will provide an expression of the linearized flow of $R_{\epsilon}$ as a path of matrices in $\SLtwoR$ with entries in $\R[\epsilon^{-1}]$ determined by $\cycword(\gamma)$. Analysis of the highest-order terms of these polynomials gave us the proof of Theorem \ref{Thm:Mod2CZ}. Analysis of the second-highest-order terms of these polynomials will yield a formula for integral Conley-Zehnder indices, $\CZ_{X, Y}$.

\subsection{Matrix model for the linearized flow}

With respect to the coordinate system $(z, p, q)$ inside of the surgery handles $R_{\epsilon} = \partial_{z}$. Hence computing the restriction of the linearized to flow to $\SurgXi$ from the bottom ($z=-\epsilon$) to a point above it ($z > -\epsilon$) in the surgery handle with respect to $(X, Y)$ amounts to writing $(X, Y)_{-\epsilon, p, q}$ in the basis $(X, Y)_{z, p, q}$. We write $F_{i}(z, p, q) \in \SLtwoR$ for this path of matrices associated to points $(z, p, q)$ in the component of $N_{\epsilon}$ associated to $\Lambda_{i}$.

By composing Equation \eqref{Eq:XYBottom} with Equations \eqref{Eq:XYExtensionPlus1}  -- in the case of $+1$-surgery -- and \eqref{Eq:XYExtensionMinus1} -- in the case of $-1$-surgery -- we have
\begin{equation}\label{Eq:HandleFlow}
\begin{aligned}
F_{i}(z, p, q) &= E(\zeta^{c_{i}}, \eta^{c_{i}})e^{J_{0}G_{i}(\eta^{c_{i}})}\begin{pmatrix}
    1 & -c_{i} \frac{\partial f_{\epsilon}}{\partial p}(\zeta^{c_{i}}) \\
    0 & 1 \end{pmatrix} e^{-J_{0} G_{i}(q)}E^{-1}(p, q)\\
    &= e^{J_{0}G_{i}(\eta^{c_{i}})}\begin{pmatrix}
    1 & -c_{i} \frac{\partial f_{\epsilon}}{\partial p}(\zeta^{c_{i}}) \\
    0 & 1 \end{pmatrix} e^{-J_{0}G(q)}(\Id + \bigO(p))
\end{aligned}
\end{equation}
We have preemptively simplified the expression with some basic arithmetic. Here $G_{i}$ is the Gauss map associated to the component $\Lambda_{i}$ of $\Lambda$. The following collection of assumptions will allow us to further simplify the above expression:

\begin{assump}\label{Assump:FlowProperties}
We refine our previous constructions as follows: At any point through which a closed Reeb orbit passes, the sections $X, Y$ of $\SurgXi$ described in Section \ref{Sec:FramingExtension} are defined according to the formula contained within that section. This can be achieved by setting the constant $\delta$ to be sufficiently small.
\end{assump}

Some consequences of the above assumptions coupled with Assumptions \ref{Assump:DiskIntersection} are:
\be
\item Equation \eqref{Eq:HandleFlow} is valid for any point contained in a closed Reeb orbit.
\item The expression $-c_{i} \frac{\partial f_{\epsilon}}{\partial p}(\zeta^{c_{i}})$ in that formula simplifies to $-c_{i}\epsilon^{-1}$ for any point lying in a closed Reeb orbit.
\ee

Combining these consequences with a conjugation, we have that $F_{i}$ in a neighborhood of a Reeb segment which exits $N_{\epsilon, i}$ near $l^{+}_{j_{1}}$ and exists near $l^{-}_{j_{2}}$ for composable Reeb chords $r_{j_{i}}, r_{j_{2}}$ is homotopic -- relative endpoints -- to a path of the form
\begin{equation}\label{Eq:ModelReturnMap}
\mathcal{F}_{j_{1}, j_{2}}(t) = e^{J_{0}t \theta_{j_{1}, j_{2}}}\begin{pmatrix}
1 & -t c^{+}_{j_{1}}\epsilon^{-1} \\
0 & 1
\end{pmatrix}(\Id + \bigO(\epsilon)) \in \SLtwoR,\ t\in [0, 1]
\end{equation}
where we use the basis $e^{i\frac{\pi}{4}}(X, Y)$.

Using Equation \eqref{Eq:ModelReturnMap}, we can write the restriction of Poincaré return map to $\xi$ of a closed Reeb orbit $\gamma$ of $\alpha_{\epsilon}$ with $\cycword(\gamma) = r_{j_{1}}\cdots r_{j_{n}}$ as
\begin{equation}\label{Eq:ReturnMapProduct}
    \Ret_{\gamma} = \mathcal{F}_{j_{n}, j_{1}}(1)\mathcal{F}_{j_{n-1}, j_{n}}(1)\cdots \mathcal{F}_{j_{1}, j_{2}}(1)
\end{equation}
by composing the flow maps as an orbit passes through the various surgery handles. If the word consists of a single chord, then we have $\Ret = \mathcal{F}_{j_{1}, j_{1}}(1)$. Note that while our expression for $\Ret$ depends on a particular representation of the associated cyclic word, its conjugacy class in $\SLtwoR$ does not.

\subsection{Integral Conley-Zehnder indices}

In this subsection we prove Theorem \ref{Thm:IntegralCZ} via induction on the word length $n$ of $\gamma$. The proof is computational, making use of the Robbin-Salamon characterization of the Conley-Zehnder index described in Equation \eqref{Eq:RSCZ}.

\subsubsection{The case $n=1$}

We begin with the case $n=1$, analyzing Equation \eqref{Eq:ModelReturnMap}. A slight modification of the proof of the following lemma along with further analysis of Equation \eqref{Eq:MatrixReturn} will provide the general case. For the sake of notational simplicity, we temporarily drop the subscripts required to describe words of length greater than $1$.

\begin{lemma}\label{Lemma:CZnEqualsOne}
Theorem \ref{Thm:IntegralCZ} is valid for Reeb orbits of word length $1$.
\end{lemma}

\begin{proof}
We homotop the path $\mathcal{F}$ so that it is parameterized with the interval $[0, 3]$, taking the form
\begin{equation}\label{Eq:TkParam}
\begin{gathered}
\mathcal{F}(t) = e^{J_{0}t_{1} \theta}\begin{pmatrix}
1 & -t_{2} c\epsilon^{-1} \\
0 & 1
\end{pmatrix}E(t_{3}), \quad E(0) = \Id,\quad E(1) = \Id + \bigO(\epsilon),\\
t_{k} = \begin{cases}
0 & t \leq k - 1\\
t - k + 1 & t \in (k-1, k)\\
1 & t \geq k.
\end{cases}
\end{gathered}
\end{equation}
With this parameterization we are performing the rotation first so that the path is non-degenerate at $\mathcal{F}(0)=\Id$. A standard computation shows that along the interval $[0, 1]$, the contributions to $\CZ_{X, Y}$ are given by $2\left\lfloor \frac{\theta}{2\pi} \right\rfloor + 1$. Then along $t\in [1, 2]$, we have
\begin{equation*}
\mathcal{F}(t) = (-1)^{\rot}\begin{pmatrix}
0 & -1 \\
1 & -t_{2}c\epsilon^{-1}
\end{pmatrix}
\end{equation*}
By the $\SLtwoR$ trace formula of Equation \eqref{Eq:DetTr}, $t \in [1, 2]$ will be crossing exactly when 
\begin{equation*}
\tr(\mathcal{F}(t)) = (-1)^{1 + \rot} t_{2} c = 2.
\end{equation*}
Therefore we find a crossing in the interval -- and a single one at that -- if and only if $c = (-1)^{1 + \rot}$. At such a crossing, if it exists, the matrix $S(t)$ of Equation \eqref{Eq:SympPathDerivative} is $S(t) = \Diag(t_{2} c \epsilon^{-1}, 0)$. So the contribution to $\CZ_{X, Y}$ can be computed as $\half(c - (-1)^{\rot})$.

Adding up the contributions along $t \in [0, 2]$, we have
\begin{equation*}
\CZ = 2\left\lfloor\frac{\theta}{2\pi}\right\rfloor + 1 + \frac{c - (-1)^{\delta}}{2} = 2\left\lfloor\frac{\theta}{2\pi}\right\rfloor + \frac{1-(-1)^{\rot}}{2} + \frac{c + 1}{2} = \rot + \delta_{1, c}.
\end{equation*}
Along the interval $[2, 3]$, the addition of the $E$ term to the formula contributes a term to the trace which is bounded by a constant which is independent of $\epsilon$. The this interval is devoid of crossings for $\epsilon$ small.
\end{proof}

\subsubsection{The case $n > 1$}

Now we prove the induction step in our index computation. We suppose that the Reeb orbit in question has word length $n + 1 > 1$ and is parameterized with an interval $[0, n+1]$. Then we can compute the Conley-Zehnder index using the path of symplectic matrices
\begin{equation*}
\phi(t) = \mathcal{F}_{j_{n}, j_{n+1}}(t_{n+1})\cdots\mathcal{F}_{j_{n+1}, j_{1}}(t_{1}), \quad 
t_{k} = \begin{cases}
0 & t \leq k - 1 \\
t - k + 1 & t \in (k-1, k)\\
1 & t \geq k
\end{cases}
\end{equation*}
by combining Equations \eqref{Eq:ModelReturnMap} and \eqref{Eq:ReturnMapProduct}. As in the proof of Lemma \ref{Lemma:CZnEqualsOne}, we can drop the $E$-terms in the equations, count the contributions to $\CZ$ coming from the rotation and shearing matrices, then re-introduce the $E$ terms noting that they do not contribute to $\CZ$ due to the large absolute values of traces. Consequently, we ignore these $E$ terms during computation. With this simplification $\phi$ takes the  following form when $t \in [n, n+1]$,
\begin{equation}\label{Eq:ReturnMapLastStep}
\begin{gathered}
\phi(t) = \tilde{\phi}(t_{n+1})\Ret_{n}, \quad \tilde{\phi}(t_{n+1}) = e^{it_{n} \theta_{j_{n+1}, j_{1}}}\begin{pmatrix}
1 & -t_{n+1} c^{+}_{j_{n+1}}\epsilon^{-1} \\
0 & 1
\end{pmatrix},\\
\Ret_{n} = \mathcal{F}_{j_{n-1}, j_{n}}(1)\cdots\mathcal{F}_{j_{n+1}, j_{1}}(1) = (-1)^{\rot_{n}}\left( J^{n} + \sum_{K=1}^{n} \left( \sum_{k\in I_{K}} \left(\prod_{i=1}^{K} -c^{+}_{j_{k_{i}}}\right)M_{k}\right) \epsilon^{-K}\right),\\
\rot_{n} = \sum_{0}^{n-1} \rot_{j_{k},j_{k+1}}.
\end{gathered}
\end{equation}
over the sub-interval $[n, n+1]$. Here indices are cyclic so that $\rot_{j_{0},j_{1}} = \rot_{j_{n+1}, j_{1}}$. The $M_{k}$ are as in Equation \eqref{Eq:MatrixReturn}.

By Theorem \ref{Thm:Mod2CZ} the trace of $\Ret_{n}$ has absolute value of order $\epsilon^{-n}$. Therefore $n, n+1 \in [0, n+1]$ are not crossing for $\epsilon$ small. The $\epsilon^{-n}$ term in $\Ret_{n}$ is given by Equation \eqref{Eq:Mn}. The $\epsilon^{1-n}$ term is also easily computable. Noting that $\Diag(0, a)J\Diag(0, b) = 0$ for $a, b\in \R$, the only $k$ for which $M_{k}$ is non-zero with $k\in I_{n}$ are $(1,\dots, n-1)$ and $(2,\dots,n)$. Thus the $\epsilon^{1-n}$ terms in $\Ret_{n}$ are
\begin{equation*}
(\prod_{1}^{n-1}-c^{+}_{j_{k_{i}}})\Diag(0, 1)J + (\prod_{2}^{n}-c^{+}_{j_{k_{i}}})J\Diag(0, 1) = \begin{pmatrix}
0 & -\prod_{2}^{n}-c^{+}_{j_{k_{i}}} \\
\prod_{1}^{n-1}-c^{+}_{j_{k_{i}}} & 0
\end{pmatrix}.
\end{equation*}
Combining this with Equation \eqref{Eq:Mn} we have
\begin{equation}\label{Eq:ReturnMapPenultimate}
\begin{aligned}
\Ret_{n} &= (-1)^{\rot_{n}}\begin{pmatrix}
0 & -\epsilon^{-n+1}\prod_{2}^{n}-c^{+}_{j_{k_{i}}} \\
\epsilon^{-n+1}\prod_{1}^{n-1}-c^{+}_{j_{k_{i}}} & \epsilon^{-n}\prod^{n}_{1} -c^{+}_{j_{k_{i}}}
\end{pmatrix} + \bigO(\epsilon^{2-n})\\
&= (-1)^{\rot_{n}}\epsilon^{-n+1}(\prod_{1}^{n}-c^{+}_{j_{k_{i}}})\begin{pmatrix}
0 & c^{+}_{j_{k_{1}}} \\
-c^{+}_{j_{k_{n}}} & \epsilon^{-1}
\end{pmatrix} + \bigO(\epsilon^{2-n}).
\end{aligned}
\end{equation}

\begin{lemma}\label{Lemma:CZnPlusOneContribution}
For $\epsilon$ sufficiently small, the contribution to $\CZ_{X, Y}$ along the interval $[n, n+1]$ in Equation \eqref{Eq:ReturnMapLastStep} is
\begin{equation*}
    \rot_{j_{n}, j_{n+1}} +\ \delta_{1,c^{+}_{j_{n+1}}}.
\end{equation*}
\end{lemma}

\begin{proof}
We begin by making some temporary notational simplifications and further subdivide the interval $[n, n+1]$ along which the map $\tilde{\phi}$ is changing and $\Ret_{n}$ is constant. We are referring here to Equation \eqref{Eq:ReturnMapLastStep} and use notation from that equation throughout the proof. We write 
\begin{equation*}
\begin{gathered}
\sigma = (-1)^{\rot_{n}}\Big(\prod_{1}^{n}-c^{+}_{j_{k_{i}}}\Big)\in \{ \pm 1\}\\ 
c_{1} = c^{+}_{j_{1}}, \quad c_{n} = c^{+}_{j_{n}}, \quad c_{n+1} = c^{+}_{j_{n+1}},\\
\rot = \rot_{j_{n}, j_{n+1}},\quad \rot_{2} = \rot \bmod_{2} \in \Z/2\Z \\
\theta_{j_{n}, j_{n+1}} = \pi(2k + \delta_{1,\rot_{2}} + \half), \quad k = \left\lfloor \frac{\theta_{j_{n}, j_{n+1}}}{2\pi}\right\rfloor \in \Z.
\end{gathered}
\end{equation*}
By combining the above notation with Equation \eqref{Eq:ReturnMapPenultimate}, we can write $\phi(t) = \tilde{\phi}(t_{n})\Ret_{n}$ where
\begin{equation*}
\tilde{\phi}(t_{n+1}) = e^{J_{0}t_{n+1} \theta_{j_{n}, j_{n+1}}}\begin{pmatrix}
1 & -t_{n+1} c_{n+1} \epsilon^{-1} \\
0 & 1
\end{pmatrix}, \quad \Ret_{n} = \sigma \epsilon^{1-n}\begin{pmatrix}
0 & c_{1}\\
-c_{n} & \epsilon^{-1}
\end{pmatrix} + \bigO(\epsilon^{2-n}).
\end{equation*}

\begin{figure}[h]
\begin{overpic}[scale=.35]{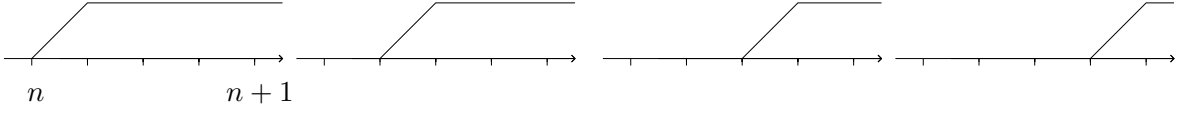}
\put(2, -3){$n$}
\put(19, -3){$n+1$}
\end{overpic}
\vspace{3mm}
\caption{From left to right are graphs of the functions $s_{1}(t), \dots, s_{4}(t)$.}
\label{Fig:RotationHomotopy}
\end{figure}

Along the subset $t \in [n, n+1]$ we homotop $\tilde{\phi}$ to take the form
\begin{equation*}
\tilde{\phi}(t_{n+1}) = e^{J_{0}\theta}\begin{pmatrix}
1 & -s_{2} c_{n+1}\epsilon^{-1}\\
0 & 1
\end{pmatrix}, \quad \theta = \pi\Big(s_{1}\frac{1}{4} + s_{3}\left(\delta_{1,\rot_{2}} + \frac{1}{4}\right) + s_{4}2k\Big)
\end{equation*}
where $s_{i}$ are functions of $t$ taking values in $[0, 1]$ as described in Figure \ref{Fig:RotationHomotopy} so that $\theta = \theta_{j_{n}, j_{n+1}}$ when $s_{1} = \cdots = s_{4} = 1$. In words, we will be applying a rotation by $\frac{\pi}{4}$, a shear, a rotation by $\pi(\delta_{1, \rot_{2}} + \frac{1}{4})$, and then finally a rotation by $2\pi k$. Taking the arguments of all trigonometric functions to be $\theta$,
\begin{equation*}
\begin{aligned}
\phi &= \begin{pmatrix}
\cos & -\sin \\ \sin & \cos
\end{pmatrix}\begin{pmatrix}
1 & -s_{2}c_{n+1}\epsilon^{-1} \\ 0 & 1
\end{pmatrix}\left(\sigma \epsilon^{1- n}\begin{pmatrix}
0 & c_{1} \\ -c_{n} & \epsilon^{-1}
\end{pmatrix} + \bigO(\epsilon^{2 - n})\right) \\
&= \sigma \epsilon^{-n}\begin{pmatrix}
s_{2}c_{n}c_{n+1} \cos & -s_{2}c_{n+1}\epsilon^{-1}\cos - \sin \\ s_{2}c_{n}c_{n+1}\sin & -s_{2}c_{n+1}\epsilon^{-1}\sin + \cos
\end{pmatrix} + \bigO(\epsilon^{1 - n}),\\
\tr(\phi) &= \sigma \epsilon^{-n}\Big(s_{2}c_{n}c_{n+1} \cos -s_{2}c_{n+1}\epsilon^{-1}\sin + \cos\Big) + \bigO(\epsilon^{1- n}).
\end{aligned}
\end{equation*}

Along our first sub-interval parameterized by $s_{1}$ we have $s_{2} = s_{3} = s_{4} = 0$ with $\theta$ increasing from $0$ to $\frac{\pi}{4}$. Here $\tr(\phi) = \sigma \epsilon^{-n}\cos + \bigO(\epsilon^{1-n})$ has large absolute value when $\epsilon$ is small. Hence for $\epsilon$ small, there are no crossings over the $s_{1}$ sub-interval and so there are no contributions to the Conley-Zehnder index.

Along the sub-interval parameterized by $s_{2}$, we have $\theta = \frac{\pi}{4}$ so that $\cos(\theta) = \sin(\theta) = 1/\sqrt{2}$. Hence
\begin{equation}\label{Eq:TrStwoSubinterval}
\tr(\phi) = \frac{\sigma \epsilon^{-n}}{\sqrt{2}}\Big(s_{2}c_{n}c_{n+1} - s_{2}c_{n+1}\epsilon^{-1} + 1\Big) + \bigO(\epsilon^{1 - n}).
\end{equation}
At both the $s_{2} = 0$ and $s_{2} = 1$ endpoints of the interval, $\tr(\phi)$ will have large absolute value of orders $\epsilon^{-n}$ and $\epsilon^{-1-n}$, respectively so that $\phi$ cannot be crossing at either of these endpoints. Over the interior of the sub-interval, we see that there is a single crossing if $c_{n+1} = 1$ and no crossings if $c_{n+1} = -1$.

If $c_{n+1} = 1$, then at the unique crossing we compute the crossing form
\begin{equation*}
S = -J_{0}\frac{\partial \phi}{\partial s_{2}}\phi^{-1} = \frac{1}{2\epsilon}\begin{pmatrix}
1 & -1 \\ -1 & 1
\end{pmatrix}, \quad \begin{pmatrix}
a & b
\end{pmatrix}S\begin{pmatrix}
a \\ b
\end{pmatrix} = \frac{1}{2\epsilon}(a - b)^{2}.
\end{equation*}
The quadratic form determined by $S$ vanishes exactly along the line $\R\left(\begin{smallmatrix}
1 \\ 1
\end{smallmatrix}\right)$. Therefore to see that the unique crossing along the $s_{2}$ sub-interval is non-degenerate, we only need to check that $\phi\left(\begin{smallmatrix}
1 \\ 1
\end{smallmatrix}\right) \neq 0$ at the crossing. Plugging $c_{n+1}=1$ into Equation \eqref{Eq:TrStwoSubinterval}, at the crossing we must have
\begin{equation*}
s_{2}(c_{n} - \epsilon^{-1}) + 1 = \bigO(\epsilon)
\end{equation*}
in order to eliminate the $\epsilon^{-n}$ and $\epsilon^{-n - 1}$ terms. Therefore at the crossing
\begin{equation*}
\phi \begin{pmatrix}
1 \\ 1
\end{pmatrix} = \frac{\sigma\epsilon^{-n}}{\sqrt{2}}\begin{pmatrix}
s_{2}(c_{n} - \epsilon^{-1}) - 1 \\
s_{2}(c_{n} - \epsilon^{-1}) + 1
\end{pmatrix} + \bigO(\epsilon^{1 - n}) = \frac{\sigma\epsilon^{-n}}{\sqrt{2}}\begin{pmatrix}
- 2 \\
0
\end{pmatrix} + \bigO(\epsilon^{1 - n})
\end{equation*}
is non-zero for $\epsilon$ small. We conclude that at the crossing $\ker(\phi - \Id)$ must be $1$-dimensional and the restriction of $S$ to $\ker(\phi - \Id)$ must be positive. Hence the $s_{2}$ sub-interval contributes $\delta_{c_{n+1}, 1}$ to the Conley-Zehnder index.

Now we study the $s_{3}$ sub-interval along which $s_{2} = 1$ and $\theta \in [\frac{\pi}{4}, \pi(\delta_{1, \rot_{2}} + \half)]$. Along this sub-interval
\begin{equation*}
\tr(\phi) = \sigma \epsilon^{-n}\Big(c_{n}c_{n+1} \cos - c_{n+1}\epsilon^{-1}\sin + \cos\Big) + \bigO(\epsilon^{1- n}).
\end{equation*}
If $\delta_{1, \rot_{2}} = 0$ and $\epsilon$ is very small, then as $\theta$ increases from $\frac{\pi}{4}$ to $\frac{\pi}{2}$, the $\sin$ term dominates, $\tr(\phi)$ maintains a large absolute value, and there are no crossings.

If $\delta_{1, \rot_{2}} = 1$, then we have a single crossing at which the crossing form is determined by the matrix
\begin{equation*}
S = -J_{0}\frac{\partial \phi}{\partial s_{2}}\phi^{-1} = \frac{3\pi}{4}\Id.
\end{equation*}
Because the crossing form is positive definite, the contribution to the Conley-Zehnder index is the dimension of $\ker (\phi - \Id)$ -- either $1$ or $2$. At the crossing, $\sin$ is $\bigO(\epsilon)$ implying that $\theta \rightarrow \pi$ (and so $\cos(\theta)\rightarrow -1$) as $\epsilon \rightarrow 0$. Therefore at the crossing,
\begin{equation*}
\left(\phi - \Id\right)\begin{pmatrix}
0 \\ 1
\end{pmatrix} = \sigma\epsilon^{-n}\begin{pmatrix}
-c_{n+1}\epsilon^{-1}\cos - \sin \\ -c_{n+1}\epsilon^{-1}\sin + \cos
\end{pmatrix} + \bigO(\epsilon^{1 - n}) = \sigma\epsilon^{-1-n}\begin{pmatrix}
-c_{n+1}\cos \\ 0
\end{pmatrix} + \bigO(\epsilon^{-n}) \neq 0
\end{equation*}
implying that $\dim\ker (\phi - \Id) = 1$. We conclude that the contribution to $\CZ$ along the $s_{3}$ interval is $\delta_{1, \rot_{2}}$.

For the $s_{4}$ sub-interval, we appeal to the loop property of $\CZ$ described in Equation \eqref{Eq:CZLoop} to see a contribution of $2k$. Combining the contributions over the four $s_{i}$ sub-intervals we get
\begin{equation*}
2k + \delta_{\rot_{2}, 1} + \delta_{c_{n+1}, 1} = \rot_{j_{n}, j_{n+1}} + \delta_{c^{+}_{j_{n+1}}, 1},
\end{equation*}
by reverting to our original notation, thereby completing the proof.
\end{proof}

The combination of the above lemmas completes our induction, thereby proving Theorem \ref{Thm:IntegralCZ}.

\subsection{Integral Maslov indices}

The proof of Theorem \ref{Thm:MaslovComputation} follows from the same methods of calculation as Theorem \ref{Thm:IntegralCZ}.

\begin{proof}[Proof of Theorem \ref{Thm:MaslovComputation}]
According to Definition \ref{Def:MaslovIndex}, we need to measure the rotation of $\Flow_{R_{\epsilon}}^{t}(T_{q^{-}_{k_{1}}}\LambdaZero)$ along each chord $(r_{k_{1}}\cdots r_{k_{n_{k}}})$ with respect to the framing $(X, Y)$. For chords $\kappa_{k}$ of word length $1$, the flow is trivial, so we restrict attention to chords of word length $> 1$. The required analysis can be carried out via analysis of Equation \eqref{Eq:ModelReturnMap}: We recall that this describes the restriction of the linearized flow of $R_{\epsilon}$ to $\SurgXi$ through a component $N_{\epsilon, j_{1}}$ of $N_{\epsilon}$ starting at a point near the tip of one chord $r_{j_{1}}$ up to a point near the tail of another chord $r_{j_{2}}$. 

We study the rotation along a single $\kappa = (r_{j_{1}}\cdots r_{j_{n}})$: The matrix expression $\mathcal{F}_{j_{1},j_{2}}(t)$ in Equation \eqref{Eq:ModelReturnMap} applies to the basis
\begin{equation*}
e^{i\frac{\pi}{4}}(X, Y)\simeq (\partial_{q} - p\partial_{z}, -\partial_{p}),    
\end{equation*}
beginning on the bottom, $\{ z = -\epsilon \}$ of the surgery handle $N_{\epsilon, j_{1}}$. For $j_{1}=k_{1}, j_{2}=k_{2}$, the strand of $\LambdaZero$ touching the starting point of the chord $r_{j_{1}}$ is such that $T\LambdaZero = \R\left( \begin{smallmatrix}0 \\ 1\end{smallmatrix}\right)$. Therefore we need to see how $\mathcal{F}_{j_{1},j_{2}}(t)$ rotates this subspace for $t\in [0, 1]$. As in the proof of Theorem \ref{Thm:IntegralCZ} we can modify the path so as to apply the shearing first, and then the rotation. 

For the shearing, we study the family of real lines in $\R^{2}$ given by
\begin{equation*}
\R
\begin{pmatrix}
    1 & -tc_{j_{1}}^{+}\epsilon^{-1} \\
    0 & 1
\end{pmatrix}(\Id + \bigO(\epsilon))\begin{pmatrix}
    0 \\ 1
\end{pmatrix},\ t\in [0, 1].
\end{equation*}
The end result is a line of the form $\R(\left( \begin{smallmatrix}1 \\ 0\end{smallmatrix}\right) + \bigO(\epsilon))$ obtained by rotating $\R\left( \begin{smallmatrix}0 \\ 1\end{smallmatrix}\right)$ by an angle of
\begin{equation}\label{Eq:MaslovRotationShear}
c^{+}_{j_{1}}\frac{\pi}{2} + \bigO(\epsilon).
\end{equation}

Then applying the rotation by through angles $t\theta_{j_{1},j_{2}}$ as in Equation \eqref{Eq:ModelReturnMap}, we rotate this subspace by 
\begin{equation}\label{Eq:MaslovRotationRot}
\theta_{j_{1}, j_{2}} = \pi \rot_{j_{1},j_{2}} + \frac{\pi}{2}
\end{equation}
which we recall from Section \ref{Sec:ChordNotation} is the rotation angle of the capping path $\eta_{i, j}$ associated to the pair of composable chords $(r_{j_{1}}, r_{j_{2}})$.

Continuing the flow by applying the remaining $\mathcal{F}_{k_{l},k_{l+1}}$ for $l=2,\dots,n_{k}-1$ provides us a total rotation angle of
\begin{equation*}
    \pi\big(\sum_{l=1}^{n_{k}-1} \rot_{k_{l},k_{l+1}} + \half + \frac{c^{+}_{k_{l}}}{2}\big) + \bigO(\epsilon) = \pi\big(\sum_{l=1}^{n_{k-1}} \rot_{k_{l},k_{l+1}} +\  \delta_{c^{+}_{k_{l}}, 1}\big) + \bigO(\epsilon)
\end{equation*}
leaving us on a neighborhood of the strand of $\Lambda$ lying at the starting point of the chord $r_{k_{n_{k}}}$ which ends on $\LambdaZero$. Each summand in the above formula is the result of adding the contributions of Equations \eqref{Eq:MaslovRotationShear} and \eqref{Eq:MaslovRotationRot}. As in the case $\wl(\kappa) = 1$, the linearized flow up to $\LambdaZero$ along $r_{k_{n_{k}}}$ is trivial in the basis $(X, Y)$. This nearly completes the construct of the section $\phi^{G}$ along the chord $\kappa$: We must apply one more rotation to ensure that the unoriented Lagrangian line closes up as described in the discussion preceding \ref{Def:MaslovIndex}.

In the case that our asymptotic indicator is $a_{k} = 1$, the total rotation along the chord will be
\begin{equation*}
    \pi\big(- \half + \sum_{l=1}^{n_{k-1}} (\rot_{k_{l},k_{l+1}} +\  \delta_{c^{+}_{1, k_{l}}})\big).
\end{equation*}
where the $\half$ is the contribution of the clockwise correction rotation at the end of the chord. From the above analysis we know that the angle for this rotation is $-\frac{\pi}{2} + \bigO(\epsilon)$.

If the asymptotic indicator is $-1$, we must travel in the opposite direction, from the tip to the tail of the chord and then apply small the clockwise rotation to obtain a total rotation angle of
\begin{equation*}
    \pi\big(- \half - \sum_{l=1}^{n_{k-1}} (\rot_{k_{l},k_{l+1}} +\  \delta_{c^{+}_{1, k_{l}}})\big).
\end{equation*}
The section $\phi^{G}$ appearing in the definition of $\Maslov_{s}$ is determined by the $\theta(\zeta_{k})$ as the framing $(X, Y)$ coincides with $(\partial_{x}-y\partial_{z}, \partial_{y})$ -- which which the $\theta(\zeta_{k})$ are computed -- on the complement of $N_{\epsilon}$, in which $\LambdaZero$ is presumed to be contained.
\end{proof}

\section{Diagrammatic index formulae}\label{Sec:IndexFormulae}

In this section we compute indices of holomorphic curves in $(\R\times \SurgL, d(e^{t}\alpha_{\epsilon}))$. We begin by covering the case of curves whose domain is a closed surface with punctures, which is a simple application of Equation \eqref{Eq:DelbarIndex} to our existing computations of Conley-Zehnder indices and Chern classes. Next we cover the case of a holomorphic disk which is asymptotic to a broken closed string on $\LambdaZero \subset \SurgLxi$ in the sense of Example \ref{Ex:BoundaryBCS}. The case $\LambdaPM = \emptyset$ recovers classic index formula appearing in combinatorial versions of $LCH$ and Legendrian $RSFT$. These index formulae are then combined to describe indices associated to holomorphic curves with arbitrary configurations of interior and boundary punctures in Theorem \ref{Thm:IndexGeneral}.

All indices computed will depend only on topological data, so mention of any specific almost complex structures are ignored.

\subsection{Index formulae for closed orbits}

Let $\{w^{+}_{1},\dots, w^{+}_{m^{+}}\}$ and $\{w^{+}_{1},\dots,w_{m^{-}}^{-}\}$ be collections of cyclic words of chords on $\Lambda$. By Theorem \ref{Thm:IntegralCZ}, we may choose some $\epsilon_{\gamma} > 0$ such that for all $\epsilon < \epsilon_{\gamma}$, the closed orbits $\gamma^{\pm}_{j}$ of $R_{\epsilon}$ corresponding to these cyclic words via Theorem \ref{Thm:ChordsToOrbits} are all non-degenerate hyperbolic á la Theorem \ref{Thm:Mod2CZ}. Write $\gamma^{+} = \{\gamma^{+}_{1},\dots,\gamma^{+}_{m^{+}}\}$ and $\gamma^{-} = \{\gamma^{-}_{1},\cdots,\gamma^{-}_{m^{-}}\}$ for the corresponding collections of orbits. 

Suppose that $(\Sigma, j)$ is a closed Riemann surface containing a non-empty collection of punctures and that $(t, U): \Sigma' \rightarrow \R\times \SurgL$ is a holomorphic curve (as in Section \ref{Sec:SFTOverview}) which is positively asymptotic to the punctures $\gamma^{+}$ and negatively asymptotic to the $\gamma^{-}$.

\begin{thm}\label{Thm:IndexInteriorPunctures}
Using the framing $(X, Y)$ described in Section \ref{Sec:FramingExtension} we can write the expected dimension of the moduli space of curves near $(t, U)$ as
\begin{equation}\label{Eq:ModuliSpaceDimension}
\ind((t, U)) = \CZ_{X, Y}(\gamma^{+}) - \CZ_{X, Y}(\gamma^{-}) - \chi(\Sigma') - 2\sum_{1}^{n} c_{i}\rot(\Lambda_{i})(U \cdot T^{c_{i}}_{i})
\end{equation}
for all $\epsilon < \epsilon_{\gamma}$ where the right-most sum runs over the connected components of $\LambdaPM$ and the $\CZ_{X, Y}$ are computed as in Theorem \ref{Thm:IntegralCZ}.
\end{thm}

\begin{proof}
By comparing with Equation \eqref{Eq:DelbarIndex}, we only need to check that
\begin{equation*}
    c_{X, Y}(U) = - \sum_{1}^{n} c_{i}\rot(\Lambda_{i})(U \cdot T^{c_{i}}_{i})
\end{equation*}
where $c_{X, Y}(U)$ is the relative Chern class of the framing $(X, Y)$. Letting $X_{U} \in U^{*}\SurgLxi$ be a section for which $T_{z}U(X_{U}) = X(U(z))$ for $z \in \Sigma'$ to compute $c_{X, Y}(U)$ provides the desired result, as $X^{-1}(0)$ is a union of connected components of $\cup T^{c_{i}}_{i}$ and the coefficients $-c_{i}\rot(\Lambda_{i})$ account for the multiplicities of the zeros of $X_{U}$ by construction of $(X, Y)$ in Section \ref{Sec:Framing}.
\end{proof}

\subsection{Index formulae for disks with boundary punctures}

Now suppose that $\{ p_{j} \}_{1}^{m} \subset \partial\disk$ is a collection of distinct points on the boundary of a disk. Write $\disk' = \disk \setminus \{p_{j}\}$ for the complement of the boundary punctures in $\disk$ and write $j$ for the standard complex structure on $\disk$. Suppose that $(t, U): \disk' \rightarrow \R\times \SurgL$ is a $(j, J)$ holomorphic map satisfying the following criteria:
\be
\item $(t, U)(\partial \disk') \subset \R\times \LambdaZero$, and
\item the punctures $\{ p_{j} \}$ are asymptotic to chords of $R_{\epsilon}$ with boundary on $\LambdaZero \subset \SurgLxi$.
\ee

As described in Example \ref{Ex:BoundaryBCS}, such a map determines a broken closed string which we will denote by $\bcs(U)$. As in the case of Equation \eqref{Eq:DelbarIndex}, we use $\ind((t, U))$ to denote the expected dimension of the space of holomorphic maps with the same $\bcs(U)$ boundary conditions as $(t, U)$ and in the same relative homotopy class obtained by allowing the locations vary and then modding out by holomorphic reparameterization in the domain (when $m < 3$).

\begin{thm}\label{Thm:IndexPuncturedDisk}
The moduli space of holomorphic disks with boundary condition $\bcs(U)$ in the homotopy class of $U$ has expected dimension
\begin{equation}\label{Eq:IndexPuncturedDisk}
    \ind((t, U)) = \Maslov_{X, Y}(\bcs(U)) + m - 1 - 2\sum_{1}^{n} c_{i}\rot(\Lambda_{i})(U \cdot T^{c_{i}}_{i})
\end{equation}
near the point $(t, U)$. The last sum appearing in the above formula is indexed over the components $\Lambda_{i}$ of $\Lambda$.
\end{thm}

\begin{proof}
We are simply plugging our definition of broken closed strings to formulae appearing in \cite{EES:LegendriansInR2nPlus1, Ekholm:Z2RSFT}. 

Assume first that $X$ is non-vanishing over $\im(U)$, so that $\partial_{t}, R_{\epsilon}, X, Y$ determines a trivialization of $U^{*}T(\R\times \SurgL)$ which splits as a pair of complex lines. Using framing-deformation invariance of $\Maslov_{X, Y}$, we may perturb $X, Y$ so that it is invariant under the flow of $R_{\epsilon}$ in which case the geometric setup described in \cite[Section 3.1]{Ekholm:Z2RSFT} applies. Our choices of ``clockwise rotations'' along positive punctures and ``counter-clockwise rotations'' along negative punctures in the definition of the path of symplectic matrices defining $\Maslov_{s}$ coincide with those used to define the Maslov numbers (which are denoted $\mu(\gamma)$) in that text. The tangent space of our Lagrangian -- $\R\times\LambdaZero$ -- splits as $\R\partial_{t}\oplus T\LambdaZero$, so the only contribution to the Maslov number in question comes from the rotation of $T\LambdaZero$ along the boundary of the disk by the direct sum formula for Maslov numbers. Then the moduli space dimension formula of \cite[Section 3.1]{Ekholm:Z2RSFT} completes our proof.

Now suppose that $X$ is non-vanishing along $\im(U)$. By the construction of the framing $X, Y$, we have that this section must be non-vanishing along $\LambdaZero$ and all of its Reeb chords, and so is non-vanishing along $\im(\bcs(U))$. Therefore the Maslov index can corrected by a relative Chern class term as in Equation \eqref{Eq:DelbarIndex}, which may be computed as signed count of intersections of $U$ with the transverse push-offs of the $\Lambda_{i}$ as in the statement of that theorem.
\end{proof}

\subsection{Index formulae for curves with interior and boundary punctures}

Now we state an index formula for holomorphic curves of general topological type. The geometric setup is as follows.

Let $(\Sigma, j)$ be a compact, connected Riemann surface with boundary components
\begin{equation*}
(\partial \Sigma)_{k},\quad k=1,\dots,\#(\partial \Sigma),
\end{equation*}
marked points $p^{\Int, \pm}_{i}$ contained in $\Int(\Sigma)$, and marked points $p^{\partial, \pm}_{i}$ contained in $\partial \Sigma$. We write $\Sigma'$ for $\Sigma$ with all of its marked points removed. Consider a holomorphic map $(t, U): \Sigma' \rightarrow \R\times \SurgLxi$ subject to the following conditions:
\be
\item The $p^{\Int, +}_{i}$ are positively asymptotic to some collection $\gamma^{+}$ of closed orbits of $R_{\epsilon}$,
\item The $p^{\Int, -}_{i}$ are negatively asymptotic to some collection $\gamma^{-}$ of closed orbits of $R_{\epsilon}$,
\item The $p^{\partial, +}_{i}$ are positively asymptotic to some collections $\kappa^{+}$ of chords of $\LambdaZero \subset \SurgLxi$,
\item The $p^{\partial, -}_{i}$ are negatively asymptotic to some collections $\kappa^{-}$ of chords of $\LambdaZero \subset \SurgLxi$, and
\item $(t, U)(\partial \Sigma') \subset \R \times \LambdaZero$.
\ee
In this setup, we have a broken closed string $\bcs_{k}$ associated to each component $(\partial \Sigma)_{k}$ of $\Sigma$. We may consider the moduli space of curves subject to the same asymptotics -- $\gamma^{\pm}$ and $\bcs_{k}$ -- allowing the complex structure on $\Sigma$ to vary and taking a quotient by $j$-holomorphic symmetries on the domain.

\begin{thm}\label{Thm:IndexGeneral}
In the above notation, the expected dimension of the moduli space of holomorphic maps is
\begin{equation*}
\begin{aligned}
\ind((t, U)) &= \CZ_{X, Y}(\gamma^{+}) - \CZ_{X, Y}(\gamma^{-}) +  \sum_{k} \Maslov_{X, Y}(\bcs_{k})\\
&- \chi(\Sigma) + \#(p^{int}) + \#(p^{\partial}) - 2\sum_{\Lambda_{i}\subset \LambdaPM}c_{i}\rot(\Lambda_{i})(U \cdot T^{c_{i}}_{i}).
\end{aligned}
\end{equation*}
\end{thm}

The proof is a simple combination of Theorems \ref{Thm:IndexInteriorPunctures} and \ref{Thm:IndexPuncturedDisk} using index additivity, cf. \cite[Section 3]{Schwarz:Thesis}.

\section{$H_{1}$ computations and push-outs of closed orbits}\label{Sec:Homology}

Here we compute the first homology $H_{1}(\SurgL)$  of $\SurgL$ and the homology classes of the closed orbits of $R_{\epsilon}$.

\begin{thm}\label{Thm:H1}
The first homology $H_{1}(\SurgL)$ is presented with generators $\mu_{i}$ and relations
\begin{equation*}
(\tb(\Lambda_{i}) + c_{i})\mu_{i}+ \sum_{j\neq i} \lk(\Lambda_{i}, \Lambda_{j})\mu_{j} = 0.
\end{equation*}
Let $\gamma$ be a Reeb orbit of $\alpha_{\epsilon}$ with $\cycword(\gamma) = r_{j_{1}}\cdots r_{j_{n}}$. Then its homology class in $H_{1}(\SurgL)$ with respect to the above basis is
\begin{equation*}
    [\gamma] = \half \sum_{k=1}^{n}(\cross_{j_{k}} + \cross_{j_{k}, j_{k+1}})
\end{equation*}
where the $k$ are considered modulo $n$.
\end{thm}

Relative homology classes $[\kappa] \in H_{1}(\SurgL, \LambdaZero)$ of chords $\kappa$ with boundary on $\LambdaZero \subset \SurgLxi$ can similarly be computed using the technique of the proof of Theorem \ref{Thm:H1} which is carried out in Section \ref{Sec:H1Orbits}. It will be clear that the method of proof allows the reader to compute $[\gamma]$ as an element of the $H_{0}$ of the free loop space of $\SurgL$. In Section \ref{Sec:PushOutDefn}, we show how the proof can be generalized to provide a general means of homotoping closed orbits of $R_{\epsilon}$ into $\R^{3}\setminus N$, a technique we will need for the proof of Theorem \ref{Thm:Trefoil}.

\subsection{Conventions for meridians and longitudes}

Before proving Theorem \ref{Thm:H1}, we quickly review some standard notation. Let $\mu_{j}$ denote a meridian of $\Lambda$ and $\lambda_{i}$ a longitude of $\Lambda$ provided by the Seifert framing and orientation of $\Lambda_{i}$. We note that with respect to the Seifert framing of $\Lambda_{i}$ the longitude provided by $\xi$, denoted $\lambda_{\xi, i}$ is 
\begin{equation*}
\lambda_{\xi, i} = \lambda_{i} + \tb(\Lambda_{i})\mu_{i}.    
\end{equation*}
Each $\mu_{i}$ is oriented so that
\begin{equation*}
(\text{meridian, longitude, outward-pointing normal})
\end{equation*}
is a basis for $T\R^{3}$ agreeing with the usual orientation over $\partial N$ (after rounding the edges of $\partial N$ in the obvious fashion). See Figure \ref{Fig:MuLambdaOrientation}.

\begin{figure}[h]
\includegraphics[scale=.8]{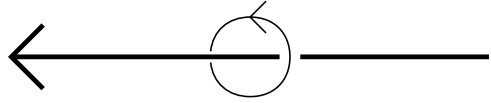}
\caption{Default orientations for meridians.}
\label{Fig:MuLambdaOrientation}
\end{figure}

\subsection{First homology of the ambient space}

The computation of $H_{1}(\SurgL)$ easily follows from the fact that contact $\pm 1$-surgery is a form of Dehn surgery. Suppose that $\R^{3}_{L}$ is a 3-manifold obtained by Dehn surgery on a smooth link $L =\cup L_{i}$ for which the surgery coefficients with respect to the Seifert framing are $p_{i}/q_{i}$ for relatively prime integers $p_{i}$ and $q_{i}$. Writing $\mu_{j}$ for the oriented meridians of the $L_{i}$ we have the following theorem from Kirby calculus -- see eg. \cite[Theorem 2.2.11]{OS:SurgeryBook}:

\begin{thm}
Denote by $\R^{3}_{L}$ a $3$-manifold determined by a surgery diagram where each component $L_{i}$ of $L$ has Dehn surgery coefficient $p_{i}/q_{i}$ for relative prime integers $p_{i}, q_{i}$. Then $H_{1}(\R^{3}_{L})$ is presented with generators $\mu_{i}$, and relations
\begin{equation*}
p_{i} \mu_{i}+q_{i}\sum_{j\neq i} \lk(L_{i}, L_{j}) \mu_{j} = 0
\end{equation*}
where $\lk(L_{i}, L_{j})$ is the linking number.
\end{thm}

When performing contact surgery on the component $\Lambda_{i}$ of $\Lambda$, the meridian $\mu_{i}$ bounding a core disk of the surgery handle is sent to 
\begin{equation*}
    \mu_{i} + c_{i} \lambda_{\xi, i} = (1 + c_{i}\tb(L_{i}))\mu_{i} + c_{i}\lambda_{i}.
\end{equation*}
Thus for Legendrian knots in $\R^{3}$ contact $\pm 1$-surgery on $\Lambda_{i}$ is topologically a $(\tb(\Lambda_{i}) \pm 1)$-surgery. From this computation, the calculation of $H_{1}(M)$ in Theorem \ref{Thm:H1} is then immediate.

\subsection{Homology classes of Reeb orbits}\label{Sec:H1Orbits}

In this section we describe how to compute homology classes of the Reeb orbits of $\alpha_{\epsilon}$. Our strategy will be the homotop orbits to the complement of $N_{\epsilon}$  in $\SurgL$ after which the following computational tool may be applied:

\begin{thm}\label{Thm:H1GammaGeneral}
Let $\gamma$ be an oriented link in $\R^{3}\setminus L$. Then the homology classes of $\gamma$ in $H_{1}(\R^{3} \setminus L)$ and $H_{1}(\R^{3}_{L})$ is given by the formula
\begin{equation*}
    [\gamma] = \sum_{i} \lk(\gamma, L_{i})\mu_{i}.
\end{equation*}
\end{thm}

\begin{proof}
Assume that $\gamma$ is embedded and let $S \subset \R^{3}$ be a Seifert surface which transversely intersects the $L_{i}$. Punch holes in $S$ near its intersections with the $L_{i}$ producing a surface $S'$ which is disjoint from $L$ and whose oriented boundary is a union of $\gamma$ and a linear combination $\sum a_{i}\mu_{i}$. Then $S'$ provides a cobordism from $\gamma$ to these $\mu_{i}$, providing an equivalence $[\gamma] = \sum a_{i}\mu_{i}$ in homology. By the definition of $\lk$, we have $a_{i} = \lk(\gamma, L_{i})$.
\end{proof}

\begin{warn}
The homotopies which we apply to closed Reeb orbits $\gamma$ are not guaranteed to preserve the isotopy class of their embedding in $\SurgL$ (assuming $\gamma$ is embedded).
\end{warn}

\begin{figure}[h]
\includegraphics[scale=.9]{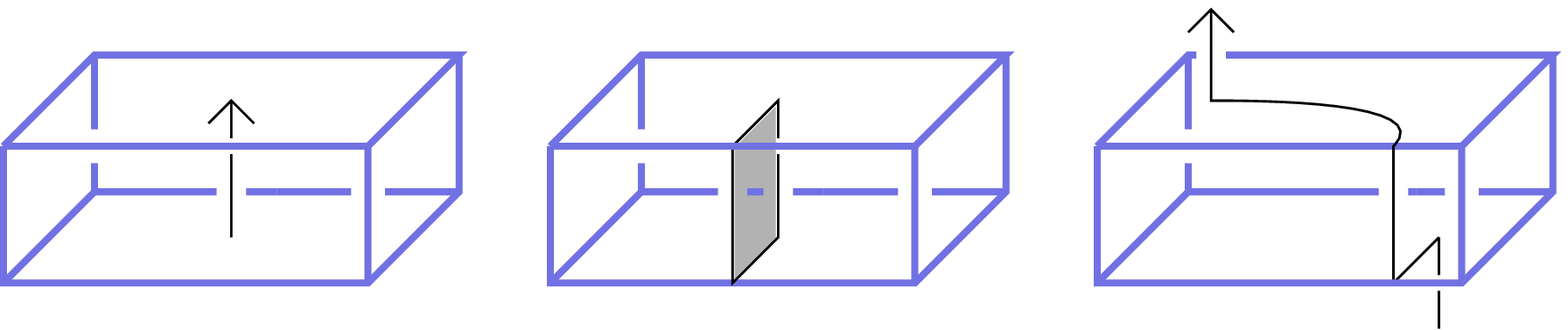}
\caption{Homotoping a Reeb orbit into $\R^{3}\setminus N_{\epsilon}$ as it passes though a $c=+1$ surgery handle.}
\label{Fig:PushingOutPlus1}
\end{figure}

Figure \ref{Fig:PushingOutPlus1} demonstrates how to homotop a segment of a Reeb orbit $\gamma$ into the exterior of the surgery handle $N_{\epsilon}$ as it passes through a component $N_{\epsilon, i}$ for which $c_{i}=1$. The boxes represent the surgery handles with $\partial_{p}$ pointing into the page, $\partial_{q}$ pointing to the left, and $\partial_{z}$ pointing up. On the left we have an arc parallel to the Reeb vector field entering the handle as seen from the inside of $N_{\epsilon, i}$. The arc extends in the $\partial_{z}$-direction through the handle, along which it can be realized realize as being contained in the boundary of a square of the form $\{ p \leq p_{0}, q=q_{0}\}$, depicted in gray. On the right, we see intersection of the boundary of this square with $\partial N_{\epsilon}$ as see from the outside of the surgery handle $\R^{3}\setminus N_{\epsilon}$. By homotoping $\gamma$ across the gray disk, we obtain the this arc shown on the right.

\begin{figure}[h]
\includegraphics[scale=.9]{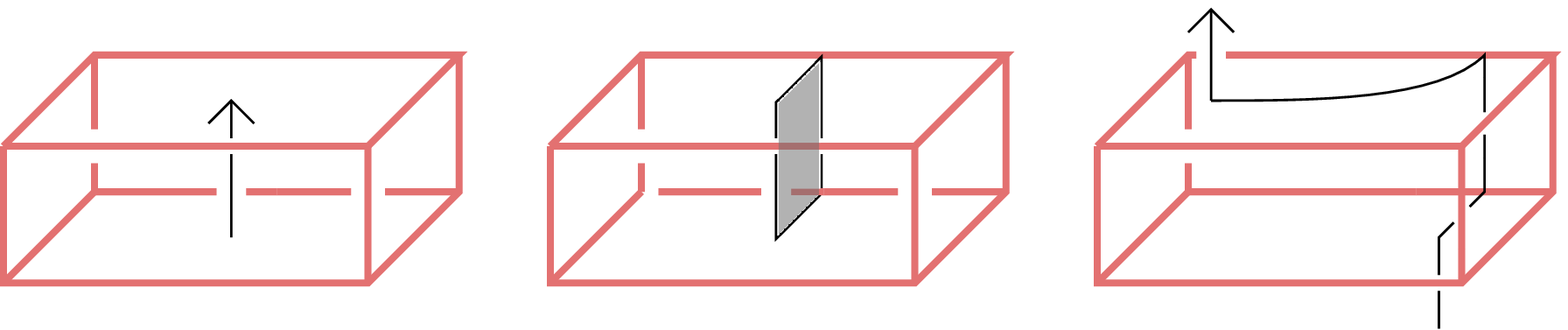}
\caption{Homotoping a Reeb orbit into $\R^{3}\setminus N_{\epsilon}$ as it passes though a $c=-1$ surgery handle.}
\label{Fig:PushingOutMinus1}
\end{figure}

Figure \ref{Fig:PushingOutMinus1} demonstrates the same procedure for orbits as they pass thought surgery handles with surgery coefficient $-1$. In this case we consider squares of the form $\{ p \geq p_{0}, q=q_{0}\}$ through which we homotop our arcs. Note that our choice of homotopy for both surgery coefficients are such that the homotoped arcs traverse $\partial N$ in the $\partial_{q}$-direction in which the components of $\Lambda$ are oriented. 

For a Reeb orbit $\gamma$, we can perform homotopies as described above at the tips of all chords in $\cycword(\gamma)$ to push it to the exterior of $N_{\epsilon}$. Away from the chords, we may arrange that the homotoped orbit traverses the $p=\mp \epsilon$ side of $N_{\epsilon, i}$ when the surgery coefficient of $\Lambda_{i}$ is $\pm 1$. The image of the $\gamma$ after homotopy is shown in the Lagrangian projection in Figure \ref{Fig:Crossings}.

\begin{figure}[h]
\begin{overpic}[scale=.8]{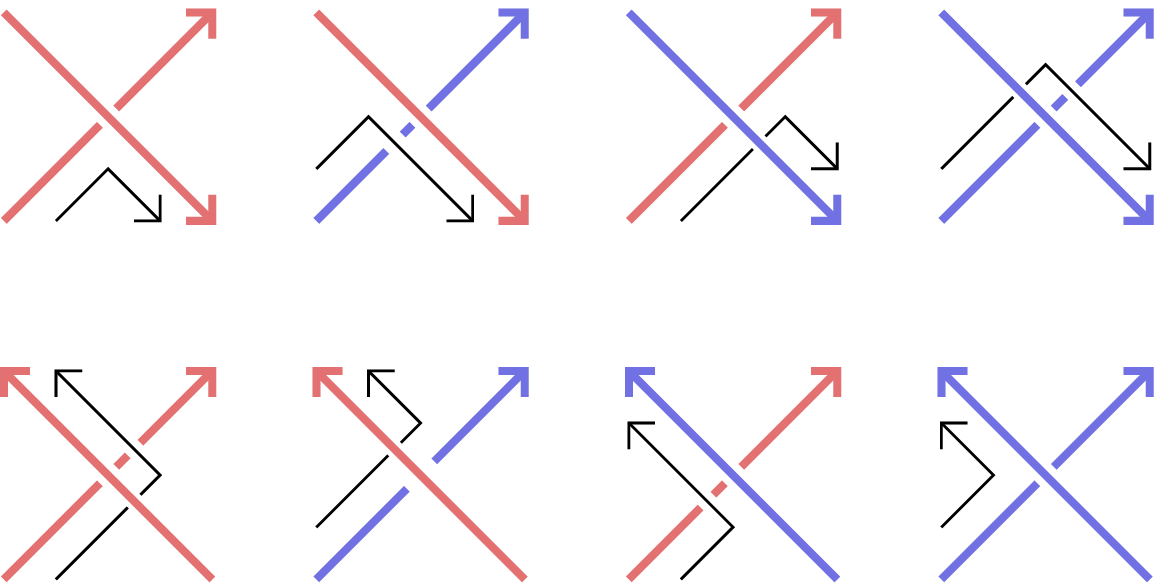}
\put(0, 25){$0\mu_{l^{-}_{j_{k}}} + 0\mu_{l^{+}_{j_{k}}}$}
\put(26, 25){$1\mu_{l^{-}_{j_{k}}} + 0\mu_{l^{+}_{j_{k}}}$}
\put(53.5, 25){$0\mu_{l^{-}_{j_{k}}} + 1\mu_{l^{+}_{j_{k}}}$}
\put(80.5, 25){$1\mu_{l^{-}_{j_{k}}} + 1\mu_{l^{+}_{j_{k}}}$}

\put(-2, -5){$-1\mu_{l^{-}_{j_{k}}} -1\mu_{l^{+}_{j_{k}}}$}
\put(26, -5){$0\mu_{l^{-}_{j_{k}}} -1 \mu_{l^{+}_{j_{k}}}$}
\put(51.5, -5){$-1\mu_{l^{-}_{j_{k}}} + 0\mu_{l^{+}_{j_{k}}}$}
\put(80.5, -5){$0\mu_{l^{-}_{j_{k}}} + 0\mu_{l^{+}_{j_{k}}}$}
\end{overpic}
\vspace*{.7cm}
\caption{Here the homotopy described in Figures \ref{Fig:PushingOutPlus1} and \ref{Fig:PushingOutMinus1} are depicted in the Lagrangian projection. The top (bottom) row shows positive (negative) crossings of $\LambdaPM$. Each subfigure may be rotated by $\pi$. Local contributions to linking numbers are indicated below each subfigure.}
\label{Fig:Crossings}
\end{figure}

The computation of homology classes of orbits in Theorem \ref{Thm:H1} then amounts to packaging the above observations algebraically:

\begin{proof}[Proof of Theorem \ref{Thm:H1}]
We homotop $\gamma$ to $\R^{3}\setminus N_{\epsilon}$ as described above and then apply Theorem \ref{Thm:H1GammaGeneral}. We write $\gamma'$ for the image of $\gamma$ under the homotopy. The linking number of two knots in $\R^{3}$ may be computed from a diagram as half of the signed count of crossings in the diagram. Therefore, in order to compute $[\gamma]$ it suffices to show that the signed count of crossings between $\gamma'$ and each $\Lambda_{i}$ is given by the $\mu_{i}$ coefficients in $\sum (\cross_{j_{k}} + \cross_{j_{k}, j_{k+1}})$.

In a neighborhood of a crossing, $\gamma'$ will be as depicted in Figure \ref{Fig:Crossings} in the Lagrangian projection, where the contribution to the signed count of crossings between $\gamma$ and the $\Lambda_{i}$ are given by the terms
\begin{equation*}
\half\Big((c_{j_{k}}^{-} + \sgn_{j_{k}})\mu_{l^{-}_{j_{k}}} + (c^{+}_{j_{k}} + \sgn_{j_{k}})\mu_{l^{+}_{j_{k}}}\Big).
\end{equation*}
The formula may be verified on a case-by-case basis for each of the eight components of the figure. This is exactly the definition of $\cross_{j_{k}}$ given in Equation \eqref{Eq:ChordCrossingMonomial}.

Away from a crossing, $\gamma'$ will continue following alongside arc-components of the $\Lambda_{i}$, to the right (in the $p > 0$ direction) of $\Lambda$ when the component of $\Lambda$ has coefficient $-1$ and to the left otherwise as it travels from a crossings $j_{k}$ to $j_{k+1}$. The contributions to the signed count of crossings with each of the $\Lambda_{i}$ are given by the coefficients of $\mu_{i}$ in $\cross_{j_{k}, j_{k+1}}$ in the formula as is clear from the definition of the crossing monomial.
\end{proof}

\subsection{Push-outs of Reeb orbits}\label{Sec:PushOutDefn}

We've demonstrated how squares of the form $\{ p \leq p_{0}, q=q_{0} \}\subset N_{\epsilon}$ in the case of $+1$ surgery and of the form $\{ p \geq p_{0}, q=q_{0} \}$ in the case of $-1$ surgery are used to homotop Reeb orbits into $\SurgL \setminus N_{\epsilon} = \R^{3} \setminus \epsilon$ so that the homotoped circles ride along some $\eta_{j_{1},j_{2}}\subset \Lambda$ according to its prescribed orientation. 

Squares of the form $\{ p \geq p_{0}, q=q_{0} \}$ inside of a $c_{i}=+1$ component of $\Lambda$ and of the form $\{ p \leq p_{0}, q=q_{0} \}$ inside of a $c_{i}=-1$ component could also be used. As may be checked with the same local model -- Figures \ref{Fig:PushingOutPlus1} and \ref{Fig:PushingOutMinus1} -- but with opposite the prescribed orientation for $\Lambda$, we may use these squares to homotop an orbit $\gamma$ to $\R^{3} \setminus \epsilon$. Using these squares will result in the homotoped arcs riding along some $\overline{\eta}_{j_{1},j_{2}}\subset \Lambda$.

We then have two choices of homotoping square each time our orbit $\gamma$ passes through $N_{\epsilon}$, with each choice corresponding to a choice of either a $\eta_{j_{1},j_{2}}$ or a $\overline{\eta}_{j_{1},j_{2}}$. Hence for a Reeb orbit $\gamma=(r_{j_{1}}\dots r_{j_{n}})$, a choice of $\zeta_{1}\cdots\zeta_{n}$ with each $\zeta_{j}\in \{\eta_{j_{k},j_{k+1}}, \overline{\eta}_{j_{k}, j_{k+1}} \}$ determines a means of homotoping $\gamma$ into $\R^{3}\setminus N_{\epsilon}$.

\begin{defn}
Provided $\zeta_{1}\cdots\zeta_{n}$ as above, we say that homotopy class of map of the circle in $\R^{3}\setminus N_{\epsilon}$ determined by homotoping $\gamma$ as described above is the \emph{push-out of $\zeta_{1}\cdots\zeta_{n}$}.
\end{defn}

In other words, each orbit string -- recall Definition \ref{Def:OrbitString} -- determines instructions for homotoping $\gamma$ into the complement of the surgery locus. Various examples are depicted in Figure \ref{Fig:TrefoilOrbitStrings}, displaying all push-outs for orbits $(r_{1})$, $(r_{2})$, and $(r_{1}r_{2})$ for $\SurgL$, where $\Lambda$ is the trefoil of Figure \ref{Fig:LagrangianResolutionEx} for both choices of surgery coefficient.

\begin{figure}[h]\begin{overpic}[scale=.6]{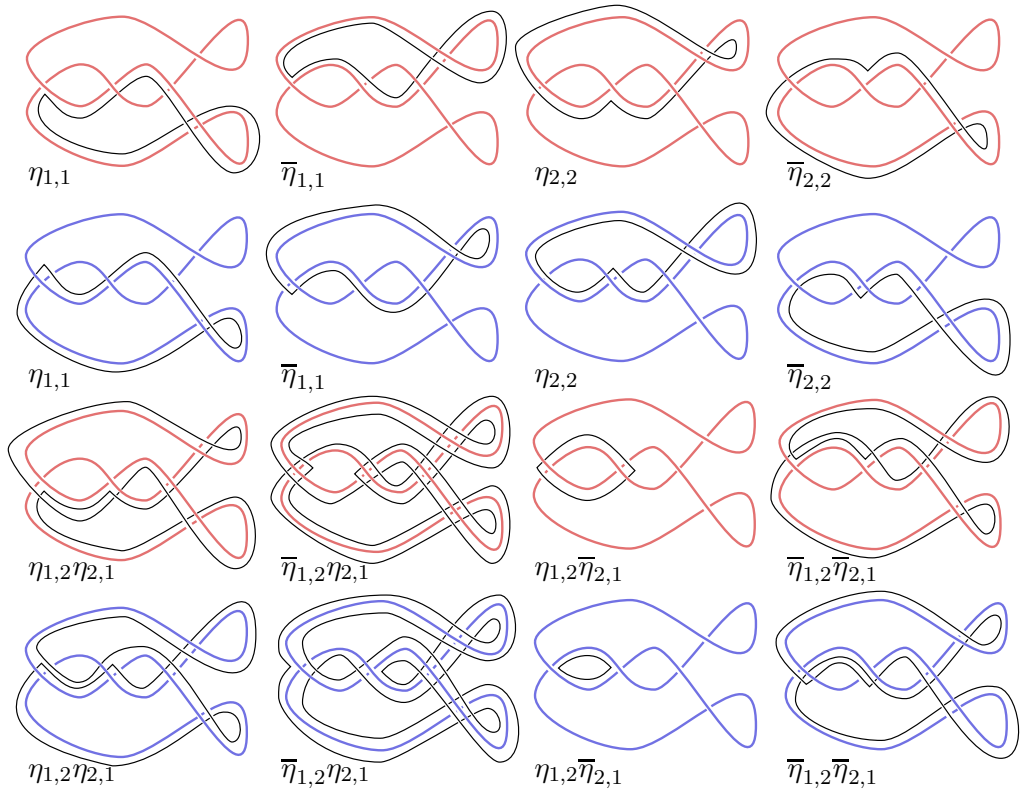}
    \put(2, 58){$\eta_{1, 1}$}
    \put(27, 58){$\overline{\eta}_{1, 1}$}
    \put(52, 58){$\eta_{2, 2}$}
    \put(77, 58){$\overline{\eta}_{2, 2}$}
    \put(2, 38){$\eta_{1, 1}$}
    \put(27, 38){$\overline{\eta}_{1, 1}$}
    \put(52, 38){$\eta_{2, 2}$}
    \put(77, 38){$\overline{\eta}_{2, 2}$}
    \put(2, 19){$\eta_{1, 2}\eta_{2, 1}$}
    \put(27, 19){$\overline{\eta}_{1, 2}\eta_{2, 1}$}
    \put(52, 19){$\eta_{1, 2}\overline{\eta}_{2, 1}$}
    \put(77, 19){$\overline{\eta}_{1, 2}\overline{\eta}_{2, 1}$}
    \put(2, -1){$\eta_{1, 2}\eta_{2, 1}$}
    \put(27, -1){$\overline{\eta}_{1, 2}\eta_{2, 1}$}
    \put(52, -1){$\eta_{1, 2}\overline{\eta}_{2, 1}$}
    \put(77, -1){$\overline{\eta}_{1, 2}\overline{\eta}_{2, 1}$}
\end{overpic}
\caption{Push-outs of Reeb orbits in $\SurgLxi$, where $\Lambda$ is the trefoil of Figure \ref{Fig:TrefoilImmersions}. Default orientations for $\Lambda$ and hence for capping paths are determined by the arrow on $\Lambda$ appearing in that figure. Each subfigure is labeled (to its lower-left) with the capping paths which determine the homotopy shown with homotoped Reeb orbits appearing in black.}
\label{Fig:TrefoilOrbitStrings}
\end{figure}

\section{Surgery cobordisms and Lagrangian disks}\label{Sec:SurgeryCobordisms}

The purpose of this section is to build symplectic cobordisms between the $\SurgLxi$ with specialized properties. We consider the following setup: Take $\Lambda \subset \Rthree$ in good position with $\LambdaZero \subset \Lambda$ non-empty. After performing surgery on $\LambdaPM \subset \Lambda$, we have a contact form $\alpha_{\epsilon}$ on $\SurgLxi$ and consider $\LambdaZero$ as a Legendrian link in $\SurgLxi$. We choose a constant $c = \pm 1$ and denote the contact manifold obtained by performing contact $c$ surgery along $\LambdaZero \subset \SurgLxi$ by $\SurgLxiPrime$, we also denote the contact form on $\SurgLxiPrime$ obtained as $\alpha_{\epsilon}$. We write $N_{\epsilon}^{0}$ for a standard neighborhood of $\LambdaZero \subset \SurgL$ as described in Section \ref{Sec:StandardNeighborhoods} of size $\epsilon$.

\begin{thm}\label{Thm:SurgeryCobordisms}
For any $\epsilon > 0$, there exists a positive constant $C > 0$ and a Liouville cobordism $(W_{c}, \lambda_{c})$ with the following properties:
\be
\item If $c = +1$, the convex end of the cobordism is $(\SurgL, e^{C}\alpha_{\epsilon})$ and the convex end is $(\SurgLPrime, e^{-C}\alpha_{\epsilon})$.
\item If $c = -1$, the convex end of the cobordism is $(\SurgLPrime, e^{C}\alpha_{\epsilon})$ and the convex is $(\SurgL, e^{-C}\alpha_{\epsilon})$.
\item $(W_{c}, \lambda_{c})$ contains a disjoint collection of disks $\disk_{c, i}$ along which $\lambda_{c} = 0$, bounding $\LambdaZero$ in the convex end of cobordism when $c=+1$ and bounding $\LambdaZero$ in the concave end of the cobordism when $c=-1$.
\item A finite symplectization $([-C, C] \times (\SurgL \setminus N_{\epsilon}^{0}), e^{t}\alpha_{\epsilon})$ of $(\SurgL \setminus N_{\epsilon}^{0})$ is contained in $(W_{c}, \lambda_{c})$, so that the restriction of its inclusion map to $(\partial [-C, C])\times (\SurgL \setminus N_{\epsilon}^{0})$ provide the obvious inclusions into $(\SurgL, e^{\pm C}\alpha_{\epsilon})$ and $(\SurgLPrime, e^{\pm C}\alpha_{\epsilon})$.
\ee
\end{thm}

We will construct $(W_{c}, \lambda_{c})$ by attaching $4$-dimensional surgery handles to $\SurgL$. As mentioned in the above theorem, the key properties of our cobordism are that
\be
\item we get exactly the contact forms $\alpha_{\epsilon}$ on its boundaries and 
\item all of the perturbations required to achieve this end happen within a standard neighborhood of $\LambdaZero$ whose size shrinks as $\epsilon$ tends to zero.
\ee
Then the analysis of Sections \ref{Sec:ModelGeometry} through \ref{Sec:CZMaslov} applies to contact forms on the ends of our cobordisms without modification.

We are only slightly modifying known handle attachment constructions -- corresponding to the case $c=-1$ above -- such as appearing in Weinstein's original work \cite{Weinstein:Handles} and Ekholm's \cite{Ekholm:SurgeryCurves}.

An outline of this section is as follows:
\be
\item In Section \ref{Sec:ContactizationForms} we collect lemmas required to perturb contact forms on contactizations, being particularly interested in standard neighborhoods of Legendrian knots.
\item In Section \ref{Sec:SquareHandle} we describe a square surgery handle sitting inside of $\R^{4}$ and outline the properties of its ambient geometry.
\item In Section \ref{Sec:HandleShaping} we flatten the corners of the handle to prepare for later attachment.
\item In Section \ref{Sec:HandleDynamics} we describe Reeb dynamics on the convex end of this handle, showing that its flow is described as a Dehn twist.
\item In Section \ref{Sec:HandleLambdaPerturbations} we modify the handle so that the Dehn twist determined by the Reeb flow is a linear Dehn twist as a described in the gluing construction of Section \ref{Sec:GluingMaps}.
\item In Section \ref{Sec:HandleAttachment} we finalizing our construction by attaching our handle to finite symplectizations of $(\SurgL, \alpha_{\epsilon})$.
\ee

\subsection{Geometry of $1$-forms on contactizations and their symplectizations}\label{Sec:ContactizationForms}

Let $(I \times W, \alpha = dz + \beta)$ be a contactization of an exact symplectic manifold $(W, \beta)$ as in Section \ref{Sec:Contactizations}.

\subsubsection{$\xi$-preserving perturbations}

We first look at how the Reeb vector field changes if we multiply $\alpha$ by a positive function, thereby preserving the contact structure.

\begin{lemma}\label{Lemma:AlphaPerturbation}
Given $H \in \Cinfty(I\times W)$, the Reeb vector field $R_{H}$ of the contact form 
\begin{equation*}
\alpha_{H} = e^{H}(dz + \beta)
\end{equation*}
on $I\times W$ is
\begin{equation*}
R_{H} = e^{-H}\bigg( \big(1 + \beta(X_{H})\big)\partial_{z} - X_{H} - \frac{\partial H}{\partial z}X_{\beta}\bigg).
\end{equation*}
where $X_{H}$ is computed with respect to $d\beta$.
\end{lemma}

This is a straightforward computation. We'll be interested in the following special case:

\begin{lemma}\label{Lemma:AlphaPerturbationSpecific}
Suppose that $H = H(z, p)$ is a smooth function on $I \times I \times \Circle$. Then the Reeb vector field of $\alpha_{H} = e^{H}(dz + p dq)$ is
\begin{equation*}
R_{H} = e^{-H}\left( \left(1 + p\frac{\partial H}{\partial p}\right)\partial_{z} - \frac{\partial H}{\partial p}\partial_{q} - p\frac{\partial H}{\partial z}\partial_{p}\right).
\end{equation*}
and the functions $pe^{H}$ is invariant under $\Flow^{t}_{R_{H}}$.
\end{lemma}

For the last item, we see that the projection of $R_{H}$ onto the $(z, p)$ coordinates is the Hamiltonian vector field associated to $dp \wedge dz$ and the function $pe^{-H}$.

\subsubsection{$\xi$-modifying perturbations}

Now we study perturbations of $\alpha$ which modify $\xi$. Similar modifications of contact forms appear in \cite[Definition 3.1.1]{BH:ContactDefinition} and \cite[Corollary 2.5]{CGHH:Sutures}.

\begin{lemma}\label{Lemma:ZPerturbation}
Given a smooth function $h \in \Cinfty(I\times W, (0, \infty))$ the contact form 
\begin{equation*}
	\alpha_{h} = h dz + \beta
\end{equation*}
is contact if and only if 
\begin{equation*}
	h d\beta + \beta\wedge dh
\end{equation*}
is a symplectic form on each $\{z\} \times W$. If this form is contact, its Reeb vector field $R_{h}$ is
\begin{equation*}
	R_{h} = \big( h - \beta(X_{h}) \big)^{-1}\bigg(\partial_{z}  - X_{h}\bigg)
\end{equation*}
where $X_{h}$ is computed with respect to $d\beta$. The contact structure $\xi_{h} = \ker(\alpha_{h})$ is given by
\begin{equation*}
	\xi_{h} = \{ hV - \beta(V)\partial_{z}\ :\ V \in TW \}.
\end{equation*}
\end{lemma}

The following technical result will allow us to modify the Reeb vector field on standard neighborhoods of Legendrians so that the flow map from the bottom to the top of the neighborhood realizes a Dehn twist $\tau_{g}$ associated to a function $g$. For applications to surgery later in this section, it will be important to keep track of the size of our neighborhood.

\begin{prop}\label{Prop:TrivialTwistRealization}
For positive constants $\epsilon_{p}, \epsilon_{g} > 0$ let $g = g(p): I_{\epsilon_{p}} \rightarrow \R$ be a smooth function which vanishes for all orders on $\partial I_{\epsilon_{p}}$ and satisfies the point-wise bound $|g(p)| \leq \epsilon_{g}$. Then for constants $\epsilon_{z}$ and $\epsilon_{t}$ satisfying
\begin{equation*}
	\epsilon_{p}\epsilon_{g} \leq \half \epsilon_{z},\ \frac{\epsilon_{t}\epsilon_{z}}{2(1 + \epsilon_{t})}
\end{equation*}
there exists a function $h = h(z, p)$ on $I_{\epsilon_{z}} \times I_{\epsilon_{p}}$ and an exact symplectic manifold
\begin{equation*}
	\left([-\epsilon_{t}, 0]\times I_{\epsilon_{z}} \times I_{\epsilon_{p}} \times \Circle, \lambda\right)
\end{equation*}
such that the following conditions hold:
\be
\item $\lambda|_{\{ -\epsilon_{t}\}\times I_{\epsilon_{z}} \times I_{\epsilon_{p}} \times \Circle} = e^{-\epsilon_{t}}(dz + pdq)$,
\item $\lambda|_{\{ 0\}\times I_{\epsilon_{z}} \times I_{\epsilon_{p}}\times \Circle} = \alpha_{h}$ where $\alpha_{h}$ is as in Lemma \ref{Lemma:ZPerturbation} for a positive function $h$,
\item $s \alpha_{h} + (1-s)(dz + pdq)$ is contact for all $s \in [0, 1]$,
\item $\alpha_{h} - (dz + pdq)$ and all of its derivatives vanish along $\partial (I_{\epsilon_{z}}\times I_{\epsilon_{p}}\times \Circle)$,
\item the Reeb vector field $R_{h}$ of $\alpha_{h}$ satisfies $dz(R_{h}) > 0$ everywhere,
\item for each point $(p, q) \in I_{\epsilon_{p}} \times \Circle$ a flow-line of $R_{h}$ passing through $(-\epsilon_{z}, p, q)$ will pass through $(\epsilon_{z}, p, q + g(p))$,
\item the Liouville vector field of $\lambda$ agrees with $\partial_{t}$ on a collar neighborhood of the boundary of its domain.
\ee
\end{prop}

\begin{proof}
We first outline the contact forms we'll need. Consider functions of the form $h = 1 + F(z)G(p)$ on $I_{\epsilon_{z}}\times I_{\epsilon_{p}}$ and $1$-forms
\begin{equation*}
	\alpha_{h} = hdz + p dq
\end{equation*}
as studied in Lemma \ref{Lemma:ZPerturbation}. We assume $F \geq 0$ and that both $F,G$ and all of their derivatives vanish on collar neighborhoods of the boundary of their domains. By Lemma \ref{Lemma:ZPerturbation}, $\alpha_{h}$ is contact if and only if
\begin{equation}\label{Eq:TwistContactCondition}
0 < 1 + FG - pF\frac{\partial G}{\partial p}.
\end{equation}

Second we outline the construction of Liouville forms which interpolate between $\alpha = dz + pdq$ and $\alpha_{h}$. Consider functions $E$ on an interval $[-\epsilon_{t}, 0]$ satisfying $E(-\epsilon_{t}) = 0$ and $E(0) = 1$ with $\frac{\partial^{k} E}{\partial t^{k}} = 0$ for all $k > 0$ at the endpoints of its domain and $\frac{\partial E}{\partial t} \geq 0$ everywhere. Define a $1$-form 
\begin{equation*}
\lambda_{EFG} \in \Omega^{1}([0, \epsilon_{t}] \times I_{\epsilon_{z}} \times I_{\epsilon_{p}} \times \Circle)
\end{equation*}
determined by
\begin{equation}\label{Eq:LambdaEFG}
	\lambda_{EFG} = e^{t}\bigg( \bigr(1 + E(t)F(z)G(p)\bigr)dz + pdq \bigg).
\end{equation}
Then we compute 
\begin{equation}\label{Eq:TwistSymplecticCondition}
d\lambda_{EFG} \wedge d\lambda_{EFG} = e^{2t}\bigg( 1 + EFG - pEF\frac{\partial G}{\partial p} + \frac{\partial E}{\partial t}FG \bigg) dt\wedge dz \wedge dp \wedge dq.
\end{equation}

We seek to specify the $E$, $F$, and $G$ so that:
\be
\item $\alpha_{h}$ is contact and its flow determines a Dehn twist by $g$,
\item $d\lambda_{EFG}$ is symplectic, and
\item the sizes of our neighborhood and symplectic cobordism -- governed by the constants $\epsilon_{z}$ and $\epsilon_{t}$ -- are reasonably small.
\ee

First we show that $G$ is determined by $g$. If $\alpha_{h}$ is contact its Reeb vector field is computed
\begin{equation*}
	R_{h} = \left(1 + FG - pF\frac{\partial G}{\partial p}\right)^{-1}\left(\partial_{z} - F\frac{\partial G}{\partial p}\partial_{q}\right).
\end{equation*}
This Reeb vector field is particularly friendly in that it preserves $p$ and provides us with a separable O.D.E. For provided an initial condition $(z_{0}, p_{0}, q_{0})$ and some $z > z_{0}$ we see that after some time $t > 0$, $\Flow^{t}_{R_{h}}$ will pass through the point $(z, p_{0}, q)$ with
\begin{equation*}
	q = q_{0} - \frac{\partial G}{\partial p}\int_{z_{0}}^{z}F(Z)dZ.
\end{equation*}
In order to realize the flow from $\{ -\epsilon_{z}\} \times \R \times \Circle$ to $\{ \epsilon_{z} \} \times \R \times \Circle$ as a Dehn twist by $g$ we need
\begin{equation*}
	G(p) = -\bigg(\int_{-\epsilon_{z}}^{\epsilon_{z}} F(z)dz\bigg)^{-1}\int_{-\infty}^{p}g(P)dP.
\end{equation*}
This quantity is well defined by our presumption that $g$ is compactly supported.

With this choice of $G$, the contact condition provided by Equation \eqref{Eq:TwistContactCondition} is equivalent to
\begin{equation}\label{Eq:TwistContactConditionSpecific}
F\cdot \bigg(\int_{-\infty}^{p}g(P)dP - p g(p)\bigg) \leq \int_{-\epsilon_{z}}^{\epsilon_{z}} F(z)dz
\end{equation}
for all $(z, p, q)$. The condition that $d\lambda_{EFG}$ is symplectic provided by Equation \eqref{Eq:TwistSymplecticCondition} is equivalent to
\begin{equation}\label{Eq:TwistSymplecticConditionSpecific}
EF\cdot\bigg( \int_{-\infty}^{p}g(P) dP - pg(p)\bigg) + \frac{\partial E}{\partial t} F\cdot \bigg(\int_{-\infty}^{p}g(P) dP\bigg) \leq \int_{-\epsilon_{z}}^{\epsilon_{z}} F(z)dz
\end{equation}
Now choose $F$ and a constant $\epsilon_{F}$ so that the following are satisfied:
\begin{equation*}
	\epsilon_{F} = \sup |F(z)|,\quad \epsilon_{F}\epsilon_{z} = \int_{-\epsilon_{z}}^{\epsilon_{z}}F(z)dz.
\end{equation*}
Its easy to see by drawing pictures bump functions that these choices can be made. Then Equation \eqref{Eq:TwistContactConditionSpecific} is satisfied so long as
\begin{equation*}
	\epsilon_{p}\epsilon_{g} \leq \half \epsilon_{z}
\end{equation*}
and since $0 \leq E \leq 1$ we have that Equation \eqref{Eq:TwistSymplecticConditionSpecific} is satisfied so long as
\begin{equation*}
	\left(2 + \frac{\partial E}{\partial t}\right)\epsilon_{p}\epsilon_{g} \leq \epsilon_{z}.
\end{equation*}
Choose $E$ so that $\sup \frac{\partial E}{\partial t} = \frac{2}{\epsilon_{t}}$. Then this last inequality we seek to satisfy becomes
\begin{equation*}
	2(1 + \epsilon_{t}^{-1})\epsilon_{p}\epsilon_{g} \leq \epsilon_{z} \iff \epsilon_{p}\epsilon_{g} \leq \frac{\epsilon_{t}\epsilon_{z}}{2(1 + \epsilon_{t})}.
\end{equation*}
\end{proof}

\subsection{The square handle}\label{Sec:SquareHandle}

Having established the above lemmas, we proceed with the construction of our symplectic handle. Here we construct a square Weinstein handle sitting in $\R^{4}$.

Consider the Liouville form on $\R^{4} = \C^{2}$,
\begin{equation*}
\lambda_{0} = \sum_{1}^{2} 2x_{i}dy_{i} + y_{i}dx_{i}.
\end{equation*}
This is a potential for the standard symplectic form $d\lambda_{0} = dx_{i}\wedge dy_{i}$ with Liouville vector field
\begin{equation*}
	X_{\lambda_{0}} = 2x_{i}\partial_{x_{i}} - y_{i}\partial_{y_{i}}
\end{equation*}
whose time $t$ flow is given by
\begin{equation}\label{Eq:XLambdaFlow}
\Flow^{t}_{X_{\lambda_{0}}}(x, y) = \left(e^{2t}x, e^{-t}y\right).
\end{equation}

For $\rho_{0} > 0$, consider also the convex set with corners
\begin{equation*}
	\disk_{\rho_{0}} \times \disk = \{ |x| \leq \rho_{0}, |y| \leq 1 \} \subset \R^{4}
\end{equation*}
whose smooth boundary strata we denote
\begin{equation*}
M^{+}_{\rho_{0}} = \partial \disk_{\rho_{0}} \times \disk,\quad M^{-}_{\rho_{0}} = \disk_{\rho_{0}} \times \partial\disk.
\end{equation*}
Then $X_{\lambda_{0}}$ is positively transverse to the $M^{\pm}_{\rho_{0}}$ if we equip $M^{+}_{\rho_{0}}$ with the outward-pointing orientation and equip $M^{-}_{\rho_{0}}$ with its inward-pointing orientation. Therefore $\lambda_{0}|_{M^{\pm}}$ is contact. Applying $\Flow^{t}_{X_{\lambda_{0}}}$ for $t \in (-\infty, 0]$ we have embeddings of the negative half-infinite symplectizations of the $(M^{\pm}_{\rho_{0}}, \lambda_{0}|_{TS^{\pm}})$ into $\R^{4}$,
\begin{equation}\label{Eq:SymplectizationFlow}
\Flow^{t}_{X_{\lambda_{0}}}\circ i^{\pm}:(-\infty, 0]\times M^{\pm}_{\rho_{0}} \rightarrow \R^{4}
\end{equation}
where $i^{\pm}:M^{\pm}_{\rho_{0}} \rightarrow \R^{4}$ denote the inclusion mappings.

\subsubsection{Reeb trajectories across the square handle}\label{Sec:SquareHandleTrajectories}

The Reeb vector field $R_{\rho_{0}}$ along $M^{+}_{\rho_{0}}$ is
\begin{equation*}
R_{\rho_{0}} = \frac{1}{2\rho_{0}^{2}} x_{i}\partial_{y_{i}} \implies \Flow^{t}_{R_{\rho_{0}}}(x, y) = \left(x, y + \frac{t}{2\rho_{0}^{2}} x\right).
\end{equation*}
Starting at points $(x_{\theta}, y_{0}) = (\rho_{0}\cos(\theta), \rho_{0}\sin(\theta), 1, 0)$, Reeb trajectories are
\begin{equation*}
\Flow^{t}_{R_{\rho_{0}}}(x_{\theta}, y_{0}) = \left(\rho_{0}\cos(\theta), \rho_{0}\sin(\theta), 1 + \frac{t}{2\rho_{0}}\cos(\theta), \frac{t}{2\rho_{0}}\sin(\theta)\right).
\end{equation*}
In order that such a trajectory does not immediately exit $M^{+}_{\rho_{0}}$ (maintaining the condition $|y| \leq 1$ for small $t \geq 0$), we must have $\theta \in [\frac{\pi}{2}, \frac{3\pi}{2}]$. These trajectories touch $\partial M^{+}_{\rho_{0}}$ when
\begin{equation*}
1 = \left(1 + \frac{t}{2\rho_{0}}\cos(\theta)\right)^{2} + \left(\frac{t}{2\rho_{0}}\sin(\theta)\right)^{2} \iff -4\rho_{0} \cos(\theta) = t
\end{equation*}
at which point the $y$ coordinate will be
\begin{equation*}
y_{\theta} = (1 -2\cos^{2}(\theta), -2\cos(\theta)\sin(\theta)) = (-\cos(2\theta), -\sin(2\theta)) = (\cos(2\theta - \pi), \sin(2\theta - \pi)).
\end{equation*}
We can then measure the angle from $y_{0}$ to $y_{\theta}$ as $2\theta - \pi \in [0, 2\pi]$.

\subsection{Identification of the concave end of the handle as a $1$-jet space}\label{Sec:HandleNegativeEndJetSpace}

We define an embedding of a standard neighborhood of a Legendrian into $M^{-}_{1}$ as
\begin{equation*}
	\Phi_{-}(z, p, q) = \left(z\cos - \frac{p}{2\pi}\sin, z\sin + \frac{p}{2\pi}\cos, \cos, \sin\right)
\end{equation*}
where the arguments of $\cos$ and $\sin$ are both $2\pi q$. The map parameterizes $M^{-}_{1}$ so that
\be
\item $2\pi q$ is the angle in the $y$-plane,
\item $z = x\cdot y$,
\item $p = x\cdot \frac{\partial y}{\partial q}$, and
\item $|x|^{2} = z^{2} + \left(\frac{p}{2\pi}\right)^{2}$.
\ee

The tangent map of $\Phi_{-}$ is computed
\begin{equation*}
T\Phi_{-} = \begin{pmatrix}
	\cos & -\frac{\sin}{2\pi} & -2\pi z \sin - p\cos\\
	\sin & \frac{\cos}{2\pi} & 2\pi z \cos - p\sin\\
	0 & 0 & -2\pi\sin\\
	0 & 0 & 2\pi\cos
\end{pmatrix}
\end{equation*}
with incoming basis $\partial_{z}, \partial_{p}, \partial_{q}$ and outgoing basis $\partial_{x_{1}},\partial_{x_{2}}, \partial_{y_{1}},\partial_{y_{2}}$ from which it follows that
\begin{equation*}
\Phi_{-}^{*}\lambda_{0} = dz + pdq.
\end{equation*}

We can extend $\Phi_{-}$ to an embedding of the symplectization of the $1$-jet space into $\R^{4}$ by
\begin{equation}\label{Eq:PhiMinusExplicit}
	\begin{aligned}
		\overline{\Phi}_{-}(t, z, p, q) &= \Flow^{t}_{X_{\lambda_{0}}}\circ\Phi_{-}(z, p, q)\\
		&= \left(e^{2t}\left(z\cos - \frac{p}{2\pi}\sin\right), e^{2t}\left(z\sin + \frac{p}{2\pi}\cos \right), e^{-t}\cos, e^{-t}\sin\right).
	\end{aligned}
\end{equation}
By Equation \eqref{Eq:XLambdaFlow} and $\Lie_{X_{\lambda_{0}}}\lambda_{0} = \lambda_{0}$, we have
\begin{equation}\label{Eq:PhiMinusStarLambda}
\overline{\Phi}_{-}^{*}\lambda_{0} = e^{t}(dz + p dq).
\end{equation}

\subsection{Shaping the handle}\label{Sec:HandleShaping}

Here we shape our handle so that the manifold obtained by the handle attachment will be smooth. Moreover, we will choose a specific shape which allows us to control Reeb dynamics on the surgered contact manifold.

Pick a positive constant $\rho_{1} < \rho_{0}$ and a smooth function $B = B(\rho): (0, \infty) \rightarrow [0, \infty)$ satisfying the conditions
\be
\item $B(\rho) = \log\left(\sqrt{\frac{\rho_{0}}{\rho}}\right) = -\half (\log (\rho) - \log(\rho_{0}))$ for $\rho \in (0, \rho_{1})$,
\item $B(\rho) = 0$ for $\rho > \rho_{0}$, and
\item $0 \leq -\frac{\partial B}{\partial \rho} < \rho^{-1}$ everywhere.
\ee
Along $\rho \in (0, \rho_{1})$, we have $\frac{\partial B}{\partial \rho} = -\frac{1}{2\rho}$ so that our last condition is satisfied. To find such a function $B$ we can take a smoothing of the piece-wise smooth function
\begin{equation}\label{Eq:PLBump}
B^{PW}(\rho) = \begin{cases}
\log\left(\sqrt{\frac{\rho_{0}}{\rho}}\right) & \rho \leq \rho_{0} \\
0 & \rho \geq \rho_{0}.
\end{cases}
\end{equation}

Let $N = I \times I \times \Circle$ be a standard neighborhood of a Legendrian $\Lambda$ with $\Lambda = \{ 0 \} \times  \{0\} \times \Circle$. Using the function $B$ we define an embedding
\begin{equation*}
\begin{gathered}
\Phi_{H}: (N \setminus \Lambda) \rightarrow \R^{4}, \quad \Phi_{H} = \Flow^{H}_{X_{\lambda_{0}}}\circ \Phi_{-},\\
H(p, z) = B(\rho(p, z)), \quad \rho(p, z) = \sqrt{z^{2} + \left(\frac{p}{2\pi}\right)^{2}}.
\end{gathered}
\end{equation*}

We outline some important properties of the map $\Phi_{H}$:
\be
\item From Equation \eqref{Eq:PhiMinusStarLambda}, $\Phi_{H}^{\ast}\lambda_{0} = e^{H}(dz + pdq)$.
\item Along the set $\{ z^{2} + \left(\frac{p}{2\pi}\right)^{2} \geq \rho_{0} \}$, $\Phi_{H}$ is the same as $\Phi_{-}$. 
\item From the first property characterizing $B$ and equation \eqref{Eq:PhiMinusExplicit} we see that on the set $\{ \rho \leq \rho_{1} \}$, the $x$ and $y$ coordinates of the embedding satisfy
\begin{equation}
\left| x\circ \Phi_{H}(z, p, q)\right| = e^{2H}\sqrt{z^{2} + \left(\frac{p}{2\pi}\right)^{2}} = \rho_{0}, \quad 
\left| y\circ \Phi_{H}(z, p, q)\right| = e^{-H} = \sqrt{\frac{\rho}{\rho_{0}}}.
\end{equation}
\ee
From the last equation, we have the equivalences
\begin{equation*}
\begin{aligned}
\Phi_{H}(\{ \rho \leq \rho_{1} \}) &= \left\{ |x| = \rho_{0}, |y| \leq \sqrt{\frac{\rho_{1}}{\rho_{0}}} \right\} \setminus \{ y = 0 \}, \\
\overline{\Phi_{H}(\{ \rho \leq \rho_{1} \})} &= \left\{ |x| = \rho_{0}, |y| \leq \sqrt{\frac{\rho_{1}}{\rho_{0}}}\right\}
\end{aligned}
\end{equation*}

\begin{figure}[h]\begin{overpic}[scale=.6]{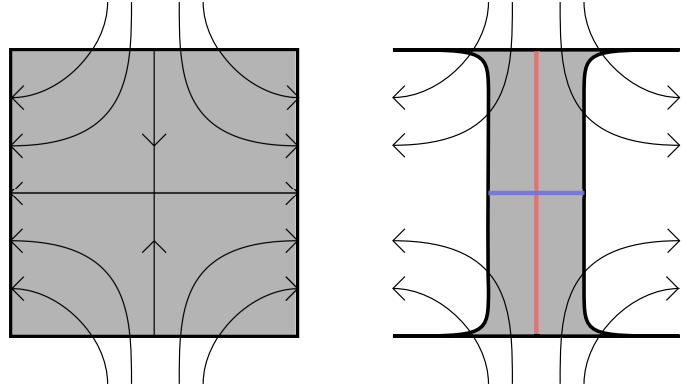}
\end{overpic}
\caption{On the left is the square handle $\disk_{\rho_{0}} \times \disk$ and on the right is the handle $W_{H}$. Flow lines of $X_{\lambda_{0}}$ transversely pass through $M^{\pm}_{\rho_{0}}$ and $M_{H}$. The Lagrangian disks $\{ x = 0 \}$ and $\{ y = 0 \}$ are shown in red and blue, respectively.}
\label{Fig:HandlePerturbation}
\end{figure}

The closure of the image of $\Phi_{H}$ in $\R^{4}$ is a smooth hypersurface $M_{H}$ which is positively transverse to $X_{\lambda_{0}}$. We write $M_{H}$ for this hypersurface and define $W_{H}\subset \R^{4}$ to be the set enclosed by $M^{-}_{1}$ and $M_{H}$. The handle $W_{H}$ is depicted in the right-hand side of Figure \ref{Fig:HandlePerturbation}.

\subsection{Analysis of $R_{H}$ over $M_{H}$}\label{Sec:HandleDynamics}

Here we analyze dynamics on $M_{H}$ of the Reeb vector field $R_{H}$ for the contact form $\alpha_{H} = \lambda_{0}|_{M_{H}}$. Because of our use of the imprecisely defined function $B$, we won't be able to solve for $\Flow_{R_{H}}^{t}$ explicitly. However, we'll be able to capture enough information about this flow for the applications to handle attachment.

On the complement of the set $\{ z = p = 0 \}$ our contact form is $\alpha_{H} = e^{H}(dz + pdq)$ whose Reeb field will be denoted $R_{H}$. Writing $\rho = \rho(p, z)$, we compute
\begin{equation*}
dH = \frac{\partial B}{\partial \rho}d\rho, \quad d\rho = \rho^{-1}\left(zdz + \frac{p}{(2\pi)^{2}}dp \right).
\end{equation*}
Now apply Lemma \ref{Lemma:AlphaPerturbationSpecific}, to compute $R_{H}$ using the coordinates $(z, p, q)$ on $M_{H} \setminus \{ y = 0\}$ as
\begin{equation}\label{Eq:HReeb}
R_{H} = e^{-H}\left( \left(1 + \frac{\partial B}{\partial \rho}\rho^{-1}\left(\frac{p}{2\pi}\right)^{2}\right)\partial_{z} - \frac{p}{\rho}\frac{\partial B}{\partial \rho}\left(\frac{1}{(2\pi)^{2}}\partial_{q} + z\partial_{p}\right) \right).
\end{equation}

\begin{figure}[h]\begin{overpic}[scale=1.0]{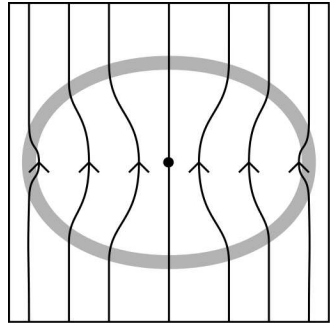}
\end{overpic}
\caption{Projections of flow-lines of $R_{H}$ to the $(p, z)$ coordinates. By Lemma \ref{Lemma:AlphaPerturbationSpecific}, these are the level sets of $p H$. Here $p$ points to the right and $z$ points upward. The dot represents the circle $|x| = 0$ along which our flow is not defined in the $(z, p, q)$ coordinate system. The region $\rho \in [\rho_{1}, \rho_{0}]$ where the function $B^{PW}$ is smoothed to obtain $B$ is shaded.}
\label{Fig:SurgeryFlowlines}
\end{figure}

Here is a collection of observations regarding $R_{H}$ and its flow:
\be
\item The $\partial_{z}$ part of $R_{H}$ is always strictly positive. This is a consequence of the inequalities
\begin{equation*}
-\frac{\partial B}{\partial \rho}\rho^{-1}\left(\frac{p}{2\pi}\right)^{2} \leq -\frac{\partial B}{\partial \rho}\rho^{-1}\rho^{2} = -\frac{\partial B}{\partial \rho}\rho < 1 \implies 1 + \frac{\partial B}{\partial \rho}\rho^{-1}\left(\frac{p}{2\pi}\right)^{2} > 0
\end{equation*}
following from the definition of $\rho$ and the third defining property of the function $B$.
\item For each $p$ and $\epsilon > 0$, a flow-line starting at the point $(-\epsilon, p, q)$ will pass through some $(\epsilon, p, q')$. This follows from the facts that $(p e^{-H})(-z, p) = (p e^{-H})(z, p)$ and that the projection of $R_{H}$ onto the $(z, p)$ plane is Hamiltonian with respect to $dp \wedge dz$ as per Lemma \ref{Lemma:AlphaPerturbationSpecific}. See Figure \ref{Fig:SurgeryFlowlines}.
\item The flow-line passing through $(-\epsilon, 0, q)$ will pass through the point $(\epsilon, 0, q + \half)$. To see this, observe that such a flow-line with such an initial condition must flow up into the circle $\{ z = p = 0 \}$ along the line $\{ p = 0\}$ and compare with the definition of the map $\Phi_{H}$.
\item A twist map $f_{H, \rho_{0}}: I_{2\pi\rho_{0}} \rightarrow \Circle$ defined by following the flow-line of $R_{H}$ passing through $(-\epsilon, p, q)$ to a point $(\epsilon, p, q + f_{H, \rho_{0}}(p))$. By the properties we've used to specify $B$, the derivatives of $f_{H, \rho_{0}}$ are supported on $I_{2\pi \rho_{0}}$ as $R_{H}$ coincides with $\partial_{z}$ outside of this region. Likewise, $R_{H} = \partial_{z}$ on $\{ |\rho| \geq \rho_{0} \}$.
\item The twist map satisfies
\begin{equation*}
f_{H, \rho_{0}}(-p) = -f_{H, \rho_{0}}(p), \quad f_{H, \rho_{0}}(-2\pi\rho_{0}) = f_{H, \rho_{0}}(2\pi\rho_{0}) = 0.
\end{equation*}
The first equality follows from the fact that the $\partial_{z}$ factor of $R_{H}$ is a function of $p^{2}$ while the $\partial_{p}$ and $\partial_{q}$ factors are anti-symmetric in $p$. The second equality follows from the previous item.
\item\label{Obs:Twisting} As $\frac{\partial B}{\partial p} \leq 0$, the $\partial_{q}$ coefficient of $R_{H}$ from Equation \eqref{Eq:HReeb} has sign equal to $\sgn(p)$ where is it non-zero. Hence $f_{H, \rho_{0}}$ always twists to the right for $p < 0$ and to the left along $p > 0$. This is the expected behavior of a positive Dehn twist.
\ee

\begin{prop}
Write $\tilde{f}_{H, \rho_{0}}: I_{2\pi\rho_{0}} \rightarrow \R$ for the lift of the twist map $f_{H, \rho_{0}}: I_{2\pi \rho_{0}} \rightarrow \R$ with initial condition
\begin{equation*}
\tilde{f}_{H, \rho_{0}}(-2\pi \rho_{0}) =0 \implies \tilde{f}_{H, \rho_{0}}(0) = -\half, \quad \tilde{f}_{H, \rho_{0}}(2\pi_{\rho_{0}}) = -1
\end{equation*}
by the preceding analysis. Suppose that $\tilde{f}: I_{2\pi_{\rho_{0}}} \rightarrow [-1, 0]$ is a decreasing function also satisfying the above equalities. Then for $H$ constructed using a function $B$ which is sufficiently $\mathcal{C}^{0}$ close to the function $B^{PW}$, the estimate
\begin{equation}\label{Eq:HandleTwistLinearApproxBound}
| \tilde{f}_{H, \rho_{0}}(p) - \tilde{f}(p) | \leq \half
\end{equation}
is satisfied for all $p \in I_{2\pi \rho_{0}}$.
\end{prop}

\begin{proof}
The analysis of Section \ref{Sec:SquareHandleTrajectories} provides a very explicit approximation of the function $\tilde{f}_{H, \rho_{0}}$. Let's consider the degenerate case when $B = B^{PW}$ as described in Equation \eqref{Eq:PLBump}, writing $H^{PW}$ for the associated piece-wise smooth function. Then $M_{H^{PW}}$ will be piece-wise smooth as a submanifold of $\R^{4}$. We have a $\mathcal{C}^{0}$ flow on $M_{H^{PW}}$ given by following $\partial_{z}$ on $M_{H^{PW}} \setminus M^{+}_{\rho_{0}}$ and by following $R_{\rho_{0}}$ on $M^{+}_{\rho_{0}}$.  We can view $M_{H^{PW}}$ as a smooth manifold by viewing its non-smooth part to be the graph of a $\mathcal{C}^{0}$ function, and observe that $M_{H}$ and $M_{H^{PW}}$ coincide along the sets $|x| \geq \rho_{0}$.

We look at flow trajectories passing over the $q=0$ slice of our neighborhood, which corresponds to the $y = y_{0}$ subset of $M_{H^{PW}}$. Then $\Phi_{H^{PW}}$ maps the arc
\begin{equation*}
A = \left\{ z = -\sqrt{\rho_{0}^{2} - \left(\frac{p}{2\pi}\right)^{2}},\quad p \in [-2\pi \rho_{0}, 2\pi \rho_{0}], \quad q=0 \right\}
\end{equation*}
to the semi-circle
\begin{equation}\label{Eq:SemiCircParam}
\left\{ \left(-\sqrt{\rho_{0}^{2} - \left(\frac{p}{2\pi}\right)^{2}}, \frac{p}{2\pi}, 1, 0\right)\quad  p \in [-2\pi \rho_{0}, 2\pi \rho_{0}]\right\} \subset M_{H^{PW}}.
\end{equation}
In the language of Section \ref{Sec:SquareHandleTrajectories}, this semi-circle is the set
\begin{equation*}
\left\{ (x_{\theta}, y_{0}) , \quad \theta \in \left[\frac{\pi}{2},\frac{3\pi}{2}\right] \right\} \subset \R^{4}
\end{equation*}
equipped with a clockwise parameterization (determined by the variable $p$). When a trajectory passes through the handle entering at angle $\theta = \theta(p) \in \left[\frac{\pi}{2},\frac{3\pi}{2}\right]$ determined by $p$ in the $x$-plane, it will come out on the top of our neighborhood at angle $2\theta - \pi$ as described in Section \ref{Sec:SquareHandleTrajectories}. Therefore when using the piece-wise smooth handle the lift $\tilde{f}_{H, \rho_{0}}$ of our continuous flow map $f_{H^{PW}, \rho_{0}}$ can be written
\begin{equation*}
\tilde{f}_{H^{PW}, \rho_{0}}(p) = \begin{cases}
0 & p < -2\pi\rho_{0} \\
\frac{1}{2\pi}(2\theta(p) - \pi) & p \in [-2\pi\rho_{0}, 2\pi \rho_{0}] \\
-1 & p > 2\pi\rho_{0},
\end{cases}
\end{equation*}
where $\theta(p)$ is the angle in the $x$-plane given by Equation \eqref{Eq:SemiCircParam}.

As the $p$ coordinate wraps around the semicircle of Equation \eqref{Eq:SemiCircParam} in a clockwise fashion, we conclude that $\tilde{f}_{H^{PW}, \rho_{0}}$ is a decreasing function. Moreover, $\tilde{f}_{H, \rho_{0}}(-2\pi\rho) = 0$, $\tilde{f}_{H^{PW}, \rho_{0}}(0) = -\half$, and $\tilde{f}_{H^{PW}, \rho_{0}}(-2\pi\rho) = -1$, just like our test function $\tilde{f}$. Therefore both $\tilde{g} = \tilde{f}, \tilde{f}_{H^{PW}, \rho_{0}}$ must satisfy
\begin{equation}\label{Eq:FliftProperties}
\tilde{g}([-2\pi \rho_{0}, 0]) = [-\half, 0], \quad \tilde{g}([0, 2\pi \rho_{0}]) = [-\half, -1].
\end{equation}
From this we conclude that Equation \eqref{Eq:HandleTwistLinearApproxBound} holds for $\tilde{f}_{H^{PW}, \rho_{0}}$ given any function $\tilde{f}$ satisfying the required properties.

Now we suppose that $B$ is smooth and $\mathcal{C}^{0}$ close to $B^{PW}$. Then the twist map $\tilde{f}_{H, \rho_{0}}$ will be $\mathcal{C}^{0}$ close to $\tilde{f}_{H^{PW}, \rho_{0}}$. This is because the twist map is entirely determined by the flow of the arc $A$ across the surgery handle. All Reeb trajectories starting at points in $A$ which pass through the $|y| < 1$ portion of the handle will intercept the region $\rho \in [\rho_{1}, \rho_{0}]$ where $B^{PW}$ is smoothed to obtain $B$. See Figure \ref{Fig:SurgeryFlowlines}.

Because of item \eqref{Obs:Twisting} in the observations preceding this proof, $\tilde{f}_{H, \rho_{0}}(p) < 0$ for $p < 0$, we can guarantee to that $\tilde{f}_{H, \rho_{0}}$ satisfies $\tilde{g}([-2\pi \rho_{0}, -\delta]) \subset [-\half, 0]$ for some arbitrarily small $\delta > 0$. Therefore we have $|\tilde{f}_{H, \rho_{0}}(p) - \tilde{f}(p)| \leq \half$ for $p\in [-2\pi \rho_{0}, -\delta]$. By continuity we can also ensure that the desired inequality holds for $p \in [-\delta, 0]$ by making $B$ $\mathcal{C}^{0}$ close enough to $B^{PW}$. The same arguments with modified notation apply to ensure that Equation \eqref{Eq:HandleTwistLinearApproxBound} holds over $[0, 2\pi\rho_{0}]$ for $B$ close enough to $B^{PW}$.
\end{proof}

\subsection{Perturbing $\lambda_{0}$}\label{Sec:HandleLambdaPerturbations}

Now that we've shown that the flow from the set $\{ z = -\rho_{0} \}$ to the set $\{ z = \rho_{0} \}$ defined by $R_{H}$ is determined by a Dehn twist by $f_{H, \rho_{0}}$, which is supported on $I_{2\pi\rho_{0}} \times \Circle$. Moreover, Equation \eqref{Eq:HandleTwistLinearApproxBound} tells us that we can use Proposition \ref{Prop:TrivialTwistRealization} to correct $\lambda_{0}$ and so that the flow over our handle will be an ``approximately linear twist'' satisfying Assumptions \ref{Assump:DiskIntersection}. We now carry out the details of this correction.

We construct a new coordinate system $(z, p, q)$ on $M_{H}$ as follows: On the set $\{ |y| = 1, |x| > \rho_{0} \}$ we have coordinates $(p, q, z)$ on $M_{H}$ coming from the embedding $\Phi_{-}$ as $M^{-}_{1}$ and $M_{H}$ overlap on this region. To get a standard coordinate system on $M_{H}$, apply the map
\begin{equation}\label{Eq:OverFlowOne}
(z, p, q) \mapsto \Flow_{R_{H}}^{z + \rho_{0}}\circ \Phi_{-}(-\rho_{0}, p, q), \quad z \in I_{\rho_{0}}.
\end{equation}
With respect to this coordinate system
\begin{equation*}
\lambda_{0}|_{M_{H}} = dz + p dq.
\end{equation*}
Due to our identification of the flow from the top to bottom of this region -- with respect to the $(z, p, q)$ coordinates on $M_{-}$ -- as being determined by a Dehn twist by $f_{H, \delta}$, we have that the change of coordinates on the overlap
\begin{equation*}
\big( M_{H} \setminus \{ \rho < \rho_{0} \} \big) \rightarrow M^{-}_{1}
\end{equation*}
is given exactly as the gluing map of Section \ref{Sec:GluingMaps} with the ``height perturbation function'' -- denoted in that section as $H_{f, \epsilon}$ -- uniquely determined by $f_{H, \rho_{0}}$.

\begin{figure}[h]\begin{overpic}[scale=.6]{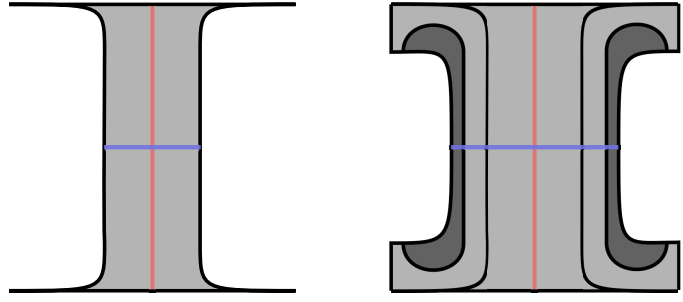}
\end{overpic}
\caption{On the left, the rounded handle $W_{H}$ of Figure \ref{Fig:HandlePerturbation}. On the right, the perturbed handle $W \subset \R^{4}$. The region along which $\lambda_{0}$ is modified -- as in Proposition \ref{Prop:TrivialTwistRealization} -- is shaded in dark gray. The lightly shaded extension of $W_{H}$ indicates extension by the Liouville flow.}
\label{Fig:HandleExtension}
\end{figure}

We seek to modify $f_{H, \epsilon}$ using Proposition \ref{Prop:TrivialTwistRealization} so that the flow over the convex boundary of our handle agrees satisfies linear dynamics assumptions described in Assumptions \ref{Assump:DiskIntersection}. To this end, let $g: I_{2\pi \rho_{0}} \rightarrow \R$ be a function satisfying the following properties:
\be
\item A Dehn twist by $f_{\rho_{0}}(0) = f_{H, \rho_{0}}(p) + g(p)$ satisfies Assumptions \ref{Assump:DiskIntersection} with $\frac{\partial f_{\rho_{0}}}{\partial p}(0) = (2\pi\rho_{0})^{-1}$.
\item $|g(p)| \leq \half$.
\item $g$ and all of its derivatives vanish along $\partial I_{2\pi \rho_{0}}$.
\ee
Such a choice of $g$ is possible by Equation \eqref{Eq:HandleTwistLinearApproxBound}. According to Proposition \ref{Prop:TrivialTwistRealization}, using 
\begin{equation}
\epsilon_{p} = \epsilon_{z} = 2\pi\rho_{0},\quad \epsilon_{g} = \half
\end{equation}
and $\epsilon_{t}$ arbitrarily large we can modify the contact form within the coordinate system on $M_{H}$ by
\be
\item adding a finite symplectization $([0, e^{\epsilon_{t}}]\times M_{H}, \lambda_{0} = e^{t}(dz + p dq))$ to obtain a handle $W \subset \R^{4}$ containing $W_{H}$,
\item perturbing $\lambda_{0}$ within a proper subset of this region to obtain a contact form $\lambda$ on $W$
\ee
so that the flow over $M_{H}$ in the coordinates $(z, p, q)$ is given by a Dehn twist by $f_{\rho_{0}}$. A schematic for this extension and perturbation is depicted on the right-hand side of Figure \ref{Fig:HandleExtension}.

Now we re-work through Equation \ref{Eq:OverFlowOne} and its consequences this time using the new Reeb vector field $R$. The map
\begin{equation}\label{Eq:OverFlowTwo}
(z, p, q) \mapsto \Flow_{R}^{z + \rho_{0}}\circ \Phi_{-}(-\rho_{0},  p, q)
\end{equation}
will provide us with a coordinate system $(z, p, q)$ on the convex boundary of $W$. Now our attaching map is determined by the composition of the Dehn twists
\begin{equation*}
(p, q) \mapsto (p, q + g(p)),\quad (p, q) \mapsto (p, q + f_{H, \rho_{0}})
\end{equation*}
yielding a Dehn twist by $f_{\rho_{0}}$ as desired.

\subsection{Attaching the handle to finite symplectizations}\label{Sec:HandleAttachment}

To finish our construction, we attach the handle $(W, \lambda)$ to a finite symplectization of $(\SurgL, \alpha_{\epsilon})$. In doing so, we will omit the specific choices of $\rho_{0}$ required, provided that they are determined by $\epsilon$ as described in Section \ref{Def:AlphaEpsilon}. Likewise, we assume that $\LambdaZero$ consists of a single connected component to simplify notation.

We first consider the the case $c = -1$, the map $\Phi_{-}$ provides us with an identification of  standard neighborhood $N_{\epsilon}^{0}$ of $\LambdaZero$. The map $\Phi_{-}$ provides us with an identification of this neighborhood with the convex end of the handle $W$. By considering $\SurgL$ as being contained in the top of a finite symplectization $[-C, 0]\times \SurgL$ we may attach the handle $W$ via this identification to obtain a $4$ manifold along which we set
\begin{equation*}
\lambda_{-1}|_{W} = \lambda,\quad \lambda_{-1}|_{[-C, 0]\times \SurgL} = e^{t}\alpha_{\epsilon}
\end{equation*}
Outside of a neighborhood of the form $\{ \rho(p, z) < \const \} \subsetneq \{ 0 \}\times N_{\epsilon}^{0}$ we may extend by a some $[0, C]\times \SurgL \setminus \{ |z| + |p| < \const \}$ over which we take 
\begin{equation*}
\lambda_{-1}|_{[0, C]\times \SurgL \setminus \{ \rho(p, z)< \const \}} = e^{t}\alpha_{\epsilon}.
\end{equation*}
The constant $C$ may be chosen so that the top of this region coincides with the convex end of the handle $W$. By the fact that the the perturbation of $\lambda$ described in the previous subsection occur away from the attaching locus, we have that $W_{-1}$ is smooth with $\lambda_{-1}$ determining a smooth form, as desired. The disk $\disk_{-1}$ is obtained by taking the intersection of the plane $\{ |x| = 0 \} \subset \R^{4}$ with the handle $W_{H} \subset W$ -- depicted as the red line in Figures \ref{Fig:HandlePerturbation} and \ref{Fig:HandleExtension} -- and then extending through $[-C, 0] \times \SurgL$ by a Lagrangian cylinder $[-C, 0]\times \LambdaZero$.

Now set $c=+1$. In the case our disk $\disk_{+1}$ is taken to be the intersection of the plane $\{ |y| = 0 \}$ with the handle $W$. According to Equation \eqref{Eq:LambdaEFG}, $\lambda|_{\disk_{+1}} = 0$. Using the coordinates $(p, q, z)$ on Equation \eqref{Eq:OverFlowTwo} we may identify a neighborhood of the boundary of this disk with a standard neighborhood of $\LambdaZero$, which we may consider is being contained in the bottom of a finite symplectization $[0, C]\times \SurgL$. We extend the disk by a Lagrangian cylinder over $\LambdaZero$ within $[0, C]\times \SurgL$ so that its boundary lies in $\{ C \}\times \SurgL$. To complete the construction of our Liouville cobordism $(W_{+1}, \lambda_{+1})$ we layer on $[-C, 0]\times \SurgL \setminus \{ |z| + |p| < \const \}$ so that the concave end of the cobordism is smooth and coincides with $(\SurgLPrime, \alpha_{\epsilon})$.

\section{Holomorphic foliations, intersection numbers, and the $\Lambda$ quiver}\label{Sec:FoliationsAndQuivers}

In this section we describe some tools which allow us to frame geometric questions regarding holomorphic curves in the $4$ manifolds relevant to this article -- symplectizations and surgery cobordisms -- as algebraic problems. We will largely be relying on intersection positivity for holomorphic curves in $4$-manifolds \cite[Appendix E]{MS:Curves} and basic algebraic topology.

These tools serve to establish some properties of holomorphic curves in $\R \times \SurgL$ and surgery cobordisms which we believe to be true intuitively but which are more difficult to articulate precisely: Curves with ``high energy'' should look like Legendrian $RSFT$ disks as they pass though the complement of the surgery locus $N_{\epsilon}$ while ``low energy'' curves should be trapped inside of the union of $N_{\epsilon}$ with a neighborhood of the chords of $\Lambda$ and have constrained asymptotics. This will be formalized in Section \ref{Sec:ExposedHidden} as the \emph{exposed/hidden alternative}.

The first three subsections deal with geometry: In Section \ref{Sec:SteinProduct}, we describe special almost complex structures on contactizations and how combinatorial $LRSFT$ disks can be ``lifted'' to holomorphic disks. Section \ref{Sec:NStandard} described how these complex structures $J$ can be used on large open subsets of symplectizations and surgery cobordisms. Next, in Section \ref{Sec:JFoliation} we show that such $J$ endows open subsets of our $4$ manifolds with a foliation by $J$ holomorphic planes. This is another area of analysis which is considerably simplified by working with $\SurgLxi$ rather than $\SurgLxiClosed$.

The remainder of the section is concerned with algebra: Section \ref{Sec:IntersectionNumbers} describes some properties of intersections between these planes and finite energy holomorphic curves asymptotic to chords and orbits of the $R_{\epsilon}$. These intersection numbers are essentially homological invariants of curves. In the event that the intersection numbers all vanish, an alternative book-keeping device can be used to keep track of holomorphic curves -- an object we call the \emph{$\Lambda$ quiver}, $Q_{\Lambda}$. This quiver can be used as an algebraic tool to encode $LCH^{cyc}$ chain complexes -- see Remark 4.1 of \cite{BEE:LegendrianSurgery} -- but we will be most interested in the fact that it is a quotient of a space homotopy equivalent to the complement of the $\C$-foliated region of our $4$-manifold.

\subsection{Model almost complex structures on symplectizations of contactizations of Stein manifolds}\label{Sec:SteinProduct}

Here we review some generalities regarding holomorphic curves in symplectizations of contactizations of Stein manifolds. For the purposes of this paper, we're really only interested in the cotangent bundles of the real line -- for $\Rthree$ is the $1$-jet space of $\R$ -- though the results are no harder to state or prove in fuller generality. The results here are known: For example, they are implicit in the convexity arguments of \cite{CGHH:Sutures} and definitions of $LCH$ moduli spaces in \cite{EES:LegendriansInR2nPlus1}.

Let $W$ be a manifold of dimension $2n$ with complex structure $J$ and suppose that $F \in C^{\infty}(\Sigma)$ is such that 
\begin{equation*}
\beta = -dF \circ J
\end{equation*}
is a Liouville form on $W$. In other words, $(W, J, F)$ is a \emph{Stein manifold} except that we have omitted any requirements regarding transversality between $X_{\beta}$ and $\partial W$. Define a contact $1$-form $\alpha = dz + \beta$ on $\R \times W$ so that
\begin{equation*}
\xi = \{ V - \beta(V)\partial_{z}\ : \ V \in TW\}.
\end{equation*}
for $V \in TW$. We can define a $J'$ adapted to the symplectization of $(\R \times W, \alpha)$ by
\begin{equation*}
J' \partial_{t} = \partial_{z},\quad J'(V - \beta(V)\partial_{z}) = JV - \beta(JV)\partial_{z}.
\end{equation*}

As previously mentioned, we're primarily concerned with the cases $W = \R \times I$ for a $1$-manifold $I$ with $\beta = pdq = -\half d(p^{2})\circ j$. We get $(\R^{2}, -ydx)$ by a change of coordinates.

\begin{lemma}\label{Lemma:HarmonicLift}
If a map $(t, z, u): \Sigma' \rightarrow \R \times \R \times W$ is $(J', j)$ holomorphic then
\be
\item $z$ is harmonic and
\item $u$ is $(J, j)$-holomorphic.
\ee
Moreover if $\Sigma'$ is simply connected and we have $(z, u)$ for which $z$ is harmonic and $u$ is $(J, j)$ holomorphic, then there exists $t: \Sigma' \rightarrow \R$ for which $(t, z, u)$ is $(J', j)$ holomorphic. Such $t$ is unique up to addition by a constant.
\end{lemma}

\begin{proof}
This is a local calculation: Take coordinates $x, y$ on $\disk$ which we may consider being contained in $\Sigma'$ with $j$ denoting the standard complex structure on $T\disk$. We will be studying Equation \eqref{Eq:DelbarBreakdown}.

Write $\delbar_{J', j}(t, z, u) = \half(T(t, z, u) + J'T(t, z, u)\circ j)$ for the usual Cauchy-Riemann operator. For $V \in TW$, we calculate
\begin{equation*}
\pi_{\alpha}(a\partial_{z} + V) = V - \beta(V)\partial_{z}
\end{equation*}
so that the $\xi$-valued part $\half(\pi_{\alpha} + J'\circ \pi_{\alpha} \circ j)$ of $\delbar_{J', j}(s, t, u)$ depends only on $u$. Then
\begin{equation*}
(\pi_{\alpha} + J'\circ \pi_{\alpha} \circ j)(s, t, u) = \delbar_{J, j}u -\beta \circ (\delbar_{J, j}u)\partial_{z}
\end{equation*}
where $\delbar_{J, j}$ is the Cauchy Riemann operator for $u$. The $TW$ part of this expression vanishes if and only if $u$ is $(J, j)$-holomorphic which would imply that the $\partial_{z}$ part of the expression vanishes as well.

Assuming that $(t, z, u)$ is $(J', j)$-holomorphic, then $u$ is $(J, j)$ holomorphic and $dt = ((z, u)^{\ast}\alpha)\circ j$, implying
\begin{equation*}
\begin{gathered}
u^{*}\beta = -u^{\ast}(dF\circ J) = -d(F\circ u)\circ j \\
(z, u)^{*}\alpha \circ j = dz\circ j + d(F\circ u),\\
d^{2} t = d((z, u)^{*}\alpha \circ j ) = d(dz \circ j) = -\Delta(z) = 0
\end{gathered}
\end{equation*}
where $\Delta$ is the Laplacian. Therefore $z$ is harmonic.

Now provided harmonic $z$ and $(J, j)$ holomorphic $u$ for simply connected $\Sigma'$, the above expression tells us that $(z, u)^{*}\alpha \circ j$ is closed, and so is exact. Therefore we have a function $t$ -- determined uniquely up to addition by scalars -- satisfying $dt = (z, u)^{*}\alpha \circ j$. Then by the above formula and Equation \eqref{Eq:DelbarBreakdown}, we have that $(t, z, u)$ is $(J', j)$ holomorphic.
\end{proof}

\begin{cor}[Drawing-to-disk correspondence]\label{Cor:DrawingDisk}
Suppose that $(W, J, F)$ is a Stein manifold of complex dimension $1$ and $\Lambda$ is a chord generic Legendrian link in $(I \times W, dz - dF\circ J)$. Suppose that 
\begin{equation*}
u: \disk \setminus \{ p_{j} \} \rightarrow W
\end{equation*}
is an orientation-preserving immersion of the disk with a finite set of boundary punctures $\{ p_{k} \}$ removed so that $u(\partial \disk \setminus \{ p_{k} \}) \subset \pi_{W}(\Lambda)$. Then there exists a set $\{ p_{k}' \}$ of boundary punctures on the disk, a diffeomorphism $\phi: \disk \setminus \{ p_{k}' \} \rightarrow \disk \setminus \{ p_{k} \}$, and functions $t, z: \rightarrow \R$ such that
\begin{equation*}
(t, z, u\circ \phi): \disk \setminus \{ p_{k}' \} \rightarrow \R \times \R \times W
\end{equation*}
is $(J', j)$ holomorphic with $(z, u\circ \phi)(\partial \disk \setminus \{ p_{k} \}) \subset  \Lambda$. Provided $\phi$, $z$ is uniquely determined and $t$ is uniquely determined up to addition by a positive constant.
\end{cor}

\begin{proof}
Because $u$ is an immersion, we can force it to be $(J, j')$ holomorphic for some almost complex structure $j'$ on $\disk$ by defining $j'\partial_{x} = (Tu)^{-1}J (Tu)\partial_{x}$. We can then find a diffeomorphism $\phi$ which is $(j', j)$ holomorphic by the uniformization theorem. 

By the chord genericity and smoothness of $\Lambda$ there exists a unique, bounded, smooth function $z_{\partial}$ on $\partial \disk \setminus \{ p_{k}' \}$ for which 
\begin{equation*}
(z_{\partial}, u\circ \phi) \in \Lambda.
\end{equation*}
Applying \cite[Chapter 6, Section 4.2]{Ahlfors}, there is a function $z:\disk \setminus \{ p_{k}' \} \rightarrow \R$ solving the Direchlet problem
\begin{equation*}
\Delta(f) =0,\quad z|_{\partial \disk \setminus \{ p_{k}' \}} = z_{\partial}
\end{equation*}
which is unique by the maximum principal. By Lemma \ref{Lemma:HarmonicLift}, we can find $t$ for which $(t, z, u\circ \phi)$ is $(J', j)$ holomorphic as desired.
\end{proof}

\subsection{$N$-standard almost complex structures}\label{Sec:NStandard}

As always, let $N_{\epsilon}$ be a tubular neighborhood of $\Lambda$ whose complement we may consider to be a codimension-$0$ submanifold of either $\R^{3}$ or $\SurgL$. Define
\begin{equation*}
\widetilde{N}_{\epsilon} = \pi_{xy}^{-1}(\pi_{xy}(N_{\epsilon}))
\end{equation*}
which we may consider as an open set in either of $\R^{3}$ or $\SurgL$. Denote its complement by $\widetilde{N}_{\epsilon}^{\complement}$. 

\begin{defn}\label{Def:NStandard}
We say that an almost complex structure $J$ on $\R\times \SurgL$ is \emph{$N$-standard} if its restriction to $\SurgXi$ agrees with the standard almost complex structure $J_{0}$ on $\SurgXi$ described by Equation \eqref{Eq:Jstd} on $\widetilde{N}_{\epsilon}^{\complement}$ as well as on a neighborhood
\begin{equation*}
N_{C, \infty} = \{ x^{2} + y^{2} + z^{2} > C \}
\end{equation*}
of the puncture of our $3$-manifold for some $\epsilon, C > 0$. In order that $J$ be adapted to the symplectization, we require $J\partial_{t} = \partial_{z}$ on $\R \times \widetilde{N}_{\epsilon}^{\complement}$.
\end{defn}

We may define $N$-standard for almost complex structures on completions of surgery cobordisms $(W_{c}, \lambda_{c})$ of Section \ref{Sec:SurgeryCobordisms} analogously as the cobordisms contain the symplectizations of $(\SurgL \setminus N_{\epsilon}, \alpha_{std})$.

For an $N$-standard almost complex structure $J$ and a $(J, j)$ holomorphic curve
\begin{equation*}
U: \Sigma' \rightarrow \R \times \SurgL,
\end{equation*}
we have that the along $U^{-1}(\R \times N_{C, \infty})$ we can write $U = (t, z, u)$. By Lemma \ref{Lemma:HarmonicLift} $z$ harmonic and $u$ holomorphic, so that $x\circ u$ and $y \circ u$ are harmonic as well. It follows that $-d(d(z^{2} + |u|^{2})\circ j)$ is non-negative as an area form on $\Sigma'$. Hence for $C' > C$, finite energy curves with punctures asymptotic to chords and orbits of $R_{\epsilon}$ cannot touch spheres of radius $C'$ by the maximum principle.

\subsubsection{Compatibility with perturbation schemes and adaption to symplectizations}

Note that perturbations of almost complex structures required to achieve transversality required to define $\SFT$ curve counts in $\R \times \SurgL$ or $(\overline{W}_{c}, \overline{\lambda}_{c})$ may be defined in arbitrarily small neighborhoods of the orbits of $R_{\epsilon}$ \cite[Section 5]{BH:ContactDefinition} and these orbits are properly contained in open sets unconstrained by the $N$-standard condition. Hence these perturbations may be carried out for $N$-standard almost complex structures while maintaining their defining properties. 

Similarly, the cobordisms $(W_{c}, \lambda_{c})$ of Section \ref{Sec:SurgeryCobordisms} are designed to support $N$-standard almost complex structures which are adapted to their cylindrical ends. For such cobordisms, we'll be additionally interested in studying somewhere injective curves positively asymptotic to chords of the Legendrian boundaries of the disks $\disk_{c, i} \subset W_{c}$ with Lagrangian boundary. See Section \ref{Sec:PlaneBubbling}. In this context, the perturbation scheme of \cite[Section 2]{Ekholm:SurgeryCurves} may be applied, which likewise only deforms Cauchy-Riemann equations in arbitrarily small neighborhoods of chords and orbits. Again, there is no lack of compatibility with the $N$-standard condition.

\begin{assump}\label{Assump:Transversality}
Throughout the remainder of this section, we assume that any almost complex structure $J$ on a symplectization or surgery cobordisms is $N$-standard and that all somewhere injective curves under consideration are regular. When discussing surgery cobordisms, we assume that $J$ is adapted to the cylindrical ends of its completion and that almost complex structures on symplectizations are adapted.
\end{assump}

\subsection{Semi-global foliation by holomorphic planes}\label{Sec:JFoliation}

Here we describe holomorphic foliations by infinite energy planes in symplectizations and surgery cobordisms.

\subsubsection{$\C$-foliations in symplectizations}

Observe that $\widetilde{N}_{\epsilon}^{\complement}$ is foliated by embedded, $\R$-parameterized Reeb orbits of the form $t \rightarrow (t, x_{0}, y_{0})$. Then $\R \times \widetilde{N}_{\epsilon}^{\complement}$ is foliated by holomorphic planes parameterized
\begin{equation*}
(s, t) \mapsto (s, t, x, y)
\end{equation*}
for $(x, y) \in \R^{2} \setminus \pi_{x, y}(N_{\epsilon})$. We denote each such unparameterized plane as $\C_{x, y}$.

\subsubsection{$\C$-foliations in surgery cobordisms}

For the following, we require that $\LambdaZero$ be non-empty. The link $\LambdaPM$ is allowed to be empty, in which case we would have $\SurgLxi = \Rthree$ and set $\alpha_{\epsilon} = dz - ydx$. Let $(W_{c}, \lambda_{c})$ be a surgery cobordism associated to the pair
\begin{equation*}
\LambdaZero \subset \SurgLxi,\quad c \in \{ \pm 1\}
\end{equation*}
as described in the introduction of Section \ref{Sec:SurgeryCobordisms}, with completion $(\overline{W}_{c}, \overline{\lambda}_{c})$. Because the handles are attached along a neighborhood of $\Lambda^{0}$, we can view $\R \times \tilde{N}^{\complement}$ as a subset of $\overline{W}_{c}$ which is also foliated by infinite energy planes $\C_{x, y}$.

\subsection{Intersection numbers}\label{Sec:IntersectionNumbers}

For the following, let $(\Sigma, j)$ be a compact Riemann surface, possibly with boundary, with fixed collections of interior points $p_{k}^{int}$ and boundary points $p_{k}^{\partial}$. As usual we write $\Sigma'$ for $\Sigma$ with all of its marked points removed. When discussing completions $(\overline{W}_{c}, \overline{\lambda}_{c})$ we write 
\begin{equation*}
\overline{\disk}_{c, i} \subset \overline{W}_{c}
\end{equation*}
for the Lagrangian planes obtained by extending the disks $\disk_{c, i}$ of Theorem \ref{Thm:SurgeryCobordisms} by the positive (negative) half-infinite Lagrangian cylinders over their Legendrian boundaries when $c=1$ (respectively, $c=-1$).

\begin{defn}
We say that a holomorphic map $U: \Sigma' \rightarrow \overline{W}_{c}$ is a \emph{$\overline{W}_{c}$ curve} if it its boundary is mapped to the $\overline{\disk}_{c, i}$, its boundary punctures are asymptotic to chords of their Legendrian boundaries, and if all interior punctures are asymptotic to closed Reeb orbits at the convex and concave ends of $\overline{W}_{c}$. 

We say that a holomorphic map $U: \Sigma' \rightarrow \R \times \SurgL$ is a \emph{$\R \times \SurgL$ curve} if the boundary of $\Sigma'$ is mapped to the Lagrangian cylinder over $\LambdaZero$, its boundary punctures are asymptotic to chords of $\LambdaZero$ and its interior punctures are asymptotic to closed orbits of $R_{\epsilon}$.
\end{defn}

We recall -- see \cite[Definition E.2.1]{MS:Curves} -- that provided a pair of maps $u_{i}:\Sigma'_{i} \rightarrow W$, from surfaces $\Sigma'_{i}$, $i=1, 2$ into a $4$-manifold $W$ whose images are disjoint outside of some open sets $S_{i} \subset \Sigma_{i}$ with compact closures outside of which the $\Sigma_{i}$ are disjoint, then we can define a \emph{intersection number} $u_{1}\cdot u_{2} \in \Z$ by perturbing the $u_{i}$ along the $S_{i}$ so that the maps are transverse and counting their intersections with signs.\footnote{We're taking a slight modification of \cite[Definition E.2.1]{MS:Curves} by defining the intersection number to be the sum of the \emph{local intersection numbers} over all points of intersection. This is feasible for holomorphic curves in $4$-manifolds with our hypotheses as distinct curves have isolated intersections \cite[Proposition E.2.2]{MS:Curves}.}

\begin{thm}\label{Thm:IntersectionContinuity}
Suppose that $U$ is a $\overline{W}_{c}$ curves or a $\R \times \SurgL$ curve. Then for $(x, y) \in \R^{2} \setminus \pi_{x, y}(N_{\epsilon})$, the intersection number $\C_{x, y}\cdot U_{s} \in \Z$ is well-defined and non-negative. Furthermore, they are homological invariants in the following sense:
\be
\item \textbf{Boundary-less curves in symplectizations}: Suppose that $U$ is a $\R \times \SurgL$ curve positively asymptotic to a collection $\gamma^{+}$ of Reeb orbits and negatively asymptotic to some $\gamma^{-}$. Then the intersection number $\C_{x, y}\cdot U$ depends only on the relative homology class 
\begin{equation*}
[\pi_{\SurgL}(U)] \in H_{2}(\SurgL, \gamma^{+} \cup \gamma^{-}).
\end{equation*}
\item \textbf{Curves in symplectizations with Lagrangian boundary}: Suppose that $U$ is a $\R \times \SurgL$ curve asymptotic to collections $\gamma^{\pm}$ of Reeb orbits and collections of chords $\kappa^{\pm}$ of $\LambdaZero$. Then the intersection number $\C_{x, y}\cdot U$ depends only on the relative homology class 
\begin{equation*}
[\pi_{\SurgL}(U)] \in H_{2}(\SurgL, \gamma^{+} \cup \gamma^{-} \cup \kappa^{+} \cup \kappa^{-} \cup \LambdaZero)
\end{equation*}
\item \textbf{Boundary-less curves in surgery cobordisms}: Suppose that $U$ is a $\overline{W}_{c}$ curve positively asymptotic to a collection of closed Reeb orbits $\gamma^{+}$ in $\partial^{+} W$ and negatively asymptotic to some collection of closed Reeb orbits $\gamma^{-}$ in $\partial^{-} W$. Then the intersection number $\C_{x, y}\cdot U$ depends only on the relative homology class 
\begin{equation*}
[U] \in H_{2}(W_{c}, \gamma^{+} \cup \gamma^{-}).
\end{equation*}
Here we view $U$ as a cobordism in the compact manifold $W_{c}$ bounding the orbit collections $\gamma^{\pm}$ in its boundary.
\item \textbf{Curves in surgery cobordisms with Lagrangian boundary}: Suppose that $U$ is a $\overline{W}_{c}$ curve positively asymptotic collection of closed Reeb orbits $\gamma^{\pm}$ in $\partial^{+} W$ with boundary punctures asmptotic to some collection $\kappa^{\pm}$ of chords of the Legendrian boundaries of disks $\disk_{k}$. Then the intersection number $\C_{x, y}\cdot U$ depends only on the relative homology class 
\begin{equation*}
[U] \in H_{2}(W_{c}, \gamma^{+} \cup \gamma^{-} \cup \kappa^{+} \cup \kappa^{-} \cup (\cup \overline{\disk}_{c, i})).
\end{equation*}
\ee
\end{thm}

\begin{proof}
To check well-definition, we need to ensure that any intersections between $\C_{x, y}$ and $U(\Sigma')$ occur away from the boundary and punctures of $\Sigma'$ so that intersection numbers are independent of perturbation required in their definition. By our boundary conditions, $U$ must be such that there exists some open neighborhood $S \subset \Sigma'$ of the punctures and boundary of $\Sigma$ which map into the complement of $\R \times (\SurgL \setminus N_{\epsilon})$. The images of the complements of the $S_{s}$ must be contained in some compact set of the form $[-C_{1}, C_{1}]\times (\SurgL \setminus N_{\epsilon})$. Likewise the images of the complements of the $S_{s}$ must be bounded in the $z$ coordinate on $\R^{3}$. Hence all intersections occur within a subset of the form $[-C_{1}, C_{1}]\times [-C_{2}, C_{2}] \times \{ (x, y) \} \subset \C_{x, y}$ implying that the $\C_{x, y}\cdot U_{s} \in \Z$ are well-defined. 

Intersection non-negativity follows from positivity of intersections of holomorphic curves in $4$-manifolds. See, for example \cite[Section E.2]{MS:Curves}. For homological invariance, we will work out the details in the case of boundary-less curves in symplectizations. The other cases follow similar reasoning.

As in the statement of the theorem, we can slightly perturb $U$ near its asymptotic ends to obtain a $2$ cycle in $[-C , C] \times \SurgL$ bounding $\{ C\} \times \gamma^{+} - \{ -C \} \times \gamma^{-}$ for some large $C > 0$. Using the coordinates on $\R \times \R^{3}$ in which we may consider $\C_{x, y}$ to be contained, each intersection between $U$ and $\C_{x, y}$ occurs at some $(t, z, x, y)$. Possibly perturbing $U$ near each such intersection to achieve transversality and isolation of intersections, the sign of each intersection is given by the sign of $T\C_{x, y} \wedge TU$ considered as an oriented ray in the orientation line bundle $\R \partial_{t}\wedge \partial_{z} \wedge \partial_{x} \wedge \partial_{y}$ for $T_{(t, z, x, y)}W$. As $T\C_{x,y} = \Span_{\R}(\partial_{t}, \partial_{z})$, this sign only depends on the $\partial_{x}, \partial_{y}$ part of the tangent map $TU$ of $U$. Hence the intersection number $\C_{x, y}\cdot U$ only depends on $(x, y)$ and $\pi_{\SurgL}\circ U$.
\end{proof}

The ending of the above proof also immediately implies the following:

\begin{lemma}
Suppose that $U:\disk \setminus \{ p_{k} \} \rightarrow \R \times \R^{3}$ is a holomorphic disk determined by an immersion $u: \disk \setminus \{ p_{k} \} \rightarrow \R^{2}$ as in Corollary \ref{Cor:DrawingDisk}. Given a point $(x, y) \in \R^{2} \setminus \pi_{x, y}(N_{\epsilon})$, the intersection number in computed
\begin{equation*}
\C_{x, y} \cdot U = \# u^{-1}\big( (x, y) \big).
\end{equation*}
\end{lemma}

\subsection{Bases and energy bounds}\label{Sec:BasesAndEnergy}

Here we'll reduce the information of the $\C_{x, y}$ down to that of a finite collection of planes. Write $\region_{k}$ for the connected components of $\R^{2} \setminus \pi_{x, y}(N_{\epsilon})$ of finite area and write
\begin{equation*}
\energy_{k} = \int_{\region_{k}} dx\wedge dy
\end{equation*}
for their areas. There is also a single connected component $\R^{2} \setminus \pi_{x, y}(N_{\epsilon})$ of infinite area which we will denote by $\region_{\infty}$.

Pick a point $(x_{k}, y_{k})$ within the interior of each $\region_{k}$ as well as a point $(x_{\infty}, y_{\infty}) \in \region_{\infty}$. We'll call such a choice of indices and points a \emph{point basis} for $\Lambda$. Provided a point basis we may abbreviate
\begin{equation*}
\C_{k} = \C_{(x_{k}, y_{k})}.
\end{equation*}

Such a choice allows us to package a simple-to-state energy estimate:

\begin{prop}\label{Prop:EnergyBound}
Let $U$ be a finite energy $\R \times \SurgL$ curve with interior punctures asymptotic to some collections of orbits of $R_{\epsilon}$ and boundary punctures asymptotic to chords of $\LambdaZero \subset \SurgL$. Then
\begin{equation*}
\energy(U) > \sum_{k} \energy_{k} \C_{k}\cdot U.
\end{equation*}
\end{prop}

\begin{proof}
For each $k \neq \infty$ for which 
\begin{equation*}
\Sigma'_{k} = U^{-1}(\R\times \R \times \region_{k})
\end{equation*}
is not empty, we have that
\begin{equation*}
\pi_{x, y}\circ \pi_{\SurgL} \circ U: \Sigma'_{k} \rightarrow \region_{k}
\end{equation*}
is a non-constant holomorphic map. By our boundary conditions, each $\Sigma'_{k}$ is disjoint from some neighborhood of the boundary and punctures of $\Sigma'$ and so must be a branched covering. The degree of the associated map
\begin{equation*}
(\overline{\Sigma}'_{k}, \partial \overline{\Sigma}'_{k}) \rightarrow (\overline{\region}_{k}, \partial \overline{\region}_{k})
\end{equation*}
is equal to $\C_{k} \cdot U$ so that our requirement that $\alpha_{\epsilon}$ coincides with $\alpha_{std} = dz - y dx$ on the compliment of $N^{\pm}$ implies
\begin{equation*}
\energy(U) > \sum_{k} \int_{\Sigma'_{k}} d\alpha_{\epsilon} = \sum_{k} \int_{\Sigma'_{k}} dx \wedge dy = \sum_{k} \energy_{k} \C_{k}\cdot U.
\end{equation*}
\end{proof}

\subsection{The $\Lambda$ quiver}\label{Sec:LambdaQuivers}

In the event that all intersection numbers $\C_{k} \cdot U$ are zero for a given curve $U$ we can employ another device to keep track of homolorphic curves and their boundary conditions.

\begin{defn}
The \emph{$\Lambda$ quiver}, denoted $Q_{\Lambda}$ is the directed graph with
\be
\item one vertex $\ell_{i}$ for each connected component $\Lambda_{i}$ of $\Lambda$ and
\item one directed edge for each chord $r_{j}$ of $\Lambda \subset \R^{3}$ starting at the vertex $\ell_{l_{j}^{-}}$ and ending at $\ell_{l_{j}^{+}}$.\footnote{We recall the $l_{j}^{\pm}$ are defined in Section \ref{Sec:ChordNotation}.}
\ee
Also define a graph $Q_{\Lambda}/\ell$ which is the quotient of $Q_{\Lambda}$ obtained by identifying all of its vertices. We write 
\begin{equation*}
\pi_{\ell}: Q_{\Lambda} \rightarrow Q_{\Lambda}/\ell
\end{equation*}
for the quotient map.
\end{defn}

\begin{figure}[h]\begin{overpic}[scale=.7]{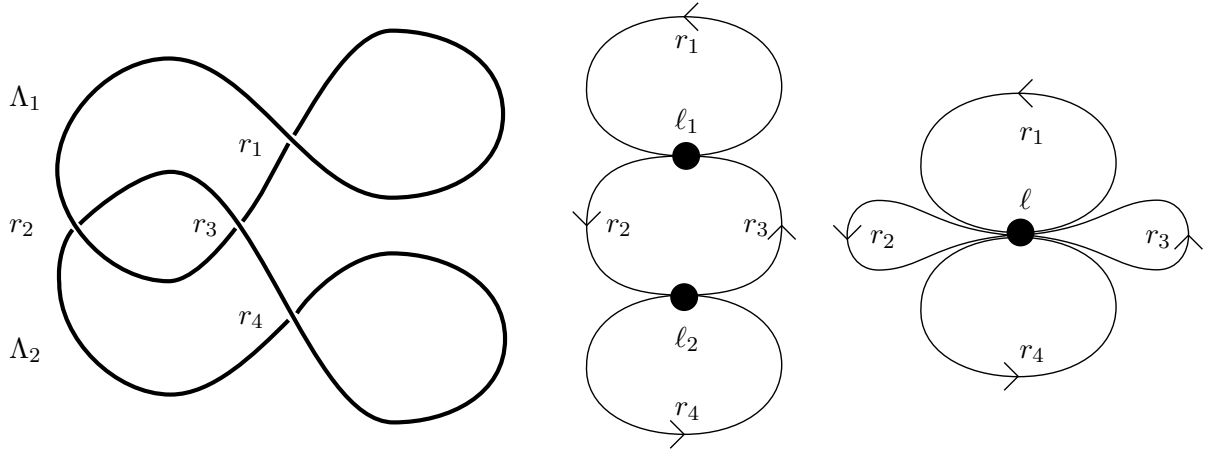}
	\put(16, 26){$r_{1}$}
	\put(-4, 19){$r_{2}$}
	\put(12, 19){$r_{3}$}
	\put(16, 11){$r_{4}$}
	\put(-4, 30){$\Lambda_{1}$}
	\put(-4, 8){$\Lambda_{2}$}

	\put(54, 35){$r_{1}$}
	\put(54, 3){$r_{4}$}
	\put(48, 19){$r_{2}$}
	\put(60, 19){$r_{3}$}	
	\put(54, 28){$\ell_{1}$}
	\put(54, 9){$\ell_{2}$}
	
	\put(84, 27){$r_{1}$}
	\put(84, 8){$r_{4}$}
	\put(71, 18){$r_{2}$}
	\put(95, 18){$r_{3}$}	
	\put(84, 21){$\ell$}
\end{overpic}
\caption{From left to right: a Legendrian Hopf link $\Lambda$ in the Lagrangian projection, the associated quiver $Q_{\Lambda}$, and the quiver $Q_{\Lambda}/\ell$.}
\label{Fig:HopfQuiver}
\end{figure}

An example is provided in Figure \ref{Fig:HopfQuiver}

\subsubsection{Algebraic aspects of $Q_{\Lambda}$ and $Q_{\Lambda}/\ell$}

The primary utility space of the space $Q_{\Lambda}/\ell$ is that its homology has a particularly nice presentation, with $H_{1}$ freely generated by the chords of $\Lambda \subset \Rthree$
\begin{equation*}
H_{1}(Q_{\Lambda}/\ell) = \oplus \Z r_{j},
\end{equation*}
while its fundamental group -- based at its unique vertex, $\ell$ -- is a free group on the chords of $\Lambda$
\begin{equation*}
\pi_{1}(Q_{\Lambda}/\ell, \ell) = \langle r_{j} \rangle.
\end{equation*}

In applications, we'll make use of the following definitions and lemma.

\begin{defn}\label{Def:PositiveLoops}
For an edge $e$ of a directed graph $G$ we define the \emph{collapse map at $e$}, denoted $\pi_{e}: G \rightarrow \Circle$, as the map which takes the quotient by $G \setminus \Int(e)$. The target is naturally pointed and oriented by the direction of $e$. A continuous map $\Phi: \Circle \rightarrow G$ from an oriented circle is \emph{non-negative} if for every edge $e$ of $g$ the composition
\begin{equation*}
\Circle \xrightarrow{\Phi} G \xrightarrow{\pi_{e}} \Circle
\end{equation*}
with the collapse map has non-negative degree. We say that the map is \emph{positive} if it is non-negative and there exists at least one $e \subset G$ for which $\pi_{e}\circ \Phi$ has positive degree.
\end{defn}

\begin{defn}\label{Def:PositiveFreeElements}
Let $\mathcal{S}$ be a set with associated free group $\langle \mathcal{S} \rangle$. We say that an element $x \in \langle \mathcal{S} \rangle$ is \emph{positive} if it can be described as a word
\begin{equation*}
x = x_{1}\cdots x_{n},\ x_{k} \in \mathcal{S}.
\end{equation*}
Alternatively, the set of positive elements in $\langle \mathcal{S} \rangle$ is equivalent to the image of the natural monoid homomorphism from the free monoid on $\mathcal{S}$ into $\langle \mathcal{S} \rangle$.

If $x$ is positive then the above factorization is necessarily unique. We say that two positive elements $x, y$ of $\langle \mathcal{S} \rangle$ are \emph{cyclically equivalent} if their positive factorizations differ by a cyclic rotation. That is, provided a factorization of $x$ as above, there exists $k$ for which
\begin{equation*}
y = x_{k}\cdots x_{n}x_{1}\cdots x_{k-1}.
\end{equation*}

We say that $x \in \langle \mathcal{S} \rangle$ is \emph{negative} if $x^{-1}$ is positive. Two negative elements $x, y$ are \emph{cyclically equivalent} if $x^{-1}$ and $y^{-1}$ are cyclically equivalent.
\end{defn}

Cyclic equivalence is no stronger then conjugacy equivalence.

\begin{lemma}\label{Lemma:CyclicFreeElements}
Suppose that $x, y \in \langle \mathcal{S} \rangle$ are positive and conjugate in $\langle \mathcal{S} \rangle$. Then they are cyclically equivalent.
\end{lemma}

\begin{proof}
Suppose there exists some $z$ for which $z x = y z$ and write $z = z_{1}\cdots z_{n}$ with the $z_{k}$ being elements of $\mathcal{S}$ or inverses of such letters. We can assume that at least one of $z_{1}$ or $z_{n}$ is positive. Otherwise we can write $z^{-1} y = x z^{-1}$ to obtain the desired hypothesis by a change of notation.

Suppose that $z_{1}$ is positive. Then the positive factorization of $y$ must start with $z_{1}$. Then $y' = z_{1}^{-1}y z_{1}$ is positive so we can write $z' x = y' z'$ with $z' = z_{2}\cdots z_{n}$. We have reduced the problem to finding a cyclic equivalence between two positive elements $x, y'$ which are conjugate by a word $z'$ of length $n - 1$. A similar argument may be applying in the case that $z_{n}$ is positive.

To complete the proof, loop through this argument $n$ times.
\end{proof}

\subsubsection{Geometric aspects of $Q_{\Lambda}$ and $Q_{\Lambda}/\ell$}

The primary utility of the space $Q_{\Lambda}$ in relation to the present discussion is given by the following result:

\begin{prop}\label{Prop:QProjections}
There exist surjective maps
\begin{equation*}
\begin{gathered}
\SurgL \setminus \widetilde{N}_{\epsilon}^{\complement} \rightarrow Q_{\Lambda}\\
\overline{W}_{c} \setminus \R \times \widetilde{N}_{\epsilon}^{\complement} \rightarrow Q_{\Lambda},
\end{gathered}
\end{equation*}
both of which we will denote by $\pi_{Q}$, such that for each chord $r_{j}$ of $\Lambda$ and each line segment $I$ directed by $\partial_{z}$ connecting $\rect^{ex}_{j}$ to $\rect^{en}_{j}$, the submanifold $\R \times I$ is mapped onto the edge $r_{j}$ of $Q_{\Lambda}$ in a way such that for each $t$ in $\R$, $\{ t \} \times I \rightarrow e_{j}$ is a homeomorphism.\footnote{We recall that the $\rect_{j}^{\ast}$ are defined in Section \ref{Sec:OverlappingRectangles}.}
\end{prop}

\begin{proof}
We start with the case in which the domain of $\pi_{Q}$ is $\SurgL \setminus \widetilde{N}_{\epsilon}^{\complement}$. We have that $\SurgL \setminus  \widetilde{N}_{\epsilon}^{\complement}$ is homotopy equivalent the union of $N_{\epsilon}$ with all of the chords $r_{j}$ of $\Lambda$. We can perform this homotopy so that the intervals connecting the $\rect^{en}_{j}, \rect^{ex}_{j}$ (forming a neighborhood of $r_{j}$) collapse onto $r_{j}$ as a fibration. Noting that $N_{\epsilon}$ is a collection of solid tori, so that $N_{\epsilon} \cup \{ r_{j} \}$ is homotopy equivalent to a $1$ dimensional CW complex. If we collapse each connected component $N_{\epsilon, i}$ of $N_{\epsilon}$ to a point $\ell_{i}$, the graph $Q_{\Lambda}$ is obtained.

The proof for $\overline{W}_{c} \setminus \R \times \widetilde{N}_{\epsilon}^{\complement}$ is nearly identical except at the last step, the addition of the surgery handles already provides the effect of attaching $2$-cells along the circles in our $1$-complex corresponding to components of $\LambdaPM$. We then collapse these $2$ cells to points, which has the same effect -- in the homotopy category -- as collapsing the circles corresponding to the components of $\Lambda$ to points.
\end{proof}

\begin{prop}
Suppose that $\gamma(t)$ parameterizes a Reeb orbit in $\SurgL$ or $\partial W_{c}$. Then $\pi_{Q}\circ \gamma$ is positive in the sense of Definition \ref{Def:PositiveLoops}. 

The open string version of this assertion is as follows: Let $U$ be a $\overline{W}_{c}$ or $\R \times \SurgL$ curve with domain $\Sigma'$ having a boundary component $\partial_{i} \Sigma \subset \partial \Sigma$ for which all punctures along $\partial_{i} \Sigma$ have positive asymptotics. Then $\pi_{Q}\circ U|_{\partial_{i}}\Sigma$ is a positive loop. If all punctures along $\partial_{i} \Sigma$ have negative asymptotics, then this loop is negative.
\end{prop}

This is clear from the construction of the map $\pi_{Q}$. For a parameterization $\gamma$ of a Reeb orbit with cyclic word $r_{j_{1}}\cdots r_{j_{n}}$ we have
\begin{equation*}
[\pi_{\ell}\circ \pi_{Q} \circ \gamma] = \sum_{1}^{n} [r_{j_{k}}] \in H_{1}(Q_{\Lambda}/\ell).
\end{equation*}
Intuitively, the map $\pi_{\ell}\circ \pi_{Q}$ induces a map on homology which abelianizes boundary conditions for holomorphic curves. We can also view $[\pi_{\ell}\circ \pi_{Q} \circ \gamma]$ as a element of the $H_{0}$ of the free loop space of $Q_{\Lambda}/\ell$ which records the word map of $\gamma$.

For a single chord $\kappa$ with boundary on some $\LambdaZero \subset \SurgLxi$, we can view $\pi_{\ell}\circ \pi_{Q} \circ \kappa$ as a pointed map 
\begin{equation*}
(\kappa, \partial \kappa) \rightarrow (Q_{\Lambda}/\ell, \ell)
\end{equation*}
as $\Lambda$ is mapped to $\ell$ by $\pi_{\ell}\circ \pi_{Q}$. In this way $\kappa$ determines a positive element of $\pi_{1}(Q_{\Lambda}/\ell)$ as well as a relative homology class
\begin{equation*}
[\pi_{\ell}\circ \pi_{Q} \circ \kappa] \in H_{1}(Q_{\Lambda}/\ell, \ell).
\end{equation*}
Both the $\pi_{1}$ and $H_{1}$ classes record the word map of $\kappa$.

\subsection{The exposed/hidden alternative}\label{Sec:ExposedHidden}

Assume that $\Lambda$ is equipped with a basis of points $(x_{k}, y_{k}) \in \R^{2} \setminus \pi_{x, y}(N_{\epsilon})$ as described in Section \ref{Sec:BasesAndEnergy}.

\begin{defn}[Exposed/hidden alternative]\label{Def:ExposedHidden}
We say that a $\R \times \SurgL$ or $\overline{W}_{c}$ curve $U$ is \emph{exposed} if there exists at least one $k$ for which $\C_{k} \cdot U > 0$. Otherwise we say that $U$ is \emph{hidden}.
\end{defn}

If a curve $U$ is exposed, then we can use the intersection numbers to keep track of the location of its image within the target manifold. If the curve is hidden then by intersection positivity, its image must be entirely contained in the complement of $\R \times \widetilde{N}_{\epsilon}^{\complement}$, whence we can apply the map $\pi_{\ell}\circ\pi_{Q}$. We state some simple applications, the first few of which tell us that the homology of $Q_{\Lambda}/\ell$ dictates whether a curve is exposed or hidden.

\begin{prop}[Homological mismatches are exposed]\label{Prop:ExposedHOneMismatches}
Suppose that $U$ is a $\R\times \SurgL$ or $\overline{W}_{c}$ curve without boundary components positively asymptotic to some collection $\gamma^{+} = \{ \gamma^{+}_{k}\}$ of closed orbits and negatively asymptotic to some collection $\gamma^{-} = \{ \gamma^{-}_{k} \}$ of Reeb orbits. If the $1$ cycle
\begin{equation*}
\sum [\pi_{\ell}\circ \pi_{Q}\circ \gamma^{+}_{k}] - \sum [\pi_{\ell}\circ \pi_{Q}\circ \gamma^{-}_{k}] \neq 0 \in H_{1}(Q_{\Lambda}/\ell)
\end{equation*}
then $U$ is exposed.
\end{prop}

\begin{proof}
If the curve was hidden then we could apply the map $\pi_{\ell}\circ\pi_{Q}$ to the image of $U$. Our hypotheses on asymptotics imply that we would get a $2$-cycle in $\SurgL \setminus \widetilde{N}_{\epsilon}^{\complement}$ or $W_{c} \setminus \R \times \widetilde{N}_{\epsilon}^{\complement}$ bounding a homologically non-trivial $1$ cycle, providing a contradiction.
\end{proof}

A slight modification applies to chords as well.

\begin{prop}[Exposure of filling curves]\label{Prop:ExposedFillings}
Suppose that $U$ is a $\R \times \SurgL$ or $\overline{W}_{c}$ curve for which all asymptotic chords and orbits are positive. Then $U$ must be exposed.
\end{prop}

\begin{prop}[Homological matches are hidden]\label{Prop:HiddenHOneMatches}
Let $h \in H_{1}(Q_{\Lambda}/\ell)$ be a positive homology class.\footnote{That is, $h$ may be represented as a sum of positive cycles.} Then there exists $\epsilon_{h}$ such that for each $\epsilon < \epsilon_{h}$ and holomorphic curve in $\R \times \SurgL$ positively asymptotics to a collection of orbits $\gamma^{+}$ and negatively asymptotic to a collection $\gamma^{-}$ of $R_{\epsilon}$ orbits with
\begin{equation*}
[\pi_{\ell} \circ \pi_{Q} \circ \gamma^{+}] = [\pi_{\ell} \circ \pi_{Q} \circ \gamma^{-}] = h \in H_{1}(Q_{\Lambda}/\ell)
\end{equation*}
then $U$ is hidden.
\end{prop}

\begin{proof}
By the action estimates of Section \ref{Sec:ActionEstimates}, we have 
\begin{equation*}
\energy(U) = \bigO\bigg( 3\epsilon \sum \wl(\gamma^{+}_{k}) \bigg).
\end{equation*}
For $\epsilon$ sufficiently small, we could guarantee that this quantity is less that the energies $\energy_{k}$ of the regions $\region_{k}$ (which grow slightly as $\epsilon$ tends to $0$ with $N_{\epsilon}$ shrinking). Therefore the energy bound of \ref{Prop:EnergyBound} would imply that $U$ must be hidden.
\end{proof}

\begin{prop}[Cyclic order preservation of open-closed interpolations]\label{Prop:CyclicOrderPreservationOC}
Suppose that $U$ is a hidden $\overline{W}_{c}$ curve whose domain is a disk with a single interior puncture and any number of boundary punctures. We require that:
\be
\item if $c = +1$, the boundary punctures are positively asymptotic to chords of $\LambdaZero$ with words $w_{1}, \dots, w_{n}$.
\item if $c = -1$, the boundary punctures are negatively asymptotic to chords of $\LambdaZero$ with words $w_{1}, \dots, w_{n}$.
\ee
Here indices follow the counter clockwise cyclic ordering of the punctures around $\partial \disk$. Then interior puncture of $U$ asymptotic to the orbit $(w_{1}\cdots w_{n})$.
\end{prop}

The $c=-1$ curves described are those used to determine homomorphisms from linearized contact homology to a cyclic version of Legendrian contact homology when performing a contact $-1$ surgery in \cite{BEE:LegendrianSurgery, Ekholm:SurgeryCurves, EkholmNg}.\footnote{We're ignoring anchors which can be avoided in some settings such as \cite{EkholmNg}.} We'll see some of the $c=1$ curves shortly in Theorem \ref{Thm:PlaneBubbling}.

\begin{proof}
Consider the map $\pi_{\ell}\circ \pi_{Q} \circ U$ from the punctured disk to the graph $Q_{\Lambda}/\ell$. Then $\partial \disk$ -- compactified appropriately -- will give us an element of the free loop space of $Q_{\Lambda}/\ell$. It is clear from the construction of the map $\pi_{Q}$ that the connected component of the free loop space of $Q_{\Lambda}/\ell$ containing this loop is indexed by $w_{1}\cdots w_{n}$. Looking at circles of varying radii in $\disk$ provides a homotopy between this loop and the one provided by the interior puncture. Again by the construction of $\pi_{Q}$, observe that if the orbit to which the puncture is asymptotic has cyclic word $r_{j_{1}}\cdots r_{j_{n}}$, then this word must also index the component of of the free loop space of $Q_{\Lambda}/\ell$ to which the puncture is associated. The connected components of the free loop space of $Q_{\Lambda}/\ell$ are in bijective correspondence with conjugacy classes on $\langle r_{j} \rangle$ so that the expressions $r_{j_{1}}\cdots r_{j_{n}}$ and $w_{1}\cdots w_{n}$ are conjugate by the existence of the aforementioned homotopy. They are also both positive in the sense of Definition \ref{Def:PositiveFreeElements} and so differ by a cyclic permutation of their letters by Lemma \ref{Lemma:CyclicFreeElements}.
\end{proof}

\begin{prop}[Triviality of hidden cylinders and strips]\label{Prop:CyclicOrderPreservationCyl}
Suppose that $U$ is a hidden holomorphic cylinder in $\R \times \SurgL$. Then $U$ is a trivial cylinder. 

If $U$ has domain $\R \times I_{C}$ for some $C > 0$, is hidden, with boundary on the Lagrangian cylinder over $\LambdaZero \subset \SurgLxi$, and with punctures asymptotic to chords of $\LambdaZero$, then $U$ is a trivial strip.
\end{prop}

\begin{proof}
If $U$ is positively asymptotic to some orbit $(r_{j_{1}}\cdots r_{j_{n}})$ then we can follow the proof of Proposition \ref{Prop:CyclicOrderPreservationOC} verbatim to conclude that $U$ is negatively asymptotic to $(r_{j_{1}}\cdots r_{j_{n}})$. Hence the energy of $U$ is zero and it must be a trivial cylinder.

The case of a holomorphic strip is even easier. Suppose the strip is parameterized $s \in \R, t \in I_{C}$ and consider the family of paths $\gamma_{s}(t) = \pi_{\SurgL}\circ U(s, t)$ with boundary on $\LambdaZero \subset \SurgL$. Then we may consider the $\pi_{\ell}\circ \pi_{Q} \circ \gamma_{s}$ as an $\R$ family of based loops in $Q_{\Lambda}/\ell$. As $s \rightarrow \infty$, the $\pi_{1}(Q_{\Lambda}/\ell)$ element recorded by this based loop is the word map of the chord to which $U$ is positively asymptotic. As $s \rightarrow -\infty$, the element recorded is the word map of the chord to which $U$ is negatively asymptotic. Hence the asymptotics are equivalent by our chord-to-chord correspondence (Theorem \ref{Thm:ChordsToChords}), the energy of $U$ is zero, and $U$ is a trivial strip.
\end{proof}

\section{Applications}\label{Sec:Applications}

In this section we apply our computational tools to study the contact homology of various contact manifolds. A summary of the results are as follows:
\be
\item In Section \ref{Sec:CHUnknot} we compute the contact homology of contact $\pm 1$ surgeries on the $\tb =-1, \rot=0$ unknot in $\R^{3}$.
\item In Section \ref{Sec:PlaneBubbling} we use the results of Section \ref{Sec:FoliationsAndQuivers} to prove a general existence result for holomorphic planes in $\R \times \SurgL$ when $\LambdaPlus \neq \emptyset$.
\item In Section \ref{Sec:OTSurgery}, we use the existence of these holomorphic planes to provide a new proof of the vanishing of $CH$ for overtwisted contact structures.
\item In Section \ref{Sec:IntersectionGrading} we state how the intersection numbers of Section \ref{Sec:FoliationsAndQuivers} can be used to define a grading $\Igrading$ on the $CH$ chain complex for $\alpha_{\epsilon}$.
\item In Sections \ref{Sec:Trefoil} we compute the homology classes and Conley-Zehnder indices of $R_{\epsilon}$ orbits appearing after application of contact surgeries to the $\tb = 1$, right-handed trefoil.
\item In Section \ref{Sec:TrefoilProof}, we combine computations of Section \ref{Sec:Trefoil} with the results of Sections \ref{Sec:PlaneBubbling} and \ref{Sec:IntersectionGrading} to prove Theorem \ref{Thm:Trefoil}.
\ee

For notational simplicity, we will ignore mention of specific contact forms $\alpha_{\epsilon}$ assuming that each contact manifold $\SurgLxi$ is equipped with such a contact form with $\epsilon$ small enough to guarantee that all orbits under consideration are hyperbolic and that Theorem \ref{Thm:IntegralCZ} may be applied. Assumptions \ref{Assump:Transversality} are also in effect. When working with symplectizations of $\Rthree$, we assume that we're using the standard almost complex structure $J_{0}$.

\subsection{Surgeries on the standard unknot}\label{Sec:CHUnknot}

Let $\Lambda$ be the Legendrian unknot with $\tb = -1$ and $\rot = 0$, depicted as a figure-8 in the Lagrangian projection in Figure \ref{Fig:UnknotPushouts}. Performing contact $-1$ surgery will produce the standard contact lens space $L(2, 1)$ -- the unit cotangent bundle of $S^{2}$, or alternatively the unit circle bundle associated to the line bundle $\bigO(-2) \rightarrow \Proj^{1}$. We'll denote this contact lens space by $(L(2, 1), \xi_{std})$. Performing $+1$ produced the standard contact $\Circle \times S^{2}$ -- see Theorem \ref{Thm:SurgeryOverview} -- denoted $(S^{1} \times S^{2}, \xi_{std})$.

\begin{figure}[h]\begin{overpic}[scale=.8]{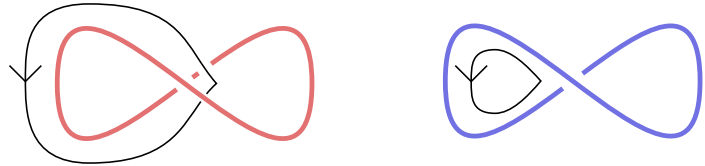}
\end{overpic}
\caption{Contact surgeries on the $\tb = -1$ unknot with push-outs of their unique embedded Reeb orbits. A $-1$ ($+1$) surgery is applied on the left (right) subfigure.}
\label{Fig:UnknotPushouts}
\end{figure}

We can arrange that the Lagrangian projection of $\Lambda$ has a single crossing corresponding to a Reeb chord we denote by $r$, so that after performing a contact $\pm 1$ surgery there is only a single embedded orbit $(r)$ with cyclic word $r$. Push-outs of $(r)$ using a choice of capping path are shown in Figure \ref{Fig:UnknotPushouts}. As $\rot(\Lambda) = 0$, the framing $(X, Y)$ described in Section \ref{Sec:Framing} is nowhere vanishing. For either choice of surgery coefficient, the first homology $H_{1}$ is generated by a meridian $\mu$ of $\Lambda$ with
\begin{equation*}
H_{1}(L(2, 1)) = (\Z / 2\Z)\mu ,\quad H_{1}(\Circle\times S^{2}) = \Z \mu.
\end{equation*}

\begin{thm}
The Conley-Zehnder gradings $|\ast|_{X, Y}$ on $\widehat{CH}(L(2, 1), \xi_{std})$ and $\widehat{CH}(S^{1} \times S^{2}, \xi_{std})$ are canonical in the sense of Proposition \ref{Prop:CanonicalZGrading}. We compute
\begin{equation*}
\begin{gathered}
\widehat{CH}(L(2, 1), \xi_{std}) = \Q[z_{0},z_{2},\dots,z_{2k},\dots] \\
|z_{2k}|_{X, Y} = 2k,\quad [z_{2k}] = \mu \in H_{1}
\end{gathered}
\end{equation*}
for the lens space and
\begin{equation*}
\begin{gathered}
\widehat{CH}(\Circle \times S^{2}, \xi_{std}) = \wedge_{k=1}^{\infty} \Q z_{2k-1}\\
|z_{2k-1}|_{X, Y} = 2k-1,\quad [z_{2k-1}] = 0 \in H_{1}
\end{gathered}
\end{equation*}
for $\Circle \times S^{2}$.
\end{thm}

\begin{proof}
For either choice of surgery coefficient $c = \pm 1$, we may compute Conley-Zehnder indices of $(r)$ using a capping path $\eta$. We see that the rotation angle of $\eta$ is $\frac{3\pi}{2}$ so that  its rotation number $1$. We conclude that
\begin{equation*}
\CZ_{X, Y}((r^{k})) = \begin{cases}
k & c = -1 \\
2k & c = +1
\end{cases}
\end{equation*}
Here and throughout the remainder of the proof, $(r^{k}) = (r\cdots r)$ is the $k$-fold cover of the embedded orbit $(r)$ for $k > 0$. To sanity check our index computations against known results, we may 
\be
\item compare the case $c=-1$ with \cite[Section 7.1]{BEE:LegendrianSurgery} in which contact $-1$ surgery is applied to $\Lambda$.
\item compare the case $c=+1$ with \cite[Lemma 4.2]{EkholmNg} in which a contact $1$-handle is attached to $\Rthree$ to obtain $\widehat{(\Circle \times S^{2}, \xi_{std})}$.
\ee
In each case a single closed, embedded orbit is produced with Conley-Zehnder index as described in the present scenario.

For the homology classes of orbits, we may apply Theorem \ref{Thm:H1}, or simply look at the push-outs depicted in Figure \ref{Fig:UnknotPushouts} to compute
\begin{equation*}
[(r)] = \begin{cases}
\mu & c = -1 \\
0 & c = +1.
\end{cases}
\end{equation*}
As the framing $(X, Y)$ is non-vanishing, we conclude that $\widehat{CH}$ is canonically $\Z$-graded for either choice of surgery coefficient, for when $c=-1$ we have a $\Q$ homology sphere and when $c=+1$, all orbits are homologically trivial.

When $c = -1$, an orbit $(r^{k})$ is bad exactly when $k \bmod_{2} = 0$. Write $z_{2k}$ for the orbit $(r^{2k - 1})$. Then the $\widehat{CH}$ chain algebra is freely generated by the $z_{2k}$ with gradings as described in the statement of the theorem. As the $\CZ_{X, Y}$ grading is even, $\partial_{CH}$ must vanish. The theorem is now complete in the case $c=-1$.

When $c = +1$, all of the $(r^{k})$ are good orbits which we will denote by $z_{2k - 1}$. These are graded as described in the statement of the theorem. As $(r)$ is the unique orbit of index $1$, $\partial_{CH}(r)$ must be a count of holomorphic planes. If this count was non-zero, then the unit in $\widehat{CH}$ would be exact. This is impossible, as $(\Circle \times S^{2}, \xi_{std})$ bounds the Liouville domain
\begin{equation*}
(S^{1} \times \disk^{3}, x dy - y dx + z d\theta)
\end{equation*}
implying that $CH(\Circle \times S^{2}, \xi_{std}) \neq 0$ and so $\widehat{CH}(\Circle \times S^{2}, \xi_{std}) \neq 0$ by Theorem \ref{Thm:CHHatOverview}. We conclude $\partial_{CH}(r) = 0$.

For $c=+1, k > 1$, the contact homology differential of $(r^{k})$ is determined by counts of pairs of pants $\Proj^{1} \setminus \{ 0, 1, \infty \}$ with 
\be
\item $\infty$ positively asymptotic to $(r^{k})$
\item $0$ negatively asymptotic to some $(r^{k_{0}})$,
\item $1$ negatively asymptotic to some $(r^{k_{1}})$, and
\item $k = k_{0} + k_{1}$ as required by the index formula, Equation \eqref{Eq:DelbarIndex}.
\ee
The energies of any such curves must be $0$ indicating that these curves must be branched covers of the trivial cylinder over $(r)$. According to calculations of Fabert \cite{Fabert:Pants}, the contact homology differential must be strictly action decreasing, implying that the counts of such curves are $0$. We conclude $\partial_{CH}(r^{k}) = 0$ completing the proof.
\end{proof}

\subsection{Bubbling planes in surgery diagrams}\label{Sec:PlaneBubbling}

In this section we use the results of Section \ref{Sec:FoliationsAndQuivers} to count holomorphic curves in completed surgery cobordisms $(\overline{W}_{+1}, \overline{\lambda}_{+1})$ determined by certain $LRSFT$ disks on Legendrian links in $\Rthree$ with only positive punctures. The arguments can be generalized to Legendrians $\LambdaZero$ in arbitrary punctured contact manifolds $\SurgLxi$, with additional notation and hypothesis. We consider $LRSFT$ disks with arbitrary numbers of positive punctures although in the applications of Section \ref{Sec:OTSurgery} and \ref{Sec:Trefoil} we'll only need to look at disks with a single positive puncture. 

As mentioned in the introduction, the inspiration for our construction is Hofer's bubbling argument \cite{Hofer:OTWeinstein}, used to prove the Weinstein conjecture -- that every Reeb vector field on a given contact manifold has a closed orbit -- for certain contact $3$-manifolds. We also have in mind the holomorphic curves in contact $-1$ surgery cobordisms of \cite{BEE:LegendrianSurgery, Ekholm:SurgeryCurves} positively asymptotic to closed orbits and negatively asymptotic to chords of a Legendrian link. In the case of $+1$ surgery, we will see some curves for which these boundary conditions have been flipped upside-down, allowing us to interpolate between chords of Legendrian links and Reeb orbits appearing after contact $+1$ surgery.

\begin{figure}[h]\begin{overpic}[scale=.7]{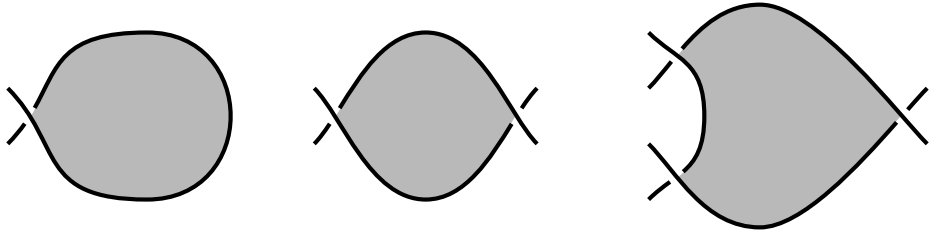}
	\end{overpic}
	\caption{Some $RSFT$ disks with only positive punctures.}
	\label{Fig:ModelPositivePunctureCurves}
\end{figure}

Suppose that $\LambdaZero \subset \Rthree$ has an immersed $LRSFT$ disk $u: \disk\setminus \{ p_{k}\} \rightarrow \R^{2}$ for some boundary punctures $\{ p_{k} \}$ as in Figure \ref{Fig:ModelPositivePunctureCurves}. Specifically, we assume that $u$ is an embedding with only positive punctures, completely covering a connected component of $\R^{2} \setminus \pi_{x,y}(\Lambda_{0})$. Write $r_{j_{1}},\dots, r_{j_{n}}$ for the chords associated to the punctures of the disk indexed in a counterclockwise fashion along its boundary and write 
\begin{equation*}
U: \disk\setminus \{ p_{k}\} \rightarrow \R \times \Rthree
\end{equation*}
for the associated holomorphic curve with boundary mapping  to $\R \times \LambdaZero$ determined by the drawing-to-disk correspondence \ref{Cor:DrawingDisk}.

Let $(x_{k}, y_{k})$ be a basis of points for $\LambdaZero$, indexed so that $(x_{1}, y_{1})$ lies in the interior of the image of $u$. Then by our hypothesis on $u$,
\begin{equation}\label{Eq:BubblingHypothesis}
\C_{k}\cdot U = \begin{cases}
1 & k=1 \\
0 & k\neq 1
\end{cases}.
\end{equation}

Consider the completed cobordism $(\overline{W}_{+1}, \overline{\lambda}_{+1})$ obtained by performing contact $+1$ surgery on $\LambdaZero$ as described by Theorem \ref{Thm:SurgeryCobordisms}. Then we may consider $U$ as having $\overline{W}_{+1}$ as its target with boundary on an embedded union of Lagrangian planes $\overline{\disk}_{+1, i}$ -- as described in Section \ref{Sec:SurgeryCobordisms} -- whose intersection with the positive end of $\overline{W}_{+1}$ is $[0, \infty) \times \LambdaZero$. We simply write $\overline{\disk}_{+1}$ for this union of planes. We may consider the planes $\C_{k}$ as being contained in any of $\R\times \R^{3}$, $\overline{W}_{+1}$, or $\R \times \SurgLPrime$.

We consider the following moduli spaces:
\be 
\item $\mathcal{M}_{\R^{3}}$ is the moduli space of holomorphic disks in $\R \times \R^{3}$ with positive punctures asymptotic to the $r_{1},\dots,r_{n}$ and boundary on $\R \times \LambdaZero$ satisfying \eqref{Eq:BubblingHypothesis}.
\item $\mathcal{M}_{\overline{W}_{+1}}$ is the moduli space of holomorphic disks in $\overline{W}_{+1}$ with positive punctures asymptotic to the $r_{1},\dots,r_{n}$ and boundary on $\overline{\disk}_{+1}$ satisfying \eqref{Eq:BubblingHypothesis}.
\item $\mathcal{M}_{\SurgLPrime}$ is the moduli space of holomorphic planes in $\R \times \SurgLPrime$ positively asymptotic to the closed orbit $(r_{j_{1}}\cdots r_{j_{n}})$ and satisfying \eqref{Eq:BubblingHypothesis}.
\ee

Within the positive end of the completed cobordism, we can translate $U$ positively in the $\R$ direction determining a half-infinite ray $[0, \infty) \subset \mathcal{M}_{\overline{W}_{+1}}$. The index of $U$ is equal to $1$ so that these curves are regular. Following the analogy with \cite{Hofer:OTWeinstein}, these disks will serve as our Bishop family.

\begin{figure}[h]\begin{overpic}[scale=.7]{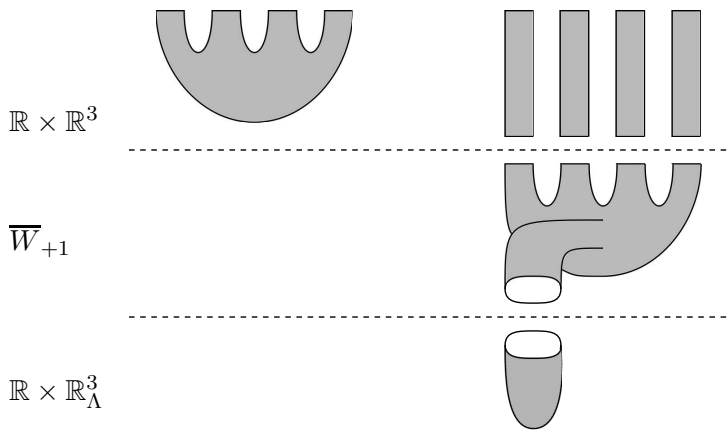}
    \put(-20, 50){$\R \times \R^{3}$}
    \put(-20, 30){$\overline{W}_{+1}$}
    \put(-20, 5){$\R \times \SurgLPrime$}
\end{overpic}
\caption{Elements of $\partial \overline{\mathcal{M}}_{\overline{W}_{+1}}$.}
\label{Fig:PlaneBreaking}
\end{figure}

\begin{thm}\label{Thm:PlaneBubbling}
The boundary of the $\SFT$ compactification $\overline{\mathcal{M}}_{\overline{W}_{+1}}$ of the moduli space $\mathcal{M}_{\overline{W}_{+1}}$ consists of two points (when curves in symplectizations are considered equivalent modulo $\R$-translation). One point is given by $\R$-translations of the curve $U$, considered as living in $\R\times \R^{3}$. The other point is given by a height $3$ $\SFT$ building consisting of
\be
\item A collection of trivial strips over the $r_{j_{k}}$ in $\R \times \R^{3}$.
\item A hidden curve $U^{o}_{c}$ in $\overline{W}_{+1}$ from a disk with $n$ boundary punctures positively asymptotic to the $r_{j_{k}}$ -- preserving the cyclic ordering of the $r_{j_{k}}$ -- and a single interior puncture negatively asymptotic to the closed Reeb orbit $(r_{j_{1}}\cdots r_{j_{n}})$.
\item A curve $U^{c}_{\emptyset}\in \mathcal{M}_{\SurgLPrime}$.
\ee
The algebraic count of such $U^{o}_{c}$ is $\pm 1$ and the algebraic count of points in $\mathcal{M}_{\SurgLPrime}$ is also $\pm 1$.
\end{thm}

The two buildings in $\partial \overline{\mathcal{M}}_{\overline{W}_{+1}}$ are shown in Figure \ref{Fig:PlaneBreaking}. The notation $U^{o}_{c}$ indicates that the curves interpolates between open and closed strings -- that is, between chords and orbits -- and this curve is shown in the center-right of Figure \ref{Fig:PlaneBreaking}. The curve $U^{c}_{\emptyset}$ is shown in the bottom-right of the figure.

\begin{proof}
The space $\partial \overline{\mathcal{M}}_{\overline{W}_{+1}}$ consists of multi-level $\SFT$ buildings such that when their levels are glued together, an index $1$ curve obeying the topological hypotheses on $\mathcal{M}_{\overline{W}_{+1}}$ is obtained. Subject to these conditions, such buildings may be of any of the following configurations:
\be
\item \textbf{Case $(1, \emptyset, \emptyset)$}: A $3$-level building consisting of an index $1$ curve in $\R \times \R^{3}$, an empty curve in $\overline{W}_{+1}$, and an empty curve in the symplectization of the surgered manifold $\R \times \SurgLPrime$.
\item \textbf{Case $(1, 0, \emptyset)$}: A $3$-level building consisting of an index $1$ curve in $\R \times \Rthree$, a collection of index $0$ curves in $\overline{W}_{+1}$, and an empty curve in $\R \times \SurgL$.
\item \textbf{Case $(0, 0, 1)$}: A $3$-level building consisting of a collection of index $0$ curves in $\R \times \R^{3}$, a collection of index $0$ curves in $\overline{W}_{+1}$, and an index $1$ curve in $\R \times \SurgL$.
\ee
The buildings are required to recover the boundary conditions of $U$ when glued in the obvious way. Buildings of height greater than $3$ are ruled out by presumption of transversality for somewhere injective curves in Assumptions \ref{Assump:Transversality}, index additivity, and the fact that all closed orbits of $R_{\epsilon}$ at the negative end of $\overline{W}_{+1}$ are assumed hyperbolic, so that there cannot be levels consisting of branched covers of trivial cylinders with $\ind \leq 0$ as described in \cite[Section 1]{HT:GluingI}.

We will show, using the intersections with the $\C_{k}$, that
\be
\item $U$ is the only possibility for the case $(1, 0, \emptyset)$, 
\item there are no curves in the case $(1, 0, \emptyset)$, and 
\item the second configuration described in the statement of the proposition -- appearing in the right-hand side of Figure \ref{Fig:PlaneBreaking} -- is the only possibility for case $(0, 0, 1)$. 
\ee

\noindent \textbf{Case $(1, \emptyset, \emptyset)$}: For the case $(1, \emptyset, \emptyset)$, our assumptions on the immersion $u$ indicate that $U$ is the only disk in $\R \times \R^{3}$ satisfying Equation \eqref{Eq:BubblingHypothesis}. We conclude that $U$ is then the only possibility in this case.
\newline \newline
\noindent \textbf{Case $(1, 0, \emptyset)$}: Next, suppose we have a holomorphic building satisfying conditions of the case $(1, 0, \emptyset)$ and note that the middle level -- a union of curves in $\overline{W}_{+1}$ we'll denote $U_{\overline{W}_{+1}}$ -- must be positively asymptotic to some number of chords and have no negative asymptotics. Hence each connected component of $U_{\overline{W}_{+1}}$ must be exposed by Proposition \ref{Prop:ExposedFillings}.\footnote{By connected component we intend that nodal configuration, such as those appearing in the appendix of \cite{CL:SFTStringTop} are broken up into their irreducible pieces, with any removable boundary singularities filled in. We maintain this convention throughout the remainder of the proof.} The conditions on intersection numbers of Equation \eqref{Eq:BubblingHypothesis} then indicates that $U_{\overline{W}_{+1}}$ must consist of a single component and that the upper level of this building $U_{\R\times \R^{3}}$ must be hidden. 

For each component of $U_{\R\times \R^{3}}$ the number of positive punctures must match the number of negative punctures, as otherwise \ref{Prop:ExposedHOneMismatches} would indicate that this component is exposed. If any component had more than a single negative puncture, then $U_{\overline{W}_{+1}}$ would have more than a single connected component in violation of the above arguments. We conclude that $U_{\R\times \R^{3}}$ must be a union of hidden strips, which are then trivial by Proposition \ref{Prop:CyclicOrderPreservationCyl}.

Since $U_{\R\times \R^{3}}$ is a collection of trivial strips, it must then have $\ind = 0$ in violation of our hypothesis. We conclude that no buildings of type $(1, 0, \emptyset)$ can exist.
\newline \newline
\noindent \textbf{Case $(0, 0, 1)$}: Finally, we address configurations of type $(0, 0, 1)$. Suppose that we have such a height $3$ building whose levels -- going from top to bottom -- will be denoted $U_{\R\times \R^{3}}$, $U_{\overline{W}_{+1}}$, and $U_{\R \times \SurgLPrime}$. By Proposition \ref{Prop:ExposedFillings}, $U_{\R \times \SurgLPrime}$ must be exposed and so by Equation \eqref{Eq:BubblingHypothesis}, both $U_{\R \times \R^{3}}$ and $U_{\overline{W}_{+1}}$ must be hidden. Then $U_{R \times \SurgLPrime}$ must consist of a holomorphic plane positively asymptotic to some orbit $\gamma$. The curve $U_{\overline{W}_{+1}}$ must then consist of a single connected component negatively asymptotic to $\gamma$, as any additional components would necessarily have trivial negative asymptotics and therefore be exposed by Proposition \ref{Prop:ExposedFillings}. As its index is zero, $U_{\R \times \R^{3}}$ must be a collection of trivial strips. We conclude that $U_{\overline{W}_{+1}}$ must consist of a punctured disk exactly as described in the statement of the proposition. We know that the negative puncture of $U_{\overline{W}_{+1}}$ must be asymptotic to $(r_{j_{1}}\cdots r_{j_{n}})$ by Proposition \ref{Prop:CyclicOrderPreservationOC}.
\newline \newline
Apart from the statement regarding algebraic counts, our proof is complete. To prove this last statement, observe that $\partial\overline{\mathcal{M}}_{\overline{W}_{+1}}$ has a count of $0$ points when taking into account some choice of orientation as it is the boundary of a $1$ manifold. We can also write
\begin{equation*}
\# \partial\overline{\mathcal{M}}_{x_{0}, y_{0}} = \#(\text{$(1, \emptyset, \emptyset)$ buildings}) + \#(\text{$(1, 0, \emptyset)$ buildings}) + \#(\text{$(0, 0, 1)$ buildings})
\end{equation*}
where the $\#(\cdots)$ are counted with signs. We know that the set of $(1, \emptyset, \emptyset)$ buildings consists of a single element yielding a count of $\pm1$ and that the set of $(1, 0, \emptyset)$ buildings must be empty by our previous arguments providing a count of $0$. Hence the number of $(0, 0, 1)$ buildings must be $\mp 1$. But this number is equal to $\# (U_{\overline{W}_{+1}})\cdot \# (U_{\R \times \SurgLPrime})$, so that both numbers must have absolute value $1$. Observing that $\# (U_{\R \times \SurgLPrime})$ coincides with a count of points in the moduli space $\mathcal{M}_{\R \times \SurgLPrime}$, the proof is complete.
\end{proof}

\subsection{Vanishing invariants of overtwisted contact manifolds}\label{Sec:OTSurgery}

Here we use the holomorphic planes of Section \ref{Sec:PlaneBubbling} to prove that the contact homologies of overtwisted contact $3$-manifolds are $0$. Throughout, we write $\MxiOT$ for a closed, overtwisted contact $3$-manifold.

\begin{thm}[\cite{Yau:VanishingCH}]\label{Thm:OTCH} $\widehat{CH}\MxiOT = CH\MxiOT = 0$.
\end{thm}

\begin{proof}
Applying Eliashberg's theorem \cite{Eliash:OTClassification, Huang:OTClassification} which asserts that isotopy classes of overtwisted contact structures on a given contact $3$-manifold are classified by the homotopy classes of their underlying oriented $2$-plane fields, we know that for each $n \in \Z$, there exists a unique overtwisted contact structure $\xi_{n}$ on $S^{3}$ whose $d_{3}$ invariant is $n - \half$. For the tight contact structure $\Sthree$ on $S^{3}$, have have $d_{3}(\xi_{std}) = -\half$.\footnote{See \cite[Section 11.3]{OS:SurgeryBook} for an overview of $d_{3}$ invariants (which we will be following in this proof) as defined by Gompf in \cite[Section 4]{Gompf:Handlebodies}.} Denoting contact-connected-sum by $\#$ and isotopic contact structures as $\simeq$,
\begin{equation*}
\MxiOT \simeq \MxiOT \# \Sthree \simeq \MxiOT \# (S^{3},\xi_{-1}) \# (S^{3}, \xi_{1}).
\end{equation*}
By the connected-sum formula of Theorem \ref{Thm:CHHatOverview}, then we only need to show that $\widehat{CH}(S^{3}, \xi_{1})=0$.

\begin{figure}[h]\begin{overpic}[scale=.9]{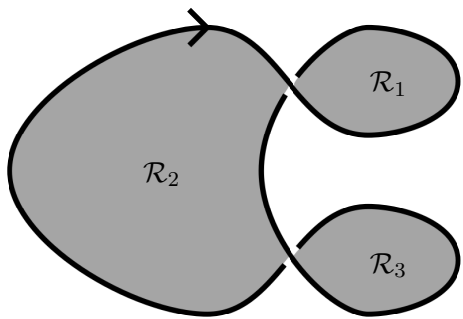}
		\put(80, 50){$\region_{1}$}
		\put(30, 30){$\region_{2}$}
		\put(80, 10){$\region_{3}$}
	\end{overpic}
	\caption{A basis for the $\tb=-2, \rot = 1$ unknot.}
	\label{Fig:Rot1UnknotDisks}
\end{figure}

A contact surgery diagram for $(S^{3}, \xi_{1})$ is provided by a contact $+1$ surgery on a $\tb = -1, \rot = 1$ unknot. See \cite[Lemma 11.3.10]{OS:SurgeryBook}. A Lagrangian resolution of this knot $\Lambda$ -- shown in Figure \ref{Fig:Rot1UnknotDisks} -- has two chords, say $r_{1}, r_{2}$. Perturbing $\Lambda$ as necessary, we may assume that the actions of the chords are distinct and that $r_{1}$ has the least action of the two chords with $r_{1}$ corresponding to the positive puncture of the disk determined by the region $\region_{1}$ of Figure \ref{Fig:Rot1UnknotDisks}. Applying the Conley-Zehnder index calculations of Theorem \ref{Thm:IntegralCZ} to the figure, we see that the Reeb orbit $(r_{1})$ has $\CZ_{X, Y} = 2$. Moreover, the orbit is contractible as can be seen by considering a push-out by the orbit string $\eta_{1, 1}$. 

As the action of $(r_{1})$ is the least among all orbits $R_{\epsilon}$ according to our chords-to-orbits correspondence (Theorem \ref{Thm:ChordOrbitCorrespondence}) and the action estimates of Proposition \ref{Prop:ActionEstimate}, $\partial_{CH}(r_{1})$ and $\partial_{SFT}(r_{1})$ are counts of planes bounding $(r_{1})$. Using the notation of Section \ref{Sec:BasesAndEnergy}, write $\energy_{i}$ for the areas of the regions $\region_{i} \subset \R^{2} \setminus \pi_{xy}(N_{\epsilon})$ shown in Figure \ref{Fig:Rot1UnknotDisks}. By taking the $\epsilon$ parameter in $\alpha_{\epsilon}$ to be sufficiently small, we may assume that $\energy_{2}, \energy_{3} > \action((r_{1}))$. Likewise by Stokes' theorem $\action((r_{1})) - \energy_{1}$ is positive. and may be assumed arbitrarily small by taking $\epsilon$ to be arbitrarily small. Then by the action-energy bound of Proposition \ref{Prop:EnergyBound} and the exposure of filling curves (Proposition \ref{Prop:ExposedFillings}), we must have that any plane $U: \C \rightarrow \R \times \SurgLPrime$ bounding $(r_{1})$ must satisfy
\begin{equation*}
\C_{k} \cdot U = \begin{cases}
	1 & k = 1 \\
	0 & k\neq 1
\end{cases}
\end{equation*}
We can view $\region_{1}$ as determining a disk with a positive puncture at the chord $r_{1}$, apply Theorem \ref{Thm:PlaneBubbling} to obtain a holomorphic plane bounding $(r_{1})$, and conclude that the count of such planes is $\pm 1$. Hence
\begin{equation*}
\partial_{CH}(r_{1}) = \pm 1 \in \Q
\end{equation*}
so that the unit in $\widehat{CH}$ is zero. This implies that $CH\MxiOT$ must also be zero by Theorem \ref{Thm:CHHatOverview}.
\end{proof}

\subsection{Intersection gradings on $\widehat{CH}$ chain complexes}\label{Sec:IntersectionGrading}

Here we describe how the intersections of finite energy curves with the planes $\C_{k}$ of section \ref{Sec:IntersectionNumbers} can define gradings on the $CC_{\ast, 0}(\alpha_{\epsilon})$ chain complexes of punctured $\Q$-homology spheres which take values in a free $\Z$-module. As described in the introduction, this is simply a variation of the transverse knot filtrations of \cite[Section 7.2]{CGHH:Sutures}.

It will be clear from their construction that analogous gradings -- which depend on a surgery presentation of our punctures contact manifold -- can be constructed for holomorphic curve invariants of $\Q$-homology spheres $\SurgLxi$ such as $\widehat{ECH}$ and the $\widehat{SFT}$. It will also be clear that the assumption that $H_{2}(M) = 0$ may be dropped by considering $\Q[H_{2}(M)]$ coefficient systems as described in \cite{Bourgeois:ContactIntro}. Likewise, such gradings can be extended to all of $CC_{\ast, \ast}$ using $\Q[H_{2}(M)]$ coefficients and spanning surfaces bounding unions of closed orbits and fixed representatives of homology classes as in \cite{Bourgeois:ContactIntro}. In Section \ref{Sec:TrefoilProof} we will use this grading to prove Theorem \ref{Thm:Trefoil}, in which case we will only need the $CC_{\ast, 0}$ version of this construction for $\Q$ homology spheres.

Let $\SurgLxi$ be a contact manifold determined by a contact surgery diagram $\LambdaPM$ with $\SurgL$ a $\Q$ homology sphere. Let $(x_{k}, y_{k})$, $k=1,\dots,K$ be a point basis for the surgery diagram determining a finite collection of infinite energy holomorphic planes $\C_{k}$ as described in Section \ref{Sec:BasesAndEnergy}.

Suppose $\gamma = \{ \gamma_{k} \}$ is a collection of Reeb orbits for which $[\gamma] = 0 \in H_{1}(\SurgL)$ and let $S_{\gamma}$ be a surface in $\SurgL$ with $\partial S_{\gamma} = \gamma$. To the surface $S_{\gamma}$ and each point $(x_{k}, y_{k})$ we define
\begin{equation*}
I_{k}(\gamma) = \big( \{ (x_{k}, y_{k}) \}  \times \R \big) \cdot S_{\gamma} \in \Z.
\end{equation*}

By Theorem \ref{Thm:IntersectionContinuity} and the fact that $H_{2}(\SurgL) = 0$, the numbers $I_{k}(\gamma)$ are independent of choice of spanning surface $S_{\gamma}$ for $\gamma$. We collect all of these numbers as monomials
\begin{equation*}
\Igrading(\gamma) = \sum_{1}^{K} I_{k}(\gamma)\iota_{k} \in \Z^{K}
\end{equation*}
for formal variables $\iota_{k}$, $k=1,\dots, K$. It follows from this definition that provided two homologically trivial collections $\gamma_{1}$, $\gamma_{2}$ of closed Reeb orbits we have
\begin{equation*}
\Igrading(\gamma_{1} \cup \gamma_{2}) = \Igrading(\gamma_{1}) + \Igrading(\gamma_{2}).
\end{equation*}
We set $\Igrading(\emptyset) = 0 \in \Z^{K}$. Then $\Igrading$ determines a $\Z^{K}$-valued grading on the $H_{1}=0$ subalgebra $CC_{\ast, 0}$ of the chain algebra $CC$ for the contact homology associated to the contact form $\alpha_{\epsilon}$ of $\SurgL$.

Now suppose that $\gamma^{+}$ and $\gamma^{-}$ are two homologically trivial collections of closed orbits and that $U$ is a map from a surface with boundary into $\SurgL$ for which $\partial U = \gamma^{+} - \gamma^{-}$. Then relative to its boundary, we have
\begin{equation*}
\bigr( \{ (x_{k}, y_{k}) \}\times \R \bigr) \cdot U = I_{k}(\gamma^{+}) - I_{k}(\gamma^{-}) \in \Z.
\end{equation*}
In particular, if $(t, U): \Sigma' \rightarrow \R \times \SurgL$ is a holomorphic curve positively asymptotic to the $\gamma^{+}$ and negatively asymptotic to the $\gamma^{-}$ then
\begin{equation}\label{Eq:IGradingDelta}
\Igrading(\gamma^{+}) - \Igrading(\gamma^{-}) = \sum \bigg( \bigr( \{ (x_{k}, y_{k}) \}\times \R \bigr) \cdot U \bigg) \iota_{k} = \sum \bigg( \C_{k} \cdot (t, U) \bigg)\iota_{k} \in \Z_{\geq 0}^{K}.
\end{equation}

In summary, the $\Igrading$ allows us to make \emph{a priori} computions of intersection numbers between holomorphic curves asymptotic to orbits with leaves of the foliation described in Section \ref{Sec:FoliationsAndQuivers}. In particular, if 
\begin{equation}\label{Eq:IGradingObstruction}
\Igrading(\gamma^{+}) - \Igrading(\gamma^{-}) \notin \Z_{\geq 0}^{K}
\end{equation}
then the coefficient of $\gamma^{-}$ in $\partial_{CH}(\gamma^{+})$ must be zero.\footnote{Here it is implicit that if the collection $\gamma^{+}$ contains more than a single orbit that a holomorphic map $(t, U)$ as above contributing to $\partial_{CH}$ will consist of a connected index $1$ holomorphic curve positively asymptotic to some orbit in $\gamma^{+}$ together with a union of trivial cylinders over the remaining orbits in the collection. This deviation from convention allows us to associate cobordisms to differentials of monomials consisting of $\gamma^{+}$ containing more than one orbit.}

\subsection{Surgery on a trefoil}\label{Sec:Trefoil}

Take $\Lambda$ to be the trefoil depicted in Figure \ref{Fig:TrefoilImmersions} with chords $r_{1},\dots, r_{5}$. This is a reproduction of Figure \ref{Fig:LagrangianResolutionEx} with a point basis shown in the right-hand side of the figure. This trefoil is the unique nondestabilizeable $m(3_{1})$ by \cite{EtnyreHonda:Knots}.

\begin{figure}[h]\begin{overpic}[scale=.9]{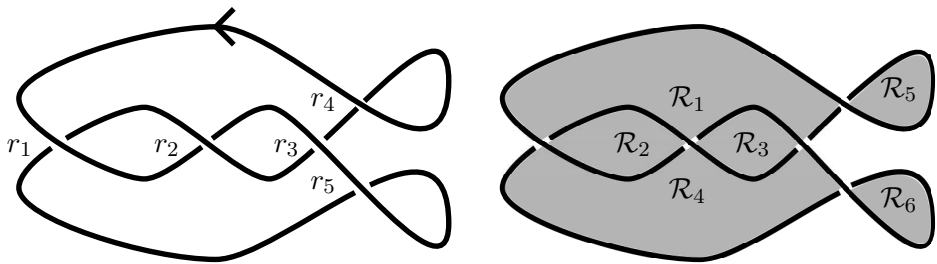}
    \put(-1, 12){$r_{1}$}
    \put(15, 12){$r_{2}$}
    \put(28, 12){$r_{3}$}
    \put(32, 17){$r_{4}$}
    \put(32, 8){$r_{5}$}
    \put(71, 17){$\region_{1}$}
    \put(65, 12){$\region_{2}$}
    \put(78, 12){$\region_{3}$}
    \put(71, 7){$\region_{4}$}
    \put(94, 18){$\region_{5}$}
    \put(94, 6){$\region_{6}$}
\end{overpic}
\caption{A Legendrian trefoil with $\tb=-1$ and $\rot=0$ in the Lagrangian projection together with a basis for $\R^{2} \setminus N$.}
\label{Fig:TrefoilImmersions}
\end{figure}

\subsubsection{Ambient geometry}

According to Theorem \ref{Thm:H1}, the first homology of $\SurgL$ is generated by the meridian $\mu$ with
\begin{equation*}
H_{1}(\SurgL) = \begin{cases}
\Z/2\Z\mu & c=1\\
\Z\mu & c=-1.
\end{cases}
\end{equation*}

Since $\Lambda$ is smoothly fibered, with fiber a punctured torus, the closed manifold obtained by contact $-1$ surgery -- a topological $0$ surgery with respect to the Seifert framing -- is a torus bundle over $\Circle$. This manifold is Liouville fillable, hence tight, and so is a torus bundle covered by the classification in \cite[Section 2]{Honda:Tight2}.

Performing $+1$-contact surgery produces a tight but non-fillable contact manifold studied in \cite{LS:TightI} -- see also the exposition \cite[Theorem 1.3.4]{OS:SurgeryBook} -- which is a Breiskorn sphere with reversed orientation, $-\Sigma(2, 3, 4)$. Non-fillability may also be viewed as a consequence of the fact that the trefoil is not slice by \cite{PlusOneFilling} as mentioned in Theorem \ref{Thm:SurgeryOverview}.

\subsubsection{Rotation numbers and crossing monomials}

Here we compute rotation numbers and crossing monomials for the trefoil which will allow us to compute Conley-Zehnder indices and homology classes of the orbits in the surgered manifolds be applying Theorems  \ref{Thm:IntegralCZ} and \ref{Thm:H1}, respectively. 

To compute the rotation numbers, we first find the rotation angles $\theta_{j_{1}, j_{2}}$ which we see are all either $\frac{\pi}{2}$, $\frac{3\pi}{2}$, or $\frac{5\pi}{2}$, producing the following table.

\begin{center}
\begin{tabular}{|c|c|c|c|c|c|}
\hline
Chord & $\rot_{j, 1}$ & $\rot_{j, 2}$ & $\rot_{j, 3}$ & $\rot_{j, 4}$ & $\rot_{j, 5}$ \\
\hline
$r_{1}$ & $0$ & $0$ & $0$ & $0$ & $1$ \\ [0.5ex] 
$r_{2}$ & $0$ & $0$ & $0$ & $0$ & $1$ \\ [0.5ex] 
$r_{3}$ & $0$ & $0$ & $0$ & $0$ & $1$ \\ [0.5ex] 
$r_{4}$ & $1$ & $1$ & $1$ & $1$ & $2$ \\ [0.5ex] 
$r_{5}$ & $0$ & $0$ & $0$ & $0$ & $1$ \\ [0.5ex] 
\hline
\end{tabular}
\end{center}

For the computation of the crossing monomials, there is only a single $\mu_{i}$ so that Remark \ref{Rmk:CrossingConnectedLambda} is applicable. The following table lists the $\mu$ coefficients of the relevant crossing monomials.

\begin{center}
\begin{tabular}{ |c|c|c|c|c|c|c|c|c| }
\hline
 Chord & $\sgn$ & $\cross_{j}: c=1$ & $\cross_{j}: c=-1$ & $\cross_{j, 1}$ & $\cross_{j, 2}$ & $\cross_{j, 3}$ & $\cross_{j, 4}$ & $\cross_{j, 5}$\\ [0.5ex] 
\hline
$r_{1}$ & $1$ & $2$ & $0$ & $0$ & $0$ & $2$ & $3$ & $1$ \\ [0.5ex]
$r_{2}$ & $1$ & $2$ & $0$ & $0$ & $0$ & $0$ & $1$ & $1$ \\ [0.5ex]
$r_{3}$ & $1$ & $2$ & $0$ & $-2$ & $0$ & $0$ & $1$ & $-1$ \\ [0.5ex]
$r_{4}$ & $-1$ & $0$ & $-2$ & $1$ & $1$ & $3$ & $4$ & $2$ \\ [0.5ex]
$r_{5}$ & $-1$ & $0$ & $-2$ & $-1$ & $1$ & $1$ & $2$ & $0$ \\ [0.5ex]
\hline
\end{tabular}
\end{center}

\subsubsection{Homology classes and indices of orbits after surgery}

Using the above computations, we can produce the following table of homology classes and Conley-Zehnder indices of Reeb orbits with word length $\leq 2$ using Theorems \ref{Thm:H1GammaGeneral} and \ref{Thm:IntegralCZ}. Multiply covered orbits have been omitted. Coefficients for $\mu$ in the case $c = 1$ are taken modulo $2$.

\begin{center}
\begin{tabular}{ |c|c|c|c|c| }
\hline
 $\cycword(\gamma)$ & $\mu: c=1$ & $\CZ_{X, Y}: c=1$ & $\mu: c=-1$ & $\CZ_{X, Y}: c=-1$\\ [0.5ex] 
\hline
$r_{1}$ & $1$ & $1$ & $0$ & $0$\\ [0.5ex]
$r_{2}$ & $1$ & $1$ & $0$ & $0$\\ [0.5ex]
$r_{3}$ & $1$ & $1$ & $0$ & $0$\\ [0.5ex]
$r_{4}$ & $0$ & $2$ & $1$ & $1$\\ [0.5ex]
$r_{5}$ & $0$ & $2$ & $-1$ & $1$\\ [0.5ex]
$r_{1}r_{2}$ & $0$ & $2$ & $0$ & $0$\\ [0.5ex]
$r_{1}r_{3}$ & $0$ & $2$ & $0$ & $0$\\ [0.5ex]
$r_{1}r_{4}$ & $1$ & $3$ & $1$ & $1$\\ [0.5ex]
$r_{1}r_{5}$ & $1$ & $3$ & $-1$ & $1$\\ [0.5ex]
$r_{2}r_{3}$ & $0$ & $2$ & $0$ & $0$\\ [0.5ex]
$r_{2}r_{4}$ & $0$ & $3$ & $0$ & $1$\\ [0.5ex]
$r_{2}r_{5}$ & $0$ & $3$ & $0$ & $1$\\ [0.5ex]
$r_{3}r_{4}$ & $1$ & $3$ & $1$ & $1$\\ [0.5ex]
$r_{3}r_{5}$ & $1$ & $3$ & $-1$ & $1$\\ [0.5ex]
$r_{4}r_{5}$ & $0$ & $4$ & $0$ & $2$\\ [0.5ex]
\hline
\end{tabular}
\end{center}

\subsection{Proof of Theorem \ref{Thm:Trefoil}}\label{Sec:TrefoilProof}

In this section, we prove Theorem \ref{Thm:Trefoil} by computing $\partial_{CH}(r_{4})$.

\subsubsection{The subalgebra $C_{0, 0}$ and intersection gradings}

As the rotation numbers of capping paths on $\Lambda$ are bounded below by $0$, Theorem \ref{Thm:IntegralCZ} tells us that the Conley-Zehnder indices of all orbits of are bounded below by their word-lengths. We conclude that $\partial_{CH}(r_{4})$ must be an element of $CC_{0, 0}$ which is a commutative algebra on generators
\begin{equation*}
    1,\ (r_{1})^{2},\ (r_{2})^{2},\ (r_{3})^{2}, (r_{1})(r_{2}),\ (r_{1})(r_{3}),\ (r_{2})(r_{3}).
\end{equation*}

We'll compute the $\Igrading$ gradings on $CC_{0, 0}$ using points $(x_{k}, y_{k})$ appearing in the centers of the regions $\region_{k}$ of Figure \ref{Fig:TrefoilImmersions}.

\begin{center}
\begin{tabular}{|c|c|c|c|c|c|c|}
\hline
$CC_{\ast, 0}$ monomial & $I_{1}$ & $I_{2}$ & $I_{3}$ & $I_{4}$ & $I_{5}$ & $I_{6}$\\ [0.5ex] 
\hline
$(r_{4})$ & $0$ & $0$ & $0$ & $0$ & $1$ & $0$ \\ [0.5ex]
$(r_{1})^{2}$ & $-1$ & $-1$ & $-2$ & $-1$ & $1$ & $1$ \\ [0.5ex] 
$(r_{2})^{2}$ & $1$ & $2$ & $2$ & $1$ & $-1$ & $-1$ \\ [0.5ex] 
$(r_{3})^{2}$ & $-1$ & $-2$ & $-1$ & $-1$ & $1$ & $1$ \\ [0.5ex] 
$(r_{1})(r_{2})$ & $0$ & $1$ & $0$ & $0$ & $0$ & $0$ \\ [0.5ex] 
$(r_{1})(r_{3})$ & $-1$ & $-1$ & $-1$ & $-1$ & $1$ & $1$ \\ [0.5ex] 
$(r_{2})(r_{3})$ & $0$ & $0$ & $1$ & $0$ & $0$ & $0$ \\ [0.5ex] 
\hline
\end{tabular}
\end{center}

To establish the calculations appearing in the above table we construct surfaces bounding $(r_{1})(r_{2})$, $(r_{2})(r_{3})$, and $(r_{2})(r_{2})$, filling in the remainder of the table using arithmetic. Such surfaces will be constructed out of simple cobordisms build out of homotopies and skein operations. For $(r_{4})$ we have an obvious disk bounding a push-out along $\overline{\eta}_{4}$ obtained by perturbing $\region_{5}$.

\begin{figure}[h]\begin{overpic}[scale=.8]{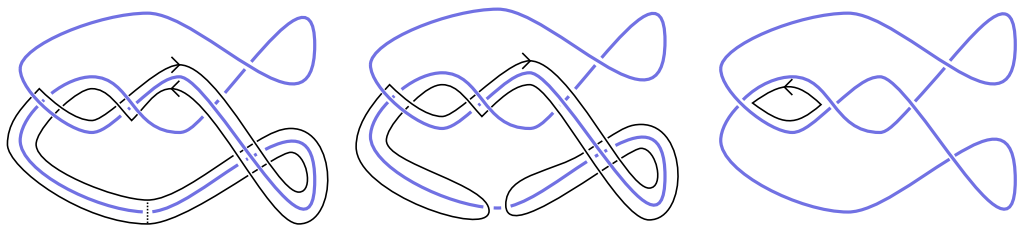}
\end{overpic}
\caption{An annulus bounding $(r_{1}) \cup (r_{2})$.}
\label{Fig:TrefoilR1R2SpanningSurface}
\end{figure}

In figure Figure \ref{Fig:TrefoilR1R2SpanningSurface} we construct a spanning surface for the union of the orbits $(r_{1}) \cup (r_{2})$. We begin by homotoping the union of orbits into the complement of $N_{\epsilon}$ as described in Section \ref{Sec:PushOutDefn}. The result -- associated to capping paths $\eta_{1}$ and $\overline{\eta}_{2}$ -- is shown on the left-most subfigure. To get from the left column of the figure to the center, we apply a skein cobordism along the dashed arc, resulting in a pair of pants cobordism. The resulting knot can be homotoped to the Reeb orbit $(r_{1}r_{2})$ as shown in the right hand side of the figure. So far our surface has avoided passing through any of the lines $\{ (x, y) = (x_{k}, y_{k})\} \subset \SurgL$. To complete our cobordism, we fill in the knot shown in the right-most subfigure using the obvious disk which is a perturbation of the disk $\region_{2}$. The union of our pair of pants with this disk provides us with an annular filling of $(r_{1}) \cup (r_{2})$ which intersects the link $\{ (x, y) = (x_{2}, y_{2})\}$ exactly once with positive sign. We conclude that
\begin{equation*}
\Igrading((r_{1})(r_{2})) = \iota_{2}.
\end{equation*}

A similar construction can be carried out to find an annular filling of $(r_{2}) \cup (r_{3})$: We start with a push-out corresponding to capping paths $\overline{\eta}_{2}$ and $\eta_{3}$, apply a skein cobordism giving us a pair of pants with boundary $(r_{2}) \cup (r_{3}) - (r_{2}r_{3})$, and then fill in $(r_{2}r_{3})$ with a perturbation of the disk $\disk_{3}$. We conclude that 
\begin{equation*}
\Igrading((r_{2})(r_{3})) = \iota_{3}.
\end{equation*}

Now we construct a spanning surface for $(r_{2}) \cup (r_{2})$. The construction is more complicated in this case: We construct two cobordisms from $(r_{2})$ from a positive and negative meridian of $\Lambda$ which can then be patched together to give us a surface with boundary $(r_{2}) \cup (r_{2})$.

\begin{figure}[h]\begin{overpic}[scale=.8]{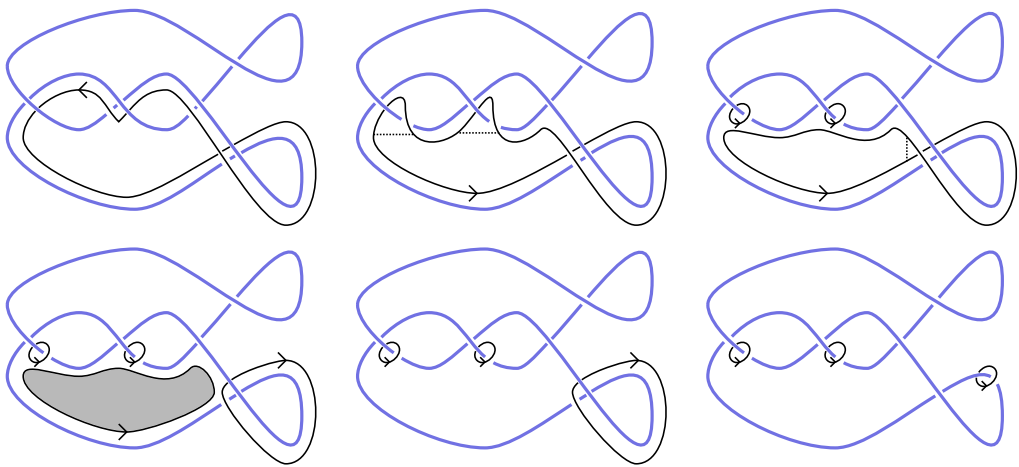}
\end{overpic}
\caption{A cobordism with boundary $(r_{2}) + \mu$.}
\label{Fig:TrefoilR2SpanningSurface}
\end{figure}

We break down the construction of one such cobordism whose boundary is $(r_{2}) + \mu$ into a sequence of elementary cobordisms as shown in Figure \ref{Fig:TrefoilR2SpanningSurface}.
\be
\item We start with a push-out of $(r_{2})$ using the capping path $\overline{\eta}_{2}$ as shown in the top-left subfigure.
\item Going from the top-left to top-center, we homotop our knot across the disks $\region_{2}$ and $\region_{3}$. Along the way we pick up two intersections with the lines associated to the points $(x_{2}, y_{2})$ and $(x_{3}, y_{3})$ with positive signs.
\item Going from the top-center to the top-right we apply skein cobordisms along the dashed arcs appearing in the top-center.
\item Going from the top-right to the bottom-left we apply another skein cobordism along the dashed arc appearing in the top-right yielding a $4$ component link.
\item Going from the bottom-left to the bottom-center we fill in one of the components of our link with a disk which is a perturbation of the disk $\region_{4}$. In doing so, we pick up a positive intersection with the line over the point $(x_{4}, y_{4})$.
\item Going from the bottom-center to the bottom-right we homotop one component of our knot over $-\region_{6}$ to a $-\mu$
\ee

Combining all of the above steps, we've constructed a homotopy from $(r_{2})$ to a collection of meridians. We can cancel a pair of them with a tube as shown in Figure \ref{Fig:MuTube}. The end result is a cobordism with boundary $(r_{2}) + \mu$ passing through the lines associated to the points $(x_{2}, y_{2}), (x_{3}, y_{3})$, and $(x_{4}, y_{4})$ once each with positive intersection number and passing through the line over $(x_{6}, y_{6})$ with negative intersection number.

\begin{figure}[h]\begin{overpic}[scale=.6]{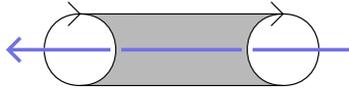}
	\end{overpic}
	\caption{A tube bounding $\mu - \mu$.}
	\label{Fig:MuTube}
\end{figure}

We can also construct a cobordism with boundary $(r_{2}) - \mu$ by flipping Figure \ref{Fig:TrefoilR2SpanningSurface} about a horizontal line, starting with a push-out of $\eta_{2}$. The resulting cobordism passes through the lines associated to the points $(x_{1}, y_{1}), (x_{2}, y_{2})$, and $(x_{3}, y_{3})$ once each with positive intersection number and passing through the line over $(x_{5}, y_{5})$ with negative intersection number.

We can connect the two cobordisms with another tube bounding $\mu - \mu$ to obtain a spanning surface for $(r_{2}) \cup (r_{2})$. By the above counts of intersections, we have
\begin{equation*}
\Igrading((r_{2})^{2}) = \iota_{1} + 2\iota_{2} + 2\iota_{3} + \iota_{4} - \iota_{5} - \iota_{6}.
\end{equation*}
Using our calculations of $\Igrading((r_{1})(r_{2})), \Igrading((r_{2})(r_{3}))$, and $\Igrading((r_{2})^{2})$, we can fill out the remainder of the above table by computing
\begin{equation*}
\begin{gathered}
\Igrading((r_{1})^{2}) = 2\Igrading((r_{1})(r_{2})) - \Igrading((r_{2})^{2}),\\
\Igrading((r_{3})^{2}) = 2\Igrading((r_{2})(r_{3})) - \Igrading((r_{2})^{2}),\\
\Igrading((r_{1})(r_{3})) = \Igrading((r_{1})(r_{2})) +  \Igrading((r_{2})(r_{3}))- \Igrading((r_{2})^{2}).
\end{gathered}
\end{equation*}

\subsubsection{Intersection numbers of curves positively asymptotic to $(r_{4})$}

Now suppose that we have a holomorphic curve $U$ positively asymptotic to $(r_{4})$ and negatively asymptotic to a collection of generators $\gamma^{-}$ from $C_{0, 0}$. Writing $\gamma^{-}$ as a monomial in $C_{0, 0}$, there are non-negative constants $C_{i, j}$ for which
\begin{equation*}
\gamma^{-} = (r_{1})^{2C_{1, 1}}(r_{2})^{2C_{2, 2}}(r_{3})^{2C_{3, 3}}((r_{1})(r_{2}))^{C_{1, 2}}((r_{1})(r_{3}))^{C_{1, 3}}((r_{2})(r_{3}))^{C_{2, 3}}.
\end{equation*}

We'll use the intersection grading to show that all of the $C_{i, j}$ must be zero so that $U$ cannot have any negative asymptotics. We can count the intersection of $U$ with the planes $\C_{k}$ as the coefficients of the $\iota_{k}$ in the expression $\Igrading((r_{4})) - \Igrading(\gamma^{-})$ as described in Equation \eqref{Eq:IGradingDelta}. Using the table above, we compute
\begin{equation*}
\begin{aligned}
\Igrading((r_{4})) - \Igrading(\gamma^{-}) &= (C_{1, 1} - C_{2, 2}  + C_{3, 3} + C_{1, 3})\iota_{1}\\
&+ (C_{1, 1} - 2C_{2, 2} + 2C_{3, 3} - C_{1, 2} + C_{1, 3})\iota_{2}\\
&+ (2C_{1, 1} - 2C_{2, 2} + C_{3, 3} + C_{1, 3} - C_{2, 3})\iota_{3}\\
&+ (C_{1, 1} - C_{2, 2} + C_{3, 3} + C_{1, 3})\iota_{4}\\
&+ (1 - C_{1, 1} + C_{2, 2} - C_{3, 3} - C_{1, 3})\iota_{5}\\
&+ (-C_{1, 1} + C_{2, 2} - C_{3, 3} - C_{1, 3})\iota_{6}
\end{aligned}
\end{equation*}
All of the $\iota_{k}$ coefficients above must be non-negative by intersection positivity.

As the $\iota_{4}$ and $\iota_{6}$ coefficients are the same with opposite sign, both must be zero so that 
\begin{equation*}
C_{2, 2} = C_{1, 1} + C_{3, 3} + C_{1, 3}. 
\end{equation*}
Therefore, we must have
\begin{equation*}
\begin{aligned}
\Igrading((r_{4})) - \Igrading(\gamma^{-}) &= (-C_{1, 1} - C_{1, 2} - C_{1, 3})\iota_{2}\\
&+ (- C_{3, 3} - C_{1, 3} - C_{2, 3})\iota_{3}\\
&+ \iota_{5}
\end{aligned}
\end{equation*}
implying that the remaining $C_{i, j}$ are all zero.

\subsubsection{Completion of the proof}

The above analysis implies that if $U$ is an index $1$ holomorphic curve contributing to $\partial_{CH}(r_{4})$ then it cannot have any negative asymptotics and must satisfy
\begin{equation}\label{Eq:R4IntersectionConstraints}
\C_{k}\cdot U = \begin{cases}
1 & k=5\\
0 & k\neq 5
\end{cases}
\end{equation}
Such a curve must be parameterized by $\C$ as per the definition of $\partial_{CH}$. To complete our proof we analyze to moduli space of finite energy curves
\begin{equation*}
\mathcal{M}_{4, 5} = \{ \C \xrightarrow{U} \R \times \SurgL\ : \text{$U$ asymptotic to $(r_{4})$, satisfying \eqref{Eq:R4IntersectionConstraints}} \}/\text{reparameterization}.
\end{equation*}
By the above analysis, $\partial_{CH}((r_{4})) = \#(\mathcal{M}_{4, 5})1$, counting points algebraically. This moduli space exactly describes the lowest levels $U^{c}_{\emptyset}$ of the height $3$ $\SFT$ buildings studied in Theorem \ref{Thm:PlaneBubbling}, which when applied to the disk $\region_{5}$ tell us that $\#(\mathcal{M}_{4, 5}) = \pm 1$. 

The proof of Theorem \ref{Thm:Trefoil} is then complete in the case of the $\tb = 1$ trefoil shown in Figure \ref{Fig:TrefoilImmersions}. By the classification torus knots in $\Rthree$ \cite{EtnyreHonda:Knots}, all other right-handed trefoils are stabilizations of this one -- contact $+1$ surgeries on these stabilized knots will be overtwisted and so will have $CH = 0$. The proof is now complete in the case that $\LambdaPlus$ consists of a single component. In the case that $\LambdaPlus = \cup_{i}^{n} \LambdaPlus_{i}$ has multiple components, we have -- as described in Section \ref{Sec:SurgeryCobordisms} -- a Liouville cobordism $(W, \lambda)$ whose convex end $(M^{+}, \xi^{+}) = (\partial^{+} W, \ker(\lambda)|_{\partial^{+}W})$ is given by contact $+1$ surgery on $\LambdaPlus_{1}$ and whose concave end $(M^{-}, \xi^{+}) = (\partial^{-} W, \ker(\lambda)|_{\partial^{-}W})$ is given by contact surgery on $\LambdaPlus$. If we index the components of $\LambdaPlus$ so that $\LambdaPlus_{1}$ is a right handed trefoil. Then $CH(M^{+}, \xi^{+}) = 0$, and so by Liouville functoriality $CH(M^{-}, \xi^{-}) = 0$ as well. The proof is now complete for all right-handed trefoils and all contact surgery surgery coefficients $\frac{1}{k}$ with $k > 0$.


\begin{thebibliography}{}

\bibitem[Al79]{Ahlfors}
L. Ahlfors, \textit{Complex Analysis: An Introduction to the Theory of Analytic Functions of One Complex Variable}, New York: McGraw-Hill, 1979.

\bibitem[Av11]{Avdek:ContactSurgery}
R. Avdek, \textit{Contact surgery and supporting open books}, Algebr. Geom. Topol. vol. 13, p.1613-1660, 2013.

\begin{comment}
\bibitem[Av12]{Avdek:Liouville}
R. Avdek, \textit{Liouville hypersurfaces and connect sum cobordisms}, arXiv:1204.3145, 2012.
\end{comment}

\bibitem[BH15]{BH:ContactDefinition}
E. Bao and K. Honda, \textit{Semi-global Kuranishi charts and the definition of contact homology}, arXiv:1512.00580, 2015.

\bibitem[BH18]{BH:Cylindrical}
E. Bao and K. Honda, \textit{Definition of cylindrical contact homology in dimension three}, J. Topology, vol. 11, p.1002-1053, 2018.

\begin{comment}
\bibitem[BM]{SmoothAtlas}
D. Bar-Natan, S. Morrison, et. al, \textit{The Rolfsen knot table}, \url{http://katlas.org/wiki/The_Rolfsen_Knot_Table}
\end{comment}

\bibitem[B02]{Bourgeois:Thesis}
F. Bourgeois, \textit{A Morse-Bott approach to contact homology}, PhD thesis, Stanford University, 2002.

\bibitem[B03]{Bourgeois:ContactIntro} 
F. Bourgeois, \textit{Introduction to contact homology}, lecture notes available at \url{https://www.imo.universite-paris-saclay.fr/~bourgeois/papers/Berder.pdf}, 2003.

\begin{comment}
\bibitem[BEE11]{BEE:Product}
F. Bourgeois, T. Ekholm, and Y. Eliashberg, \textit{Symplectic homology product via Legendrian surgery}, Proc. Natl. Acad. Sci., vol. 108, Number 20, p.8114–8121, 2011.
\end{comment}

\bibitem[BEE12]{BEE:LegendrianSurgery}
F. Bourgeois, T. Ekholm, and Y. Eliashberg, \textit{Effect of Legendrain Surgery}, Geom. Topol., vol. 16, p.301-389, 2012.

\bibitem[BEHW03]{SFTCompactness}
F. Bourgeois, Y. Eliashberg, H. Hofer, and K. Wysocki, \textit{Compactness results in symplectic field theory}, Geom. Topol., vol. 7, p.799-888, 2003.

\bibitem[BN10]{AlgebraicallyOvertwisted}
F. Bourgeois and K. Niederkr{\"u}ger, \textit{Towards a good definition of algebraically overtwisted}, Expo. Math., vol. 28, 85–100, 2010.

\begin{comment}
\bibitem[CG86]{CG:Cobordisms}
A. Casson, C. M. Gordon, \textit{Cobordism of classical knots}, from: "{\'A} la recherche de la topologie perdue" (editors L Guillou, A Marin), Progr. Math. 62, Birkh{\"a}user 181, 1986.

\bibitem[CS99]{StringTopology}
M. Chas and D. Sullivan \textit{String topology}, preprint (1999), math.GT/9911159

\bibitem[CS04]{StringDiagrams}
M. Chas, and D. Sullivan, \textit{Closed string operators in topology leading to Lie bialgebras and higher string algebra}, The legacy of Niels Henrik Abel, p.771–784, Springer, Berlin, 2004.
\end{comment}

\bibitem[C02]{Chekanov:LCH}
Y. Chekanov, \textit{Differential algebra of Legendrian links}, Invent. Math., vol. 150, p.441–483, 2002.

\begin{comment}
\bibitem[CN15]{KnotAtlas}
W. Chongchitmate and L. Ng, \emph{The Legendrian knot atlas},
\url{http://www.math.duke.edu/~ng/atlas/}, updated version, 2015.

\bibitem[Ci02a]{Cieliebak:Subcritical}
K. Cieliebak, \textit{Subcritical Stein manifolds are split}, arXiv:math/0204351v1, 2002.

\bibitem[Ci02b]{Cieliebak:SubcriticalSH}
K. Cieliebak, \textit{Handle attaching in symplectic homology and the Chord Conjecture}, J.
Eur. Math. Soc., p.115–142, 2002.
\end{comment}

\bibitem[CE12]{SteinToWeinstein} 
K. Cieliebak and Y. Eliashberg, \textit{From Stein to Weinstein and back}, vol. 59 of American
Mathematical Society Colloquium Publications. American Mathematical Society, Providence, RI,
2012.

\bibitem[CL07]{CL:SFTStringTop}
K. Cieliebak and J. Latschev, \textit{The role of string topology in symplectic field theory},  New perspectives and challenges in symplectic field theory, CRM Proc. Lecture Notes, Amer. Math. Soc., vol. 49, p. 113–146, 2009.

\bibitem[CGH11]{CGH:HFequalsECH}
V. Colin, P. Ghiggini and K. Honda, \textit{Equivalence of Heegaard Floer homology and embedded contact homology via open book decompositions}, PNAS,
vol. 108, p.8100-8105, 2011

\bibitem[CGHH10]{CGHH:Sutures}
V. Colin, P. Ghiggini, K. Honda, and M. Hutchings, \textit{Sutures and contact homology I}, Geom. Topol., vol. 15, p.1749-1842, 2011.

\bibitem[CET19]{PlusOneFilling}
J. Conway, J. Etnyre, and B. Tosun, \textit{Symplectic fillings, contact surgeries, and Lagrangian disks}, IMRN, 2019.

\bibitem[DG04]{DG:Surgery}
F. Ding and H. Geiges, \textit{A Legendrian surgery presentation of contact 3-manifolds}, Math. Proc. Cambridge Philos. Soc.,
vol. 136, p.583-598, 2004.

\begin{comment}
\bibitem[DG08]{DG:HandleMoves}
F. Ding and H. Geiges, \textit{Handle moves in contact surgery diagrams}, J. Topology, vol. 2, Number 1, p.105-122, 2008.
\end{comment}

\bibitem[Ek08]{Ekholm:Z2RSFT}
T. Ekholm, \textit{Rational symplectic field theory over $\Z_{2}$ for exact Lagrangian cobordisms}, J. Eur. Math. Soc. (JEMS) 10, p.641–704, 2008.

\bibitem[Ek19]{Ekholm:SurgeryCurves}
T. Ekholm, \textit{Holomorphic curves for Legendrian surgery}, arXiv:1906.07228, 2019.

\bibitem[EES05]{EES:LegendriansInR2nPlus1}
T. Ekholm, J. Etnyre, and M. Sullivan, \textit{The contact homology of Legendrian submanifolds of $\R^{2n+1}$}, J. Diff. Geom., vol. 71, p.177-305, 2005.

\bibitem[EkN15]{EkholmNg}
T. Ekholm and L. Ng, \textit{Legendrian contact homology in the boundary of a subcritical Weinstein 4-manifold}, J. Diff. Geom., vol. 101, p.67-157, 2015.

\begin{comment}
\bibitem[E98]{Eliashberg:LCH}
Y. Eliashberg, \textit{Invariants in contact topology}, Proceedings of the International Congress of Mathematicians, vol. II (Berlin), number Extra vol. II, p.327–338 (electronic), 1998.
\end{comment}

\bibitem[EGH00]{EGH:SFTIntro}
Y. Eliashberg, A Givental, H. Hofer, \textit{Introduction to symplectic field theory}, Geom. Funct. Anal., Special vol., Part II, p.560-673, 2000.

\bibitem[El89]{Eliash:OTClassification}
Y. Eliashberg, \textit{Classification of overtwisted contact structures on 3-manifolds}, Invent. Math., vol. 98, p.623-637, 1989.

\bibitem[El91]{Eliash:Filling}
Y. Eliashberg, \textit{On symplectic manifolds with some contact properties}, J. Diff. Geom., vol. 33, p.233–238, 1991.

\bibitem[Et05]{Etnyre:KnotNotes}
J. Etnyre, \textit{Legendrian and Transversal Knots}, Handbook of Knot Theory Elsevier B. V., p.105-185, 2005.

\begin{comment}
\bibitem[Et08]{Etnyre:ContactSurgery}
J. Etnyre, \textit{On contact surgery}, Proc. Amer. Math. Soc., vol. 136, p.3355-3362, 2008.
\end{comment}

\bibitem[EH01]{EtnyreHonda:Knots}
J. Etnyre and K. Honda, \textit{Knots and Contact Geometry I: Torus Knots and the Figure Eight Knot}, J. Symplectic Geom., vol. 1, p.63-120, 2001.

\bibitem[EtN18]{EtnyreNg:LCHSurvey}
J. Etnyre and L. Ng, \textit{Legendrian contact homology in $\R^{3}$}, arXiv:1811.10966v3, 2018.

\begin{comment}
\bibitem[ENV13]{ENV:Twists}
J. Etnyre, L. Ng, and V. Vertesi, \textit{Legendrian and transverse twist knots}, J. Eur. Math. Soc. (JEMS), vol. 15, no. 3, p.969–995, 2013.
\end{comment}

\bibitem[EO08]{EO:OBInvariants}
J. Etnyre and B. Ozbagci, \textit{Invariants of contact structures from open books}, Trans. Amer. Math. Soc., vol. 360, p.3133–3151, 2008.

\begin{comment}
\bibitem[EV18]{EV:Satellites}
J. Etnyre and V. Vertesi, \textit{Legendrian satellites}, IMRN, Issue 22, p.7241–7304, 2018.
\end{comment}

\bibitem[F11]{Fabert:Pants}
O. Fabert, \textit{Obstruction bundles over moduli spaces with boundary and the action filtration in symplectic field theory}, Mathematische Zeitschrift, vol. 269, p.325–372, 2011.

\bibitem[FH13]{FH:Anosov}
P. Foulon and B. Hasselblatt, \textit{Contact Anosov flows on hyperbolic 3–manifolds}, Geom. Topol., vol. 17, p.1225–1252, 2013.

\begin{comment}
\bibitem[F06]{Fukaya:LagrangianSubmanifolds}
K. Fukaya, \textit{Applications of Floer homology of Lagrangian submanifolds to symplectic topology}, Morse Theoretic Methods in Nonlinear Analysis and in Symplectic Topology (P.
Biran, O. Cornea, eds.), Nato Science Series, vol. 217, p.231-276, 2006.
\end{comment}

\bibitem[GZ13]{GZ:FourBall}
H. Geiges and K. Zehmisch, \textit{How to recognize a 4-ball when you see one}, M{\"u}nster J. Math., vol. 6, p.525–554, 2013.

\begin{comment}
\bibitem[G02]{Giroux:ContactOB}
E. Giroux, \textit{G\'{e}om\'{e}trie de contact: de la dimension trois vers les dimensions sup\'{e}rieures}, Proceedings of the International Congress of Mathematicians, vol. II, Higher Ed. Press, Beijing, p.405-414, 2002.
\end{comment}

\bibitem[Go98]{Gompf:Handlebodies}
R. Gompf, \textit{Handlebody construction of Stein surfaces}, Ann. of Math., vol. 148,
p.619-693, 1998.

\begin{comment}
\bibitem[GS99]{GS:KirbyCalculus}
R. Gopf and A. Stipsicz, \textit{4-manifolds and Kirby Calculus}, Graduate Studies in Mathematics 20, Amer. Math. Society, Providence, RI, 1999.
\end{comment}

\bibitem[Gr85]{Gromov:JCurves}
M. Gromov, \textit{Pseudoholomorphic curves in symplectic manifolds}, Invent. Math, vol. 82, p.307–347, 1985.

\bibitem[Ha02]{Hatcher:AlgebraicTopology}
A. Hatcher, \textit{Algebraic topology}, Cambridge University Press, 2002.

\bibitem[Hi03]{Hind:Filling}
R. Hind, \textit{Stein fillings of lens spaces}, Commun. Contemp. Math., vol. 5, p.967–982, 2003.

\bibitem[Hof93]{Hofer:OTWeinstein}
H. Hofer, \textit{Pseudoholomorphic curves in symplectizations with applications to the Weinstein conjecture in dimension three}, Inv. Math., vol. 114, p.515-563, 1993.

\begin{comment}
\bibitem[H00]{Honda:Tight1}
K. Honda, \textit{On the classification of tight contact structures I}, Geom. Topol., vol. 4, Number 1, p.309-368, 2000.
\end{comment}

\bibitem[Hon00]{Honda:Tight2}
K. Honda, \textit{On the classification of tight contact structures II}, J. Diff. Geom., vol. 55, p.83-143, 2000.

\bibitem[Hon02]{Honda:OTSurgery}
K. Honda, \textit{Gluing tight contact structures}, Duke Math. J.
vol. 115, p.435-478, 2002.

\bibitem[HKM09]{HKM:ContactClass}
K. Honda, W. Kazez, and G. Mati'{c}, \textit{On the contact class in Heegaard Floer homology}, J. Diff. Geom., vol. 83, p.289-311, 2009.

\bibitem[Hua13]{Huang:OTClassification}
Y. Huang, \textit{A proof of the classification theorem of overtwisted contact structures via convex surface theory}, J. Symplectic Geom., vol. 11, p.563-601, 2013.

\bibitem[Hut14]{Hutchings:ECHNotes}
M. Hutchings, \textit{Lecture Notes on Embedded Contact Homology}, in Contact and Symplectic Topology, Bolyai Society Mathematical Studies, vol. 26, Springer, p.389-484, 2014.

\bibitem[HT07]{HT:GluingI}
M. Hutchings and C. Taubes, \textit{Gluing pseudoholomorphic curves along branched covered cylinders I}, J. Symplectic Geom., vol. 5, p.43–137, 2007.

\begin{comment}
\bibitem[K87]{Kauffman}
L. H. Kauffman, \textit{State models and the Jones polynomial}, Topology, vol. 26, Number 3, p.395–407, 1987.

\bibitem[Ko00]{Khovanov}
M. Khovanov, \textit{A categorification of the Jones polynomial}, Duke Math. J., vol. 101, Number 3, p.359-426, 2000.

\bibitem[KM93]{KM:MilnorConj}
P. Kronheimer and T. Mrowka. \textit{Gauge theory for embedded surfaces I}. Topology, 32, p.773–826, 1993.
\end{comment}

\bibitem[KLT10]{KLT:HFSW}
C. Kutluhan, Y.-J. Lee, and C. Taubes, \textit{HF=HM I: Heegaard Floer homology and Seiberg-Witten Floer homology}, arXiv:1007.1979, 2010.

\begin{comment}
\bibitem[L98]{Lisca:Nonfillable}
P. Lisca, \textit{Symplectic fillings and positive scalar curvature}, Geom. Topol. 2, p.103-116, 1998.

\bibitem[L08]{Lisca:Filling}
P. Lisca, On symplectic fillings of lens spaces, Trans. Amer. Math. Soc. 360, no. 2, p.765–799, 2008.
\end{comment}

\bibitem[LS04]{LS:TightI}
P. Lisca and A. I. Stipsicz, \textit{Ozsváth–Szábo invariants and tight contact three-manifolds I}, Geom. Topol., vol. 8, p.925-945, 2004.

\begin{comment}
\bibitem[LS06]{LS:ContactClassNotes}
P. Lisca and A. I. Stipsicz, \textit{Notes on the Contact Ozsváth-Szabó Invariants}, Pacific J. Math, vol. 228, Number 2, p.277–295, 2006.


\bibitem[L02]{Liu:Moduli}
C.C.M. Liu, \textit{Moduli of J-holomorphic curves with Lagrangian boundary conditions and open Gromov-Witten invariants for an $\Circle$-equivariant pair}, Ph.D. thesis (Harvard University), arXiv:math.SG/0210257, 2002.
\end{comment}

\bibitem[M90]{McDuff:RationalRuled}
D. McDuff, \textit{The structure of rational and ruled symplectic 4-manifolds}, J. Amer. Math. Soc., vol. 3, p.679–712, 1990.

\bibitem[M91]{McDuff:Filling}
D. McDuff, \textit{Symplectic manifolds with contact type boundaries}, Invent. Math., vol. 103, p.651–671, 1991.

\bibitem[MS99]{MS:SymplecticIntro}
D. McDuff and D. Salamon, \textit{Introduction to symplectic topology}, Second edition. Oxford Mathematical Monographs. The Clarendon Press, Oxford University Press, New York, 1998.

\bibitem[MS04]{MS:Curves}
D. McDuff and D. Salamon, \textit{J-holomorphic curves and symplectic topology}, American Mathematical Society, 2004.

\bibitem[MZ21]{MZ:RSFT}
A. Moreno and Z. Zhou, \textit{A landscape of contact manifolds via rational SFT}, 	arXiv:2012.04182, 2021.

\begin{comment}
\bibitem[N01]{Ng:Satellites}
L. Ng, The Legendrian satellite construction, arXiv:math/0112105, 2001.
\end{comment}

\bibitem[N03]{Ng:ComputableInvariants}
L. Ng, \textit{Computable Legendrian invariants}, Topology, vol. 42, p.55–82, 2003.

\bibitem[N10]{Ng:RSFT}
L. Ng, \textit{Rational symplectic field theory for Legendrian knots}, Invent. math., vol. 182, Issue 3, p.451–512, 2010.

\begin{comment}
\bibitem[NR]{NR:Helix}
L. Ng and D. Rutherford, \textit{Satellites of Legendrian knots and representations of the Chekanov–Eliashberg algebra}, Algebr. Geom. Topol., vol. 13, Number 5, p.3047-3097, 2013.

\bibitem[OO05]{OO:Filling} 
H. Ohta and K. Ono, \textit{Simple singularities and symplectic fillings}, J. Differential
Geom. 69, no. 1, p.1–42, 2005.
\end{comment}

\bibitem[O05]{Ozbagci:Stabilization}
B. Ozbagci, \textit{A note on contact surgery diagrams}, Int. J. Math., vol. 16, p.87-99, 2005.

\begin{comment}
\bibitem[O09]{Ozbagci:CotangentBundle}
B. Ozbagci, \textit{Stein and Weinstein structures on disk cotangent bundles of surfaces}, Archiv der Mathematik, vol. 113, Issue 6, p.661–670, 2019.
\end{comment}

\bibitem[OzbSt04]{OS:SurgeryBook}
B. Ozbagci and A. I. Stipsicz, \textit{Surgery on contact 3-manifolds and Stein surfaces}, Bolyai Society Mathematical Studies, Springer-Verlag, vol. 13, 2004.

\bibitem[OzvSz05]{OS:ContactClass}
P. Ozsv{\'a}th and Z. Szab{\'o}, \textit{Heegaard Floer homology and contact structures}, Duke Math. J., vol. 129, p.39–61, 2005.

\bibitem[P19]{Pardon:Contact}
J. Pardon, \textit{Contact homology and virtual fundamental cycles}, J. Amer. Math. Soc. vol. 32, p.825-919, 2019.

\begin{comment}
\bibitem[Ra10]{Ras:MilnorConj}
J. Rasmussen, \textit{Khovanov homology and the slice genus}, Invent. math., vol. 182, p.419–447, 2010.

\bibitem[Ro90]{Rolfsen}
D. Rolfsen, \textit{Knots and links} (Corrected reprint of the 1976 original), Mathematics Lecture Series, vol. 7, Publish or Perish Inc., Houston, TX, 1990.
\end{comment}

\bibitem[RS93]{RS:Index}
J. Robbin and D. Salamon, \textit{The Maslov index for paths}, Topology, vol. 32, p.827–844, 1993.

\bibitem[Ro19]{Rooney:ECH}
J. Rooney, \textit{Cobordism maps in embedded contact homology}, arXiv:1912.01048, 2019.

\bibitem[Ru97]{Rudolph}
L. Rudolph, \textit{The slice genus and the Thurston–Bennequin invariant of a knot}, Proc. Amer. Math. Soc., vol. 125, p.3049–3050, 1997.

\bibitem[S95]{Schwarz:Thesis}
M. Schwarz, \textit{Cohomology operations from $\Circle$ cobordisms in Floer homology}, PhD thesis, ETH Zurich, 1995.

\bibitem[S07]{BiasedSH}
P. Seidel, \textit{A biased view of symplectic cohomology}, arXiv:math/0704.2055, 2007.

\begin{comment}
\bibitem[Sm62]{Smale:HCobordism}
S. Smale, \textit{On the structure of manifolds}, Amer. Journ. Math., Number 3, p.387-399, 1962.

\bibitem[Su05]{Sullivan:Coproduct}
D. Sullivan, \textit{Sigma models and string topology in Graphs And Patterns}, in Mathematics And Theoretical
Physics: Proceedings Of The Stony Brook Conference On Graphs And Patterns In Mathematics And Theoretical Physics, Dedicated to Dennis Sullivan’s 60th Birthday Proceedings of Symposia in Pure Mathematics,
American Mathematical Society, 2005.

\bibitem[T97]{Traynor:Helix}
L. Traynor, \emph{Legendrian circular helix links}, Mathematical Proceedings of the Cambridge Philosophical Society 122, p.301-314, 1997.

\bibitem[W15]{Wand:LegendrianSurgery}
A. Wand, \textit{Tightness is preserved by Legendrian surgery}, Ann. of Math., vol. 128, Issue 2, p.723-738, 2015.
\end{comment}

\bibitem[Wei91]{Weinstein:Handles}
A. Weinstein, \textit{Contact surgery and symplectic handlebodies}, Hokkaido Math. J., vol. 20, p.241-251, 1991.

\bibitem[Wen10]{Wendl:Foliations}
C. Wendl, \textit{Strongly fillable contact manifolds and J-holomorphic foliations}, Duke Math. J., vol. 151, p.337-384, 2010.

\bibitem[Wen13]{Wendl:NonExact}
C. Wendl, \textit{Non-exact symplectic cobordisms between contact 3-manifolds}. J. Diff. Geom., vol. 95, p.121–182, 2013.

\bibitem[Wen15]{Wendl:Signs}
C. Wendl, \textit{Signs (or how to annoy a symplectic topologist)}, blog post available at \url{https://symplecticfieldtheorist.wordpress.com/2015/08/23/signs-or-how-to-annoy-a-symplectic-topologist/}, 2015.

\bibitem[Wen16]{Wendl:SFTNotes}
C. Wendl, \textit{Lectures on Symplectic Field Theory}, arXiv:1612.01009, 2016.

\bibitem[Y06]{Yau:VanishingCH}
M. L. Yau, \emph{Vanishing of the contact homology of overtwisted contact 3-
manifolds, with an appendix by Y. Eliashberg}, Bull. Inst. Math. Acad. Sin., vol. 1, p.211–229, 2006.

\end{thebibliography}
\end{document}